\documentclass[10pt,reqno]{article}
\pdfoutput=1
\usepackage{amsmath}
\usepackage{amsfonts,amsthm,amssymb}
\usepackage{mathabx}
\usepackage{bm}
\usepackage{amscd}
\usepackage[latin2]{inputenc}
\usepackage{t1enc}
\usepackage[mathscr]{eucal}
\usepackage{indentfirst}
\usepackage{graphicx}
\usepackage{graphics}
\usepackage{pict2e}
\usepackage{epic}
\numberwithin{equation}{section}
\usepackage[margin=1.7cm]{geometry}
\usepackage{epstopdf} 
\usepackage[colorlinks,linkcolor=blue]{hyperref}
\usepackage[capitalise,noabbrev]{cleveref}
\usepackage{todonotes}
\makeatletter
\def\thanks#1{\protected@xdef\@thanks{\@thanks
		\protect\footnotetext{#1}}}
\makeatother
\crefformat{equation}{(#2#1#3)}
\crefmultiformat{equation}{(#2#1#3)}{ and~(#2#1#3)}{, (#2#1#3)}{ and~(#2#1#3)}
\crefrangeformat{equation}{(#3#1#4) to~(#5#2#6)}
\usepackage[normalem]{ulem}

\allowdisplaybreaks

\newcommand{\ww}{\check{w}}
\newcommand{\dl}{\tilde{\mathrm{div}}}
\newcommand{\di}{\mathrm{div}}
\newcommand{\nlt}{\mathcal{NL}}
\newcommand{\lt}{\mathcal{L}}
\newcommand{\ma}{m^{\frac{1}{4}}M}
\newcommand{\de}{{R}}
\newcommand{\lde}{{\phi}}
\newcommand{\nno}{{(1)}}
\newcommand{\nzno}{{(2)}}
\newcommand{\nxn}{{(3)}}
\newcommand{\nyn}{{(4)}}
\newcommand{\npn}{{(5)}}
\newcommand{\nwii}{{(6)}}
\newcommand{\noiii}{{(7)}}
\newcommand{\nwi}{{(8)}}
\newcommand{\nxnh}{{(9)}}
\newcommand{\pnh}{{(10)}}
\newcommand{\nwh}{{(11)}}

\theoremstyle{plain}
\newtheorem{Thm}{Theorem}[section]
\newtheorem*{Thm*}{Theorem}
\newtheorem{Lem}[Thm]{Lemma}

\newtheorem{Prop}[Thm]{Proposition}

\theoremstyle{definition}
\newtheorem{Def}[Thm]{Definition}

\newtheorem{Rem}[Thm]{Remark}
\newtheorem{?}[Thm]{Problem}

\newcommand{\Doo}{\mathbf{D}_{00}}

\newcommand{\p}{\partial}
\newcommand{\R}{\mathbb{R}}

\newcommand{\na}{\nabla}
\newcommand{\tna}{\tilde{\nabla}}
\newcommand{\eps}{\varepsilon}
\newcommand{\T}{\mathbb{T}}

\newcommand{\norm}[1]{\left\| #1 \right\|}

\newcommand{\Torus}{\mathbb{T}}

\newcommand{\dv}{\text{div}}

\newcommand{\ra}{\rangle}
\newcommand{\la}{\langle}

\newcommand{\abs}[1]{\left\lvert#1\right\rvert}

\begin{document}
	
	\begin{titlepage}
		\title{Nonlinear stability threshold for 3D compressible Couette flow}
	\author{ Rui Li \thanks{\footnotesize Department of Mathematics, The Chinese University of Hong Kong, Shatin, Hong Kong, China \href{mailto:ruili001@cuhk.edu.hk}{\texttt{ruili001@cuhk.edu.hk}}} 
     \and Fei Wang\thanks{School of Mathematical Sciences, CMA-Shanghai, Shanghai Jiao Tong University, 
		 Shanghai, China \href{mailto:fwang256@sjtu.edu.cn}{\texttt{fwang256@sjtu.edu.cn}}} 
    \and Lingda Xu\thanks{Hetao Institute of Mathematics and Interdisciplinary Sciences, Guangdong Province, China \href{mailto:xuldmath@gmail.com}{\texttt{xuldmath@gmail.com}}} 
    \and Zeren Zhang\thanks{School of Mathematical Sciences, Shanghai Jiao Tong University, Shanghai, China \href{mailto:zhangzr0018@sjtu.edu.cn}{\texttt{zhangzr0018@sjtu.edu.cn}}}}  
    \date{}
	\end{titlepage}
	
	\maketitle
\begin{abstract}
We establish the nonlinear stability threshold $O(\nu^{3/2})$ for the three-dimensional Couette flow governed by the compressible Navier--Stokes equations. 
While stability thresholds are well understood in two dimensions for both compressible and incompressible flows, and in three dimensions for incompressible flows, the three-dimensional compressible case remains open due to additional structural features, strong mode interactions, and wave coupling. The proof is based on a refined frequency-space approach. For zero modes, we improve upon two-dimensional methods by clearly separating and precisely estimating the main contributions from diffusion waves, acoustic waves, and the lift-up mechanism, leading to a systematic way to handle their nonlinear coupling. For the non-zero modes, we introduce new multiplier estimates and a decomposition based on the structure of the compressible system, which allows us to track the interaction between dissipation and acoustic effects.
\end{abstract}


\tableofcontents
\section{Introduction}
		
    In this paper, we consider the 3D non-dimensional isentropic compressible Navier-Stokes (NS) equations in $\mathbb{T}\times\mathbb{R}\times\mathbb{T}$:
	\begin{equation}\label{ns-ori}
		\left\{\begin{aligned}
			&\partial_t \tilde{\rho} +\di(\rho \tilde{v})=0,\\
			&\partial_t (\tilde{\rho} \tilde{v})+\di(\tilde{\rho} \tilde{v}\otimes \tilde{v})+\frac{1}{M}\nabla P(\tilde{\rho})=\nu\Delta \tilde{v}+(\nu+\nu')\nabla\di \tilde{v},
		\end{aligned}\right.
	\end{equation}
	where $\tilde{\rho}$ is the density of the fluid, $\tilde{v}$ is the velocity,  $P(\tilde{\rho})$ is the pressure which is a smooth function in a neighborhood of $1$, $M>0$ is the Mach number, and $\nu>0,\nu+\nu'\geq0$ are the shear and bulk viscosity coefficients, respectively. For simplify, we assume that $M=1$, $P'(1)=1$, and $\nu\approx(\nu+\nu')$ in the present paper.
	We introduce $({\lde},v)$ to be the perturbation of $(\tilde{\rho},\tilde{v})$ around compressible Couette flow $(1,(y,0,0)^T)$, i.e.,
	\begin{equation*}
		{\lde}=\tilde{\rho}-1,\quad v=\tilde{v}-(y,0,0)^T.
	\end{equation*}
    The equations of $({\lde},v)$ read
	\begin{equation}\label{perturbation-1}
		\left\{\begin{aligned}
			&\partial_t {\lde}+y\partial_x{\lde}+\textrm{div}\ v=F_1,\\
		&\partial_t v+y\partial_xv+(v^2,0,0)^T+\nabla {\lde}-\nu\Delta v-(\nu+\nu')\nabla \textrm{div}\ v=F_2,\\[1.5mm]
        &({\lde}(x,0),v(x,0))=({\lde}_{in}(x),v_{in}(x)),
		\end{aligned}\right.
	\end{equation}
	where
	\begin{equation*}
		\begin{aligned}
			&F_1=-v\cdot \nabla {\lde}-{\lde}\textrm{div}\ v,\\
			&F_2=-v\cdot \nabla v-\left(\frac{P'(1+{\lde})}{1+{\lde}}-1\right)\nabla {\lde}-\frac{{\lde}}{{\lde}+1}(\nu\Delta v+(\nu+\nu')\nabla \textrm{div}\ v).
		\end{aligned}
	\end{equation*}
    \subsection*{Literature review}
Since Reynolds' pioneering experiments in 1883 \cite{Reynolds}, the stability of shear flows in fluid equations has been a central topic in \textit{hydrodynamic stability theory} \cite{D-R, Rayleigh}, with planar Couette flow as the simplest model.
The stability of planar Couette flow has been investigated since the seminal work of Rayleigh \cite{Rayleigh} and Kelvin \cite{kelvin}. For an incompressible fluid, the linear analysis of Couette flow was already studied by Kelvin \cite{kelvin} in 1887. Other classical results have been obtained via an eigenvalue (or normal mode) analysis in many different cases; however, the classical stability analysis in general does not agree with the numerical and physical observations \cite{Bedrossian-Germain-Masmoudi-2019,D-R,Rayleigh,Romanov}.

\noindent\textbf{The Sommerfeld Paradox and Non-Normality}

A significant discrepancy exists between theory and observation. Romanov \cite{Romanov} proved that Couette flow is spectrally stable for all Reynolds numbers, a result that appears to contradict both numerical simulations and physical experiments \cite{Chapman, D-R, GG, OK, SH, Trefethen, Yaglom}, which indicate that shear flows become unstable and undergo transition to turbulence at sufficiently high Reynolds numbers. This discrepancy is known as the Sommerfeld paradox. 

To resolve this, Trefethen et al. \cite{Trefethen} observed that a common feature in these problems is the \textit{non-normality} of the operators involved. In particular, this implies the possibility of large \textit{transient growths} (which are not captured via a pure eigenvalue analysis) that can take the dynamics out of the linear regime before the stability mechanisms take over. The Couette flow is the simpler flow where these phenomena are present, therefore the stability analysis of this particular case is the prototypical example to understand some of the mechanisms involved in the dynamics.

\noindent\textbf{The Transition Threshold Problem}

To elucidate the mechanisms underlying the transition to turbulence, Trefethen et al. \cite{Trefethen} first formulated the transition threshold problem: namely, to quantify the magnitude of perturbations required to trigger instability and to determine their scaling with the Reynolds number. More recently, Bedrossian, Germain, and Masmoudi \cite{Bedrossian-Germain-Masmoudi-2017, Bedrossian-Germain-Masmoudi-2019} provided a rigorous mathematical framework for this problem, stated as follows:

\begin{quote}
\textit{Given a norm $\|\cdot\|_{X}$, find a $\beta=\beta(X)$ so that}
\begin{align*}
    &\|v_{in}\|_{X} \leq Re^{-\beta} \to \textrm{stability},\\
    &\|v_{in}\|_{X} \geq Re^{-\beta}\to \textrm{instability}.
\end{align*}
\end{quote}
The exponent $\beta$ is referred to as the transition threshold.

\noindent\textbf{Results for Incompressible Flow near Couette}

For an incompressible and homogeneous fluid in the Euler regime, Arnold \cite{Arnold} obtained an elegant stability result for a particular class of shear flows. However, some relevant flows, such as the Couette, do not belong to this class. A breakthrough in the understanding of the nonlinear stability properties for the planar Couette flow was achieved by Bedrossian and Masmoudi \cite{Bedrossian-Germain-Masmoudi-2015}. In particular, in the domain $\mathbb{T} \times \mathbb{R}$ they proved the asymptotic stability for the vorticity in a high-regularity space (Gevrey class $2-$) which implies the \textit{inviscid damping} for the velocity field. Namely, the vorticity is \textit{mixed} by the background flow and the velocity field strongly converges in $L^2$ to a shear flow \textit{close} to Couette flow with polynomial rate of convergence. The nonlinear inviscid damping of stable monotone shear flows in $\mathbb{T}\times [0,1]$ was independently established by Ionescu and Jia \cite{IJ,IJ1}, and Masmoudi and Zhao \cite{MZ1}. More recently, Chen, Wei, Zhang, and Zhang \cite{CWZZ} proved the nonlinear inviscid damping of Couette flow for the inhomogeneous incompressible Euler equations in $\mathbb{T}\times\mathbb{R}$, while Zhao \cite{Z25} extended the result to a class of monotone shear flows with non-constant background density in $\mathbb{T}\times[0,1]$.
 
When viscosity is present, more stability results are available. The stability mechanism present at the inviscid level can also combine with the dissipation and one observes an \textit{enhanced dissipation} of non-zero mode of the perturbations around shear flows.
This is possible as the advection causes an energy cascade towards small spatial scales where dissipation takes over. A substantial body of work in applied mathematics and physics has been devoted to obtaining the threshold $\beta$ (see, e.g., \cite{BT, Chapman, DBL, LK, LHS, RSBH, Waleffe, Yaglom}). Over the past decade, rigorous mathematical results have appeared in rapid succession:
\begin{itemize}
    \item \textbf{2D Incompressible in an infinite domain ($\mathbb{T}\times\mathbb{R}$):} Bedrossian, Masmoudi, and Vicol \cite{Bedrossian-Masmoudi-Vicol-2016} proved that $\beta=0$ when the perturbation space $X$ is of Gevrey class $2-$. If $X$ is taken to be a Sobolev space, Bedrossian, Wang and Vicol \cite{BWV} have established the upper bound $\beta\leq \frac{1}{2}$. Then
    Masmoudi and Zhao \cite{MZ}, and Wei and Zhang \cite{WZ-2023} have established the upper bound $\beta\leq \frac{1}{3}$.
    \item \textbf{2D Incompressible in a Channel ($\mathbb{T}\times [-1, 1]$):} Chen, Li, Wei, and Zhang \cite{CLWZ} derived $\beta\le \frac{1}{2}$ for the two-dimensional problem under the no-slip boundary condition. Bedrossian, He, Iyer, Li, and Wang \cite{Bedrossian-He- Iyer-Wang-2025a} extended this result to background shear flows near Couette. Recently, Bedrossian, He, Iyer, and Wang \cite{Bedrossian-He- Iyer-Wang-2025c, Bedrossian-He- Iyer-Wang-2024, Bedrossian-He- Iyer-Wang-2025b} established uniform inviscid damping (i.e., $\beta = 0$) in Gevrey class $2-$ under Navier boundary conditions. 
    \item \textbf{3D Incompressible in domains with or without boundary:} In three dimensions, the picture changes significantly due to the \textit{lift-up effect}, a linear mechanism that induces transient growth. Consequently, the transition threshold problem becomes substantially more difficult. In $\mathbb{T}\times\mathbb{R}\times \mathbb{T}$, Bedrossian, Germain, and Masmoudi \cite{Bedrossian-Germain-Masmoudi-2020, Bedrossian-Germain-Masmoudi-2022, Bedrossian-Germain-Masmoudi-2017, Bedrossian-Germain-Masmoudi-2019} obtained $\beta\le 1$ when $X$ is a Gevrey space and $\beta\le \frac{3}{2}$ when $X$ is a Sobolev space. Wei and Zhang \cite{WZ-cpam} later improved the Sobolev result to $\beta\le 1$. Recently, Chen, Wei, and Zhang \cite{CWZ-MAMS} obtained $\beta\le 1$ in the channel  $\mathbb{T}\times [-1, 1]\times\mathbb{T}$  for Sobolev perturbations under the no-slip boundary condition.
\end{itemize}

\noindent\textbf{Stability of Compressible Flow near Couette}

Compared with the incompressible stability results, the literature is significantly less developed for compressible case. The extension of the standard stability analysis to the compressible case has been already considered starting from the 40s' \cite{CRS,DEH,FI,Glatzel1,HSH,MDA}.

\begin{itemize}
    \item \textbf{Physical and Numerical Context:} The linearization around the Couette flow in the 2D isentropic compressible Euler dynamics was considered in the physics literature. In particular, the 2D inviscid problem (with an additional Coriolis forcing term) has been considered as a first model to understand the formation of \textit{spiral arms} in a rotating disk galaxy by Goldreich and Lynden-Bell \cite{GLB1,GLB2}. In \cite{GLB2} they directly consider the linearized initial value problem and they derive a second order ODE satisfied by the density in the Fourier space. From this equation, appealing to some formal approximation, they deduce an instability phenomenon that appears specifically due to the compressibility of the flow. More precisely, they obtain that $|\rho(t)| \sim O(t^{1/2})$. The problem (without the Coriolis force) was then studied also by Chagelishvili et al. in \cite{CRS,CRS1} where, with analogous computations, it is observed that $|\rho(t)| + |v(t)| \sim O(t^{1/2})$.
    For the viscous compressible planar Couette flow, Glatzel in 1988 \cite{Glatzel,Glatzel1} investigated linear stability properties via a normal mode analysis. Hanifi et al. in \cite{HSH} have numerically investigated a transient growth mechanism in the \textit{non-isothermal} case, showing that the maximum transient growth scales as $O(\nu^{-2})$ and increases with increasing Mach number. Then, Farrell and Ioannou in \cite{FI} considered the linear problem and showed a rapid transient energy growth, that at large Mach numbers greatly exceeds the expected one in the incompressible case, which is then damped due to the effect of viscosity.
    
    \item \textbf{Mathematical Results:} On the mathematical side, Kagei \cite{Kagei} established asymptotic stability of Couette flow for small Reynolds numbers, while Li and Zhang \cite{LZ1} later treated the case with slip boundary conditions.
    Chen and Ju \cite{CJ-1,CJ-2} consider the case of non-isentropic  flow.
    Regarding the transition threshold problem, Antonelli, Dolce, and Marcati \cite{Antonelli-Dolce-Marcati-2021} proved linear stability and enhanced dissipation for two-dimensional Couette flow on $\mathbb{T}\times\mathbb{R}$ at high Reynolds number; they recovered the $t^{1/2}$ growth of the $L^2$ norm previously observed in the physical literature \cite{GLB2}. Zeng, Zi, and Zhang \cite{Zeng-Zhang-Zi-2022} obtained analogous results on $\mathbb{T}\times\mathbb{R}\times\mathbb{T}$. Recently, Huang, Li, and Xu \cite{HLX} addressed the nonlinear problem and showed that $\beta\leq 11/3$ when the perturbation space $X$ is taken to be a Sobolev space. 
    Recently, this result was improved to $\beta\le1$ by Li, Wang and Zhang \cite{LWZ}.
\end{itemize}
\subsection*{Motivations and main ideas of the present work}
It is worth noting that in the work of Huang-Li-Xu \cite{HLX}, even in two dimensions the stability threshold remains as high as $11/3$, largely due to the loss of powers of $\nu$ in deriving optimal time-decay estimates. In contrast, Li-Wang-Zhang \cite{LWZ} improved the threshold to $1$ by developing an energy framework based on the first-order derivatives of the velocity instead of the velocity itself, which provides a more effective starting point for the stability analysis. Meanwhile, the nonlinear stability of the three-dimensional compressible Couette flow remains open. To address this problem, we develop a unified analytical framework adapted to the compressible Couette flow, where we are able to estimate the velocity field and the density directly.

Our main contributions are summarized as follows. In the double zero-mode regime, we exploit the wave-speed structure and dissipative mechanisms inherent to compressible flows to decouple two distinct diffusive components from the solution: diffusion waves propagating at the acoustic speed due to initial perturbations, and Huygens-type diffusive modes generated by the nonlinear structure of the governing equations. In contrast to the approach of Huang-Li-Xu \cite{HLX}, our decomposition of the corresponding wave subsystems is structurally more transparent, and the associated energy estimates are substantially refined. This yields a sharper stability threshold and provides a uniform estimate for the velocity field itself without relying on any time-decay estimates, thereby establishing a unified analytical framework for the double-zero mode.
In particular, we decompose  $\int_{\Torus^2}v^1 dxdz$  into
$$\int_{\Torus^2}v^1 dxdz = \tilde{W}_2-\tilde{W}_1+\tilde{\Xi}_2-\tilde{\Xi}_1 + \theta_A + \mathcal{A}$$ (see \eqref{anti}, \eqref{equ-Xi}), \eqref{thetaA}, and \eqref{equ-A} for more details). 
Since $\p_y \theta_A$ and $\p_y^2 \theta_A$ appear in the equation for $\mathcal{A}$ \eqref{equ-A}, controlling $\norm{\mathcal{A}}_{L^2}$ requires estimates for $\norm{\p_y \theta_A}_{L^2}$ and $\int_0^t\norm{\p_y^2 \theta_A}_{L^2}^2 d\tau$; moreover, proving the uniform boundedness of $\int_{\Torus^2}v^1 dxdz$ also requires an estimate for $\norm{\theta_A}_{L^\infty}$.
From \eqref{thetaA},  the principal part of the equation for $\theta_A$ can be written as
\begin{align*}
    \p_t \theta_A-\nu\p_y^2\theta_A=-a\theta_1-a\p_y\tilde{ W}_1\p_y\theta_A+\cdots.
\end{align*}
Following the method in \cite{HLX} would require optimal decay estimates for $\p_y^k \tilde{W}_i$ ($k=1,2;i=1,2$) to obtain the desired estimate for $\theta_A$, but this may lead to a substantial loss of $\nu$ and force the initial perturbation to satisfy $\varepsilon\sim\nu^{\frac{11}{3}}$. To avoid this loss, we introduce a new strategy. When estimating $\norm{\theta_A}_{L^\infty}$ and $\norm{\p_y \theta_A}_{L^2}$, we apply Duhamel's principle: for terms similar to the slowly decaying $\theta_1$,  we use space-time pointwise estimates, whereas for terms like $\p_y \tilde{W}_1 \p_y\theta_A$ does not exhibit an explicit time decay rate, we need a more refined decomposition of the time integral (see  Lemma \ref{estimateonta}). To estimate $\int_0^t\norm{\p_y^2 \theta_A}_{L^2}^2 d\tau$, we split $\theta_A$ into $\theta_A^{(1)}$ and $\theta_A^{(2)}$: the source term in the equation for $\theta_A^{(1)}$ resembles $\theta_1$ and possesses explicit decay rates, while that for $\theta_A^{(2)}$ resembles $\p_y \tilde{W}_1 \p_y\theta_A$ and is amenable to energy estimates.  We then estimate $\theta_A^{(1)}$ via Duhamel's principle and $\theta_A^{(2)}$ via the energy method; combining these yields the estimate for $\int_0^t\norm{\p_y^2 \theta_A}_{L^2}^2 d\tau$ (see  Lemma \ref{estimateonta-1}).

We note that our decomposition of the double-zero wave modes is essential not only for distinguishing the mechanisms of compressible and incompressible fluids, but also for deriving uniform $L^\infty$ bounds. 
This decomposition is necessitated by the fact that interactions between waves cause the estimates to become uncontrollable. 
By explicitly analyzing the specific interactions among wave patterns, we exploit the distinct propagation directions of the two wave types to obtain refined interaction estimates. 
This constitutes the key to achieving uniform bounds.

For the other modes, we construct a refined bootstrap argument and rigorously analyze the nonlinear interactions across distinct frequency bands. Coupled with precise multiplier estimates and a uniform boundedness estimate for the velocity field itself, this yields uniform-in-time estimates for the perturbation.
There are two main growth mechanisms here. The first is the $\nu^{-1/2}$ growth of $(\partial_y\phi,\textrm{div}\ v)$ due to the lack of divergence-free condition, and the second is the lift-up effect of the simple zero mode $v^1_{0\neq}$.
Controlling the interaction between this two mechanisms is one of the main challenges. For instance, consider the leading term of $W^1_{\neq}$ in the moving frame (see \cref{equ-W1} for details):
\begin{align*}
    I=\int_0^T \langle m(\tilde{\partial}_YD_{\neq}\partial_YV^1_{0\neq}),mW^1_{\neq}\rangle d t, 
\end{align*}
where the derivative operators and Fourier multiplier $m$ are given in \cref{mnote} and {\it Def} \ref{Def-m}, respectively. 
Recalling the equation for $\tilde{\partial}_Y{\de}_{\neq}$ (see \cref{sys-ND} for details) :
			\begin{align*}
				&\partial_t \tilde{\partial}_Y{\de}+\partial_XR+\tilde{\partial}_YD=-\tilde{\partial}_Y(V\cdot \tilde{\nabla} {\de})-\tilde{\partial}_Y({\de}D).
			\end{align*}
			By integration by parts in time, we end up with the main contribution
			\begin{align*}
                \int_{0}^{T}\langle m((\partial_YV^{1}_{0\neq})\tilde{\partial}_Y{\de}_{\neq}), m (V^1_{00}\partial_XW^1_{\neq}+\partial_YV^1_{0\neq}\tilde{\partial}_YD_{\neq})\rangle dt.
			\end{align*}
Therefore, to close energy, we need a uniform estimate of the double zero mode $V^1_{00}$ and initial data to satisfy $\eps \sim \nu^{3/2}$.

Overall, this work provides a systematic and robust methodology for determining the nonlinear stability threshold of compressible Couette flow. The analytical techniques developed herein are expected to serve as a reference paradigm for related problems in mathematical fluid dynamics.

\vspace{0.3cm}

	{{\bf{Notations:}} Throughout this paper, we use the following notations}
	
	\begin{itemize}
		
		\item Defining the zero and non-zero modes for an integrable function $f\in \Torus\times\R\times\Torus$,
		\begin{align*}
			\Doo f:=f_{00}:=\int_{\Torus^2}fdxdz,\qquad {f}_{0\neq}:=\int_\Torus f dx-\int_{\Torus^2} fdxdz,\qquad f_{\neq}:=f-\int_{\Torus} f dx.
		\end{align*}
		\item $\delta>0$ denotes a small constant independent of $\nu$, $\nu+\nu'$ and $t$. 
		\item Given two quantities $A$ and $B$, we denote $A\lesssim B$ and $A \approx B$ if there exists a positive universal constant $C$ such that $A\leq CB$ and $C^{-1}B \leq A \leq CB$ respectively. 
		\item  For a vector $x$, we denote $\la x \ra:=(1+\abs{x}^2)^{\frac{1}{2}}$.
		\item The Fourier transform $\hat{f}(k,\eta,l)$ of a function $f(x,y,z)$ is defined by
		\begin{align*}
			\hat{f}(k,\eta,l)=\int _{\Torus\times \R\times\Torus} f(x,y,z)e^{-2\pi i(kx+\eta y+lz)}dxdydz.
		\end{align*}
		Then the inverse Fourier transform is given by 
		\begin{align*}
			f(x,y,z)= \sum\limits_{k,l\in \mathbb{Z}} \int_{\R} \hat{f}(k,\eta,l) e^{2\pi i(kx+\eta y+lz)} d\eta.
		\end{align*}
		\item The Sobolev space $H^N(\Torus \times \R \times\Torus)$, $N\geq0$ is given by the norm 
		$$ \norm{f}_{H^N}^2=\norm{\la \nabla \ra^N f}_{L^2}^2=\sum_{k,l\in \mathbb{Z}} \int \la k,\eta,l \ra^{2N} \abs{\hat{f}}^2 (k,\eta,l)d\eta. $$
		
		
	\end{itemize}
	\
	
	The rest of the present paper is organized as follows. In Section 2, we derive a new integrated system for the double zero modes, construct $\theta_A$ to capture the main part of $v_{00}^1$, and state the main theorem. In Section 3, we reformulate the perturbation systems, introduce the Fourier multipliers, and obtain the corresponding commutators.
In Section 4, we present the bootstrap argument. 
 Then we prove a uniform bound for  $v_{00}^1$ in Section 5. In Sections 6 and 7, we study the lower regularity energy estimate for the compressible part and the incompressible part, respectively. In Section 8, we consider the energy estimate on the zero mode. Finally, in Section 9, we obtain the high regularity energy estimate.

	\section{Analytical framework and main results}
	\subsection{The perturbation system for double zero modes}
	
	{Compared with the incompressible case, a key difference of the compressible setting is that the double zero mode (i.e., zero mode in both $x$ and $z$ direction) of the $y$ component of the velocity field, $v^2_{00}$, is no longer zero, which in turn induces growth in $v^1_{00}$.} {To quantify this lift-up effect, we need to perform a refined decomposition of the double zero mode system, thereby circumventing the time-growth estimate obtained directly from the energy method i.e., $\norm{v_{00}^1}_{L^2}\lesssim (1+t)^{\frac12}$ (see \cite{Antonelli-Dolce-Marcati-2021,Zeng-Zhang-Zi-2022}).} To this end, applying $\Doo$ to \cref{perturbation-1}$_2$, one has
	\begin{align}\label{psi1}
		\p_t v_{00}^1-\nu\p_y^2v_{00}^1=-v_{00}^2-v_{00}^2\p_yv_{00}^1-\frac{{\lde}_{00}}{{\lde}_{00}+1}\nu\p_y^2v_{00}^1+\tilde{F}_{21},
	\end{align}
	where
	\begin{align}\label{F1}
		\begin{aligned}
			\tilde{F}_{21}:=&-\Doo\left[\frac{{\lde}}{{\lde}+1}\big(\nu\Delta v^1+(\nu+\nu')\p_x\dv v\big)\right]-\Doo(v\cdot\nabla v^1)+v_{00}^2\p_yv_{00}^1\\
			&+\frac{{\lde}_{00}}{{\lde}_{00}+1}\nu\p_y^2v_{00}^1 +\Doo\left[\p_x {\lde}\frac{{\lde}}{{\lde}+1}-\frac{(P'(\rho)-1)\p_x\phi}{{\lde}+1}\right].
		\end{aligned}
	\end{align}
	To study $v^2_{00}$ appearing in the right side of the above equation, we note that the system \cref{perturbation-1} is not in conservation form. Instead, we study the conserved system for $({\lde},\varphi^2)$ with $\varphi^2=\tilde{\rho} v_2$. Taking $\Doo$ on the equations \cref{ns-ori}$_1$ and \cref{ns-ori}$_3$ respectively, we obtain a new system for the zero mode $({\lde}_{00},{\varphi} ^2_{00})$ as follows, 
	\begin{align}
		\begin{cases}\label{equ-rhou2}
			\p_t {\lde}_{00} + \p_y {\varphi} ^2_{00} = 0, \\[2mm]
			\p_t {\varphi} ^2_{00} +  \p_y{\lde}_{00} -( 2\nu +\nu')\p_y^2 {\varphi} ^2_{00} =\partial_{y} E,
		\end{cases}
	\end{align}
	where 
	\begin{align}\label{E}
		\begin{aligned}
			E=&\left[\tilde{F}_2-(2\nu+\nu')\p_y\left(\frac{{\varphi} ^2_{00}}{\rho_{00}}-\int_{\Torus^2}\frac{\varphi^2}{\tilde{\rho}}dxdz\right)\right]-(2\nu+\nu')\p_y\left({\varphi}^{2}_{00}-\frac{{\varphi} ^2_{00}}{\rho_{00}}\right)\\
			&-\left[\bigg(P(\rho_{00})-P(1)-P'(1){\lde}_{00}\bigg)+ \frac{({\varphi} ^2_{00})^2}{\rho_{00}}\right]-\left(v_{00}^2-\frac{{\varphi} ^2_{00}}{\rho_{00}}\right){\varphi} ^2_{00}\\
			:=&E_1+E_2+E_3+E_4,\\
			\tilde{F}_2:=&v_{00}^2  {\varphi} ^2_{00} +P(\rho_{00})-\big(\int_{\Torus^2} v^2 \varphi^2 d x dz + \int_{\Torus^2} P(\tilde{\rho}) d xdz\big ),
		\end{aligned}
	\end{align}
	with the initial data $({\lde}_{00},{\varphi} ^2_{00})(y,0)$.
	
	\subsection{Construction of ansatz}
	
	In this subsection, we study the system~\cref{equ-rhou2}, which may be rewritten  as
	\begin{align}\label{equ-green}
		\bar{w}_t+\tilde{A}\bar{w}_y=\tilde{B} \bar{w}_{yy}+E_y \mathbf{e}_2,
	\end{align}
	where $\mathbf{e}_2:=(0,1)^{T}$ and
	\begin{align}\label{xcvvcx}
		\bar{w}=\left(\begin{array}{c}{\lde}_{00}\\ \varphi^2_{00} \end{array}\right), \quad \tilde{A}=\left(\begin{array}{cc}0 & 1 \\ {1} & 0\end{array}\right), \quad \tilde{B}=\left(\begin{array}{cc}0 & 0 \\ 0 & \bar{\nu}\end{array}\right),\quad \bar{\nu}:=2\nu+\nu'.
	\end{align}
	We note that $\tilde{A}$ is symmetric and the eigenvalues are $\sigma_1=-1$ and $\sigma_2=1$, with the corresponding left and right eigenvectors given as 
	\begin{align}\notag
		(l_1,l_2)^t=\frac{1}{2a}\left(\begin{array}{cc}-1&1\\
			1& 1
		\end{array}\right), \quad	(r_1,r_2)=a\left(\begin{array}{cc}-1&1\\ 1&1\end{array}\right)
	\end{align} 
	respectively, where
	\begin{equation*}
		a:=\frac{2}{P''(1)+2}.
	\end{equation*}
	We decompose the solution $w$ along the right eigenvector directions
	\begin{align}\label{phi-psi}
		w=L\bar{w}, \quad \text{i.e.} \quad \bar{w}=R w,
	\end{align}
	where $L=(l_1,l_2)^{t}$, $R=(r_1,r_2).$ Then we may diagonalize \cref{equ-rhou2} as
	\begin{align}\label{2025-9-23-1}
		\begin{cases}
			w_{1t}-w_{1y}=\frac{\bar{\nu}}{2}w_{1yy}+\frac{\bar{\nu}}{2}w_{2yy}+\frac{1}{2a}E_y,\\
			w_{2t}+w_{2y}=\frac{\bar{\nu}}{2}w_{2yy}+\frac{\bar{\nu}}{2}w_{1yy}+\frac{1}{2a}E_y,
		\end{cases}
	\end{align}
	with the initial data $(w_1,w_2)(y,0)$.
	
	Next, we seek a suitable ansatz to capture the large time behavior of $(w_1,w_2)$. For this, we construct the following diffusion wave
	\begin{align}\label{equ-theta}
		\begin{cases}
			\theta_{1t}-\theta_{1y}+(\frac{\theta_1^2}{2})_y=\frac{\bar{\nu}}{2}\theta_{1yy},& y \in \mathbb{R}, t>0,\\
			\theta_{2t}+\theta_{2y}+(\frac{\theta_2^2}{2})_y=\frac{\bar{\nu}}{2}\theta_{2yy},& y\in \mathbb{R}, t>0,\\
			\int_{\R}\theta_i(y,0)dy=l_i\cdot\int_{\R}\bar{w}(y,0)dy,& y \in \mathbb{R}, \quad i=1,2.
		\end{cases}
	\end{align}
	One can use the Hopf-Cole transformation to obtain the explicit expressions for $\theta_1$ and $\theta_2$
	\begin{align}\label{theta}
		\begin{aligned}
			&\theta_1(x_1, t)=\frac{\bar{\nu}^{\frac{1}{2}}}{\sqrt{2(1+t)}} \Gamma_1\left(\frac{x+(1+t)}{\sqrt{2\bar{\nu}(1+t)}}\right),\qquad\qquad\theta_2(x_1, t)=\frac{\bar{\nu}^{\frac{1}{2}}}{\sqrt{2(1+t)}} \Gamma_2\left(\frac{x-(1+t)}{\sqrt{2\bar{\nu}(1+t)}}\right),\\
			&\Gamma_i(y)=\frac{\left( e^ \frac{\eta_i}{\bar{\nu}}-1\right) \exp \left(-y^2\right)}{\sqrt{\pi}+\left(e^ \frac{\eta_i}{\bar{\nu}}-1\right) \int_y^{+\infty} \exp \left(-\xi^2\right) d \xi},\qquad  \qquad\eta_i:=l_i\cdot\int_{\R}\bar w(y,0)dy.
		\end{aligned}
	\end{align} 	
	
	Setting $\ww_{i}:=w_{i}-\theta_i$, $i=1,2$, one has
			\begin{align}\label{barv}
				\begin{cases}
					\ww_{1t}-\bar{c}\ww_{1y}=\frac{\bar{\mu}}{2}\ww_{1yy}+\frac{\bar{\mu}}{2}\ww_{2yy}+\frac{1}{2a}{E}_y+(\frac{\theta_1^2}{2})_y+\frac{\bar{\mu}}{2}\theta_{2yy},\\
					\ww_{2t}+\bar{c}\ww_{2y}=\frac{\bar{\mu}}{2}\ww_{2yy}+\frac{\bar{\mu}}{2}\ww_{1yy}+\frac{1}{2a}{E}_y+(\frac{\theta_2^2}{2})_y+\frac{\bar{\mu}}{2}\theta_{1yy},\\
					\int_{\R}\ww_i(y,0)dy=0,\quad i=1,2,
				\end{cases}& y \in \mathbb{R},\quad t>0.
			\end{align}
			In view of  \cref{E},  the leading term of ${E}_y$ is 
			\begin{align}\label{tildeE}
				-\frac{P''(1)}{2}[\left(\phi_{00}\right)^2]_y-[\left( \varphi_{00}^2 \right)^2]_y.
			\end{align}
			Since $\phi_{00}=a\left(w_2-w_1\right)$ and $\varphi_{00}^2=a\left(w_1+w_2\right)$, the leading part of $-\frac{1}{2a}\left\{ \frac{P''(1)}{2}[\left(\phi_{00}\right)^2]_y+[\left( \varphi_{00}^2 \right)^2]_y\right\}+\left(\frac{\theta_i^2}{2}\right)_y$ is $-(\frac{\theta_{3-i}^2}{2})_y$, which is bounded by $(1+t)^{-\frac{3}{2}}e^{-\frac{(y\pm t)^2}{1+t}}$.
			 However, the decay of $\ww$ is still not sufficient due to the bad term $-\left(\frac{\theta_{3-i}^2}{2}\right)_{y}$, and hence we shall construct another diffusion wave $\Xi_i$ to approximate $\ww$. Motivated by \cite{HLX}, we define the following Huygens-type coupled diffusion wave to capture the leading terms
	\begin{align}\label{equ-xi}
		\begin{cases}
			\Xi_{1t}-  \Xi_{1y}+\left[\theta_2^2 / 2+\theta_1\Xi_1+\theta_2 \Xi_2+\frac{2-P''(1)}{2+P''(1)}(\theta_1\Xi_2+\theta_2\Xi_1)\right]_y-\frac{\bar{\nu}}{2} \Xi_{1yy}=0, & y \in \mathbb{R}, t>0, \\ 
			\Xi_{2t}+ \Xi_{2y}+\left[\theta_1^2 / 2+\theta_1\Xi_1+\theta_2 \Xi_2+\frac{2-P''(1)}{2+P''(1)}(\theta_1\Xi_2+\theta_2\Xi_1)\right]_y-\frac{\bar{\nu}}{2}  \Xi_{2yy}=0, & y \in \mathbb{R}, t>0, \\ \Xi_1(y, 0)=\Xi_2(y, 0)=0, & y \in \mathbb{R}.\end{cases}
	\end{align}
	{{Since $\theta_{3-i}$ and the Green's function for $\Xi_i$ propagate in different directions, 
			Duhamel's principle can be applied to estimate  $\Xi_i$, yielding the decay rate 
			$(1+t)^{-1/2}$ for $\norm{\Xi_i}_{L^2}$, which is better than the rate 
			$(1+t)^{-1/4}$ for $\norm{\theta_i}_{L^2}$ (see  Lemma \ref{estimateontheta} and  Lemma \ref{estimateonxi} below).}}
	Setting
	\begin{align}\label{tildev}
		&\tilde{w}_i:=w_i-\mathcal{C}_i,\qquad i'=3-i,\\
		&\mathcal{C}_i = \theta_i +\Xi_i+(-1)^{i}\frac{\bar{\nu}}{4}\theta_{i'y}+(-1)^{i}\frac{\bar{\nu}}{4}\Xi_{i'y},\nonumber
	\end{align}
	then we have
	\begin{align}\label{equ-tildev}
		\begin{cases}
			\tilde{w}_{1t}-\tilde{w}_{1y}=\frac{\bar{\nu}}{2}\tilde{w}_{1yy}+\frac{\bar{\nu}}{2}\tilde{w}_{2yy}+\tilde{K}_{1y},\\
			\tilde{w}_{2t}+\tilde{w}_{2y}=\frac{\bar{\nu}}{2}\tilde{w}_{2yy}+\frac{\bar{\nu}}{2}\tilde{w}_{1yy}+\tilde{K}_{2y},
		\end{cases}
	\end{align}
	where
	\begin{align}\label{tildeK}
		\tilde{K}_{i}&:=\tilde{E}_{i}+\tilde{N},\\ \nonumber
		\tilde{E}_{i}&:=\frac{1}{2a}{E}_{2}+\left(\frac{1}{2a}{E}_{3}+\theta_1^2 / 2+\theta_2^2 / 2+\theta_1\Xi_1+\theta_2 \Xi_2+\frac{2-P''(1)}{2+P''(1)}(\theta_1\Xi_2+\theta_2\Xi_1)\right)\\
		&+(-1)^{i}\frac{\bar{\nu}}{4}\left(\theta_1^2 / 2+\theta_2^2 / 2+\theta_1\Xi_1+\theta_2 \Xi_2+\frac{2-P''(1)}{2+P''(1)}(\theta_1\Xi_2+\theta_2\Xi_1)-\frac{\bar{\nu}}{2}\theta_{iy}-\frac{\bar{\nu}}{2}\Xi_{iy}\right)_{y},\nonumber\\ 
		\tilde{N}&:=\frac{1}{2a}(E_{1}+E_{4}).\nonumber
	\end{align}
	It should be noted that the decay rate of  $\tilde{K}_{iy}$ in \cref{equ-tildev} is improved to $$(1+t)^{-2}e^{-\frac{(x\pm t)^2}{(1+t)}}+\p_y(\Xi^2)$$ due to the cancellations in $$\frac{1}{2a}{E}_{3}+\theta_1^2 / 2+\theta_2^2 / 2+\theta_1\Xi_1+\theta_2 \Xi_2+\frac{2-P''(1)}{2+P''(1)}(\theta_1\Xi_2+\theta_2\Xi_1).$$ 
	
	{From \cref{perturbation-1}$_1$, we find the formal relation
    \begin{align*}
        \p_t\int_{-\infty}^{y} \phi_{00} d\tilde{y} + v_{00}^2 = \int_{-\infty}^{y}\int_{\T^2} F_1 dxd\tilde{y}dz,
    \end{align*}
    which may cancel the growth term $v_{00}^2$ appearing in the equation of $v_{00}^1$ \cref{psi1}.  This motivates us to define the  anti-derivatives of $\tilde{w}_i$ and $\Xi_i$.}
	In view of \cref{equ-theta}-\cref{equ-tildev}, one has $$\int_{-\infty}^\infty \tilde{w}_i(y,t)dy=0, \quad \int_{-\infty}^\infty \Xi_i(y,t)dy=0, \quad i=1,2.$$ 
	We define the anti-derivatives of $(\tilde{w}_1,\tilde{w}_2)$ and $\left({\Xi}_1(z,t),{\Xi}_2(z,t)\right)$ as 
	\begin{align}
		\begin{array}{lll}
			\mathbf{\tilde{W}}:=\left(\tilde{W}_1,\tilde{W}_2\right):=&\int_{-\infty}^y\left(\tilde{w}_1(z,t),\tilde{w}_2(z,t)\right)dz,\quad\mathbf{\tilde{\Xi}}:=\left(\tilde{\Xi}_1,\tilde{\Xi}_2\right):=\int_{-\infty}^y\left({\Xi}_1(z,t),{\Xi}_2(z,t)\right)dz.\label{equ--anti}
		\end{array}
	\end{align}
	Then the systems \cref{equ-tildev} and \cref{equ-xi} become,  respectively,
	\begin{align}\label{anti}
		\mathbf{\tilde{W}}_t+A\mathbf{\tilde{W}}_y=B\mathbf{\tilde{W}}_{yy}+\mathbf{\tilde{K}},
	\end{align}
	where
	\begin{align}\notag
		A=\left(\begin{array}{cc}-1& 0 \\ 0 & 1\end{array}\right), \quad {B}=\left(\begin{array}{cc}\frac{\bar{\nu}}{2} & \frac{\bar{\nu}}{2} \\ \frac{\bar{\nu}}{2} & \frac{\bar{\nu}}{2}\end{array}\right),\quad \mathbf{\tilde{K}}=\left(\begin{array}{c}\tilde{K}_1\\ \tilde{K}_2\end{array}\right),
	\end{align}
	and
	\begin{align}\label{equ-Xi}
		\begin{cases}
			\tilde{\Xi}_{1t}-
			\tilde{\Xi}_{1y}+{\left[\theta_2^2 / 2+\theta_1\Xi_1+\theta_2 \Xi_2+\frac{2-P''(1)}{2+P''(1)}(\theta_1\Xi_2+\theta_2\Xi_1)\right]}=\frac{\bar{\nu}}{2}  \tilde{\Xi}_{1yy}, & y \in \mathbb{R}, t>0, \\ 
			\tilde{\Xi}_{2t}+ \tilde{\Xi}_{2y}+\left[\theta_1^2 / 2+\theta_1\Xi_1+\theta_2 \Xi_2+\frac{2-P''(1)}{2+P''(1)}(\theta_1\Xi_2+\theta_2\Xi_1)\right]=\frac{\bar{\nu}}{2}  \tilde{\Xi}_{2yy}, & y \in \mathbb{R}, t>0, \\ \tilde{\Xi}_1(y, 0)=\tilde{\Xi}_2(y, 0)=0, & y \in \mathbb{R}.\end{cases}
	\end{align}
	
	Next, based on the construction above, in order to better quantify the life-up effect, we approximate $v_{00}^1$ to the leading order. We recall the equation for $v_{00}^1$
	\begin{align}\label{equ-psi1}
		\p_t v^1_{00}-\nu\p_y^2v^1_{00}+v^2_{00}=-v^2_{00}\p_yv^1_{00}-\frac{{\lde}_{00}}{{\lde}_{00}+1}\nu\p_y^2v^1_{00}+\tilde{F}_{21}.
	\end{align}
To better quantify the lift up effect, naturally $a(\tilde{W}_2-\tilde{W}_1+\tilde{\Xi}_2-\tilde{\Xi}_1)$ would be a new good unknown. However, this quantity behaves differently from $v^1_{00}$ due to nonlinearity. {Motivated by \cite{HLX}, we define $\theta_A$, which serves as  the leading part of $v_{00}^1-a(\tilde{W}_2-\tilde{W}_1+\tilde{\Xi}_2-\tilde{\Xi}_1)$,}
	\begin{align}\label{thetaA}
		\left\{\begin{aligned}
			&{\p_t\theta_{A}}-\nu\p_y^2\theta_{A}=-a(\mathcal{C}_1+\mathcal{C}_2-\Xi_1-\Xi_2)-\frac{a\nu'}{2}\left(\Xi_2-\Xi_1\right)_y+\frac{a}{2}(\theta_1^2-\theta_2^2)\\
			&\quad\;\qquad\qquad\qquad-a(\mathcal{C}_1+\mathcal{C}_2)\p_y\theta_A+a^2(\mathcal{C}_1+\mathcal{C}_2)(\mathcal{C}_2-\Xi_2-\mathcal{C}_1+\Xi_1)\\
			&\quad\;\qquad\qquad\qquad-a(\tilde{w}_1+\tilde{w}_2)\p_y\theta_A+a^2(\tilde{w}_1+\tilde{w}_2)(\mathcal{C}_2-\mathcal{C}_1),\\
			&\theta_A(y,0)=0.
		\end{aligned}\right.
	\end{align}
Denoting the error term by $\mathcal A:=v^1_{00}-a(\tilde{W}_2-\tilde{W}_1+\tilde{\Xi}_2-\tilde{\Xi}_1)-\theta_A$, one has  
	\begin{align}\label{equ-A}
		&\qquad\begin{aligned}
			\p_t{\mathcal{A}}- \nu \p_y^2{\mathcal{A}}-a\nu\p_y(\tilde{w}_2-\tilde{w}_1)=-v^2_{00}\p_y\mathcal{A}-\nu\frac{{\lde}_{00}}{{\lde}_{00}+1}\p_y^2\mathcal{A}+F_{A},
		\end{aligned}\\
		&\begin{aligned}\label{FA}
			{F_{A}:=}&a^2(\tilde{w}_1+\tilde{w}_2)(\Xi_1-\Xi_2)+a\nu\frac{{\lde}_{00}}{{\lde}_{00}+1}\left(\tilde{w}_1-\tilde{w}_2+\Xi_1-\Xi_2\right)_y-a\nu\frac{{\lde}_{00}}{{\lde}_{00}+1}\p_y^2\theta_A\\
			&-\frac{1}{\rho_{00}}({\lde}_{00})^2\varphi^2_{00}+\Doo(\tilde{\rho} v^2-\rho_{00}v_{00}^2)+{\lde}_{00}\Doo\left(\frac{\varphi^2}{\tilde{\rho}}-\frac{\varphi^2_{00}}{\rho_{00}}\right)+F_1-a (\tilde{K}_{2}-\tilde{K}_{1})\\
			&+\left\{\frac{{\lde}_{00} \varphi^2_{00}}{\rho_{00}}-\Doo\left(\frac{\varphi^2}{\tilde{\rho}}-\frac{\varphi^2_{00}}{\rho_{00}} \right) \right\}\Big(\p_y \theta_A + a\left(\tilde{w}_2-\tilde{w}_1-\Xi_1+\Xi_2 \right)\Big).
		\end{aligned}
	\end{align}
	\begin{Rem}
		{Since the propagation direction of $\mathcal{C}_1+\mathcal{C}_2-\Xi_1-\Xi_2$ differs from that of the Green's function for $\theta_A$, we can use Duhamel's principle to obtain the uniform boundedness of $\theta_A$ -- the estimate that cannot be achieved by energy estimates alone (see  Lemma \ref{estimateonta} below).}
	\end{Rem}
	\begin{Rem}
		Note that the leading term $v^2_{00}$ in \cref{equ-psi1}  is no longer present; instead it is replaced by $-a\nu\p_y(\tilde{w}_2-\tilde{w}_1)$ in the new equation \cref{equ-A}, which weakens the lift-up effect. {Therefore, $a(\tilde{W}_2-\tilde{W}_1+\tilde{\Xi}_2-\tilde{\Xi}_1)+\theta_A$ can approximate $v_{00}^1$ well.}
	\end{Rem}
	\subsection{Main results and comments}
	Now we are ready to state the main results.
	\begin{Thm}\label{MT}
		There exist a sufficiently small constant $\delta_0>0$ 
		and a suitable small constant $\varepsilon_0$ independent of $\nu$ such that if $0<\nu\leqslant\delta_0$
		and $\nu+\nu'$ is of the same order as $\nu$, and the initial data satisfies
		\begin{equation}\label{initial}
			\|(\tilde{W}_1,\tilde{W}_2)_{in}\|_{H^2}+\|(\phi,\varphi_2)_{in}\|_{L^1}+\|(\phi,v)_{in}\|_{H^5}\leqslant\varepsilon=\varepsilon_0\nu^{3/2},
		\end{equation}
		then the system \cref{perturbation-1} admits a unique global solution $(\phi,v)$. Moreover, the profiles $(R,V)(t,x,y,z)=(\phi,v)(t,x+yt,y,z)$ satisfy the following estimates:
		\begin{align*}
			&\|(V^1_{\neq},V^3_{\neq})\|_{L^\infty H^3}+\nu^{1/6}\|(V^1_{\neq},V^3_{\neq})\|_{L^2 H^3}\lesssim\varepsilon,\\
			&\|(R_{\neq},V^2_{\neq})\|_{L^\infty H^3}+\nu^{1/6}\|(R_{\neq},V^2_{\neq})\|_{L^2 H^3}\lesssim\nu^{-1/6}\varepsilon,\\
			&\nu\|V^1_{0\neq}\|_{L^\infty H^5}+\nu^{1/2}(\|V^1_{00}\|_{L^\infty L^\infty}+\|\partial_YV^1_{00}\|_{L^\infty H^4})\lesssim\varepsilon,\\
			&\|R_{0}\|_{L^\infty H^3}+\|V^2_{0}\|_{L^\infty H^3}+\|V^3_0\|_{L^\infty H^4}\lesssim\varepsilon.
		\end{align*}
	\end{Thm}
	\begin{Rem}
		Compared to the 2D stability threshold $\varepsilon \sim \nu^{11/3}$ in \cite{HLX}, the initial data of \cref{MT} $\varepsilon = \varepsilon_0 \nu^{3/2}$ represents a substantial sharpening, thereby extending the nonlinear stability theory to the more physically relevant three-dimensional setting. Moreover, note that in \cite{HLX} the optimal time decay rates of $(R_{00},V_{00}^2)$  are required to derive a uniform estimate for $V_{00}^1$; in contrast, at this stage we do not need these rates.
	\end{Rem}
	\begin{Rem}
		The approach in \cref{MT} differs fundamentally from that in \cite{LWZ}. The analysis in \cite{LWZ} constructs energy estimates beginning with the first-order derivatives of the velocity field. In contrast, our framework establishes control directly for the original solution quantities $(R, V)$.
	\end{Rem}
	\begin{Rem}
		Compared to 3D incompressible flows, we obtain the same stability threshold as \cite{Bedrossian-Germain-Masmoudi-2017} using the Fourier multiplier method. Moreover,  due to the lack of divergence-free conditions, $(R_{\neq},V^2_{\neq})$ exhibits additional growth with respect to $\nu^{-\frac16}$, which is much smaller than the linear result $\|V^2_{\neq}(t)\|_{L^2}\lesssim \nu^{-\frac{1}{3}}\eps$ in \cite{Zeng-Zhang-Zi-2022}.  In this case, a uniform estimate of the double zero mode $V_{00}^1$ is required to suppress the strong interaction between this growth and the 3D lift-up effect; see \cref{est:I21}  for details.
	\end{Rem}
	\section{Reformulation of the system, Fourier multipliers, and commutators}
	\subsection{Decomposition into compressible and incompressible parts}
We decompose v into compressible and incompressible components, reflecting their distinct physical properties. The decomposition is given by
	\begin{equation}\label{decomp}
		q=\Delta v,\qquad d=\textrm{div}\ v,\qquad\omega=q-\nabla d,
	\end{equation}
	then the equation for $q$ read
	\[
	\partial_t q+y\partial_x q+2\partial_{xy}v+(q^2,0,0)^\top+\nabla\Delta {\lde}-\nu\Delta q-(\nu+\nu')\Delta\nabla\textrm{div}v=\Delta F_2,
	\]
	where
	\begin{align*}
		\Delta F_2=&-v\cdot\nabla q-q\cdot \nabla v-2\partial_iv^j \partial_{ij} v\\
		&-\Delta\left(\left(\frac{P'(1+{\lde})}{1+{\lde}}-1\right)\nabla {\lde}+\frac{{\lde}}{{\lde}+1}(\nu\Delta v+(\nu+\nu')\nabla \textrm{div}\ v)\right).
	\end{align*}
The system we are working on is given by
	\begin{align}\notag
		\left\{\begin{aligned}
			&\partial_t {\lde}+y\partial_x {\lde}+d=F_1,\\
			&\partial_t d+y\partial_x d+2\partial_x v^2+\Delta {\lde}-\bar{\nu} \Delta d= \textrm{div} F_2, \\
			&\partial_t \omega+y\partial_x \omega+(\omega^2, 0, 0)^\top+(\partial_y d, -\partial_x d,0)^\top+2\partial_{xy}v-2\partial_x\nabla v^2-\nu\Delta \omega=\Delta F_2-\nabla\textrm{div} F_2,\\[1.5mm]
			&\big({\lde},d,\omega\big)(x,y,z,0)=\big({\lde}_{in},d_{in},\omega_{in}\big)(x,y,z).
		\end{aligned}\right.
	\end{align}
	where $\bar{\nu}=2\nu+\nu'$ and nonlinear term given by
	\begin{align*}
		F_1=&-v\cdot \nabla {\lde}-{\lde}d,\\
		\textrm{div} F_2=&-v\cdot\nabla d-\partial_i v^j\partial_j v^i\\
		&-\textrm{div} \left(\left(\frac{P'(1+{\lde})}{1+{\lde}}-1\right)\nabla {\lde}+\frac{{\lde}}{{\lde}+1}(\nu\Delta v+(\nu+\nu')\nabla d)\right),\\
		(\Delta-\nabla\textrm{div})F_2=&-v\cdot\nabla \omega-q\cdot\nabla v-2\partial_iv^j \partial_{ij} v+\nabla v\cdot\nabla d+\nabla(\partial_i v^j\partial_j v^i)\\
		&-(\Delta-\nabla\textrm{div})\left(\frac{{\lde}}{{\lde}+1}(\nu\Delta v+(\nu+\nu')\nabla d)\right).
	\end{align*}
	\subsection{Coordinate transformations}
	To mod out the Couette flow, we take the following coordinate transformations:
	\begin{equation*}
		X=x-yt,\ Y=y,\ Z=z,
	\end{equation*}
and use the following notations
    \begin{align}\label{mnote}
		\left\{
		\begin{array}{ll}
			\partial_{x} = \partial_{X},\quad \tilde{\partial}_{XY} := \partial_{X}\tilde{\partial}_{Y},\quad \tilde{\p}_{XYZ}:=\partial_{X}\tilde{\partial}_{Y}\p_Z,\qquad & \tilde{\nabla} = (\partial_{X}, \tilde{\partial}_{Y},\partial_Z),\\[2mm]
			\partial_{y} = \tilde{\partial}_{Y} := \partial_{Y} - t\partial_{X},\qquad \p_z=\p_Z,& \tilde{\mathrm{div}}\, = \tilde{\nabla}\cdot,\\[2mm]
			\Delta = \tilde{\Delta} := \partial_{XX} + (\partial_{Y} - t\partial_{X})^{2}+\p_{ZZ},\qquad & \tilde{\nabla}^{\perp} = (-\tilde{\partial}_{Y}, \partial_{X}).
		\end{array}
		\right.
	\end{align}
We denote the  symbol associated to $-\tilde{\Delta}$ by 
	$$p=k^2+(\eta-kt)^2+l^2.$$
In the new coordinate system, let
	\begin{align*}
		\big(R,V,\rho,D,\Omega\big):=(\phi,v,\tilde{\rho},d,\omega)(t,X,Y,Z).
	\end{align*}
	For compressible part, we study the symmetrical system of $({\de},p^{-\frac{1}{2}}D)$ which satisfies
	\begin{align}\label{sys-ND}
		\left\{\begin{aligned}
			&\partial_t {\de}+p^{\frac{1}{2}}(p^{-\frac{1}{2}}D)=\nlt_{\de},\\[2mm]
			&\partial_t (p^{-\frac{1}{2}}D)-\frac{\partial_tp}{2p}(p^{-\frac{1}{2}}D)-p^{\frac{1}{2}}{\de}-2\frac{\partial_X}{p^{\frac{3}{2}}}(W^2-\partial_X{\de}+\nu\partial_XD)+\nu p(p^{-\frac{1}{2}}D)=p^{-\frac{1}{2}}\nlt_D,
		\end{aligned}\right.
	\end{align}
	where we denote
	\begin{equation}\label{t-ND}
		\begin{aligned}
			&\nlt_{\de}=-V\cdot \tilde{\nabla} {\de}-{\de}D,\\
			&\nlt_{D}=-V\cdot\tilde{\nabla} D-\tilde{\partial}_iV^j\tilde{\partial}_jV^i-\dl(F({\de})\tilde{\nabla}{\de}+G({\de})(\nu\tilde{\Delta} V+(\nu+\nu')\tilde{\nabla}D)),
		\end{aligned}
	\end{equation}
	and 
	\begin{align*}
		F({\de})=\frac{P'(1+{\de})}{1+{\de}}-1,\quad G({\de})=\frac{{\de}}{{\de}+1}.
	\end{align*}
	To study the incompressible part, motivated by \cite{Antonelli-Dolce-Marcati-2021,Zeng-Zhang-Zi-2022}, we introduce the good unknowns
	\begin{align*}
		W^1=\Omega^1-\tilde{\partial}_Y{\de}+\nu\tilde{\partial}_YD,\qquad\quad W^2=\Omega^2+\partial_X{\de}-\nu\partial_XD.
	\end{align*}
	The $z$-direction is less affected by the compressible portion, so we still use unknown $\Omega^3$. Now the incompressibility condition is expressed as:
	\begin{equation} \label{divf}
		\partial_X W^1+\tilde{\partial}_Y W^2+\partial_Z \Omega^3=0.
	\end{equation}
Then the system reads
\begin{align}
	\label{equ-W1}
	&\partial_t W^1+\lt_{W^1}+\nu pW^1=\nlt_{W^1}, \\
	\label{equ-W2}
	&\partial_t W^2+\lt_{W^2}+\nu pW^2=\nlt_{W^2}, \\
	\label{equ-O3}
		&\partial_t\Omega^3+\lt_{\Omega^3}+\nu p\Omega^3=\nlt_{\Omega^3},
\end{align}
	where the linear operators are given by 
	\begin{align}
	\mathcal{L}_{W^1}=&W^2-2\partial_X{\de}+2\nu\partial_XD-\frac{\partial_t p}{p}W^1+2(\partial_{XX}+\tilde{\p}_{Y}^{2} )\tilde{\Delta}^{-1}\partial_X{\de}-2\partial_{XX}\tilde{\Delta}^{-1}W^2-2\nu\partial_{XX}\tilde{\Delta}^{-1}\partial_XD \nonumber \\
	\label{l-W2}
	&+2\nu\tilde{\partial}_{XY}\tilde{\Delta}^{-1}W^2-2\nu\tilde{\partial}_{XY}\tilde{\Delta}^{-1}\partial_X{\de}+2\nu^2\tilde{\partial}_{XY}\tilde{\Delta}^{-1}\partial_XD-\nu(\nu+\nu')\tilde{\Delta}\tilde{\partial}_YD,\\
	\label{l-O3}
		\lt_{W^2}=&2\nu\partial_X^2p^{-1}(W^2-\partial_X{\de}+\nu\partial_XD)+\nu\frac{\partial_tp}{p}\partial_XD-\nu(\nu+\nu')p\partial_XD,\\
			\lt_{\Omega^3}=&-\frac{\partial_tp}{p}\Omega^3-2\frac{\partial_{XZ}}{p}(W^2-2\partial_X{\de}+\nu\partial_XD). 
	\end{align}
The nonlinear terms are collected below:
	\begin{equation}\label{t-W1}
		\begin{aligned}
			\nlt_{W^1}=&-V\cdot\tilde{\nabla}W^1+\partial_XV^2\tilde{\partial}_YD+\partial_XV^3\partial_ZD-Q^2\tilde{\partial}_YV^1-Q^3\partial_ZV^1-\Omega^1\partial_XV^1\\
			&+2\sum_{i=2,3}\left(\tilde{\partial}_iV^1\partial_{X}^2V^i-\tilde{\partial}_iV^j\tilde{\partial}_{ij}V^1\right)+\sum_{i,j=2,3}\partial_X(\tilde{\partial}_iV^j\tilde{\partial}_jV^i)\\
			&+(\tilde{\partial}_YV\cdot\tilde{\nabla}{\de})+\tilde{\partial}_Y({\de}D)-\nu(\tilde{\partial}_YV)\cdot\tilde{\nabla}D-\nu\tilde{\partial}_Y(\tilde{\partial}_iV^j\tilde{\partial}_jV^i)\\
			&-(\tilde{\Delta}-\tilde{\nabla}\tilde{\textrm{div}})\left(G({\de})(\nu\tilde{\Delta} V+(\nu+\nu')\tilde{\nabla} \tilde{\textrm{div}} V)\right)^1\\
			&-\nu\tilde{\partial}_Y\tilde{\textrm{div}} \left(F({\de})\tilde{\nabla} {\de}+G({\de})(\nu\tilde{\Delta} V+(\nu+\nu')\tilde{\nabla} \tilde{\textrm{div}} V)\right),
		\end{aligned}
	\end{equation}
	\begin{equation}\label{t-W2}
		\begin{aligned}
			\nlt_{W^2}=&-V\cdot\tilde{\nabla}W^2+\sum_{i=1,3}(\tilde{\partial}_YV^i\partial_iD-Q^i\partial_iV^2)-\Omega^2\tilde{\partial}_YV^2\\
			&+2\sum_{i=1,3}(\partial_iV^2\tilde{\p}_{Y}^{2} V^i-\tilde{\partial}_iV^j\tilde{\partial}_{ij}V^2)+\sum_{i,j=1,3}\tilde{\partial}_Y(\partial_iV^j\partial_jV^i)\\
			&-\partial_XV\cdot\tilde{\nabla}{\de}-\partial_X({\de}D)+\nu\partial_XV\cdot\tilde{\nabla}D+\nu\partial_X(\tilde{\partial}_iV^j\tilde{\partial}_jV^i)\\
			&-(\tilde{\Delta}-\tilde{\nabla}\tilde{\textrm{div}})\left(G({\de})(\nu\tilde{\Delta} V+(\nu+\nu')\tilde{\nabla} \tilde{\textrm{div}} V)\right)^2\\
			&+\nu\partial_X\tilde{\textrm{div}} \left(F({\de})\tilde{\nabla} {\de}+G({\de})(\nu\tilde{\Delta} V+(\nu+\nu')\tilde{\nabla} \tilde{\textrm{div}} V)\right),
		\end{aligned}    
	\end{equation}
and 
	\begin{equation}\label{t-O3}
		\begin{aligned}
			\nlt_{\Omega^3}=&-V\cdot\tilde{\nabla}\Omega^3-Q^1\partial_XV^3-Q^2\tilde{\partial}_YV^3-\Omega^3\partial_ZV^3+\partial_ZV^1\partial_XD+\partial_ZV^2\tilde{\partial}_YD\\[2mm]
			&+2\sum_{j=1,2}(\tilde{\partial}_jV^3\partial_{Z}^2V^j-\tilde{\partial}_jV^i\tilde{\partial}_{ij}V^3)+\sum_{i,j=1,2}\partial_Z(\tilde{\partial}_iV^j\tilde{\partial}_jV^i)\\
			&-(\tilde{\Delta}-\tilde{\nabla}\tilde{\textrm{div}})\left(G({\de})(\nu\Delta V+(\nu+\nu')\tilde{\nabla} \tilde{\textrm{div}} V)\right)^3.
		\end{aligned}
	\end{equation}
	Moreover, using \cref{decomp} and \cref{divf}, we can recover $V$ by 
	\begin{align*}
		&V^1=-\partial_Xp^{-1}D+\tilde{\partial}_Y\partial_X^{-1}p^{-1}W^2-\tilde{\partial}_Yp^{-1}{\de}+\nu \tilde{\partial}_Yp^{-1}D+\partial_Z\partial_X^{-1}p^{-1}\Omega^3,\\
		&V^2=-\tilde{\partial}_Yp^{-1}D-p^{-1}(W^2-\partial_X{\de}+\nu \partial_XD),\\
		&V^3=-\partial_Zp^{-1}D-p^{-1}\Omega^3.
	\end{align*}
	
	Turning to zero-mode, the equation of $Q^1_{0\neq}=\Delta V^1_{0\neq}$ which suffers a strong lift-up effect, reads
	\begin{align*}
		\partial_tQ^1_{0\neq}+\Delta V^2_{0\neq}-\nu\Delta Q^1_{0\neq}=&-\Delta(V^2_0\partial_YV^1_0)_{0\neq}-\Delta(V^3_0\partial_ZV^1_0)_{0\neq}\\
		&-\tilde{\Delta}(V_{\neq}\cdot \tilde{\nabla} V^1_{\neq})_{0\neq}-\tilde{\Delta}(F({\de})\partial_X{\de}-G({\de})(\nu\tilde{\Delta}V^1+(\nu+\nu')\partial_XD))_{0\neq}.
	\end{align*}
	For $V^1_{00}$, also study the first-order derivative quantity $U^1_{00}=\partial_YV^1_{00}-{\de}_{00}+\nu D_{00}$ which satisfies
	\begin{align*}
		\partial_t U^1_{00}-\nu\partial^2_{Y} U^1_{00}=&-\partial_Y(V^2_{00}U^1_{00})-\partial_Y(V^2_{0\neq}\partial_YV^1_{0\neq}+V^3_{0\neq}\partial_ZV^1_{0\neq})+\partial_Y({\de}_{0\neq}V^2_{0\neq})-\nu\partial_Y(V_{0\neq}\cdot\nabla V^2_{0\neq})\\
		&-\tilde{\partial}_Y(V_{\neq}\cdot\tilde{\nabla}V^1_{\neq})+\tilde{\partial}_Y(V^2_{\neq}{\de}_{\neq})-\nu\tilde{\partial}_Y(V_{\neq}\cdot\tilde{\nabla}V^2_{\neq})\\
		&-\tilde{\partial}_Y(F({\de})\partial_X{\de}-G({\de})(\nu\tilde{\Delta}V^1+(\nu+\nu')\partial_XD))_{00}\\
		&-\nu\tilde{\partial}_Y(F({\de})\tilde{\partial}_Y{\de}-G({\de})(\nu\tilde{\Delta}V^2+(\nu+\nu')\tilde{\partial}_YD))_{00}.
	\end{align*}
	In fact, it holds that
	\begin{equation*}
		\partial_Y U^1_{00}=W^1_{00}.
	\end{equation*}
	Moreover, we need to consider $V^3_{00}$ which satisfies
	\begin{align*}
		\partial_t V^3_{00}-\nu\partial^2_{Y}V^3_{00}=&-(V^2_0\partial_YV^3_0)_{00}-(V^3_0\partial_ZV^3_0)_{00}\\
		&-(V_{\neq}\cdot \tilde{\nabla} V^3_{\neq})_{00}-(F({\de})\partial_Z{\de}-G({\de})(\nu\tilde{\Delta}V^3+(\nu+\nu')\partial_ZD))_{00}.
	\end{align*}

	\subsection{The Fourier multipliers}
	In this subsection, we introduce the following Fourier multipliers, which are extensively used in the previous studies, see \cite{Antonelli-Dolce-Marcati-2021,Bedrossian-Germain-Masmoudi-2017,Bedrossian-Germain-Masmoudi-2019,Bedrossian-Masmoudi-Vicol-2016,Zeng-Zhang-Zi-2022,Zillinger-2017} and the references therein,
	\begin{equation}\label{def-M}
		\left\{\begin{aligned}
			&\displaystyle\frac{\partial_tM_1}{M_1}=-\frac{\tilde{C}(k^4+(kl)^2)^{1/2}}{k^2+(\eta-kt)^2+l^2},\ \mathrm{for}\ k\neq0,\\[2mm]
			&\displaystyle\frac{\partial_tM_2}{M_2}=-\frac{\nu^{1/3}k^2}{k^2+\nu^{2/3}(\eta-kt)^2},\ \mathrm{for}\ k\neq0,\\[2mm]
			&M_j(t,0,\eta,l)=M_j(0,k,\eta,l)=1,\ j=1,2.
		\end{aligned}\right.
	\end{equation}
	where $\tilde{C}$ is a sufficiently large constant to be determined later,
	and we define $M=M_1M_2$.
	\begin{Def}\label{Def-m}
		The Fourier
		multiplier $m$ is given by
		\begin{enumerate}
			\item if $k=0$: $m(t,0,\eta,l)=1$;
			\item if $k\neq0$: $\eta/k\leq-1000\nu^{-1/3}$: $m(t,k,\eta,l)=1$;
			\item if $k\neq0,-1000\nu^{-1/3}\leq\eta/k\leq0$:
			\begin{equation*}
				m(t,k,\eta,l)=\left\{
				\begin{aligned}
					&\frac{k^2+\eta^2+l^2}{k^2+(\eta-kt)^2+l^2},\quad \text{if}\ \ \  0<t<\eta/k+1000\nu^{-1/3},\\[2mm]
					&\frac{k^2+\eta^2+l^2}{k^2+(1000k\nu^{-1/3})^2+l^2},\quad \text{if}\ \ \ t>\eta/k+1000\nu^{-1/3};\\
				\end{aligned}\right.
			\end{equation*}
			\item if $k\neq0,\eta/k>0$:
			\begin{equation*}
				m(t,k,\eta,l)=\left\{
				\begin{aligned}
					&1,\quad\qquad\qquad\qquad\qquad \text{if}\ t<\eta/k,\\
					&\frac{k^2+l^2}{k^2+(\eta-kt)^2+l^2},\quad \text{if}\ \eta/k<t<\eta/k+1000\nu^{-1/3},\\
					&\frac{k^2+l^2}{k^2+(1000k\nu^{-1/3})^2+l^2},\quad \text{if}\ t>\eta/k+1000\nu^{-1/3};\\
				\end{aligned}\right.
			\end{equation*}
		\end{enumerate}
		with the property that
		\begin{align}\label{def-p}
			&\frac{k^2+l^2}{p}\leq m(k,\eta,l),\qquad\text{and}\quad\quad\nu^{2/3}\lesssim m\lesssim 1.
		\end{align}
	\end{Def}
	\subsection{The commutators}
	We first recall that the  symbol associated to $-\tilde{\Delta}$ 
	$$p=k^2+(\eta-kt)^2+l^2,$$ and then introduce the commutators about $p$:
	\begin{Prop}\label{est:com-p}
		For the symbol $p$ defined in \cref{def-p}, one has,
		\begin{align*}
			&\Big|p^{-\frac{1}{2}}(k,\eta,l)-p^{-\frac{1}{2}}(k',\eta',l')\Big|\lesssim p^{\frac{1}{2}}(k-k',\eta-\eta',l-l')p^{-\frac{1}{2}}(k,\eta,l)p^{-\frac{1}{2}}(k',\eta',l'),\\
			&\Big|p^{\frac{1}{2}}(k,\eta,l)-p^{\frac{1}{2}}(k',\eta',l')\Big|\lesssim p^{\frac{1}{2}}(k-k',\eta-\eta',l-l'),\\
			&\left|\frac{\partial_t p}{p}(k,\eta,l)-\frac{\partial_t p}{p}(k',\eta',l')\right|\lesssim p^{\frac{1}{2}}(k-k',\eta-\eta',l-l')|k'|^{-1}.   \end{align*}
	\end{Prop}
	The proof is standard and is omitted for brevity. Next, we introduce the following useful lemma,
	\begin{Lem}\label{lemma:commutator}
		Let $s > 0$ and $d \geq 1$ be the spatial dimension. Let $r, q \in [1, \infty]$ satisfy $\frac{1}{r} + \frac{1}{q} = \frac{1}{2}$.
		
		\begin{enumerate}
			\item Suppose that
			\begin{align}\notag
				\nabla f \in L^\infty(\mathbb{T}\times\mathbb{R}\times\mathbb{T}), \quad \nabla^s f \in L^p(\mathbb{T}\times\mathbb{R}\times\mathbb{T}), \qquad
				\nabla^{s-1} g \in L^2(\mathbb{T}\times\mathbb{R}\times\mathbb{T}), \quad g \in L^q(\mathbb{T}\times\mathbb{R}\times\mathbb{T}).
			\end{align}
			
			Then
			\begin{align}\notag
				\big\| [\nabla^s, f] g \big\|_{L^2} \lesssim \|\nabla f\|_{L^\infty} \|\nabla^{s-1} g\|_{L^2} + \|\nabla^s f\|_{L^r} \|g\|_{L^q}.
			\end{align}
			
			\item Suppose that
			\[
			\nabla f \in L^\infty(\mathbb{T}\times\mathbb{R}\times\mathbb{T}), \quad \nabla\langle\nabla\rangle^2 f \in L^2(\mathbb{T}\times\mathbb{R}\times\mathbb{T}), \qquad
			\langle\nabla\rangle^2 g \in L^2(\mathbb{T}\times\mathbb{R}\times\mathbb{T}), \quad g \in L^\infty(\mathbb{T}\times\mathbb{R}\times\mathbb{T}),
			\]
			where $\langle\nabla\rangle = (1-\Delta)^{1/2}$. Then
			\[
			\big\| [\langle\nabla\rangle^3, f] g \big\|_{L^2} \lesssim \|\nabla f\|_{L^\infty} \|\langle\nabla\rangle^2 g\|_{L^2} + \|\nabla\langle\nabla\rangle^2 f\|_{L^2} \|g\|_{L^\infty}.
			\]
		\end{enumerate}
	\end{Lem}

	We need to study the following commutators
	\begin{Prop}\label{est:com}
		For the Fourier multiplier $M$ and $m$ defined in \eqref{def-M} and Def \ref{Def-m}, one has
		\begin{align}
			&m(t,k,\eta,l)\lesssim\langle\eta-\eta',l-l'\rangle^2 m(t,k,\eta',l'),\label{est:m}\\
			&\Big|m(t,k,\eta,l)-m(t,k,\eta',l')\Big|\big|k\big|\lesssim m(k,\eta',l')\big|\eta-\eta',l-l'\big|\langle\eta-\eta',l-l'\rangle,\label{est-m2}\\
			&\Big|m^{\frac{1}{4}}(t,k,\eta,l)-m^{\frac{1}{4}}(t,k,\eta',l')\Big|\big|k\big|\lesssim m^{\frac{1}{4}}(t,k,\eta',l')\big|\eta-\eta',l-l'\big|\langle\eta-\eta',l-l'\rangle\label{est:m4},\\
			&\Big|M(t,k,\eta,l)-M(t,k,\eta',l')\Big|\big|k\big|\lesssim\big|\eta-\eta',l-l'\big|.
		\end{align}
	\end{Prop}
	\begin{proof}
		Here, we only prove \cref{est-m2} and  \cref{est:m4}. The proofs for the remaining estimates are classical; see, for instance, \cite{Bedrossian-Germain-Masmoudi-2019,BWV}. Denote $$H = \eta/k,\ H' = \eta'/k,\ L = l/k,\ L' = l'/k,\ \bar{H} = H-H',\ \bar{L} = L-L'.$$ 
		Then there exists a constant $C>0$, independent of 
		$k,\eta,\eta',l,l',\nu$, and $t$, such that
		\[
		\left| \frac{m(k,\eta,l)}{m(k,\eta',l')} - 1 \right| \leq C \; |\bar{H},\bar{L}| \; \langle \bar{H},\bar{L} \rangle .
		\tag{A.1}
		\]
		We proceed by case analysis according to the signs of $H, H'$ and their relations with $t$ and $\nu^{-1/3}$. 
		The proof consists of two parts, where we treat the easiest cases in Part 1 and leave the more complicated ones in Part 2.
		
		\noindent\textbf{Part 1. Two special cases.}
		
		\begin{enumerate}
			\item[1.] $H > 0,\ H' > 0,\ t > H + 1000\nu^{-1/3}$ and $t > H' + 1000\nu^{-1/3}$.  
			In this case, $m$ and $m'$ are independent of $\eta,\eta'$. A direct computation gives
			\[
			\frac{m}{m'} = \frac{1+L^2}{1+(L')^2} \cdot \frac{1+(1000\nu^{-1/3})^2+(L')^2}{1+(1000\nu^{-1/3})^2+L^2}.
			\]
			Hence
			\[
			\Bigl|\frac{m}{m'}-1\Bigr| \lesssim \abs{\bar{L}}\la \bar{L} \ra \lesssim \frac{1}{k}|l-l'|\langle l-l'\rangle.
			\]
			
			\item[2.]$H > 0,\ H' > 0,\ t > H' + 1000\nu^{-1/3}$ and $H+1000\nu^{-1/3} > t > H$.  
			Here, $m$ depends on $t-H$ while $m'$ does not. Writing the ratio explicitly, we obtain
			\[
			\Bigl|\frac{m}{m'}-1\Bigr| \lesssim |\bar{H},\bar{L}|\langle\bar{H},\bar{L}\rangle\lesssim\frac{1}{k}|\eta-\eta',l-l'|\langle \eta-\eta',l-l'\rangle.
			\]
		\end{enumerate}

		\noindent\textbf{Part 2. The remaining cases.}
		
		Now we consider the general situation. For clarity, we first treat the case $L = L'$; the case $L \neq L'$ is treated similarly. According to the definition of $m$, its expression depends on the intervals where $H$ and $t$ lie, which we divide into five main cases. In each subcase, we compute the difference $\frac{m}{m'}-1$ explicitly and show it is controlled by $|\bar{H}|$ or $|\bar{H}|^2$. The details are as follows.
		
		\begin{enumerate}
			\item[1.] $H>0$, $t>H+\nu^{-1/3}$, $m(\eta)=\dfrac{1+L^2}{1+\nu^{-2/3}+L^2}$. \\
			\begin{itemize}
				\item[1.1] $H'>0$, $t>H'+\nu^{-1/3}$: $m(\eta')$ has the same form. The difference is
				\[
				\frac{m}{m'}-1 = \frac{(H^2-(H')^2)(1+L^2)}{(1+\nu^{-2/3}+L^2)(1+(H')^2+L^2)} = \frac{\bar{H}(H+H')(1+L^2)}{(1+\nu^{-2/3}+L^2)(1+(H')^2+L^2)},
				\]
				which is bounded by $|\bar{H}|(1+|\bar{H}|)$.
				\item[1.2] $H'<-\nu^{-1/3}$: $m(\eta')=1$. The difference is
				\[
				\frac{m}{m'}-1 = \frac{-\nu^{-2/3}}{1+\nu^{-2/3}+L^2},
				\]
				and since $\nu^{-1/3} \leq |\bar{H}|$, this is bounded by $|\bar{H}|^2$.
			\end{itemize}
			
			\item[2.] $H>0$, $H<t<H+\nu^{-1/3}$, $m(\eta)=\dfrac{1+L^2}{1+(t-H)^2+L^2}$. \\
			\begin{itemize}
				\item[2.1] $-\nu^{-1/3}<H'<0$, $t>H'+\nu^{-1/3}$: $m(\eta')=\dfrac{1+(H')^2+L^2}{1+\nu^{-2/3}+L^2}$. The difference is
				\[
				\frac{m}{m'}-1 = \frac{[\nu^{-2/3}-(t-H)^2](1+L^2) + (t-H)^2(H^2-(H')^2)}{(1+(t-H)^2+L^2)(1+\nu^{-2/3}+L^2)}.
				\]
				Using $|\nu^{-1/3}+H-t| \leq |\bar{H}|$ and $|H'| \leq |\bar{H}|$, we obtain the bound $|\bar{H}|(1+|\bar{H}|)$.
				\item[2.2] $-\nu^{-1/3}<H'<0$, $t<H'+\nu^{-1/3}$: $m(\eta')=\dfrac{1+(H')^2+L^2}{1+(t-H')^2+L^2}$. The difference is
				\[
				\frac{m}{m'}-1 = \frac{[(t-H')^2-(t-H)^2](1+L^2) + (t-H)^2(H^2-(H')^2)}{(1+(t-H)^2+L^2)(1+(t-H')^2+L^2)}.
				\]
				Since $(t-H')^2-(t-H)^2 = \bar{H}(\bar{H}+2(t-H))$, we obtain the bound $|\bar{H}|(1+|\bar{H}|)$.
				\item[2.3] $H'<-\nu^{-1/3}$: $m(\eta')=1$. The difference is
				\[
				\frac{m}{m'}-1 = \frac{-(t-H)^2}{1+(t-H)^2+L^2},
				\]
				and since $|t-H| \leq \nu^{-1/3} \leq |\bar{H}|$, this is bounded by $|\bar{H}|^2$.
			\end{itemize}
			
			\item[3.] $H>0$, $t<H$, $m(\eta)=1$. \\
			\begin{itemize}
				\item[3.1] $-\nu^{-1/3}<H'<0$, $t>H'+\nu^{-1/3}$: $m(\eta')=\dfrac{1+(H')^2+L^2}{1+\nu^{-2/3}+L^2}$. The difference is
				\[
				\frac{m}{m'}-1 = \frac{\nu^{-2/3}-(H')^2}{1+\nu^{-2/3}+L^2},
				\]
				which is bounded by $|\bar{H}|^2$ since $\nu^{-1/3}, |H'| \leq |\bar{H}|$.
				\item[3.2] $-\nu^{-1/3}<H'<0$, $t<H'+\nu^{-1/3}$: $m(\eta')=\dfrac{1+(H')^2+L^2}{1+(t-H')^2+L^2}$. The difference is
				\[
				\frac{m}{m'}-1 = \frac{t(t-2H')}{1+(t-H')^2+L^2},
				\]
				and since $t \leq H \leq |\bar{H}|$, this is bounded by $|\bar{H}|^2$.
				\item[3.3] $H'<-\nu^{-1/3}$: $m(\eta')=1$. The difference is $0$.
			\end{itemize}
			
			\item[4.] $-\nu^{-1/3}<H<0$, $t>H+\nu^{-1/3}$, $m(\eta)=\dfrac{1+H^2+L^2}{1+\nu^{-2/3}+L^2}$. \\
			\begin{itemize}
				\item[4.1] $-\nu^{-1/3}<H'<0$, $t>H'+\nu^{-1/3}$: $m(\eta')$ has the same form. The difference is
				\[
				\frac{m}{m'}-1 = \frac{(H^2-(H')^2)(1+\nu^{-2/3}+L^2)}{(1+\nu^{-2/3}+L^2)^2} = \frac{\bar{H}(H+H')}{1+\nu^{-2/3}+L^2},
				\]
				which is bounded by $|\bar{H}|(1+|\bar{H}|)$.
				\item[4.2] $-\nu^{-1/3}<H'<0$, $t<H'+\nu^{-1/3}$: $m(\eta')=\dfrac{1+(H')^2+L^2}{1+(t-H')^2+L^2}$. The difference is
				\[
				\frac{m}{m'}-1 = \frac{[\nu^{-2/3}-(t-H')^2](1+H^2+L^2) + (t-H')^2(H^2-(H')^2)}{(1+\nu^{-2/3}+L^2)(1+(t-H')^2+L^2)}.
				\]
				Using $|\nu^{-1/3}+H'-t| \leq |\bar{H}|$ and $|H^2-(H')^2| \leq |\bar{H}|(1+|\bar{H}|)$, we obtain the bound $|\bar{H}|(1+|\bar{H}|)$.
				\item[4.3] $H'<-\nu^{-1/3}$: $m(\eta')=1$. The difference is
				\[
				\frac{m}{m'}-1 = \frac{H^2-\nu^{-2/3}}{1+\nu^{-2/3}+L^2},
				\]
				and since $|\nu^{-1/3}+H| \leq |\bar{H}|$, this is bounded by $|\bar{H}|^2$.
			\end{itemize}
			
			\item[5.] $-\nu^{-1/3}<H<0$, $t<H+\nu^{-1/3}$, $m(\eta)=\dfrac{1+H^2+L^2}{1+(t-H)^2+L^2}$. \\
			\begin{itemize}
				\item[5.1] $-\nu^{-1/3}<H'<0$, $t<H'+\nu^{-1/3}$: $m(\eta')$ has the same form. The difference is
				\[
				\frac{m}{m'}-1 = \frac{[(t-H')^2-(t-H)^2](1+H^2+L^2) + (t-H)^2(H^2-(H')^2)}{(1+(t-H)^2+L^2)(1+(t-H')^2+L^2)}.
				\]
				Since $(t-H')^2-(t-H)^2 = \bar{H}(\bar{H}+2(t-H))$, we obtain the bound $|\bar{H}|(1+|\bar{H}|)$.
				\item[5.2] $H'<-\nu^{-1/3}$: $m(\eta')=1$. The difference is
				\[
				\frac{m}{m'}-1 = \frac{H^2-(t-H)^2}{1+(t-H)^2+L^2} = \frac{t(2H-t)}{1+(t-H)^2+L^2},
				\]
				and since $t < H+\nu^{-1/3} \leq |\bar{H}|$, this is bounded by $|\bar{H}|$.
			\end{itemize}
		\end{enumerate}
		Collecting the above, we have shown that
		\begin{align}
			\Bigl| \frac{m}{m'} - 1 \Bigr| \lesssim |\bar{H}| (1 + |\bar{H}|) \leq \frac{1}{k} |\eta-\eta'| \, \langle \eta-\eta' \rangle. \tag{A.2}
		\end{align}
		When $L \neq L'$, the expressions contain additional terms of the form $L^2 - (L')^2 = (L-L')(L+L')$, which are bounded by $\frac{1}{k}|l-l'|\langle l-l'\rangle$. Combining this with (A.2) yields
		\[
		\Bigl| \frac{m}{m'} - 1 \Bigr| \lesssim \bigl( |\bar{H}| + |\bar{L}| \bigr) \bigl( 1 + |\bar{H}| + |\bar{L}| \bigr) \leq C \, |\bar{H},\bar{L}| \; \langle \bar{H},\bar{L} \rangle ,
		\]
				proving the first inequality.
		
		Turning to \cref{est:m4}, noting the fact that
		\begin{align*}
			\left|\frac{m^{\frac{1}{4}}(t,k,\eta,l)}{m^{\frac{1}{4}}(t,k,\eta',l')}-1\right|=\left|\frac{m(t,k,\eta,l)}{m(t,k,\eta',l')}-1\right|\left(\frac{m^{\frac{1}{4}}(t,k,\eta,l)}{m^{\frac{1}{4}}(t,k,\eta',l')}+1\right)^{-1}\left(\frac{m^{\frac{1}{2}}(t,k,\eta,l)}{m^{\frac{1}{2}}(t,k,\eta',l')}+1\right)^{-1},
		\end{align*}
		the estimate \cref{est:m4} is a direct result of \cref{est-m2}, concluding the proof.
	\end{proof}
	\section{Bootstrap argument}
	To prove \cref{MT}, we need to use a bootstrap argument. By a standard local well-posedness argument, we state the following lemma without showing more details.
	\begin{Lem}\label{local-exis-1}
		Under the same  assumptions as \cref{MT}, there exists a small constant $t_0>0$ independent of $\nu$ such that if $\|(\phi,v)_{in}\|_{H^5}\leqslant\varepsilon$, then there holds
		\begin{align*}
			\sup_{t\in[0,2t_0]}\|(R(t),V(t))\|_{H^5}\leqslant2\varepsilon.
		\end{align*}
	\end{Lem}
	From now on, all time norms are taken over the interval $[0,T]$ unless otherwise stated.\\
	\textbf{Bootstrap hypotheses}:  Fix $B_j=\prod_{i=1}^{j}C_i$ with $j=0,1,2,3$ large constants determined by the proof below and let $T>t_0$ be the maximal time such that the following estimates hold on $[0,T]$:\\
	\emph{the incompressible part}:
	\begin{subequations}\label{bs-ip}
		\begin{align}
			&\|mMp^{\frac{j}{2}}W^1_{\neq}\|_{L^\infty H^{3-j}}+\nu^{1/2}\|\tilde{\nabla}mMp^{\frac{j}{2}}W^1_{\neq}\|_{L^2 H^{3-j}}\notag\\
			&\quad+\nu^{1/6}\|mMp^{\frac{j}{2}}W^1_{\neq}\|_{L^2 H^{3-j}}+\|mM\partial_Xp^{\frac{j-1}{2}}W^1_{\neq}\|_{L^2 H^{3-j}}\leq 10 C_0\nu^{-1/3} B_j\nu^{-j/3}\varepsilon;\label{bs-w1}\\
			&\|\ma p^{\frac{j}{2}}W^2\|_{L^\infty H^{3-j}}+\nu^{1/2}\|\tilde{\nabla}\ma p^{\frac{j}{2}}W^2\|_{L^2 H^{3-j}}\notag\\
			&\quad+\nu^{1/6}\|\ma p^{\frac{j}{2}}W^2_{\neq}\|_{L^2 H^{3-j}}+\|\ma\partial_Xp^{\frac{j-1}{2}}W^2_{\neq}\|_{L^2 H^{3-j}}\leq 10 B_j\nu^{-j/3}\varepsilon;\label{bs-w2}\\
			&\|mMp^{\frac{j}{2}}\Omega^3\|_{L^\infty H^{3-j}}+\nu^{1/2}\|\tilde{\nabla}mMp^{\frac{j}{2}}\Omega^3\|_{L^2 H^{3-j}}\notag\\
			&\quad+\nu^{1/6}\|mMp^{\frac{j}{2}}\Omega^3_{\neq}\|_{L^2 H^{3-j}}+\|mM\partial_Xp^{\frac{j-1}{2}}\Omega^3_{\neq}\|_{L^2 H^{3-j}}\leq 10 B_j\nu^{-j/3}\varepsilon\label{bs-o3};
		\end{align}	
	\end{subequations}
	\emph{the compressible part}:
	\begin{subequations}\label{bs-cp}
		\begin{align}
			&\|{\de}_0,p^{-\frac{1}{2}}D_0\|_{L^\infty H^3}+\nu^{1/2}\|\nabla {\de}_0\|_{L^2 H^2}+\nu^{1/2}\|D_0\|_{L^2H^3}\leq 10\varepsilon;\label{bs-n0}\\
			&\|\partial_Z{\de}_0,\partial_Zp^{-\frac{1}{2}}D_0\|_{L^\infty H^3}+\nu^{1/2}\|\partial_Z{\de}_0\|_{L^2 H^3}+\nu^{1/2}\|\partial_ZD_0\|_{L^2H^3}\leq 10 \varepsilon;\label{bs-zn0}\\
			&\|\ma \nabla_{X,Z}p^{\frac{j}{2}}({\de},p^{-\frac{1}{2}}D)_{\neq}\|_{L^\infty H^{3-j}}+\nu^{1/2}\|\tilde{\nabla}\ma\nabla_{X,Z}p^{\frac{j}{2}}({\de},p^{-\frac{1}{2}}D)_{\neq}\|_{L^2 H^{3-j}}\notag\\
			&\quad+\nu^{1/6}\|\ma \nabla_{X,Z}p^{\frac{j}{2}}({\de},p^{-\frac{1}{2}}D)_{\neq}\|_{L^2 H^{3-j}}+\|\ma\nabla_{X,Z}\partial_Xp^{\frac{j-1}{2}}({\de},p^{-\frac{1}{2}}D)_{\neq}\|_{L^2 H^{3-j}}\leq 10 B_j\nu^{-j/3}\varepsilon;\label{bs-xn}\\
			&\|p^{\frac{j}{2}}(p^{\frac{1}{2}}{\de},D)\|_{L^\infty H^{3-j}}+\nu^{1/2}\|p^{\frac{1}{2}}{\de}_0\|_{L^2 H^{3}}+\nu^{1/2}\|\nabla D_0\|_{L^2 H^{3}}\notag\\
			&\quad+\nu^{1/6}\|p^{\frac{j}{2}}p^{\frac{1}{2}}{\de}_{\neq}\|_{L^2 H^{3-j}}+\nu^{1/2}\|p^{\frac{j}{2}}\tilde{\nabla} D_{\neq}\|_{L^2 H^{3-j}}\leq 10 C_0\nu^{-1/2} B_j\nu^{-j/3}\varepsilon;\label{bs-yn}\\
			&\|p^{\frac{j}{2}}\partial_X(p^{\frac{1}{2}}{\de},D)_{\neq}\|_{L^\infty H^{3-j}}
			+\nu^{1/6}\|p^{\frac{j}{2}}\partial_Xp^{\frac{1}{2}}{\de}_{\neq}\|_{L^2 H^{3-j}}+\nu^{1/2}\|p^{\frac{j}{2}}\partial_X\tilde{\nabla} D_{\neq}\|_{L^2 H^{3-j}}\leq 10 C_0\nu^{-1/2} B_j\nu^{-j/3}\varepsilon;\label{bs-xyn}\\
			&\|p^{\frac{j}{2}}\partial_Z(p^{\frac{1}{2}}{\de},D)\|_{L^\infty H^{3-j}}+\nu^{1/2}\|p^{\frac{1}{2}}\partial_Z{\de}_0\|_{L^2 H^{3}}+\nu^{1/2}\|\nabla \partial_ZD_0\|_{L^2 H^{3}}\notag\\
			&\quad+\nu^{1/6}\|p^{\frac{j}{2}}\partial_Zp^{\frac{1}{2}}{\de}_{\neq}\|_{L^2 H^{3-j}}+\nu^{1/2}\|p^{\frac{j}{2}}\tilde{\nabla} \partial_ZD_{\neq}\|_{L^2 H^{3-j}}\leq 10 C_0\nu^{-1/2} B_j\nu^{-j/3}\varepsilon;\label{bs-zyn}\\
			&\|p^{\frac{j}{2}}(p{\de},p^{\frac{1}{2}}D)\|_{L^\infty H^{3-j}}+\nu^{1/2}\|p{\de}_0\|_{L^2 H^{3}}+\nu^{1/2}\|\nabla p^{\frac{1}{2}}D_0\|_{L^2 H^{3}}\notag\\
			&\quad+\nu^{1/6}\|p^{\frac{j}{2}}p{\de}_{\neq}\|_{L^2 H^{3-j}}+\nu^{1/2}\|p^{\frac{j}{2}}\tilde{\nabla} p^{\frac{1}{2}}D_{\neq}\|_{L^2 H^{3-j}}\leq 10 C_0^2\nu^{-5/6} B_j\nu^{-j/3}\varepsilon\label{bs-pn};
		\end{align}
	\end{subequations}
	\emph{zero mode of $V^1,V^3$}:
	\begin{subequations}\label{bs-zm}
		\begin{align}
			&\|U^1_{00}\|_{L^\infty H^4}+\nu^{1/2}\|\partial_YU^1_{00}\|_{L^2H^4}\leq 10\nu^{-1/2}\varepsilon;\label{bs-u100}\\
			&\|\Delta V^1_{0\neq}\|_{L^\infty H^3}+\nu^{1/2}\|\nabla \Delta V^1_{0\neq}\|_{L^2H^3}\leq 10\nu^{-1}C_0\varepsilon;\label{bs-pv10}\\
			&\|V^3_{00}\|_{L^\infty H^3}+\nu^{1/2}\|\partial_YV^3_{00}\|_{L^2H^3}\leq 10\varepsilon,\label{bs-v300}		
		\end{align}
	\end{subequations}
	where $\nabla_{X,Z}=(\partial_X,\partial_Z)$. 
	
	Next, we close the bootstrap hypotheses in the
	 Proposition \ref{prop-bs} below.
	\begin{Prop}\label{prop-bs}
		Under the same  assumptions as \cref{MT} and bootstrap hypotheses, 
		the same estimates \cref{bs-ip}-\cref{bs-zm} hold with all the constants on the right-hand side divided by $2$ on $[0,T]$. 
	\end{Prop}
	
	Before proving  Proposition \ref{prop-bs}, let us briefly comment on the structure of bootstrap hypotheses. Due to the lift-up effect term, there is an amplification of order $\nu^{-1/3}$ in the right-hand side of \cref{bs-w1}. The nonlinear interaction involving the lift-up of $V^1_{0\neq}$ leads to an amplification for $W^2_{\neq}$ near the critical time, which is quantified by the inclusion of the multiplier $m^{\frac{1}{4}}$ in the norm of \cref{bs-w2}. It is worth noting that $V^2_{\neq}$ is dominated by the compressible part $\tilde{\partial}_Yp^{-1}D_{\neq}$, so this setting does not affect the main estimate. For the third component, we need to consider a low order estimate \cref{bs-v300}, since $\Omega^3$ does not contain $V^3_{00}$. Turning to the compressible part, we bring any derivative estimate of the symmetric variables $(R,p^{-\frac{1}{2}}D)$ up to $(pR,p^{\frac{1}{2}}D)$, which has the same order as the incompressible part. To compensate for the absence of dissipation for $pR$, motivated by \cite{LWZ}, we use a high regularity argument, i.e., consider the estimates for $j=1,2,3$ simultaneously. Finally, for the estimate of the zero mode of the first component, $V^1_{0\neq}$ exhibits a standard lift-up effect, which is consistent with the linear result in \cite{Zeng-Zhang-Zi-2022}, and the $\nu^{-1/2}$ order growth of $U^1_{00}$ stems from this effect.
	
	Immediately, we can obtain following estimates for $j=0$ from bootstrap hypotheses, and the similar results hold for $j=1,2,3$.
	\begin{Prop}\label{est:main}
		Under the bootstrap hypotheses, the following estimates hold:\\
		1. double zero mode estimates:
		\begin{align*}
			&\|{\de}_{00}\|_{L^\infty H^3}+\nu^{1/2}\|\partial_Y {\de}_{00}\|_{L^\infty H^3}+\nu^{5/6}\|\partial^2_{Y} {\de}_{00}\|_{L^\infty H^3}\\
			&\quad+\nu^{1/2}(\|\partial_Y{\de}_{00}\|_{L^2H^2}+\nu^{1/2}\|\partial_Y{\de}_{00}\|_{L^2H^3}+\nu^{5/6}\|\partial^2_{Y}{\de}_{00}\|_{L^2H^3})\lesssim\varepsilon,\\
			&\|\partial_YV^1_{00}\|_{L^\infty H^4}+\nu^{1/2}\|\partial^2_{Y}V^1_{00}\|_{L^2H^3}+\nu^{5/6}\|\partial^3_{Y}V^1_{00}\|_{L^2 H^3}\lesssim\nu^{-1/2}\varepsilon,\\
			&\|V^2_{00}\|_{L^\infty H^3}+\nu^{1/2}\|V^2_{00}\|_{L^\infty H^4}+\nu^{5/6}\|V^2_{00}\|_{L^\infty H^5} \\
			&\quad+\nu^{1/2}(\|\partial_Y V^2_{00}\|_{L^2 H^3}+\nu^{1/2}\|\partial_Y V^2_{00}\|_{L^2 H^4}+\nu^{5/6}\|\partial_Y V^2_{00}\|_{L^2 H^5})\lesssim\varepsilon, \\
			&\|V^3_{00}\|_{L^\infty H^5}+\nu^{1/2}\|\partial_YV^3_{00}\|_{L^2H^5}\lesssim\varepsilon;
		\end{align*}
		2. simple zero modes estimates:
		\begin{align*}
			&\|\partial_Z{\de}_{0\neq}\|_{L^\infty H^3}+\nu^{1/2}\|\partial_Z\nabla {\de}_{0\neq}\|_{L^\infty H^3}+\nu^{5/6}\|\Delta {\de}_{0\neq}\|_{L^\infty H^3}\\&\quad+\nu^{1/2}(\|\partial_Z{\de}_{0\neq}\|_{L^2H^3}+\nu^{1/2}\|\partial_Z\nabla {\de}_{0\neq}\|_{L^2H^3}+\nu^{5/6}\|\Delta {\de}_{0\neq}\|_{L^2H^3})\lesssim\varepsilon,\\
			&\|V^1_{0\neq}\|_{L^\infty H^5}+\nu^{1/2}\|\nabla V^1_{0\neq}\|_{L^2 H^5}\lesssim\nu^{-1}\varepsilon,\\
			&\|\partial_Z V^2_{0\neq}\|_{L^\infty H^3}+\nu^{1/2}\|\partial_ZV^2_{0\neq}\|_{L^\infty H^4}+\nu^{1/2}(\|\nabla \partial_Z V^2_{0\neq}\|_{L^2H^3}+\nu^{1/2}(\|\nabla\partial_ZV^2_{0\neq}\|_{L^2H^4}))\lesssim\varepsilon,\\
			&\| V^3_{0\neq}\|_{L^\infty H^4}+\nu^{1/2}\|V^3_{0\neq}\|_{L^\infty H^5}+\nu^{1/2}(\|V^3_{0\neq}\|_{L^2 H^5}+\nu^{1/2}\|\nabla V^3_{0\neq}\|_{L^2 H^5})\lesssim\varepsilon;
		\end{align*}
		3. nonzero modes estimates:
		\begin{align*}			
			&\|m^{\frac{1}{4}}\nabla_{X,Z}{\de}_{\neq}\|_{L^\infty H^3}+\nu^{1/2}\|\nabla_{X,Z}p^{\frac{1}{2}}{\de}_{\neq}\|_{L^\infty H^3}+\nu^{5/6}\|p{\de}_{\neq}\|_{L^\infty H^3}+\|m^{\frac{1}{4}}\partial_Xp^{-\frac{1}{2}}\nabla_{X,Z}{\de}_{\neq}\|_{L^2H^3}\\&\quad+\nu^{1/6}\|m^{\frac{1}{4}}\nabla_{X,Z}{\de}_{\neq}\|_{L^2H^3}+\nu^{2/3}\|\nabla_{X,Z}p^{\frac{1}{2}}{\de}_{\neq}\|_{L^2H^3}+\nu\|p{\de}_{\neq}\|_{L^2H^3}\lesssim\varepsilon,\\
			&\|\nabla_{X,Z}V^1_{\neq}\|_{L^\infty H^3}+\nu^{1/6}\|m^{\frac{1}{4}}p^{\frac{1}{2}}V^1_{\neq}\|_{L^\infty H^3}+\nu^{2/3}\|\nabla_{X,Z}p^{\frac{1}{2}}V^1_{\neq}\|_{L^\infty H^3}+\nu\|pV^1_{\neq}\|_{L^\infty H^3}\\&\quad+\nu^{1/6}\|\nabla_{X,Z}V^1_{\neq}\|_{L^2H^3}+\nu^{1/2}\|\nabla_{X,Z}p^{\frac{1}{2}}V^1_{\neq}\|_{L^2H^3}+\nu^{1/2}\|m^{\frac{1}{2}}pV^1_{\neq}\|_{L^2H^3}\lesssim\varepsilon,\\
			&\|m^{\frac{1}{4}}\nabla_{X,Z}V^2_{\neq}\|_{L^\infty H^3}+\nu^{1/2}\|\nabla_{X,Z}p^{\frac{1}{2}}V^2_{\neq}\|_{L^\infty H^3}+\nu^{5/6}\|pV^2_{\neq}\|_{L^\infty H^3}\\
			&\quad+\nu^{1/6}\|m^{\frac{1}{4}}\nabla_{X,Z}V^2_{\neq}\|_{L^2H^3}+\nu^{1/2}\|\nabla_{X,Z}^2V^2_{\neq}\|_{L^2H^3}+\nu^{1/2}\|m^{\frac{1}{4}}\nabla_{X,Z}p^{\frac{1}{2}}V^2_{\neq}\|_{L^2H^3}+\nu\|pV^2_{\neq}\|_{L^2H^3}\lesssim\varepsilon,\\
			&\|\nabla_{X,Z}V^3_{\neq}\|_{L^\infty H^3}+\|m^{\frac{1}{2}}p^{\frac{1}{2}}V^3_{\neq}\|_{L^\infty H^3}+\nu^{1/2}\|\nabla_{X,Z}p^{\frac{1}{2}}V^3_{\neq}\|_{L^\infty H^3}+\nu^{2/3}\|pV^3_{\neq}\|_{L^\infty H^3}\\
			&\quad+\nu^{1/6}\|\nabla_{X,Z}V^3_{\neq}\|_{L^2H^3}+\nu^{1/2}\|\nabla_{X,Z}p^{\frac{1}{2}}V^3_{\neq}\|_{L^2H^3}+\nu^{1/2}\|m^{\frac{1}{2}}pV^3_{\neq}\|_{L^2H^3}\lesssim\varepsilon.
		\end{align*}
	\end{Prop}	
	\begin{proof}
		For $\partial_YV^1_{00}$, recalling the definition of $U^1_{00}$, we have
		\begin{align*}
			\|\partial^2_{Y}V^1_{00}\|_{L^\infty H^3}=&\|\partial_YU^1_{00}+\partial_Y{\de}_{00}-\nu\partial_YD_{00}\|_{L^\infty H^3}\lesssim\nu^{-1/2}\varepsilon,\\
			\|\partial^2_{Y}V^1_{00}\|_{L^2 H^3}=&\|\partial_{Y}U^1_{00}+\partial_{Y}{\de}_{00}-\nu\partial_{Y}D_{00}\|_{L^2 H^3}\\
			\lesssim&\nu^{-1}\varepsilon+\nu^{-1}\varepsilon+\nu\nu^{-1}\varepsilon\lesssim\nu^{-1}\varepsilon,\\
			\|\partial^3_{Y}V^1_{00}\|_{L^2 H^3}=&\|\partial^2_{Y}U^1_{00}+\partial^2_{Y}{\de}_{00}-\nu\partial^2_{Y}D_{00}\|_{L^2 H^3}\\
			\lesssim&\nu^{-1}\varepsilon+\nu^{-4/3}\varepsilon+\nu\nu^{-4/3}\varepsilon\lesssim\nu^{-4/3}\varepsilon.
		\end{align*}
		For $V^3_{0\neq}$, noting the bounds of the compressible part, we get
		\begin{align*}
			\|V^3_{0\neq}\|_{L^\infty H^4}=&\|\partial_Zp^{-1}D_{0\neq}+p^{-1}\Omega^3_{0\neq}\|_{L^\infty H^4}\\
			\lesssim&\|\partial_Zp^{-\frac{1}{2}}D\|_{L^\infty H^3}+\|\Omega^3_{0\neq}\|_{L^\infty H^3}\lesssim\varepsilon,\\
			\|V^3_{0\neq}\|_{L^\infty H^5}=&\|\partial_Zp^{-1}D_{0\neq}+p^{-1}\Omega^3_{0\neq}\|_{L^\infty H^5}\\
			\lesssim&\|\partial_ZD\|_{L^\infty H^3}+\|\Omega^3_{0\neq}\|_{L^\infty H^3}\lesssim\nu^{-1/2}\varepsilon.
		\end{align*}
		For $pV^1_{\neq}$, noting that
		\begin{equation*}
			V^1_{\neq}=-\partial_Xp^{-1}D_{\neq}+\tilde{\partial}_Y\partial_X^{-1}p^{-1}W^2_{\neq}-\tilde{\partial}_Yp^{-1}{\de}_{\neq}+\nu \tilde{\partial}_Yp^{-1}D_{\neq}+\partial_Z\partial_X^{-1}p^{-1}\Omega^3_{\neq},
		\end{equation*}
		we deduce that
		\begin{align*}
			\|m^{\frac{1}{2}}pV^1_{\neq}\|_{L^2H^3}\lesssim&\|p^{\frac{1}{2}}m^{\frac{1}{4}}\partial_Xp^{-\frac{1}{2}}D_{\neq}\|_{L^2H^3}+\|m^{\frac{1}{4}}\tilde{\partial}_Y(W^2_{\neq},{\de}_{\neq})\|_{L^2H^3}+\nu\|\partial_YD_{\neq}\|_{L^2H^3}+\|p^{\frac{1}{2}}m\Omega^3_{\neq}\|_{L^2H^3}\\
			\lesssim&\nu^{-1/2}\varepsilon.
		\end{align*}
		For $\nabla_{X,Z}^2V^2_{\neq}$, it holds that
		\begin{align*}
			\|\nabla_{X,Z}^2V^2_{\neq}\|_{L^2H^3}\lesssim\|\nabla_{X,Z}^2p^{-\frac{1}{2}}D_{\neq}\|_{L^2H^3}+\|\nabla_{X,Z}^2p^{-1}W^2_{\neq}\|_{L^2H^3}\lesssim\nu^{-1/2}\varepsilon.
		\end{align*}
		The remaining estimates follow from a similar argument, and we omit the details.
	\end{proof}
	Finally, we prove the main result.
    \begin{proof}[Proof of \cref{MT}]
        A continuity argument implies that $T=+\infty$ by combining the local existence of the solution (Lemma \ref{local-exis-1}) with  Proposition \ref{prop-bs}. Then \cref{MT} is a direct corollary of this, together with the uniform bound of $V^1_{00}$ from  \cref{2026-3-21-2}.
    \end{proof}

	\section{Estimates on double zero mode}
	The main goal of this section is to give a uniform bounded estimate of $\norm{v_{00}^1}_{L_t^\infty L_x^\infty}$.
	To achieve this goal, we firstly give the estimates of  $\theta_i$, $i=1,2$, which will be used frequently.
	\begin{Lem}\label{estimateontheta}
		For the initial data \cref{initial},  it holds that for $i=1,2$, 
		\begin{align}
			\norm{\p_y^k \theta_i(\cdot,t)}_{L^p}\lesssim \varepsilon\left[\bar{\nu}(1+t)\right]^{-\frac{1}{2}+\frac{1}{2p}-\frac{k}{2}}, \quad  k\geq0,\ p\geq 1.
		\end{align}
	\end{Lem}
	\begin{proof}
		Note that $\abs{\eta_i}=O(1)\varepsilon$,  Lemma \ref{estimateontheta} can be directly derived by the formulas \cref{equ-theta,theta}. The proof is then omitted. 	
	\end{proof}
	
	\subsection{Estimates on coupled diffusion waves}
Note that the Green's function of the equation $v_t+\sigma v_x=\frac{\nu}{2}v_{xx} $ is  $$H(x,t;\sigma,\nu)=\frac{1}{\sqrt{2\pi \nu t}}e^{-\frac{(x-\sigma t)^2}{2\nu t}}.$$
	Let
	$H(\sigma,\mu):=H(y-s,t-\tau;\sigma,\mu).$ 
	By Duhamel's principle, one has from \cref{equ-Xi} that 
	\begin{align}\label{XI}
		\tilde{\Xi}_i(y, t)=-\int_0^t \int_{-\infty}^{\infty} H(\sigma_i,\frac{\bar{\nu}}{2}) \left(\theta_{i'}^2 / 2+\theta_1\Xi_1+\theta_2 \Xi_2+\frac{2-P''(1)}{2+P''(1)}(\theta_1\Xi_2+\theta_2\Xi_1)\right)(s, \tau) d s d \tau,
	\end{align}
	where $i'=3-i$ and $\sigma_i=(-1)^i $. By \cref{XI}, we have
	\begin{Lem}\label{estimateonxi}
		Under the same conditions of \cref{MT},  it holds that for $i=1,2$, 
		\begin{align}
			&\norm{\tilde{\Xi}_i(\cdot,t)}_{L^\infty}\lesssim \varepsilon^2 \nu^{-\frac12}(\bar{\nu}(1+t))^{-\frac{1}{4}},\nonumber\\
			&\norm{\p_y^{k+1}\tilde{\Xi}_i(\cdot,t)}_{L^p}\lesssim\varepsilon^2\left(\bar{\nu}(1+t)\right)^{-\frac{3}{4}+\frac{1}{2p}-\frac{k}{2}}, \quad  k\geq0,\ p\geq 1.
		\end{align}
	\end{Lem}
	\begin{proof}
		This proof is similar to that of \cite{HLX}, so we omit it here.
	\end{proof}
	
	\subsection{Estimate on $\tilde W$}
	This subsection is devoted to obtain $\norm{\Big(\tilde{W}_1(\cdot,t),\tilde{W}_2(\cdot,t)\Big)}_{L^2}$ and $\norm{\Big(\tilde{W}_1(\cdot,t),\tilde{W}_2(\cdot,t)\Big)}_{L^\infty}$.
	\begin{Lem}\label{1or}
		Under the same  assumptions as \cref{MT} and  Proposition \ref{est:main}, it holds that 
		\begin{align}
			\begin{aligned}
				&\norm{\left(\tilde{W}_1(\cdot,t),\tilde{W}_2(\cdot,t)\right)}^2_{L^2}+\bar{\nu}\int_0^t\norm{\Big(\p_y\tilde{W}_1,\p_y\tilde{W}_2\Big)}_{L^2}^2d\tau \lesssim \varepsilon^2, \qquad \norm{\Big(\tilde W_1, \tilde W_2\Big)}_{L^\infty}\lesssim\varepsilon \nu^{-\frac18}.
			\end{aligned}
		\end{align}
	\end{Lem}
	\begin{proof}
		Recalling \cref{tildeK,anti}, $\mathbf{\tilde{W}}$ satisfies
		\begin{align}\label{AA}
			\mathbf{\tilde{W}}_t+A\mathbf{\tilde{W}}_y=B\mathbf{\tilde{W}}_{yy}+\mathbf{\tilde{K}},
		\end{align}
		where
		\begin{align}
			A=\left(\begin{array}{cc}-1& 0 \\ 0 & 1\end{array}\right), \quad {B}=\left(\begin{array}{cc}\frac{\bar{\nu}}{2} & \frac{\bar{\nu}}{2} \\ \frac{\bar{\nu}}{2} & \frac{\bar{\nu}}{2}\end{array}\right),\quad \mathbf{\tilde{K}}=\left(\begin{array}{c}\tilde{K}_1\\ \tilde{K}_2\end{array}\right).
		\end{align}
		To capture the viscous effect, we denote $(\mathbb{W}_1,\mathbb{W}_2):=(\tilde{W}_2-\tilde{W}_1,\tilde{W}_2+\tilde{W}_1)$, then one has
		\begin{align}\label{hatv}
			\begin{cases}
				\mathbb{W}_{1t}+\mathbb{W}_{2y}=\tilde{K}_{2}-\tilde{K}_{1},\\
				\mathbb{W}_{2t}+\mathbb{W}_{1y}=\bar{\nu}\mathbb{W}_{2yy}+\tilde{K}_{1}+\tilde{K}_{2}.
			\end{cases}
		\end{align}
		From \cref{tildeK}, one has
		\begin{align}\label{K2-K1}
			\tilde{K}_2-\tilde{K}_1=&\frac{\bar{\nu}}{2}\p_y\left[\theta_1^2 / 2+\theta_2^2 / 2+\theta_1\Xi_1+\theta_2 \Xi_2+\frac{2-P''(1)}{2+P''(1)}(\theta_1\Xi_2+\theta_2\Xi_1) \right]\\ \nonumber
			&-\frac{\bar{\nu}^2}{8}\p_y^2(\theta_1+\theta_2+\Xi_1+\Xi_2),
		\end{align}
		which implies from  Lemma \ref{estimateontheta} and  Lemma \ref{estimateonxi} that, for $k\geq0$, 
		\begin{align}\label{estimateon2-1}
			\norm{\p_y^k\left(\tilde{K}_2-\tilde{K}_1\right)}_{L^1}\lesssim\varepsilon \nu^2\left[\bar{\nu}(1+t)\right]^{-1-\frac{k}{2}},\quad\norm{\p_y^{k}\left(\tilde{K}_2-\tilde{K}_1\right)}_{L^2}&\lesssim \varepsilon\bar{\nu}^{2}\left[\bar{\nu}(1+t)\right]^{-\frac{5}{4}-\frac{k}{2}}.
		\end{align}
		The estimate of $\tilde{K}_1+\tilde{K}_2$ is complicated.  The formulas \cref{tildeK} and \cref{E}  give that 
		\begin{align}\label{K2+K1}
			\abs{\tilde{K}_1+\tilde{K}_2-\frac{E_2}{a}}=O(1)\abs{\tilde{E}^{(1)}+\tilde{E}^{(2)}+\tilde{N}},
		\end{align}
		where 	$E_2=-(2\nu+\nu')\p_y\left(\frac{\varphi^2_{00}{\lde}_{00}}{\rho_{00}}\right)$, 
		\begin{align}\label{E2}
			\frac{\varphi^2_{00}{\lde}_{00}}{\rho_{00}}&=a^2 \frac{(\tilde{W}_{1y}+\tilde{W}_{2y}+\mathcal{C}_1+\mathcal{C}_2)(\tilde{W}_{2y}-\tilde{W}_{1y}+\mathcal{C}_2-\mathcal{C}_1)}{\tilde{W}_{2y}-\tilde{W}_{1y}+\mathcal{C}_2-\mathcal{C}_1+1}\\ \nonumber
			&=O(1)\sum_{i,j=1}^2 \left(\abs{\tilde{W}_{iy}\tilde{W}_{jy}}+ \abs{\tilde{W}_{iy}\theta_i}+\abs{\theta_i\theta_j}\right).
		\end{align}
		and
		\begin{align}\label{e1}
			\abs{\tilde{E}^{(1)}}=&O(1)\Bigg|\sum_{i=1}^2\left[\mathcal{C}_i^2-\theta_i^2-2\left(\theta_1\Xi_1+\theta_2 \Xi_2\right)\right.\\\notag
			&\left.\qquad\qquad+\frac{2-P''(1)}{2(2+P''(1))}\mathcal{C}_{i}\mathcal{C}_{i'}-\frac{2-P''(1)}{2+P''(1)}(\theta_1\Xi_2+\theta_2\Xi_1)\right]\\\notag
			&+\frac{\bar{\nu}}{4}\left(\theta_1^2 / 2+\theta_2^2 / 2+(\theta_1 +\theta_2)(\Xi_1+ \Xi_2)+\frac{\bar{\nu}}{2}\theta_{iy}+\frac{\bar{\nu}}{2}\Xi_{iy}\right)_y\Bigg|,\\\label{ee}
			\abs{\tilde{E}^{(2)}}=&O(1)\sum_{i,j=1}^2|\tilde{W}_{iy}\tilde{W}_{jy}+\tilde{W}_{iy}C_i|=O(1)\sum_{i,j=1}^2\left(|\tilde{W}_{iy}\tilde{W}_{jy}|+|\tilde{W}_{iy}\theta_{j}|+|\tilde{W}_{iy}\Xi_j|\right),\\\label{e2}
			\abs{\tilde{N}}=&O(1)\abs{\Doo(({\lde}_{\neq})^2+{\lde}_{\neq}v_{\neq}^2+({\lde}_{0\neq})^2+{\lde}_{0\neq}v_{0\neq}^2)}\notag\\
			&+O(\nu)\abs{\p_y\Doo(({\lde}_{\neq})^2+{\lde}_{\neq}v_{\neq}^2+({\lde}_{0\neq})^2+{\lde}_{0\neq}v_{0\neq}^2))}.
		\end{align}
		From \cref{tildev}, the worst terms in $\tilde{E}^{(1)}$ are $ \theta_i \cdot \p_y \theta_i$ and  $\Xi_1 \cdot \Xi_2$.
		By \cref{theta},   Lemma \ref{estimateontheta},  Lemma \ref{estimateonxi}, one has
		\begin{align}
			&\norm{\p_y^{k}\tilde{E}^{(1)}}_{L^1}\lesssim \varepsilon^2 \nu\left[\bar{\nu}(1+t)\right]^{-1-\frac{k}{2}},\quad \norm{\p_y^k\tilde{E}^{(1)}}_{L^2}\lesssim\varepsilon^2 \nu\left[\bar{\nu}(1+t)\right]^{-\frac{5}{4}-\frac{k}{2}}, \quad\text{for} \quad k=0,1,2.\label{2025-9-17-1}
		\end{align}
		\begin{flushleft}
			\textbf{Basic estimate}
		\end{flushleft}
		Multiplying \cref{hatv}$_1$ by $\mathbb{W}_1$ and  \cref{hatv}$_2$ by $\mathbb{W}_2$, adding them up, and  integrating over $(0,t)\times\R$, one has
		\begin{align}
			\begin{aligned}\nonumber
				\norm{(\mathbb{W}_1,\mathbb{W}_2)}^2_{L^2}+\bar{\nu}	\int_0^t \norm{\mathbb{W}_{2y}}_{L^2}^2d\tau&\lesssim \norm{\mathbb{W}_1(\cdot,0),\mathbb{W}_2(\cdot,0)}_{L^2}^2+ \abs{\int_0^T\int_{\R}(\tilde{K}_2-\tilde{K}_1)\mathbb{W}_1+(\tilde{K}_1+\tilde{K}_2)\mathbb{W}_2dydt}.
			\end{aligned}
		\end{align} 
		By \cref{estimateon2-1}, we have
		\begin{align}
			&	\int_0^T \int_\R \abs{(\tilde{K}_2-\tilde{K}_1)\mathbb{W}_1} dy dt
			\lesssim\int_0^T\norm{\tilde{K}_2-\tilde{K}_1}_{L^1}\norm{\mathbb{W}_1}_{L^2}^{\frac{1}{2}}\norm{\mathbb{W}_{1y}}_{L^2}^{\frac{1}{2}}dt\\ \nonumber
			\lesssim& C_\delta\int_0^T\bar{\nu}^{-\frac13}\norm{\tilde{K}_2-\tilde{K}_1}^{\frac{4}{3}}_{L^1}\norm{\mathbb{W}_1}_{L^2}^{\frac{2}{3}}+\delta\bar{\nu}\norm{\mathbb{W}_{1y}}_{L^2}^{2}dt
			\lesssim C_\delta\varepsilon^2 \nu+\int_0^T\delta\bar{\nu}\norm{\mathbb{W}_{1y}}^{2}_{L^2}dt.
		\end{align}
		Recall \cref{K2+K1},  $\abs{(\tilde{K}_1+\tilde{K}_2-\frac{E_2}{a})\mathbb{W}_2}\lesssim\abs{\left(\tilde{E}^{(1)}+\tilde{E}^{(2)}+\tilde{N}\right)\mathbb{W}_2} $. From \cref{2025-9-17-1} and  Proposition \ref{est:main}, we have
		\begin{align}\label{a}
			&	\int_0^T\int_R \abs{(\tilde{E}^{(1)}+\tilde{N})\mathbb{W}_2}dydt \lesssim\int_0^T\norm{\tilde{E}^{(1)}}_{L^1}\norm{\mathbb{W}_2}^{\frac{1}{2}}_{L^2}\norm{\mathbb{W}_{2y}}_{L^2}^{\frac{1}{2}}+\norm{\tilde{N}}_{L^1}\norm{\mathbb{W}_2}_{L^\infty}dt\nonumber\\ 
			\lesssim& C_\delta \bar{\nu}^{-\frac13} \int_0^T\norm{\tilde{E}^{(1)}}^{\frac{4}{3}}_{L^1}\norm{\mathbb{W}_2}_{L^2}^{\frac{2}{3}} d\tau+\delta\bar{\nu}\int_0^t\norm{\mathbb{W}_{2y}}_{L^2}^{2}dt + \norm{\mathbb{W}_2}_{L^\infty} \int_0^t \norm{\tilde N_1}_{L^1} d\tau  \nonumber\\
			\lesssim& C_\delta\varepsilon^{\frac83}\nu^{-\frac{1}{3}}+C_\delta\varepsilon^3\nu^{-\frac98}+\int_0^T\delta\bar{\nu}\norm{\mathbb{W}_{2y}}_{L^2}^{2}dt,
		\end{align}
		where we have used 
		\begin{align*}
			&\norm{\mathbb{W}_i}_{L^\infty}\lesssim \norm{\mathbb{W}_i}_{L^2}^{\frac12}\left(\norm{({\lde}_{00},v_{00}^2)}_{L^2}^{\frac12}+\sum_{j=1}^2\norm{\theta_j}_{L^2}^{\frac12} \right)\lesssim \varepsilon \nu^{-\frac18}, \\
			&\int_0^t\norm{\tilde N}_{L^1} d\tau \lesssim \int_0^t \norm{{\lde}_{\neq}}_{L^2}\norm{v^2_{\neq}}_{L^2}+\norm{{\lde}_{\neq}}_{L^2}^2+\norm{{\lde}_{0\neq}}_{L^2}^2+\norm{{\lde}_{0\neq}}_{L^2}\norm{v^2_{0\neq}}_{L^2}d\tau \lesssim \varepsilon^2 \nu^{-1}.
		\end{align*}
		By \cref{ee} and  Proposition \ref{est:main}, we have
		\begin{align}\label{b}
			&\int_0^T\int_{\R}|\tilde{E}^{(2)}\cdot\mathbb{W}_2|dydt\lesssim \int_0^T\delta\bar{\nu} \norm{(\mathbb{W}_{1y},\mathbb{W}_{2y})}^2_{L^2}dt\\ 
			&\qquad+C_\delta\varepsilon^2{\nu}^{-1}\int_0^T\int_{\R}\left[\bar{\nu}(1+t)\right]^{-1}e^{\frac{-(y\pm(1+t))^2}{\bar{\nu}(1+t)}}\mathbb{W}_2^2 dydt+C_\delta\varepsilon^4\nu^{-1}, \nonumber
		\end{align}
		where we have used the fact that $$\sum_{i,j=1}^2|\tilde{W}_{iy}\theta_j|\lesssim  \sum_{i,j=1}^2 \varepsilon(\bar{\nu}(1+t))^{-\frac{1}{2}}e^{-\frac{\abs{y-(-1)^j(1+t)}^2}{\bar{\nu}(1+t)}}\abs{\tilde{W}_{iy}},$$ due to \cref{theta} and $\abs{\eta_i}=O(1)\varepsilon$.
		Also, we have
		\begin{align}\label{c}
			\begin{aligned}
				\abs{\int_0^T\int_\R E_2\cdot\mathbb{W}_2dydt}\lesssim& \nu\sum_{i,j,k=1,2}\int_0^T\int_{\R}|\mathbb{W}_{iy}\mathbb{W}_{jy}\mathbb{W}_{ky}|+|\theta_{i}\theta_{j}\mathbb{W}_{ky}|+|\theta_{i}\mathbb{W}_{jy}\mathbb{W}_{ky}|dydt\\
				\lesssim& \delta\bar{\nu}\int_0^T\norm{\mathbb{W}_{1y},\mathbb{W}_{2y}}_{L^2}^2dt+C_\delta\varepsilon^4\nu^{-\frac12},
			\end{aligned}
		\end{align}
		where we have used the  Proposition \ref{est:main}
		\begin{align*}
			\norm{\p_y \mathbb{W}_j}_{L^\infty} \lesssim \norm{({\lde}_{00},v_{00}^2)}_{L^\infty}+\sum_{i=1}^2\norm{\theta_i}_{L^\infty}\lesssim\varepsilon\nu^{-\frac12}.
		\end{align*}
		Then it holds that 
		\begin{align}\label{v1v2}
			\begin{aligned}
				&\norm{(\mathbb{W}_1,\mathbb{W}_2)}_{L^2}^2+\bar{\nu}	\int_0^T\norm{\mathbb{W}_{2y}}_{L^2}^2dt \notag\\
				\lesssim & \delta\bar{\nu}\int_0^T \norm{\mathbb{W}_{1y}}_{L^2}^2dt+\varepsilon^2+C_\delta\varepsilon^2 \nu^{-1}\int_0^T\int_{\R} \left[\bar{\nu}(1+t)\right]^{-1}e^{\frac{-(y\pm(1+t))^2}{\bar{\nu}(1+t)}}\mathbb{W}_2^2 dydt.
			\end{aligned}
		\end{align} 
		\begin{flushleft}
			\textbf{Estimate on $\norm{\mathbb{W}_{1y}(\cdot,t)}^2_{L^2}$}
		\end{flushleft}
		
		Multiplying \cref{hatv}$_2$ by $\mathbb{W}_{1y}$, one has
		\begin{align}
			\mathbb{W}_{1y}^2+(\mathbb{W}_2\mathbb{W}_{1y})_t+\mathbb{W}_{2y}\mathbb{W}_{1t}-\bar{\nu}\mathbb{W}_{2yy}\mathbb{W}_{1y}=(\tilde{K}_1+\tilde{K}_2)\mathbb{W}_{1y}+(\cdots)_y,
		\end{align}	
		which implies from \cref{hatv}$_1$ (i.e. $\mathbb{W}_{1t}=-\mathbb{W}_{2y}+\tilde{K}_2-\tilde{K}_1, \mathbb{W}_{2yy}=-\mathbb{W}_{1yt}+(\tilde{K}_2-\tilde{K}_1)_y$) that 
		\begin{align}\label{V1y0}
			&	\mathbb{W}_{1y}^2+(\mathbb{W}_2\mathbb{W}_{1y}+\frac{\bar{\nu}}{2}\mathbb{W}_{1y}^2)_t \notag\\
			=&\mathbb{W}_{2y}^2+(\tilde{K}_1-\tilde{K}_2)\mathbb{W}_{2y}+\bar \nu(\tilde{K}_2-\tilde{K}_1)_y\mathbb{W}_{1y}+(\tilde{K}_1+\tilde{K}_2)\mathbb{W}_{1y}+(\cdots)_y.
		\end{align}	
		Integrating \cref{V1y0} on $\R\times [0,T]$ gives that 
		\begin{align}\label{V1y}
			&\int_{\R}\mathbb{W}_2\mathbb{W}_{1y}+\frac{\bar{\nu}}{2}\mathbb{W}_{1y}^2dy+	\int_0^T\norm{\mathbb{W}_{1y}}_{L^2}^2dt\\ \nonumber
			\lesssim&\int_0^T \norm{\tilde{K}_{2}-\tilde{K}_{1}}_{L^2}^2 dt+\bar{\nu}\int_0^T \norm{\tilde{K}_{2y}-\tilde{K}_{1y}}_{L^2}^2 dt+\int_0^T\bar{\nu}\norm{\mathbb{W}_{1y}}_{L^2}^2+\norm{\mathbb{W}_{2y}}_{L^2}^2dt\\\nonumber
			&+\abs{\int_{0}^T\int_{\R}\left(\tilde{K}_{1}+\tilde{K}_{2}\right)\mathbb{W}_{1y}dydt}+\norm{(\tilde{W}_{1y},\tilde{W}_{2})(\cdot,0)}_{L^2}^2.\nonumber
		\end{align}
		Using \cref{2025-9-17-1} and  Proposition \ref{est:main}, we have
		\begin{align}
			&\begin{aligned}
				\int_0^T \norm{\tilde{K}_2-\tilde{K}_1}_{L^2}^2 + \bar{\nu}  \norm{\tilde{K}_{2y}-\tilde{K}_{1y}}_{L^2}^2 dt \lesssim \varepsilon^2\nu^{\frac{3}{2}},
			\end{aligned}\\
			&\begin{aligned}\label{nmnm}
				\int_0^T \int_\R \abs{(\tilde{E}^{(1)}+\tilde{N})\mathbb{W}_{1y}}dy \lesssim \delta \int_0^T\norm{\mathbb{W}_{1y}}_{L^2}^2 dt + C_\delta\int_0^T \norm{(\tilde{E}^{(1)},\tilde{N})}_{L^2}^2dt \lesssim \delta \int_0^T\norm{\mathbb{W}_{1y}}_{L^2}^2 dt+C_\delta\varepsilon^4\nu^{-\frac{4}{3}},
			\end{aligned}\\
			&\begin{aligned}
				\int_0^T\int_{\R}|\tilde{E}^{(2)}\mathbb{W}_{1y}|dydt\lesssim& \sum_{i.j=1,2} \int_0^T\int_{\R}|\mathbb{W}_{iy}\mathbb{W}_{jy}\mathbb{W}_{1y}|+|\mathcal{C}_{i}\mathbb{W}_{jy}\mathbb{W}_{1y}|dydt
				\lesssim \bar{\nu}^{\alpha-\frac12}\int_0^T\norm{\mathbb{W}_{1y},\mathbb{W}_{2y}}_{L^2}^2dt,
			\end{aligned}\\
			&\begin{aligned}
				\bigg|\int_0^T\int_{\R} {E}_2\mathbb{W}_{1y}dydt \bigg|\lesssim& \bar{\nu}\sum_{i,j=1,2} \int_0^T\int_{\R}(|\mathbb{W}_{iy}\mathbb{W}_{jy}|+\abs{\theta_{i}\theta_{j}}+|\theta_{i}\mathbb{W}_{jy}|)|\mathbb{W}_{1yy}|dydt\\
				\lesssim&\bar{\nu}\sum_{i,j=1}^2\int_0^T\norm{\mathbb{W}_{iy}}_{L^2}^2\norm{\mathbb{W}_{1yy}}_{L^\infty}+\norm{\mathbb{W}_{iy}}_{L^2}^2+\norm{\theta_i}_{L^\infty}^2\norm{\mathbb{W}_{1yy}}_{L^2}^2\\
				&\qquad\quad\;\;+\norm{\theta_i}_{L^\infty}\norm{\theta_j}_{L^2}\norm{\mathbb{W}_{1yy}}_{L^2} dt\\
				\lesssim&\delta\bar{\nu}\int_0^T\norm{(\mathbb{W}_{1y},\mathbb{W}_{2y})}_{L^2}^2dt+C_\delta\varepsilon^3\nu^{-\frac{1}{2}},
			\end{aligned}
		\end{align}
		where we have used
		\begin{align*}
			\int_0^t \norm{\tilde N}_{L^2}^2 d\tau \lesssim \norm{({\lde}_{\neq},{\lde}_{0\neq},v_{0\neq}^2,v_{\neq}^2)}_{L^\infty}^2\int_0^t \norm{\nabla({\lde}_{\neq},{\lde}_{0\neq},v_{0\neq}^2,v_{\neq}^2)}_{L^2}^2 d\tau \lesssim \varepsilon^4\nu^{-\frac43}.
		\end{align*}
		Then we arrive at
		\begin{align}\label{V1y1}
			\begin{aligned}
				&\int_{\R}\mathbb{W}_2\mathbb{W}_{1y}+\frac{\bar{\nu}}{2}\mathbb{W}_{1y}^2dy+	\int_0^T\norm{\mathbb{W}_{1y}}_{L^2}^2dt
				\lesssim\varepsilon^2+\int_0^T \norm{\mathbb{W}_{2y}}_{L^2}^2 dt.
			\end{aligned}
		\end{align}
		\begin{flushleft}
			\textbf{Estimate on $	\int_0^T\int_{\R}\left[\bar{\nu}(1+t)\right]^{-1}e^{-\frac{(y\pm(1+t))^2}{\bar{\nu}(1+t)}}\mathbb{W}_2^2dydt$}
		\end{flushleft}
			
			We only consider the estimate of $\int_0^T\int_{\R}\left[\bar{\nu}(1+t)\right]^{-1}e^{-\frac{(y-(1+t))^2}{\bar{\nu}(1+t)}}\mathbb{W}_2^2dydt$, the other case is handled similarly. Note that $\mathbb{W}_2=\tilde{W}_1+\tilde{W}_2$, we go back the system \cref{AA} and estimate $\int_0^T\int_{\R}[\bar{\nu}(1+t)]^{-1}e^{-\frac{(y-(1+t))^2}{\bar{\nu}(1+t)}}(\tilde{W}_1^2+\tilde{W}_2^2)dydt$. From $\cref{AA}_1$, $\tilde{W}_1$ {propagates backward with speed $-1$}, while the weight function $
			e^{-\frac{(y-(1+t))^2}{\bar{\nu}(1+t)}}$ travels forward with speed $1$, and hence we can expect stronger estimate for $\tilde{W}_1$ than $\int_0^T\int_{\R}[\bar{\nu}(1+t)]^{-1}e^{-\frac{(y-(1+t))^2}{\bar{\nu}(1+t)}}\tilde{W}_1^2dydt$. Indeed, 
			set
			\begin{align}\label{q}
				\eta_1=\exp \left(\int_{-\infty}^{y} h(z, t) d z\right), \quad h(t,z)=\frac{1}{\sqrt{2 \pi \bar{\nu}(1+t)}} \exp \left(-\frac{\left(z-(1+t)\right)^2}{2 \bar{\nu}(1+t)}\right),
			\end{align} where $h$ satisfies
			$$
			h_t+h_y=\frac{\bar{\nu}}{2} \p_y^2h,\qquad
			1 \leq \eta_1 \leq e,
			$$
			and
			$$
			\p_t\eta_{1}=\eta_1 \int_{-\infty}^{y} h_t(z, t) d z=\eta_1\left(\frac{\bar{\nu}}{2} \p_yh- h\right), \quad \p_y\eta_{1}=\eta_1 h .
			$$
			Multiplying \cref{AA}$_1$ by $\eta_1 \tilde{W}_1$, we can get
			$$
			\begin{aligned}
				& \p_t\left(\eta_1 \frac{\tilde{W}_1^2}{2}\right)-\left(\p_t\eta_{1}- \p_y\eta_{1}\right) \frac{\tilde{W}_1^2}{2} =\frac{\bar{\nu}}{2}\p_y^2 \tilde{W}_{1} \eta_1 \tilde{W}_1+\frac{\bar{\nu}}{2}\p_y^2\tilde{W}_2\eta_1\tilde{W}_1+\tilde{K}_1 \eta_1 \tilde{W}_1+(\cdots)_y .
			\end{aligned}
			$$
			Note that
			$$
			\begin{aligned}
				\p_t\eta_{1}-\p_y \eta_{1} & =-2 \eta_1 h+\eta_1 \frac{\bar{\nu}}{2} \p_yh,
			\end{aligned}
			$$
			we can get
			\begin{align}\label{qq}
				\begin{aligned}
					& \left(\eta_1 \frac{\tilde{W}_1^2}{2}\right)_t+2\eta_1 h \tilde{W}_1^2=\eta_1 \bar{\nu} \p_yh\frac{\tilde{W}_1^2}{4}+\frac{\bar{\nu}}{2}\p_y^2 \tilde{W}_{1} \eta_1 \tilde{W}_1+\frac{\bar{\nu}}{2}\p_y^2\tilde{W}_2\eta_1\tilde{W}_1+\tilde{K}_1 \eta_1 \tilde{W}_1+(\cdots)_y.
				\end{aligned}
			\end{align}
			Since $\p_yh\approx \bar{\nu}^{-\frac12}(1+t)^{-\frac12}h$, $\bar{\nu}\p_yh \tilde{W}_1^2$ could be controlled by $h\tilde{W}_1^2$. 
			Also, we have 
			\begin{align}\label{qqq}
				\begin{aligned}
					&\abs{\int_\R \eta_1(\p_y^2\tilde{W}_1+\p_y^2\tilde{W}_2)\tilde{W}_1dy}\lesssim \norm{\left(\p_y\tilde{W}_1,\p_y\tilde{W}_2\right)}_{L^2}^2+ \norm{h\tilde{W}_1}_{L^2}^2.
				\end{aligned}
			\end{align}
			Similar to \cref{a}-\cref{c}, one has
			\begin{align}\label{22222}
				\bigg|\int_0^T\int_{\R}\tilde{K}_1\tilde{W}_1\eta_1dydt\bigg|\lesssim &C_\delta\varepsilon^2 \nu^{-1}\int_0^T\int_{\R}\left[\bar{\nu}(1+t)\right]^{-1}e^{-\frac{(y\pm(1+t))^2}{\bar{\nu}(1+t)}}\tilde{W}_1^2 dydt\nonumber\\
				&+\bar{\nu}\int_0^T \bigg(\norm{\p_y\tilde{W}_1}_{L^2}^2+\norm{\p_y\tilde{W}_2}_{L^2}^2\bigg)dt+C_\delta\varepsilon^3\nu^{-\frac98}.
			\end{align}
			Thus, integrating \cref{qq} on $\R\times [0,T]$ gives that
			\begin{align}\label{new1}
				\norm{\tilde{W}_1(\cdot,T)}_{L^2}^2+	\int_0^T\int_{\R}h\tilde{W}_1^2dydt\lesssim& \norm{\tilde{W}_1(\cdot,0)}_{L^2}^2+\bar{\nu}\int_0^T\norm{\left(\p_y\tilde{W}_1,\p_y\tilde{W}_2\right)}_{L^2}^2dt\\ \nonumber
				&+C_\delta\varepsilon^2 \nu^{-1}\int_0^T\int_{\R}\left[\bar{\nu}(1+t)\right]^{-1}e^{-\frac{(y+(1+t))^2}{\bar{\nu}(1+t)}}\tilde{W}_1^2 dydt +C_\delta\varepsilon^3\nu^{-\frac98}.
			\end{align}
			From the formula \cref{q} of $h$, we exactly obtain an estimate of $\int_0^T\int_{\R}\left[\bar{\nu}(1+t)\right]^{-\frac12}e^{-\frac{(y-(1+t))^2}{\bar{\nu}(1+t)}}\tilde{W}_1^2dydt$.
			
			The estimate of $\int_0^T\int_{\R}\left[\bar{\nu}(1+t)\right]^{-1}e^{-\frac{(y-(1+t))^2}{\bar{\nu}(1+t)}}\tilde{W}_2^2dydt$ is subtle since both $\tilde{W}_2$ and the weight function $
			e^{-\frac{(y-(1+t))^2}{\bar{\nu}(1+t)}}$ propagate forward with the same speed $1$. 
			Setting
			$$
			n(y, t)=\int_{-\infty}^y h(z, t) d z,
			$$
			it is easy to check that
			\begin{align}\label{0000000}
				\p_tn=\frac{\bar{\nu}}{2} \p_yh- h, \quad 0<n<1 .
			\end{align}
			Multiplying \cref{AA}$_2$ by $n^2\tilde{W}_2$, we get
			\begin{align}\label{new11}
				&	\int_0^T<  \p_ t \tilde{W}_{2}, \tilde{W}_2 n^2>d t- \int_0^T \int_{\R} \tilde{W}_2^2 n h d y d t \nonumber\\
				=&\int_0^T \int_{\R}  \frac{\bar{\nu}}{2}\p_y^2\tilde{W}_2\tilde{W}_2n^2+\frac{\bar{\nu}}{2}\p_y^2\tilde{W}_1\tilde{W}_2n^2+\tilde{K}_2\tilde{W}_2n^2d y d t\nonumber\\
				\lesssim&\bar{\nu}\int_0^T\norm{\p_y\tilde{W}_1,\p_y\tilde{W}_2}_{L^2}^2dt+\delta\bar{\nu}\int_0^T\int_{\R} h^2\tilde{W}_2^2dydt+\int_0^T\int_{\R}\tilde{K}_2\tilde{W}_2n^2dydt.
			\end{align}
			Similar to \cref{a}-\cref{c}, one has
			\begin{align}\label{new12}
				\int_0^T\int_{\R}\tilde{K}_2\tilde{W}_2n^2dydt\lesssim &C_\delta\varepsilon^2 \nu^{-1}\int_0^T\int_{\R}\left[\bar{\nu}(1+t)\right]^{-1}e^{-\frac{(y\pm(1+t))^2}{\bar{\nu}(1+t)}}\tilde{W}_2^2 dydt\nonumber\\
				&+\bar{\nu}\int_0^T \norm{\p_y\tilde{W}_1}_{L^2}^2+\norm{\p_y\tilde{W}_2}_{L^2}^2dt+C_\delta\varepsilon^3\nu^{-\frac{9}{8}}.
			\end{align}
			Combining   Lemma \ref{Ptii}, \cref{new11} and \cref{new12}, one has
			\begin{align}\label{h2V}
				\int_0^T \int_{\R} h^2 \tilde{W}^2_{2} d y d t\lesssim& \bar{\nu}^{-1}\norm{\tilde{W}_2(\cdot,0)}_{L^2}^2+\int_0^T\norm{\left(\p_y\tilde{W}_1,\p_y\tilde{W}_2\right)}_{L^2}^2dt+C_\delta\varepsilon^3\nu^{-\frac{17}{8}}\\\nonumber
				&+C_\delta\varepsilon^2\nu^{-2}\int_0^T\int_{\R}\left[\bar{\nu}(1+t)\right]^{-1}e^{-\frac{(y+(1+t))^2}{\bar{\nu}(1+t)}} \tilde{W}_2^2 dydt.
			\end{align}
			Therefore, from \cref{new1} and \cref{h2V}, we obtain that 
			\begin{align}\label{new13}
				\int_0^T\int_{\R}h\tilde{W}_1^2dydt+\bar{\nu}\int_0^T\int_{\R}h^2\tilde{W}_2^2dydt\lesssim& \norm{(\tilde{W}_1,\tilde{W}_2)(\cdot,0)}_{L^2}^2+\bar{\nu}\int_0^T\norm{\left(\p_y\tilde{W}_1,\p_y\tilde{W}_2\right)}_{L^2}^2dt\\ \nonumber
				&+C_\delta\varepsilon^2 \nu^{-1}\int_0^T\int_{\R}\tilde{h}^2(\tilde{W}_1^2+\tilde{W}_2^2) dydt +C_\delta\varepsilon^3\nu^{-\frac{9}{8}}.
			\end{align}
			where $\tilde{h}(t,y)=\frac{1}{\sqrt{2 \pi \bar{\nu}(1+t)}} \exp \left(-\frac{\left(y+(1+t)\right)^2}{2 \bar{\nu}(1+t)}\right)$. We also have the same estimate for the case $\tilde h$, that is, 
			\begin{align}\label{33333}
				\int_0^T\int_{\R}\tilde{h}\tilde{W}_2^2dydt+\bar{\nu}\int_0^T\int_{\R}\tilde{h}^2\tilde{W}_1^2dydt\lesssim& \norm{(\tilde{W}_1,\tilde{W}_2)(\cdot,0)}_{L^2}^2+\bar{\nu}\int_0^T\norm{\left(\p_y\tilde{W}_1,\p_y\tilde{W}_2\right)}_{L^2}^2dt\\ \nonumber
				&+C_\delta\varepsilon^2 \nu^{-1}\int_0^T\int_{\R}h^2(\tilde{W}_1^2+\tilde{W}_2^2) dydt +C_\delta\varepsilon^3\nu^{-\frac{9}{8}}.
			\end{align}
			Therefore, we conclude that
			\begin{align}\label{new14}
				\bar{\nu}\int_0^T\int_{\R}\left[\bar{\nu}(1+t)\right]^{-1}e^{-\frac{(y\pm(1+t))^2}{\bar{\nu}(1+t)}}\mathbb{W}_2^2dydt\lesssim& \norm{(\mathbb{W}_1,\mathbb{W}_2)(\cdot,0)}_{L^2}^2+\bar{\nu}\int_0^T\norm{\left(\mathbb{W}_{1y},\mathbb{W}_{2y}\right)}_{L^2}^2dt\\ \nonumber
				&+C_\delta\varepsilon^3\nu^{-\frac{9}{8}}.
			\end{align}
			Taking $\cref{v1v2}+\cref{V1y1}\times\bar{\nu}$ and using \cref{new14}, one can prove 
			\begin{align}\label{new15}
				\norm{(\mathbb{W}_1,\mathbb{W}_2)(\cdot,T)}_{L^2}^2+\bar{\nu}^2\norm{\mathbb{W}_{1y}(\cdot,T)}_{L^2}^2+\bar{\nu}\int_0^T\norm{(\mathbb{W}_{1y},\mathbb{W}_{2y})}_{L^2}^2dt+	\bar{\nu}\int_0^T\int_{\R}(h^2+\tilde{h}^2)\mathbb{W}_2^2dydt\lesssim \varepsilon^2,
			\end{align}
			where we have used $\norm{(\mathbb{W}_1,\mathbb{W}_2)(\cdot,0)}_{L^2}\lesssim \varepsilon$. By  Proposition \ref{est:main} and  Lemma \ref{estimateontheta} and  Lemma \ref{estimateonxi}, we have
			\begin{align*}
				\norm{\tilde{W}}_{L^\infty}\lesssim \norm{\tilde W}_{L^2}^{\frac12}\norm{\p_y \tilde W}_{L^2}^{\frac12}\lesssim \norm{\tilde W}_{L^2}^{\frac12}\left(\norm{({\lde}_{00},v_{00}^2)}_{L^2}^{\frac12} + \norm{\theta_i}_{L^2}^{\frac12} +\norm{\Xi_i}_{L^2}^{\frac12} \right) \lesssim \varepsilon \nu^{-\frac18}.
			\end{align*} 
			This yields  Lemma \ref{1or}.
		\end{proof}
					In the proof above, we have used a weighted inequality from Huang-Li-Matsumura \cite{HLM}, which we state below.
		\begin{Lem}\label{Ptii} 
			For any $T>0$, if $V(t, y)$ satisfies
			$$
			V(t,y) \in L^{\infty}\left(0, T ; L^2({\R})\right), \quad \p_yV \in L^2\left(0, T ; L^2({\R})\right), \quad \p_tV \in L^2\left(0, T ; H^{-1}({\R})\right),
			$$
			it holds that 
			$$
			\begin{aligned}
				\frac{1}{2}\int_0^T \int_{\R} h^2 V^2 d y d t \leq  & \frac{1}{\bar{\nu}}\int_{\R} V^2(y, 0) d y+2 \int_0^T\left\|\p_yV\right\|_{L^2}^2 d t \\
				& +\frac{2}{\bar{\nu}}\left(\int_0^T<\p_t V, V n^2>d t- \int_0^T \int_{\R} nhV^2 d y d t\right).
			\end{aligned}
			$$
		\end{Lem}
		\begin{proof}
			From \cref{0000000}, it is straightforward to check that
			\begin{align*}
				&\frac{2}{\bar{\nu}}\left(\int_0^T<\p_t V, V n^2>d t- \int_0^T \int_{\R} nhV^2  d y d t\right)\\
				=&\frac{1}{\bar{\nu}}\left(\int_\R (Vn)^2(y,T)dy-\int_\R (Vn)^2(y,0)dy\right)-\int_0^T \int_\R \p_yhnV^2 dydt\\
				\geq& -\frac{1}{\bar{\nu}}\int_\R (Vn)^2(y,0)dy +2 \int_0^T \int_\R hnV V_y dydt + \int_0^T \int_\R V^2 h^2 dydt\\
				\geq & -\frac{1}{\bar{\nu}}\int_\R (Vn)^2(y,0)dy - 2 \int_0^T \norm{V_y}_{L^2}^2 dt + \frac{1}{2}\int_0^T \int_\R V^2h^2dydt,
			\end{align*}
			which completes the proof of  Lemma \ref{Ptii}. 
		\end{proof}
		
		\subsection{Estimate on $\theta_A$}
		This subsection aims to estimate the main part of $v^1_{00}$ (i.e., $\theta_A$). 
		\begin{Lem}\label{estimateonta}
			Under the same  assumptions as \cref{MT} and  Proposition \ref{est:main},  it holds that
			\begin{align}
				\norm{\theta_A(y, t)}_{L^\infty} \lesssim \varepsilon ,\quad \norm{\p_y\theta_{A}(y, t)}_{L^2}\lesssim\varepsilon \left[\nu(1+t)\right]^{-\frac{1}{4}}.
			\end{align}
		\end{Lem}
		\begin{proof}
			By Duhamel's principle and \cref{thetaA}, one has 
			\begin{align}\label{Theta-A}
				{\theta_A(y, t)}=\int_0^t \int_{-\infty}^{\infty} &H(0,\nu) \bigg(a(\Xi_1+\Xi_2-\mathcal{C}_1-\mathcal{C}_2)-\frac{a\nu'}{2}\left(\Xi_2-\Xi_1\right)_y+\frac{a}{2}(\theta_1^2-\theta_2^2)\\\nonumber
				&\qquad-a(\mathcal{C}_1+\mathcal{C}_2)\p_y\theta_A+a^2(\mathcal{C}_1+\mathcal{C}_2)(\mathcal{C}_2-\Xi_2-\mathcal{C}_1+\Xi_1)\\ \nonumber
				&\qquad-a(\tilde{w}_1+\tilde{w}_2)\p_y\theta_A+a^2(\tilde{w}_1+\tilde{w}_2)(\mathcal{C}_2-\mathcal{C}_1)\bigg) d s d \tau. \nonumber
			\end{align}
			Let
			\begin{align}\label{MMMM}
				\hat{M}^{(1)}(t):=\sup_{0 \leq \tau \leq t}\left\{\norm{\theta_A(y, \tau)}_{L^\infty}+ \norm{\p_y\theta_{A}(y, \tau)}_{L^2}\left[\nu(1+\tau)\right]^{\frac{1}{4}}\right\}.
			\end{align}
			For $k=0,1$, a direct computation gives that 
			\begin{align}\label{thetaA1}
				\p^k_y\theta_A 
				=&\int_0^{t} \int_{-\infty}^{\infty} \p^{k}_yH(0,\nu)S_{\mathcal{A}}^{(1)}(s, \tau) d s d \tau+\int_{0}^t \int_{-\infty}^{\infty} \p_y^{k}H(0,\nu) S_{\mathcal{A}}^{(2)}d s d \tau\\ \nonumber
				&\qquad\qquad-\int_0^{t} \int_{-\infty}^{\infty} aH(0,\nu)\p^k_s(\theta_1+\theta_2)d s d \tau:=I_1+I_2+I_3,
			\end{align}
			where $S_\mathcal{A}$ is defined as
			\begin{align*}
				S_{\mathcal{A}}^{(1)}:=&\frac{a\bar{\nu}}{4}(\theta_{2y}-\theta_{1y})-\frac{a}{4}(\nu'-2\nu)\left(\Xi_2-\Xi_1\right)_y+\frac{a}{2}(\theta_1^2-\theta_2^2)\\
				&-a(\mathcal{C}_1+\mathcal{C}_2)\p_y\theta_A+a^2(\mathcal{C}_1+\mathcal{C}_2)(\mathcal{C}_2-\Xi_2-\mathcal{C}_1+\Xi_1)\\
				S_{\mathcal{A}}^{(2)}:=&-a(\tilde{w}_1+\tilde{w}_2)\p_y\theta_A+a^2(\tilde{w}_1+\tilde{w}_2)(\mathcal{C}_2-\mathcal{C}_1) d s d \tau,\\
				\mbox{and}\quad			\abs{S_{\mathcal{A}}^{(1)}}=&O(1)\sum_{i,j=1,2}\left(\bar{\nu}\abs{\theta_{iy}}+\theta_i^2+\abs{\theta_i}\abs{\p_y\theta_A}\right),\\
				\abs{S_A^{(2)}}=&O(1)\sum_{i,j=1,2}\left(\abs{\p_y\tilde{W}_j}\abs{\p_y\theta_A}+\abs{\p_y\tilde{W}_i\theta_j}\right).
			\end{align*}
			By  Lemma \ref{estimateontheta}, Lemma \ref{estimateonxi} and \cref{MMMM}, one has,
			\begin{align}
				\norm{S_{\mathcal{A}}^{(1)}}_{L^1 }\lesssim \varepsilon\left[\nu(1+t)\right]^{-\frac{1}{2}}\hat{M}^{(1)}(t)+\varepsilon\bar{\nu}\left[\bar{\nu}(1+t)\right]^{-\frac{1}{2}}.
			\end{align}
			Thus, for $k=0,1$, there holds 
			\begin{align}\label{I7I8}
				\norm{I_1}_{L^2}\lesssim \varepsilon\nu^{-1}\left[\bar\nu(1+t)\right]^{-\frac{k}{2}+\frac{1}{4}}\hat{M}^{(1)}(t)+\varepsilon\left[\bar\nu(1+t)\right]^{-\frac{k}{2}+\frac{1}{4}}.
			\end{align}
			Next, we turn to $I_2$. Since we do not have the decay rate of $\norm{\p_y \tilde W_j}_{L^2}$, the estimate of $I_2$ is different from that of $I_1$ and more complicated. It is easy to see that
			\begin{align*}
				\norm{S_A^{(2)}}_{L^1} \lesssim (\hat{M}^{(1)}(t)+\varepsilon) [\nu(1+t)]^{-\frac14}\sum_{j=1}^2 \norm{\p_y \tilde W_j}_{L^2}.
			\end{align*}
			For $k=0$, by  Lemma \ref{1or}, we have
			\begin{align}\label{2026-3-20-1}
				\norm{I_2}_{L^\infty} & \lesssim (\hat{M}^{(1)}(t)+\varepsilon) \nu^{-\frac34}\sum_{j=1}^2\int_0^t (t-\tau)^{-\frac12}(1+\tau)^{-\frac14}\norm{\p_y \tilde W_j}_{L^2} d\tau \notag\\
				& =  (\hat{M}^{(1)}(t)+\varepsilon) \nu^{-\frac34}\sum_{j=1}^2\left(\int_0^{\frac t2} + \int_{\frac t2}^{t-1}+\int_{t-1}^t \right)(t-\tau)^{-\frac12}(1+\tau)^{-\frac14}\norm{\p_y \tilde W_j}_{L^2} d\tau \notag\\
				&\lesssim (\hat{M}^{(1)}(t)+\varepsilon) \nu^{-\frac34} (1+t)^{-\frac12} \left( \int_0^{\frac t2} (1+\tau)^{-\frac12} d\tau\right)^{\frac12} \left(\int_0^{\frac t2} \sum_{j=1}^2\norm{\p_y \tilde W_j}_{L^2}^2 d\tau\right)^{\frac12} \notag\\
				&\quad+(\hat{M}^{(1)}(t)+\varepsilon) \nu^{-\frac34} (1+t)^{-\frac14} \left( \int_{\frac t2}^{t-1} (t-\tau)^{-1} d\tau\right)^{\frac12} \left(\int_{\frac t2}^{t-1} \sum_{j=1}^2\norm{\p_y \tilde W_j}_{L^2}^2 d\tau\right)^{\frac12} \notag\\
				&\quad+(\hat{M}^{(1)}(t)+\varepsilon) \nu^{-\frac34} (1+t)^{-\frac14} \int_{t-1}^{t} (t-\tau)^{-\frac12} d\tau \sum_{j=1}^2\norm{\p_y \tilde W_j}_{L_t^\infty L^2_x} \notag\\
				& \lesssim \varepsilon\nu^{-\frac54}\left( \hat{M}^{(1)}(t) +\varepsilon \right).
			\end{align}
			For $k=1$, by  Lemma \ref{1or}, we have
			\begin{align}\label{2026-3-20-2}
				\norm{I_2}_{L^2} & \lesssim (\hat{M}^{(1)}(t)+\varepsilon) \nu^{-1}\sum_{j=1}^2\int_0^t (t-\tau)^{-\frac34}(1+\tau)^{-\frac14}\norm{\p_y \tilde W_j}_{L^2} d\tau \notag\\
				& =  (\hat{M}^{(1)}(t)+\varepsilon) \nu^{-1}\sum_{j=1}^2\left(\int_0^{\frac t2} + \int_{\frac t2}^{t-1}+\int_{t-1}^t \right)(t-\tau)^{-\frac34}(1+\tau)^{-\frac14}\norm{\p_y \tilde W_j}_{L^2} d\tau \notag\\
				&\lesssim (\hat{M}^{(1)}(t)+\varepsilon) \nu^{-1} (1+t)^{-\frac34} \left( \int_0^{\frac t2} (1+\tau)^{-\frac12} d\tau\right)^{\frac12} \left(\int_0^{\frac t2} \sum_{j=1}^2\norm{\p_y \tilde W_j}_{L^2}^2 d\tau\right)^{\frac12} \notag\\
				&\quad+(\hat{M}^{(1)}(t)+\varepsilon) \nu^{-1} (1+t)^{-\frac14} \left( \int_{\frac t2}^{t-1} (t-\tau)^{-\frac32} d\tau\right)^{\frac12} \left(\int_{\frac t2}^{t-1} \sum_{j=1}^2\norm{\p_y \tilde W_j}_{L^2}^2 d\tau\right)^{\frac12} \notag\\
				&\quad+(\hat{M}^{(1)}(t)+\varepsilon) \nu^{-1} (1+t)^{-\frac14} \int_{t-1}^{t} (t-\tau)^{-\frac34} d\tau \sum_{j=1}^2\norm{\p_y \tilde W_j}_{L_t^\infty L^2_x} \notag\\
				& \lesssim \varepsilon\nu^{-\frac54}\left( \hat{M}^{(1)} +\varepsilon \right)[\nu(1+t)]^{-\frac14}.
			\end{align}
			Now, we focus on the $I_3$. We introduce the following useful formula
			\begin{align}\label{semi-group}
				&-\frac{\left(s-\sigma(1+\tau)\right)^2}{4 \mu_1(1+\tau)}-\frac{\left(y-s-\sigma'(t-\tau)\right)^2}{4 \mu_2(t-\tau)}\nonumber \\
				= & -\frac{\mu_1(1+\tau)+\mu_2(t-\tau)}{4 \mu_1 \mu_2(1+\tau)(t-\tau)}\left[s-\sigma(1+\tau)-\frac{\mu_1(1+\tau)\left(y-\sigma'(t-\tau)-\sigma(1+\tau)\right)}{\mu_1(1+\tau)+\mu_2(t-\tau)}\right]^2\\
				& - \frac{\left[y-\sigma(1+\tau)-\sigma'(t-\tau)\right]^2}{4\left[\mu_1(1+\tau)+\mu_2(t-\tau)\right]},\nonumber
			\end{align}
			Since the wave speeds of $\theta_i$ and $H(0,\nu)$ are different, for $\sigma=(-1)^i, i=1,2$, applying \cref{semi-group}  we obtain
			\begin{align*}
				I_3=&\varepsilon \int_0^t \int_{\R} \bigg\{[\nu(t-\tau)]^{-\frac12}[\bar \nu(1+\tau)]^{-\frac12} \exp \Big\{-\frac{\bar\nu(1+\tau)+\nu(t-\tau)}{4\nu\bar\nu(1+\tau)(t-\tau)}\Big[s-\sigma(1+\tau)-\frac{\bar\nu(1+\tau)(y-\sigma(1+\tau))}{\bar\nu(1+\tau)+\nu(t-\tau)}\Big]^2 \Big\}\\
				&\quad\qquad\qquad\qquad\qquad\qquad\qquad\qquad\times\exp\Big\{-\frac{[y-\sigma(1+\tau)]^2}{4[\bar\nu(1+\tau)+\nu(t-\tau)]}\Big\}\bigg\} ds d\tau\\
				=&\varepsilon\int_0^t \frac{1}{[ \bar\nu (1+\tau)+\nu(t-\tau)]^{\frac12}} \exp\Big\{-\frac{[y-\sigma(1+\tau)]^2}{4[\bar\nu(1+\tau)+\nu(t-\tau)]}\Big\} d\tau.
			\end{align*}
			Then we have
			\begin{align}\label{I9}
				\norm{I_3}_{L^\infty}\lesssim \varepsilon\quad(k=0), \qquad \norm{I_3}_{L^2}\lesssim  \varepsilon\left[\nu(1+t)\right]^{-\frac{1}{4}}\quad(k=1).
			\end{align}
			Combining \cref{thetaA1,I7I8,I9}, we deduce that $\hat{M}^{(1)}(t)\lesssim \varepsilon+\varepsilon\nu^{-\frac54}\hat{M}^{(1)}(t)$.  Note that $\varepsilon=\nu^\frac32$, one has $\hat{M}^{(1)}(t)\lesssim \varepsilon$ as $\nu$ is small. Substituting $\hat{M}^{(1)}(t)\lesssim \varepsilon$ into \cref{MMMM} yields  Lemma \ref{estimateonta}.     
		\end{proof}
		\begin{Lem}\label{estimateonta-1}
			Under the same  assumptions as \cref{MT} and  Proposition \ref{est:main},  it holds that
			\begin{align}
				\int_0^t\norm{\p_y^2\theta_{A}(y, \tau)}_{L^2}^2d\tau\lesssim\varepsilon^2\nu^{-\frac32}. 
			\end{align}
		\end{Lem}
		\begin{proof}
			Since $\norm{\p_y^2 \tilde W_j}_{L^2}$ does not show the decay rate, using the same method as for  Lemma \ref{estimateonta}, we cannot prove  Lemma \ref{estimateonta-1}. Therefore, we split $\theta_A$ into $\theta_A^{(1)}$ and $\theta_A^{(2)}$, where $\theta_A^{(1)}$ \cref{thetaA-1} is handled by using the same method as for $\theta_A$  Lemma \ref{estimateonta}, and $\theta_A^{(2)}$ \cref{thetaA-2} is estimated by using the energy method
			\begin{align}
				& \p_t\theta_{A}^{(1)}-\nu\p_y^2\theta_{A}^{(1)}=S_{\mathcal{A}1}-a(\theta_1+\theta_2),
				\label{thetaA-1}\\
				& \p_t\theta_{A}^{(2)}-\nu\p_y^2\theta_{A}^{(2)}=S_{\mathcal{A}2}\label{thetaA-2},\\
				&\theta_{A}^{(1)}|_{t=0}=\theta_{A}^{(2)}|_{t=0}=0,\notag
			\end{align}
			where
			\begin{align*}
				S_{\mathcal{A}1}:=&\frac{a\bar{\nu}}{4}(\theta_{2y}-\theta_{1y})-\frac{a}{4}(\nu'-2\nu)\left(\Xi_2-\Xi_1\right)_y+\frac{a}{2}(\theta_1^2-\theta_2^2)\\
				&-a(\mathcal{C}_1+\mathcal{C}_2)\p_y\theta_A^{(1)}+a^2(\mathcal{C}_1+\mathcal{C}_2)(\mathcal{C}_2-\Xi_2-\mathcal{C}_1+\Xi_1),\\
				S_{\mathcal{A}2}:=&-a(\mathcal{C}_1+\mathcal{C}_2)\p_y\theta_A^{(2)}-a(\p_y\tilde{W}_1+\p_y\tilde{W}_2)\p_y\theta_A^{(2)}\notag\\
				&-a(\p_y\tilde{W}_1+\p_y\tilde{W}_2)\p_y\theta_A^{(1)}+a^2(\p_y\tilde{W}_1+\p_y\tilde{W}_2)(\mathcal{C}_2-\mathcal{C}_1),\\
				\mbox{and}\quad			\abs{S_{\mathcal{A}1}}=&O(1)\sum_{i,j=1,2}\left(\bar{\nu}\abs{\theta_{iy}}+\theta_i^2+\abs{\theta_i}\abs{\p_y\theta_A^{(1)}}\right),\\
				\abs{S_{\mathcal{A}2}}=&O(1)\sum_{i,j=1,2}\left[\abs{\theta_i \p_y\theta_A^{(2)}}+\abs{\Xi_i\p_y\theta_A^{(2)}}+\abs{\p_y \tilde{W}_i}\left( \abs{\p_y\theta_A^{(2)}}+\abs{\p_y\theta_A^{(1)}}+\abs{\theta_j}\right)\right].
			\end{align*}
			We first give an estimate for $\theta_A^{(1)}$.  By Duhamel's principle and \cref{thetaA-1}, one has 
			\begin{align}\label{Theta-A-1}
				{\theta_A^{(1)}(y, t)}=\int_0^t \int_{-\infty}^{\infty} &H(0,\nu)\left[S_{\mathcal{A}1} -a(\theta_1+\theta_2)\right] d s d \tau. 
			\end{align}
			For $k=0,1,2$, let
			\begin{align}\label{MMMM-1}
				\hat{M}^{(2)}(t):=\sup_{0 \leq \tau \leq t}\left\{\norm{\theta_A^{(1)}(y, \tau)}_{L^\infty}+\sum_{j=1}^k \norm{\p_y^j\theta_{A}^{(1)}(y, \tau)}_{L^2}\left[\nu(1+\tau)\right]^{\frac{j}{2}-\frac{1}{4}}\right\}.
			\end{align}
			$\norm{\theta_A^{(1)}(y, \tau)}_{L^\infty}$ and $\norm{\p_y\theta_A^{(1)}(y, \tau)}_{L^2}$ can be estimated by the same method as  Lemma \ref{estimateonta}. For $\p_y^2 \theta_A^{(1)}$, a direct computation gives that 
			\begin{align}\label{thetaA1-1}
				\p^2_y\theta_A^{(1)} 
				=&\int_0^{\frac t2} \int_{-\infty}^{\infty} \p^{2}_yH(0,\nu)S_{\mathcal{A}1}(s, \tau) d s d \tau+\int_{\frac t2}^t \int_{-\infty}^{\infty} \p_y H(0,\nu) \p_s S_{\mathcal{A}1}d s d \tau\\ \nonumber
				&\qquad\qquad-\int_0^{t} \int_{-\infty}^{\infty} aH(0,\nu)\p^2_s(\theta_1+\theta_2)d s d \tau:=I_4+I_5+I_6.
			\end{align}
			For $k=0,1$, by  Lemma \ref{estimateontheta}, Lemma \ref{estimateonxi} and \cref{MMMM-1}, one has,
			\begin{align}
				\norm{\p_y^kS_{\mathcal{A}1}}_{L^1 }\lesssim \varepsilon\left[\nu(1+t)\right]^{-\frac{1}{2}-\frac{k}{2}}\hat{M}^{(2)}(t)+\varepsilon\bar{\nu}\left[\bar{\nu}(1+t)\right]^{-\frac{1}{2}-\frac{k}{2}}.
			\end{align}
			Thus, there holds 
			\begin{align}\label{I7I8-1}
				\norm{I_4}_{L^2}+\norm{I_5}_{L^2}\lesssim \varepsilon\nu^{-1}\left[\bar\nu(1+t)\right]^{-\frac34}\hat{M}^{(2)}(t)+\varepsilon\left[\bar\nu(1+t)\right]^{-\frac34}.
			\end{align}
			Since the propagation speeds of $\theta_i$ and $H(0,\nu)$ are different,  using the same method as \cref{I9}, we obtain
			\begin{align}\label{I9-1}
				\norm{I_6}_{L^2}\lesssim  \varepsilon\left[\nu(1+t)\right]^{-\frac{3}{4}}.
			\end{align}
			Combining \cref{thetaA1-1,I7I8-1,I9-1}, we deduce that $\hat{M}^{(2)}(t)\lesssim \varepsilon+\varepsilon\nu^{-1}\hat{M}^{(2)}(t)$.  Note that $\varepsilon=\nu^\frac32$, one has $\hat{M}^{(2)}(t)\lesssim \varepsilon$ as $\nu$ is small. Substituting $\hat{M}^{(2)}(t)\lesssim \varepsilon$ into \cref{MMMM-1} yields $\int_0^t\norm{\p_y^2\theta_{A}^{(1)}(y, t)}_{L^2}^2d\tau\lesssim\varepsilon^2\nu^{-\frac32}$.
			
			Next we give an energy estimate for $\norm{\p_y \theta_A^{(2)}}_{L^2}^2+\nu\int_0^T \norm{\p_y^2 \theta_A^{(2)}}_{L^2}^2 d\tau$. Let 
			$\hat M^{(3)}(t):=\sup_{0\le\tau \le t}\norm{\p_y \theta_A^{(2)}}_{L^2}^2.$
			Multiplying $\p_y$\cref{thetaA-2} by $\p_y\mathcal{\theta}_A^{(2)}$ and integrating the resulting equation, one has
			\begin{align*}
				&\frac{1}{2}\norm{\p_y \theta_A^{(2)}(\tau)}_{L^2}^2\Big|_{\tau=0}^{\tau=t}+\nu\int_0^t \norm{\p_y^2\theta_A^{(2)}}_{L^2}^2 d\tau \\
				\le& \frac{\nu}{2}\int_0^t \norm{\p_y^2\theta_A^{(2)}}_{L^2}^2 d\tau +4 \varepsilon^2 \nu^{-1}\int_0^t\int_{\R}[\nu(1+t)]^{-1}e^{-\frac{|y\pm \bar c(1+t)|^2}{1+t}} \abs{\p_y\theta_A^{(2)}}^2 dyd\tau \\
				& +4\nu^{-1}\sum_{i,j=1}^2\left[ \norm{\p_y \theta_A^{(1)}}_{L^\infty L^2}^2 \int_0^ t \norm{\Xi_i}_{L^\infty}^2 d\tau+ \left(\norm{\p_y\theta_A^{(1)}}_{L^\infty L^\infty}^2+\norm{\theta_i}_{L^\infty L^\infty}^2 \right)\int_0^t \norm{\p_y \tilde W_j}_{L^2}^2 d\tau\right]\\
				&+ 4 \nu^{-3} \sum_{i=1}^2\norm{\p_y\theta_A^{(2)}}_{L^\infty L^2}^2\norm{\p_y\tilde W_i}_{L^\infty L^2}^2 \int_0^t \norm{\p_y \tilde W_i}_{L^2}^2 d\tau\\
				\le& 4\varepsilon^4\nu^{-3}+\varepsilon^2\nu^{-2}\hat M^{(3)}(t)+\frac{\nu}{2}\int_0^t \norm{\p_y^2\theta_A^{(2)}}_{L^2}^2 d\tau +4 \varepsilon^2 \nu^{-1}\int_0^t\int_{\R}[\nu(1+t)]^{-1}e^{-\frac{|y\pm \bar c(1+t)|^2}{1+t}} \abs{\p_y\theta_A}^2 dyd\tau,
			\end{align*}
			where we have used  Lemma \ref{estimateontheta},  Lemma \ref{estimateonxi},  Lemma \ref{1or},  Lemma \ref{estimateonta},  Proposition \ref{est:main} and the following 
			\begin{align*}
				& \sum_{i=1}^2 \int_0^t \int_{\R} \abs{\p_y \tilde W_i}\abs{\p_y \theta_A^{(2)}}\abs{\p_y^2\theta_A^{(2)}} dy d\tau \le \sum_{i=1}^2 \int_0^t \norm{\p_y \tilde W_i}_{L^2} \norm{\p_y \theta_A^{(2)}}_{L^2}^{\frac12}\norm{\p_y^2\theta_A^{(2)}}_{L^2}^{\frac32} d\tau \\
				\le & \frac{\nu}{4}\int_0^t \norm{\p_y^2 \theta_A^{(2)}}_{L^2}^2 d\tau + 4\nu^{-3}\sum_{i=1}^2 \norm{\p_y \tilde W_i}_{L^\infty L^2}^2 \norm{\p_y \theta_A^{(2)}}_{L^\infty L^2}^2 \int_0^t \norm{\p_y \tilde W_i}_{L^2}^2 d\tau.
			\end{align*}
			Since $\theta_A^{(2)}$ does not have propagation speed, while the weight function $
			e^{-\frac{(y\pm (1+t))^2}{\bar{\nu}(1+t)}}$ travels forward with speed $\pm1$, we can use the same method as \cref{q}-\cref{new1} to obtain
			\begin{align}
				\bar{\nu}\int_0^T\int_{R}\left[\bar{\nu}(1+t)\right]^{-1}&e^{-\frac{(y\pm(1+t))^2}{\bar{\nu}(1+t)}} \abs{\p_y \theta_A^{(2)}}^2 dydt \lesssim  \bar{\nu}^{\frac{1}{2}}\int_0^T\int_\R\eta_1h\abs{\p_y \theta_A^{(2)}}^2dydt
				\lesssim {\nu}^{\frac{3}{2}}\int_0^T\norm{\p_y^2 \theta_A^{(2)}}_{L^2}^2dt\nonumber.
			\end{align}
			where
			\begin{align}
				\eta_1=\exp \left(\int_{-\infty}^{y} h(z, t) d z\right), \quad h(t,z)=\frac{1}{\sqrt{2 \pi \bar{\nu}(1+t)}} \exp \left(-\frac{\left(z\pm (1+t)\right)^2}{2 \bar{\nu}(1+t)}\right),
			\end{align}
			Then we arrive at
			\begin{align}
				&\norm{\p_y\theta_{A}^{(2)}(\cdot,T)}_{L^2}^2+\nu\int_0^T\norm{\p_y^2\theta_{A}^{(2)}}_{L^2}^2dt
				\lesssim \varepsilon^2.
			\end{align}
			Since $$\int_0^t \norm{\p^2_y \theta_A}_{L^2}^2 d\tau\le\int_0^t \norm{\p^2_y \theta_A^{(1)}}_{L^2}^2 d\tau+\int_0^t \norm{\p^2_y \theta_A^{(2)}}_{L^2}^2 d\tau \lesssim \varepsilon^2\nu^{-\frac32},$$ we finished the proof of  Lemma \ref{estimateonta-1}.
		\end{proof}
		\subsection{Estimate on $\mathcal{A}$}
		This subsection is devoted to the  $\mathcal{A}$. Recall
		\begin{align}\label{equ-AA}
			\p_t{\mathcal{A}}- \nu \p_y^2{\mathcal{A}}-a\nu\p_y(\tilde{w}_2-\tilde{w}_1)=-v^2_{00}\p_y\mathcal{A}-\nu\frac{{\lde}_{00}}{{\lde}_{00}+1}\p_y^2\mathcal{A}+F_{A},
		\end{align}
		where
		\begin{align*}
			&\abs{F_A}=O(1)(\abs{F_{A1}}+\abs{F_{A2}}),\\
			&\abs{F_{A1}}=O(1)\sum_{i,j=1,2}\Big[\abs{\p_y\tilde{W}_i\Xi_j}+\nu\abs{\p_y\tilde{W}_i\p_y^2\tilde{W}_j}
			+\nu\abs{\p_y^2\theta_A}\left(\abs{\theta_j}+\abs{\p_y\tilde{W}_i}\right)+\bar{\nu}^2\abs{\p_y^2\theta_i} \Big],\\
			&\abs{F_{A2}}=O(1)\Doo\left( v_{\neq} \cdot \nabla v^1_{\neq}+v_{0\neq}\cdot\nabla v^1_{0\neq}+{\lde}_{\neq}\p_x{\lde}_{\neq}+\nu {\lde}_{\neq}\Delta v_{\neq}^{1} +\nu {\lde}_{0\neq}\Delta v_{0\neq}^{1}\right).
		\end{align*}
		Then, we have
		\begin{Lem}\label{Aor}
			Under the same  assumptions as \cref{MT} and  Proposition \ref{est:main}, it holds that
			\begin{align}
				&\norm{\mathcal{A}(\cdot,t)}_{L^2}^2+\nu \int_0^t \norm{\p_y\mathcal{A}(\cdot,\tau)}_{L^2}^2 d\tau \lesssim \varepsilon^2 \nu^{-1}.
			\end{align}
		\end{Lem}
		\begin{proof}
			Let 
			$\hat M^{(4)}(t):=\sup_{0\le\tau \le t}\norm{\mathcal{A}}_{L^2}^2.$ Since $\abs{{\lde}_{00},v^2_{00}}=O(1)\sum_{i,j=1}^2\left[\abs{\tilde{W}_{iy}}+\abs{\theta_j}\right]$, multiplying \cref{equ-AA} by $\mathcal{A}$, one has
			\begin{align}
				&\frac{1}2\norm{\mathcal{A}(\cdot,T)}^2+\nu\int_0^T\norm{\mathcal{A}_y}^2dt\notag\\ \nonumber
				= &\frac12\norm{\mathcal{A}(\cdot,0)}^2+ \int_0^T\int_\R F_A \mathcal{A} - v^2_{00} \mathcal{A}\mathcal{A}_y-\nu \frac{{\lde}_{00}}{{\lde}_{00}+1} \p_y^2 \mathcal{A} \mathcal{A} + a\nu(\p_y\tilde{W}_2-\p_y\tilde{W}_1)\p_y \mathcal{A}  dydt\\ \nonumber
				\lesssim&\norm{\mathcal{A}(\cdot,0)}^2+\varepsilon^2 \nu^{-1}\int_0^T\int_{\R}\left[\bar{\nu}(1+t)\right]^{-1}e^{-\frac{(y\pm(1+t))^2}{\bar{\nu}(1+t)}}\mathcal{A}^2dydt+\nu\int_0^T \norm{\p_y\tilde{W}_2-\p_y\tilde{W}_1}_{L^2}^2 dt+\delta\nu\int_0^T\norm{\mathcal{A}_y}_{L^2}^2dt\\ \nonumber
				&+\nu^{-1}\norm{\mathcal{A}}_{L^\infty L^2}\norm{\p_y\mathcal{A}}_{L^\infty L^2}\sum_{i,j,k=1}^2\int_0^T \norm{\p_y\theta_i}_{L^2}^2+\norm{\p_y^k\tilde{W}_j}_{L^2}^2dt+\int_0^T\int_{\R}(F_{A1}+F_{A2})\mathcal{A}dydt\\
				\leq&\varepsilon^2 \nu^{-1}+\frac{1}{10}\left(\nu\int_0^T\norm{\mathcal{A}_y}_{L^2}^2dt+\hat M^{(4)}(t) \right)+\int_0^T\int_{\R}(F_{A1}+F_{A2})\mathcal{A}dydt\notag\\
				&+C\varepsilon^2 \nu^{-1}\int_0^T\int_{\R}\left[\bar{\nu}(1+t)\right]^{-1}e^{-\frac{(y\pm(1+t))^2}{\bar{\nu}(1+t)}}\mathcal{A}^2dydt,
			\end{align}
			where we have used  Lemma \ref{estimateontheta},  Lemma \ref{estimateonxi},  Lemma \ref{1or},  Lemma \ref{estimateonta},  Proposition \ref{est:main} and the following 
			\begin{align*}
				&\int_0^T \int_{\R} v^2_{00} \mathcal{A} \mathcal{A}_y dydt \leq \delta \nu\int_0^T \norm{\mathcal{A}_y}^2 dt + C \nu^{-1} \sum_{i,j=1}^2 \int_0^T \int_{\R} \left(\theta_j^2 + \tilde{W}_{iy}^2 \right) \mathcal{A}^2 dydt\\
				&\qquad\qquad\qquad\qquad\;\;\lesssim\varepsilon^2 \nu^{-1}\int_0^T\int_{\R}\left[\bar{\nu}(1+t)\right]^{-1}e^{-\frac{(y\pm(1+t))^2}{\bar{\nu}(1+t)}}\mathcal{A}^2dydt+\delta \nu\int_0^T \norm{\mathcal{A}_y}_{L^2}^2 dt\\
				&\qquad\qquad\qquad\qquad\quad\;\;\;+\nu^{-1}\norm{\mathcal{A}}_{L^\infty L^2}\norm{\p_y\mathcal{A}}_{L^\infty L^2}\sum_{i=1}^2\int_0^T \norm{\p_y\tilde{W}_i}_{L^2}^2dt,\\
				& \norm{\mathcal A}_{L^\infty L^\infty} \le \norm{\mathcal A}_{L^\infty L^2}^{\frac12}\norm{\p_y\mathcal A}_{L^\infty L^2}^{\frac12} \\
				&\qquad\qquad\;\le \sum_{i,j=1}^2 [\hat M^{(4)}(t)]^{\frac14}\left(\norm{\p_y v_{00}^1}_{L^\infty L^2}+ \norm{\theta_i}_{L^\infty L^2}+\norm{({\lde}_{00},v^2_{00})}_{L^\infty L^2}+\norm{\p_y\Xi_j}_{L^\infty L^2}+\norm{\p_y\theta_A}_{L^\infty L^2}\right)^{\frac12}\\
				&\qquad\qquad\;\le \varepsilon^{\frac12} \nu^{-\frac14}[\hat M^{(4)}(t)]^{\frac14},\\
				& \int_0^T\|\partial_y^2 \tilde{W}^2_j\|^2_{L^2} dt\leq \sum_{j=1,2}\int_0^T \norm{(\p_y {\lde}_{00},\p_yv_{00}^2)}_{L^2}^2 + \norm{(\p_y \theta_j,\p_y\Xi_j)}_{L^2}^2 + \norm{\p_y({\lde}_{0\neq}v^2_{0\neq})}_{L^2}^2+ \norm{\p_y({\lde}_{\neq}v^2_{\neq})}_{L^2}^2 d\tau\\
				&\qquad\qquad\qquad\;\; \;\le\varepsilon^2 \nu^{-\frac32}.
			\end{align*}
			For the source term $F_{A1},F_{A2}$, by  Lemma \ref{estimateontheta}- Lemma \ref{1or},  Lemma \ref{estimateonta} and  Proposition \ref{est:main}, we have
			\begin{align*}
				\int_0^T\int_{\R} F_{A1}\mathcal{A}dydt \lesssim &\nu \sum_{i=1}^2\int_0^t \norm{\p_y \tilde W_i}_{L^2}^2 d\tau + \nu^{-1}\norm{\mathcal A}_{L^\infty L^2}^2 \sum_{i=1}^2\int_{0}^t \norm{\Xi_i}_{L^\infty}^2 d\tau + \nu\int_0^T \norm{\p_y^2 \theta_A}_{L^2}^2 d\tau + \delta \nu\int_0^T \norm{\p_y \mathcal A}_{L^2}^2 d\tau \\
				&+\nu \norm{\mathcal A}_{L^\infty L^\infty}^2 \sum_{i,j=1}^2\int_0^t \norm{\p_y \tilde W_i}_{L^2}^2 + \norm{\p_y^2 \tilde W_j}_{L^2}^2 + \norm{\p_y^2 \theta_A}_{L^2}^2 d\tau+\nu^5 \norm{\mathcal A}_{L^\infty L^2}^{\frac23}\sum_{i=1}^2\int_0^t \norm{\p_y^2 \theta_i}_{L^1}^{\frac43} d\tau \\
				&+\varepsilon^2\nu\int_0^t\int_{\R} [\nu(1+t)]^{-1}e^{-\frac{|y\pm \bar c(1+t)|^2}{1+t}}\abs{\mathcal A}^2 dyd\tau\\
				\leq&\varepsilon^2 \nu^{-1}+\frac{1}{10}\hat M^{(4)}(t)+ \frac{\nu}{10}  \int_0^t \norm{\p_y \mathcal A}_{L^2}^2 d\tau+C\varepsilon^2\nu\int_0^t\int_{\R} [\nu(1+t)]^{-1}e^{-\frac{|y\pm \bar c(1+t)|^2}{1+t}}\abs{\mathcal A}^2 dyd\tau,\\
				\int_0^T\int_{\R} F_{A2}\mathcal{A}dydt \lesssim & \norm{\mathcal A}_{L^\infty L^\infty}\int_0^t \Big[\norm{v_{\neq}}_{L^2}\norm{\nabla v_{\neq}^{1}}_{L^2} + (\norm{v_{0\neq}^2}_{L^2}+\norm{v_{0\neq}^3}_{L^2})\norm{\nabla_{y,z}v_{0\neq}^1}_{L^2} + \norm{{\lde}_{\neq}}_{L^2}\norm{\nabla {\lde}_\neq}_{L^2} \\
				&\qquad\qquad\qquad+\nu \norm{\phi_{\neq}}_{L^2}\norm{\Delta v_{\neq}^1}_{L^2} + \nu \norm{\phi_{0\neq}}_{L^2}\norm{\Delta_{y,z}v_{0\neq}^1}_{L^2} \Big] d\tau \\
				\le &\frac{1}{10}\hat M^{(4)}(t) + C \varepsilon^2 \nu^{-1},
			\end{align*}
			Since $\theta_A^{(2)}$ does not have propagation speed, while the weight function $
			e^{-\frac{(y\pm (1+t))^2}{\bar{\nu}(1+t)}}$ travels forward with speed $\pm1$, we can use the same method as \cref{q}-\cref{new1} to obtain
			\begin{align}
				\bar{\nu}\int_0^T\int_{R}\left[\bar{\nu}(1+t)\right]^{-1}&e^{-\frac{(y\pm(1+t))^2}{\bar{\nu}(1+t)}} \abs{\mathcal A}^2 dydt \lesssim  \bar{\nu}^{\frac{1}{2}}\int_0^T\int_\R\eta_1h\abs{\mathcal A}^2dydt
				\lesssim {\nu}^{\frac{3}{2}}\int_0^T\norm{\p_y \mathcal A}_{L^2}^2dt\nonumber.
			\end{align}
			where
			\begin{align}
				\eta_1=\exp \left(\int_{-\infty}^{y} h(z, t) d z\right), \quad h(t,z)=\frac{1}{\sqrt{2 \pi \bar{\nu}(1+t)}} \exp \left(-\frac{\left(z\pm (1+t)\right)^2}{2 \bar{\nu}(1+t)}\right),
			\end{align}
			Then we arrive at
			\begin{align}
				&\norm{\mathcal{A}(\cdot,T)}_{L^2}^2+\nu\int_0^T\norm{\mathcal{A}_y}_{L^2}^2dt
				\lesssim \varepsilon^2 \nu^{-1}.
			\end{align}
			Therefore Lemma \ref{Aor} is completed.
		\end{proof}
		
		\subsection{Estimate on $v_{00}^1$}
		In this subsection, we give an estimate of $\norm{v_{00}^1}_{L^\infty}$
		\begin{Thm}\label{2026-3-21-2}
			Under the same  assumptions as \cref{MT} and  Proposition \ref{est:main}, it holds that
			\begin{align}
				\norm{v_{00}^1}_{L^\infty}\lesssim \varepsilon\nu^{-\frac12}.
			\end{align}
		\end{Thm}
		\begin{proof}
			Combining  Lemma \ref{estimateontheta}- Lemma \ref{1or},  Lemma \ref{estimateonta}, Lemma \ref{Aor} and  Proposition \ref{est:main}, we have
			\begin{align*}
				\norm{v_{00}^1}_{L^\infty} &\le \norm{\mathcal A}_{L^\infty}+\sum_{j=1}^2\big(\norm{\tilde W_j}_{L^\infty}+\norm{\Xi_j}_{L^\infty}\big)+\norm{\theta_A}_{L^\infty}\\
				&\lesssim\norm{\mathcal{A}}_{L^2}^{\frac12}\left(\norm{\p_y v_{00}^1}_{L^\infty L^2}+ \norm{\theta_i}_{L^\infty L^2}+\norm{({\lde}_{00},v^2_{00})}_{L^\infty L^2}+\norm{\p_y\Xi_j}_{L^\infty L^2}+\norm{\p_y\theta_A}_{L^\infty L^2}\right)^{\frac12}+\varepsilon \nu^{-\frac18} \\
				&\lesssim \varepsilon\nu^{-\frac12}.
			\end{align*}
			Then we have finished the proof of \cref{2026-3-21-2}.
		\end{proof}
		
		\section{Energy estimates for the compressible part at low regularity}
		In this section, we prove the compressible part of  Proposition \ref{prop-bs} for $j=0$.
		\subsection{$H^3$ estimate on $({\de}_0,p^{-\frac{1}{2}}D_0)$}
		In this subsection, we improve the bootstrap hypothesis \cref{bs-n0}.
		\begin{Lem}
			Under the same assumptions as \cref{MT} and  Proposition \ref{prop-bs}, it holds that,
			\begin{align}
				\sup_{0\leq t\leq T}\|({\de}_0,p^{-\frac{1}{2}}D_0)(t)\|_{H^3}^2&
				+\nu\int_{0}^{T}\|\na(V^2_0,V^3_0)\|_{L^2}^2\,dt
				\nonumber\\
				&+(\nu+\nu')\int_{0}^{T}\|D_0\|_{H^4}^2\,dt+c_0\nu\int_{0}^{T}\|\na{\de}_0\|_{H^2}^2\,dt
				\lesssim \eps^2+\nu^{-1}\eps^3 .
			\end{align}
			
		\end{Lem}
		Due to the absence of $L^2_t$-estimates for ${\de}_{00}$, we need to obtain energy estimates in $L^2$ and $\dot{H}^s$, $s=1,2,3,$ respectively.
		\subsubsection{$L^2$ estimate of $({\de}_0,p^{-\frac{1}{2}}D_0)$}\label{sec411}
		
		\begin{Lem} \label{Lem-L2-1}
			Under the same assumptions as \cref{MT} and  Proposition \ref{prop-bs}, it holds that,
			\begin{align}
				\sup_{0\leq t\leq T}\|({\de}_0,p^{-\frac{1}{2}}D_0)(t)\|_{L^2}^2
				+\nu\int_{0}^{T}\|\na(V^2_0,V^3_0)\|_{L^2}^2\,dt
				+(\nu+\nu')\int_{0}^{T}\|D_0\|_{L^2}^2\,dt
				\lesssim \eps^2+\nu^{-1}\eps^3 .
			\end{align}
			
		\end{Lem}
		\begin{proof}

			The system for the zero modes $(\rho_0,V^2_0,V^3_0)$ reads
			\begin{equation}\label{eq:rhoV-system-zero}
				\left\{\begin{aligned}
					&\p_t\rho_0+\p_Y(\rho_0V^2_0)+\p_Z(\rho_0 V^3_0)=G_1,\\
					&\rho_0\p_t V^i_0+P'(\rho_0)\p_i\rho_0-\bigl(\nu\Delta V^i_0+(\nu+\nu')\p_iD_0\bigr)+\rho_0(V_0\cdot\na V^i_0)=G_i,\qquad i=2,3,
				\end{aligned}\right.
			\end{equation}
			where the nonlinear remainders are
			\begin{align*}
				G_1&=-\p_Y(\rho_{\neq}V^2_{\neq})_0-\p_Z(\rho_{\neq}V^3_{\neq})_0,\\[4pt]
				G_i&=-\rho_0\Bigl[\Bigl(\frac{P'(\rho)}{\rho}\Bigr)_0-\frac{P'(\rho_0)}{\rho_0}\Bigr]\p_i\rho_0
				+\rho_0\Bigl[\Bigl(\frac{1}{\rho}\Bigr)_0-\frac{1}{\rho_0}\Bigr]\bigl(\nu\Delta V^i_0+(\nu+\nu')\p_iD_0\bigr)\\
				&\quad -\rho_0\Bigl[\Bigl(\frac{P'(\rho)}{\rho}\Bigr)_{\neq}\tna_i\rho_{\neq}\Bigr]_0
				+\rho_0\Bigl[\Bigl(\frac{1}{\rho}\Bigr)_{\neq}\bigl(\nu\tilde{\Delta} V^i_{\neq}+(\nu+\nu')\tna_iD_{\neq}\bigr)\Bigr]_0
				-\rho_0(V_{\neq}\cdot\tna V^i_{\neq})_0 .
			\end{align*}
			Since $\rho_0=1+{\de}_0$ we employ the physical energy
			\[
			E^{(1)}(t)=\Bigl\la\rho_0,\int_{1}^{\rho_0}\frac{P(s)-P(1)}{s^2}\,ds\Bigr\ra_{L^2}
			+\frac12\|\sqrt{\rho_0}\,(V^2_0,V^3_0)\|_{L^2}^2\;\gtrsim\;\|({\de}_0,p^{-1/2}D_0)\|_{L^2}^2 .
			\]
			A standard energy computation yields
			\begin{align*}
				E^{(1)}(T)&+\nu\|\na(V^2_0,V^3_0)\|_{L^2L^2}^2+(\nu+\nu')\|D_0\|_{L^2L^2}^2 \\[2pt]
				&=E^{(1)}(0)+\int_{0}^{T}\Bigl\la G_1,\;
				\int_{1}^{\rho_0}\frac{P(s)-P(1)}{s^2}ds+\frac{P(\rho_0)-P(1)}{\rho_0}
				+\frac12|V^2_0,V^3_0|^2\Bigr\ra_{L^2}\,dt \\[2pt]
				&\quad +\sum_{i=2,3}\int_{0}^{T}\la V^i_0,\,G_i\ra_{L^2}\,dt
				\;=:\,E^{(1)}(0)+I_1+I_2 .
			\end{align*}
			Because $P'(1)=1$ and $\rho_{\neq}={\de}_{\neq}$, Taylor expansion gives
			\[
			\int_{1}^{\rho_0}\frac{P(s)-P(1)}{s^2}ds\approx\frac12{\de}_0^2,\qquad
			\frac{P(\rho_0)-P(1)}{\rho_0}\approx {\de}_0 .
			\]
			Consequently,
			\begin{align*}
				|I_1|&\lesssim\|{\de}_{\neq}\|_{L^2H^3}\|(V^2_{\neq},V^3_{\neq})\|_{L^2H^3}
				\bigl(\|{\de}_0\|_{L^\infty L^2}+\|(V^2_0,V^3_0)\|_{L^\infty L^2}^2\bigr)\\[2mm]
				&\lesssim\nu^{-1/3}\eps\;\nu^{-1/3}\eps\;\eps^2
				\lesssim\nu^{-2/3}\eps^3 .
			\end{align*}
			For $I_2$ we use the expansions
			\[
			\Bigl(\frac{P'(\rho)}{\rho}\Bigr)_0=\frac{P'(\rho_0)}{\rho_0}+O\bigl(({\de}_{\neq}^2)_0\bigr),\qquad
			\Bigl(\frac{P'(\rho)}{\rho}\Bigr)_{\neq}=O({\de}_{\neq}),
			\]
			\[
			\Bigl(\frac{1}{\rho}\Bigr)_0=\frac{1}{\rho_0}+O\bigl(({\de}_{\neq}^2)_0\bigr),\qquad
			\Bigl(\frac{1}{\rho}\Bigr)_{\neq}=O({\de}_{\neq}),
			\]
			together with the bootstrap bounds.  This yields
			\begin{align*}
				|I_2|&\lesssim\Bigl(\|\na {\de}_0\|_{L^2L^2}
				+\nu\|\Delta(V^2_0,V^3_0)\|_{L^2L^2}
				+\|\tna {\de}_{\neq}\|_{L^2L^2} \\
				&\qquad\qquad +\nu\|\tilde{\Delta}(V^2_{\neq},V^3_{\neq})\|_{L^2L^2}
				+(\nu+\nu')\|\na D_0\|_{L^2L^2}
				+(\nu+\nu')\|\tna D_{\neq}\|_{L^2L^2}\Bigr)\\
				&\qquad\times\|(V^2_0,V^3_0)\|_{L^\infty L^2}\|{\de}_{\neq}\|_{L^2H^3}+\|V_{\neq}\|_{L^2H^3}\|\tna(V^2_{\neq},V^3_{\neq})\|_{L^2L^2}
				\|(V^2_0,V^3_0)\|_{L^\infty L^2}\\[2pt]
				&\lesssim\bigl(\nu^{-1/2}\eps+\nu\nu^{-1/2}\eps
				+(\nu+\nu')\nu^{-1/2}\eps+\nu^{-2/3}\eps
				+\nu\nu^{-1}\eps+(\nu+\nu')\nu^{-1}\eps\bigr)\nu^{-1/3}\eps^2
				+\nu^{-1/3}\eps\,\nu^{-2/3}\eps^2 \\[2pt]
				&\lesssim\nu^{-1}\eps^3 .
			\end{align*}
			Collecting the two contributions we obtain
			\[
			\sup_{0\leq t\leq T}\|({\de}_0,p^{-\frac{1}{2}}D_0)(t)\|_{L^2}^2
			+\nu\int_{0}^{T}\|\na(V^2_0,V^3_0)\|_{L^2}^2\,dt
			+(\nu+\nu')\int_{0}^{T}\|D_0\|_{L^2}^2\,dt
			\lesssim \eps^2+\nu^{-1}\eps^3 .
			\]
			The proof is concluded by noticing that $\eps\sim \nu^{3/2}$ for $\nu$ small.
		\end{proof}
		
		\subsubsection{Control of $\|\na^s({\de}_0,p^{-\frac{1}{2}}D_0)\|_{L^2}$ for $s=1,2,3$}\label{sec412}
		\begin{Lem}\label{Lem-Hs-1}
			Under the same assumptions as \cref{MT} and  Proposition \ref{prop-bs}, it holds that,
			\begin{align}
				\sup_{t\in[0,T]}\|\na^s({\de}_0,p^{-1/2}D_0)(t)\|_{L^2}^2
				+\nu\int_{0}^{T}\|\na^sD_0\|_{L^2}^2\,dt
				+c_0\nu\int_{0}^{T}\|\na^s{\de}_0\|_{L^2}^2\,dt
				\lesssim\eps^2+\nu^{-1}\eps^3 .
			\end{align}
			
		\end{Lem}
		\begin{proof}
			
			We now consider the higher-order energies
			\[
			E^{\nno}_s(t)=\frac12\|\na^s({\de}_0,p^{-\frac{1}{2}}D_0)\|_{L^2}^2
			-c_0\bar{\nu}\, Re\bigl\la p^{-\frac{1}{2}}\na^s{\de}_0,\;\na^s(p^{-\frac{1}{2}}D_0)\bigr\ra_{L^2},
			\]
			where the constant $c_0>0$ is chosen so that $E_s$ is coercive.  Differentiating and using the equations for ${\de}_0$ and $D_0$ gives
			\begin{align*}
				E^{\nno}_s(T)&+\bar{\nu}\|\na^sD_0\|_{L^2L^2}^2+c_0\bar{\nu}\|\na^s{\de}_0\|_{L^2L^2}^2 \\[2pt]
				&=E^{\nno}_s(0)+c_0\bar{\nu}\|\na^sp^{-1/2}D_0\|_{L^2L^2}^2
				-c_0\int_{0}^{T}\bar{\nu}\bigl\la\na^s{\de}_0,\;\na^s(\bar{\nu} D_0)\bigr\ra_{L^2}\,dt \\[2pt]
				&\quad+\int_{0}^{T}\bigl\la\na^s(\mathcal{NL}_{\de})_0,\;\na^s{\de}_0\bigr\ra_{L^2}\,dt
				+\int_{0}^{T}\bigl\la\na^s(p^{-\frac{1}{2}}\na\!\cdot\!\mathcal{NL}_V)_0,\;\na^sp^{-\frac{1}{2}}D_0\bigr\ra_{L^2}\,dt \\[2pt]
				&\quad+c_0\bar{\nu}\int_{0}^{T} Re\Bigl[
				\bigl\la\na^sp^{-1/2}(\mathcal{NL}_{\de})_0,\;\na^sp^{-\frac{1}{2}}D_0\bigr\ra_{L^2}
				+\bigl\la\na^sp^{-1/2}{\de}_0,\;\na^s(p^{-\frac{1}{2}}\na\!\cdot\!\mathcal{NL}_V)_0\bigr\ra_{L^2}\Bigr]\,dt \\[2pt]
				&=:E^{\nno}_s(0)+c_0\bar{\nu}\|\na^sp^{-1/2}D_0\|_{L^2L^2}^2+J_1+J_2+J_3+J_4 .
			\end{align*}
			Here the nonlinearities are
			\[
			\mathcal{NL}_{\de}=-V\cdot\tna {\de}-{\de}D,\qquad
			\mathcal{NL}_V=-V\cdot\tna V-\bigl[F({\de})\tna {\de}+G({\de})(\nu\tilde{\Delta}V+(\nu+\nu')\tna D)\bigr].
			\]
			
			The term $J_1$ is linear and easily bounded by
			\[
			|J_1|\lesssim\bar{\nu}^2\|\na^s{\de}_0\|_{L^2L^2}\|\na^sD_0\|_{L^2L^2}
			\le\frac{\bar{\nu}^2}{4}\|\na^sD_0\|_{L^2L^2}^2+C\bar{\nu}^2\|\na^s{\de}_0\|_{L^2L^2}^2 .
			\]
			
			We split $J_2$ according to the structure of $\mathcal{NL}_{\de}$:
			\begin{align*}
				J_2&=\int_{0}^{T}\bigl\la\na^s(V_0\cdot\na {\de}_0),\na^s{\de}_0\bigr\ra_{L^2}\,dt
				+\int_{0}^{T}\bigl\la\na^s(V_{\neq}\cdot\tna {\de}_{\neq})_0,\na^s{\de}_0\bigr\ra_{L^2}\,dt \\
				&\quad+\int_{0}^{T}\bigl\la\na^s({\de}_0D_0),\na^s{\de}_0\bigr\ra_{L^2}\,dt
				+\int_{0}^{T}\bigl\la\na^s({\de}_{\neq}D_{\neq})_0,\na^s{\de}_0\bigr\ra_{L^2}\,dt \\
				&=:J_{21}+J_{22}+J_{23}+J_{24}.
			\end{align*}
			For $J_{21}$ we integrate by parts and use the commutator estimate of Lemma \ref{lemma:commutator}:
			\begin{align*}
				J_{21}&=-\int_{0}^{T}\bigl\la[\na^s,V_0]\cdot\na {\de}_0,\;\na^s{\de}_0\bigr\ra_{L^2}\,dt
				+\int_{0}^{T}\bigl\la D_0\na^s{\de}_0,\;\na^s{\de}_0\bigr\ra_{L^2}\,dt \\
				&\lesssim\Bigl(\|\na(V^2_0,V^3_0)\|_{L^\infty L^\infty}\|\na^s{\de}_0\|_{L^2L^2}
				+\|\na^s(V^2_0,V^3_0)\|_{L^\infty L^2}\|\na {\de}_0\|_{L^2L^\infty}
				+\|D_0\|_{L^\infty L^\infty}\|\na^s{\de}_0\|_{L^2L^2}\Bigr)\|\na^s{\de}_0\|_{L^2L^2}\\
				&\lesssim\|(V^2_0,V^3_0)\|_{L^\infty H^3}\|\na {\de}_0\|_{L^2H^2}^2
				\lesssim\eps\,\nu^{-1/2}\eps\,\nu^{-1/2}\eps
				=\nu^{-1}\eps^3 .
			\end{align*}
			The remaining three terms are directly bounded with the help of  Proposition \ref{est:main}:
			\begin{align*}
				|J_{22}|+|J_{23}|+|J_{24}|
				&\lesssim\|V_{\neq}\|_{L^2H^3}\|\tna {\de}_{\neq}\|_{L^2H^3}\|\na^s{\de}_0\|_{L^\infty L^2}
				+\|{\de}_0\|_{L^\infty H^3}\|D_0\|_{L^2H^3}\|\na^s{\de}_0\|_{L^2L^2}\\
				&\quad+\|{\de}_{\neq}\|_{L^2H^3}\|D_{\neq}\|_{L^2H^3}\|\na^s{\de}_0\|_{L^\infty L^2}\\
				&\lesssim\nu^{-1/3}\eps\,\nu^{-2/3}\eps\,\eps
				+\eps\,\nu^{-1/2}\eps\,\nu^{-1/2}\eps
				+\nu^{-1/3}\eps\,\nu^{-2/3}\eps\,\eps
				\lesssim\nu^{-1}\eps^3 .
			\end{align*}
			
			Next we treat $J_3$.  Using the pointwise bound
			\begin{equation}\label{ineq:div}
				|(k,\eta-kt,l)|p^{-1/2}\le1,
			\end{equation}
			and integration by parts we obtain
			\begin{align*}
				|J_3|&\lesssim\|(V^2_0,V^3_0)\|_{L^\infty H^3}\|\na(V^2_0,V^3_0)\|_{L^2H^3}\|\na^sp^{-1/2}D_0\|_{L^2L^2}\\
				&\quad+\|V_{\neq}\|_{L^2H^3}\|\tna(V^2_{\neq},V^3_{\neq})\|_{L^2H^3}\|\na^sp^{-1/2}D_0\|_{L^\infty L^2}\\
				&\quad+\|{\de}_0\|_{L^\infty H^2}\|\na {\de}_0\|_{L^2H^2}\|\na^sD_0\|_{L^2L^2}
				+\|{\de}_{\neq}\|_{L^2H^3}\|\tna {\de}_{\neq}\|_{L^2H^3}\|\na^sp^{-1/2}D_0\|_{L^\infty L^2}\\
				&\quad+\|{\de}_0\|_{L^\infty H^2}\Bigl(\nu\|\Delta(V^2_0,V^3_0)\|_{L^2H^2}
				+(\nu+\nu')\|\na D_0\|_{L^2H^2}\Bigr)\|\na^sD_0\|_{L^2L^2}\\
				&\quad+\|{\de}_{\neq}\|_{L^2H^3}\Bigl(\nu\|\tilde{\Delta}(V^2_{\neq},V^3_{\neq})\|_{L^2H^3}
				+(\nu+\nu')\|\tna D_{\neq}\|_{L^2H^3}\Bigr)\|\na^sp^{-1/2}D_0\|_{L^\infty L^2}\\
				&\lesssim\eps\,\nu^{-1/2}\eps\,\nu^{-1/2}\eps
				+\nu^{-1/3}\eps\,\nu^{-2/3}\eps\,\eps
				+\eps\,\nu^{-1/2}\eps\,\nu^{-1/2}\eps
				+\nu^{-1/3}\eps\,\nu^{-2/3}\eps\,\eps\\
				&\quad+\eps\bigl(\nu\nu^{-1/2}\eps+(\nu+\nu')\nu^{-1/2}\eps\bigr)\nu^{-1/2}\eps
				+\nu^{-1/3}\eps\bigl(\nu\nu^{-1}\eps+(\nu+\nu')\nu^{-1}\eps\bigr)\eps\\
				&\lesssim\nu^{-1}\eps^3 .
			\end{align*}
			
			The term $J_4$ is handled in the same way, giving a bound of order $\nu^{-1}\eps^3$ as well.  Collecting all the estimates and choosing $\eps$ sufficiently small (depending on $\nu$) we finally obtain
			\[
			\sup_{t\in[0,T]}\|\na^s({\de}_0,p^{-1/2}D_0)(t)\|_{L^2}^2
			+\nu\int_{0}^{T}\|\na^sD_0\|_{L^2}^2\,dt
			+c_0\nu\int_{0}^{T}\|\na^s{\de}_0\|_{L^2}^2\,dt
			\lesssim\eps^2+\nu^{-1}\eps^3 .
			\]
			Thus, for $\eps$ small enough, the desired $H^3$ estimate for the zero-mode compressible part follows. Then Lemma \ref{Lem-Hs-1} follows.
		\end{proof}
		
		\subsection{$H^3$ estimates on $m^{\frac{1}{4}}M \partial_X({\de}_{\neq},p^{-\frac{1}{2}}D_{\neq})$\ and $m^{\frac{1}{4}}M \partial_Z({\de},p^{-\frac{1}{2}}D)$}\label{4.4}
		In this subsection, we improve \cref{bs-zn0} and \cref{bs-xn} with $j=0$. Estimates on $ \partial_X({\de}_{\neq},p^{-\frac{1}{2}}D_{\neq})$\ and $\partial_Z({\de},p^{-\frac{1}{2}}D)$ are similar, we only estimate the later one for short.

		\subsubsection{Estimate for $\| \partial_Z({\de}_{0\neq},p^{-\frac{1}{2}}D_{0\neq})\|_{H^3}$}
		Firstly, we define the energy
		\begin{equation*}
			E^{\nzno}(t)=\frac{1}{2}\|\partial_Z({\de}_{0\neq},p^{-\frac{1}{2}}D_{0\neq})\|^2_{H^3}-c_0\bar{\nu} Re\langle p^{-1/2}\partial_Z {\de}_{0\neq},\partial_Z p^{-1/2}D_{0\neq}\rangle_{H^3}.
		\end{equation*}
		Then one has
		\begin{Lem}\label{Lem-dz}
			Under the same assumptions as \cref{MT} and  Proposition \ref{prop-bs}, it holds that,
			\begin{align}
				\sup_{t\in[0,T]}E^{(2)}(t)
				+\bar{\nu}\|\partial_Z D_0\|^2_{L^2H^3}+c_0\bar{\nu}\|\partial_Z {\de}_0\|^2_{L^2H^3}
				\lesssim E^{(2)}(0)+\nu^{-3/2}\eps^3 .
			\end{align}
		\end{Lem}
		\begin{proof}
			
			The energy estimate reads
			\begin{align*}
				&E^{\nzno}(T)+\bar{\nu}\|\partial_Z D_0\|^2_{L^2H^3}+c_0\bar{\nu}\|\partial_Z {\de}_0\|^2_{L^2H^3}\\
				=&E^{\nzno}(0)+c_0\bar{\nu}\|\partial_Z p^{-\frac{1}{2}}D_0\|^2_{L^2H^3}-c_0\bar{\nu} Re\langle \partial_Z {\de}_0,\bar{\nu}\partial_Z D_0\rangle_{H^3}\\
				&+ \int_{0}^{T}\langle \partial_Z(\mathcal{NL}_{{\de}})_0,\partial_Z {\de}_0\rangle_{H^3} dt+\int_{0}^{T}\langle \partial_Z ( p^{-\frac{1}{2}}\di\mathcal{NL}_{V})_0,\partial_Zp^{-\frac{1}{2}}D_0\rangle_{H^3} dt\\
				&+c_0\bar{\nu}\int_{0}^{T} Re\langle \partial_Z p^{-\frac{1}{2}}(\mathcal{NL}_{{\de}})_0,\partial_Z p^{-\frac{1}{2}}D_0\rangle_{H^3}+ Re \langle \partial_Z p^{-\frac{1}{2}}{\de}_0,\partial_Z ( p^{-\frac{1}{2}}\di\mathcal{NL}_{V})_0\rangle_{H^3}dt\\
				=&:E^{\nzno}(0)+c_0\bar{\nu}\|\partial_Z p^{-1/2}D_0\|^2_{L^2H^3}-c_0\bar{\nu} Re\langle \partial_Z {\de}_0,\bar{\nu}\partial_Z D_0\rangle_{H^3}+I^{\nzno}_1+I^{\nzno}_2+I^{\nzno}_3.
			\end{align*}
			For $I^{\nzno}_1$, we divide it into 
			\begin{align*}
				I^{\nzno}_1=&\int_{0}^{T}\langle \partial_Z(V_0\cdot \nabla {\de}_0),\partial_Z {\de}_0\rangle_{H^3}dt+\int_{0}^{T}\langle \partial_Z(V_{\neq}\cdot\tilde{\nabla} {\de}_{\neq})_0,\partial_Z {\de}_0\rangle_{H^3}dt\\
				&+\int_{0}^{T}\langle \partial_Z({\de}_0D_0),\partial_Z {\de}_0\rangle_{H^3}dt+\int_{0}^{T}\langle \partial_Z({\de}_{\neq}D_{\neq})_0,\partial_Z {\de}_0\rangle_{H^3}dt=:\sum_{i=1}^{4}I^{\nzno}_{1i}.
			\end{align*}
			Using the commutator estimate and integration by part, we obtain
			\begin{align*}
				I^{\nzno}_{11}=&\int_{0}^{T}-\langle (\partial_ZV_{0})\cdot\nabla {\de}_0, \partial_Z {\de}_0\rangle_{H^3} dt\\
				&\qquad\qquad\qquad+\int_{0}^{T}-\langle [\langle\nabla\rangle^3, V_0]\cdot\nabla \partial_Z{\de}_0,\langle\nabla\rangle^3 \partial_Z{\de}_0\rangle_{L^2}+\langle D_0\langle\nabla\rangle^3 \partial_Z{\de}_0,\langle\nabla\rangle^3\partial_Z {\de}_0\rangle_{L^2}dt\\
				\lesssim&{\|\partial_Z(V^2_0,V^3_0)\|_{L^2H^3}\|\nabla {\de}_0\|_{L^\infty H^3}\|\partial_Z {\de}_0\|_{L^2H^3}}+(\|(V^2_0,V^3_0)\|_{L^\infty H^3}+\|p^{-1/2}D_0\|_{L^\infty H^3})\|\partial_Z{\de}_0\|^2_{L^2H^3}\\[2mm]
				\lesssim&(\nu^{-1/2}\varepsilon)^3+\varepsilon(\nu^{-1/2}\varepsilon)^2\lesssim\nu^{-3/2}\varepsilon^3.
			\end{align*}
			Turning to the other three terms,  Proposition \ref{est:main} implies that
			\begin{align*}
				I^{\nzno}_{12}+I^{\nzno}_{13}+I^{\nzno}_{14}\lesssim&\|(1,\partial_Z)V_{\neq}\|_{L^2H^3}\|\tilde{\nabla} (1,\partial_Z){\de}_{\neq}\|_{L^2H^3}\|\partial_Z {\de}_0\|_{L^\infty H^3}\\
				&+\|(1,\partial_Z){\de}_0\|_{L^\infty H^3}\|(1,\partial_Z)D_0\|_{L^2H^3}\|\partial_Z{\de}_0\|_{L^2H^3}\\
				&+\|(1,\partial_Z){\de}_{\neq}\|_{L^2H^3}\|(1,\partial_Z)D_{\neq}\|_{L^2H^3}\|\partial_Z{\de}_0\|_{L^\infty H^3}\\
				\lesssim&\nu^{-1/3}\varepsilon\nu^{-2/3}\varepsilon^2+\varepsilon\nu^{-1/2}\varepsilon\nu^{-1/2}\varepsilon+\nu^{-1/3}\varepsilon\nu^{-2/3}\varepsilon^2\lesssim\nu^{-1}\varepsilon^3.
			\end{align*}
			For $I^{\nzno}_2$, applying \cref{ineq:div}, we integrate by parts to obtain
			\begin{align*}
				I^{\nzno}_2 &\lesssim \|(1,\partial_Z)(V^2_0,V^3_0)\|_{L^\infty H^3} \|(1,\partial_Z)\nabla(V^2_0,V^3_0)\|_{L^2H^3} \|\partial_Z p^{-1/2} D_0\|_{L^2H^3} \\
				&\quad + \|\nabla_{X,Z}V_{\neq}\|_{L^2H^3} \|\tilde{\nabla}\nabla_{X,Z}(V^2_{\neq},V^3_{\neq})\|_{L^2H^3} \|\partial_Z p^{-1/2} D_0\|_{L^\infty H^3} \\
				&\quad + \|{\de}_0\|_{L^\infty H^3} \|\nabla(1,\partial_Z) {\de}_0\|_{L^2H^2} \|\partial_Z D_0\|_{L^2H^3} \\
				&\quad + \|\nabla_{X,Z}{\de}_{\neq}\|_{L^2 H^3} \|\tilde{\nabla}\nabla_{X,Z} {\de}_{\neq}\|_{L^2H^3} \|\partial_Z p^{-1/2} D_0\|_{L^\infty H^3} \\
				&\quad + \|{\de}_0\|_{L^\infty H^3} \Bigl( \nu\|\Delta (1,\partial_Z)(V^2_0,V^3_0)\|_{L^2H^2} + (\nu+\nu')\|\nabla (1,\partial_Z)D_0\|_{L^2H^2} \Bigr) \|\partial_Z D_0\|_{L^2 H^3} \\
				&\quad + \|\nabla_{X,Z}{\de}_{\neq}\|_{L^2H^3} \Bigl( \nu\|\tilde{\Delta}\nabla_{X,Z}(V^2_{\neq},V^3_{\neq})\|_{L^2H^3} + (\nu+\nu')\|\tilde{\nabla}\nabla_{X,Z}D_{\neq}\|_{L^2H^3} \Bigr) \|\partial_Z p^{-1/2} D_0\|_{L^\infty H^3} \\
				&\lesssim \varepsilon\nu^{-1/2}\varepsilon\nu^{-1/2}\varepsilon + \nu^{-1/3}\varepsilon\nu^{-2/3}\varepsilon^2 + \varepsilon\nu^{-1/2}\varepsilon\nu^{-1/2}\varepsilon + \nu^{-1/3}\varepsilon\nu^{-2/3}\varepsilon^2 \\
				&\quad + \varepsilon\bar{\nu}\nu^{-1/2}\varepsilon\nu^{-1/2}\varepsilon + \nu^{-1/3}\varepsilon\bar{\nu}\nu^{-1}\varepsilon^2 \\
				&\lesssim \nu^{-1}\varepsilon^3.
			\end{align*}
			The analysis for $I_3^{(2)}$
			follows similarly, so we omit the details. Combining all the estimates of $I_{1}^{(2)}, I_{2}^{(2)},I_{3}^{(2)}$, Lemma \ref{Lem-dz} follows.
		\end{proof}
		
		\subsubsection{Estimate for $\lVert \ma\partial_Z({\de}_{\neq},p^{-\frac{1}{2}}D_{\neq})\rVert_{H^3}$}
		The associated energy functional is defined by
		\begin{align*}
			E^{\nxn}(t) 
			&= \frac{1}{2}\lVert m^{\frac{1}{4}}M\partial_Z({\de}_{\neq},p^{-\frac{1}{2}}D_{\neq})\rVert_{H^3}^2 
			+ \frac{1}{4} Re\left\langle \frac{\partial_t p}{p^{\frac{3}{2}}}m^{\frac{1}{4}}M\partial_Z{\de}_{\neq},\, m^{\frac{1}{4}}M\partial_Zp^{-\frac{1}{2}}D_{\neq}\right\rangle_{H^3} \\
			&\quad - c_1\bar{\nu}^{\frac{1}{3}} Re\left\langle p^{-\frac{1}{2}}m^{\frac{1}{4}}M\partial_Z{\de}_{\neq},\, m^{\frac{1}{4}}M\partial_Zp^{-\frac{1}{2}}D_{\neq}\right\rangle_{H^3}.
		\end{align*}
		One has
		\begin{Lem}\label{Lem-E3}
			Under the same assumptions as \cref{MT} and  Proposition \ref{prop-bs}, it holds that,
			\begin{align}
				\sup_{t\in[0,T]}E^{(3)}(t)&
				+ \bar{\nu}\lVert m^{\frac{1}{4}}M\partial_ZD_{\neq}\rVert_{L^2H^3}^2 
				+ \frac{1}{4}\Bigl\lVert\sqrt{-\frac{\partial_t m}{m}}\,m^{\frac{1}{4}}M\partial_Z({\de}_{\neq},p^{-\frac{1}{2}}D_{\neq})\Bigr\rVert_{L^2H^3}^2\nonumber \\
				&\quad + \Bigl\lVert\sqrt{-\frac{\partial_t M}{M}}\,m^{\frac{1}{4}}M\partial_Z({\de}_{\neq},p^{-\frac{1}{2}}D_{\neq})\Bigr\rVert_{L^2H^3}^2 
				+ c_1\bar{\nu}^{\frac{1}{3}}\lVert \ma \partial_Z{\de}_{\neq}\rVert_{L^2H^3}^2 
				\lesssim E^{(3)}(0)+\nu^{-3/2}\eps^3 .
			\end{align}
		\end{Lem}
		\begin{proof}
			
			Differentiating $E^{\nxn}(t)$ and integrating by parts yield the energy inequality
			\begin{align*}
				&E^{\nxn}(T) + \bar{\nu}\lVert m^{\frac{1}{4}}M\partial_ZD_{\neq}\rVert_{L^2H^3}^2 
				+ \frac{1}{4}\Bigl\lVert\sqrt{-\frac{\partial_t m}{m}}\,m^{\frac{1}{4}}M\partial_Z({\de}_{\neq},p^{-\frac{1}{2}}D_{\neq})\Bigr\rVert_{L^2H^3}^2 \\
				&\quad + \Bigl\lVert\sqrt{-\frac{\partial_t M}{M}}\,m^{\frac{1}{4}}M\partial_Z({\de}_{\neq},p^{-\frac{1}{2}}D_{\neq})\Bigr\rVert_{L^2H^3}^2 
				+ c_1\bar{\nu}^{\frac{1}{3}}\lVert \ma \partial_Z{\de}_{\neq}\rVert_{L^2H^3}^2 \\
				&\leq E^{\nxn}(0) + LS^{\nxn} + LE^{\nxn} + I_1^{\nxn} + I_2^{\nxn} + I_3^{\nxn} + I_4^{\nxn},
			\end{align*}
			where the lower-order and error terms are given by
			\begin{align*}
				LS^{\nxn} &= \int_{0}^{T} \Bigl[ \frac{1}{4}\left\langle\frac{\partial_t p}{p}\ma\partial_Z {\de}_{\neq},\ma\partial_Z {\de}_{\neq}\right\rangle_{H^3} 
				+ \frac{1}{4}\left\langle\frac{\partial_t p}{p}\ma\partial_Z p^{-\frac{1}{2}}D_{\neq},\ma\partial_Z p^{-\frac{1}{2}}D_{\neq}\right\rangle_{H^3} \Bigr] dt, \\
				LE^{\nxn} &= 2\int_{0}^{T}\left\langle \ma \frac{\partial_{XZ}}{p^{\frac{3}{2}}}(W^2_{\neq}-\partial_X{\de}+\nu\partial_XD)_{\neq},\ma \partial_Zp^{-\frac{1}{2}}D_{\neq}\right\rangle_{H^3} \\
				&\quad -\frac{1}{4} Re\left\langle \left(\frac{\partial_{tt}p}{p^{\frac{3}{2}}}+2\frac{(\partial_tp)^2}{p^{\frac{5}{2}}}\right)\ma\partial_Z{\de}_{\neq},\ma\partial_Zp^{-\frac{1}{2}}D_{\neq}\right\rangle_{H^3} \\
				&\quad +\frac{1}{2} Re\left\langle \frac{\partial_t p}{p^{\frac{3}{2}}}\left(\frac{\partial_t m}{4m}+\frac{\partial_tM}{M}\right)\ma\partial_Z{\de}_{\neq},\ma\partial_Zp^{-\frac{1}{2}}D_{\neq}\right\rangle_{H^3} \\
				&\quad +\frac{1}{4}\left\langle\frac{\partial_tp}{p^{\frac{3}{2}}}\ma \partial_Z{\de}_{\neq},\ma\Bigl(\frac{1}{2}\frac{\partial_tp}{p}\partial_Zp^{-\frac{1}{2}}D_{\neq}+ 2\ma \frac{\partial_{XZ}}{p^{\frac{3}{2}}}(W^2_{\neq}-\partial_X{\de}+\nu\partial_XD)_{\neq}-\bar{\nu} p^{\frac{1}{2}}\partial_ZD_{\neq}\Bigr)\right\rangle_{H^3} \\
				&\quad + c_1\bar{\nu}^{\frac{1}{3}}\lVert\ma \partial_Zp^{-\frac{1}{2}} D_{\neq}\rVert_{H^3}^2 
				+ c_1\bar{\nu}^{\frac{1}{3}} Re\left\langle \frac{\partial_t p}{2p^{\frac{3}{2}}}\ma\partial_Z{\de}_{\neq},\ma\partial_Zp^{-\frac{1}{2}}D_{\neq}\right\rangle_{H^3} \\
				&\quad -2c_1\bar{\nu}^{\frac{1}{3}} Re\left\langle p^{-\frac{1}{2}}\left(\frac{\partial_t m}{4m}+\frac{\partial_tM}{M}\right)\ma\partial_Z{\de}_{\neq},\ma\partial_Zp^{-\frac{1}{2}}D_{\neq}\right\rangle_{H^3} \\
				&\quad -c_1\bar{\nu}^{\frac{1}{3}}\left\langle p^{-\frac{1}{2}}\ma \partial_Z{\de}_{\neq},\ma\Bigl(\frac{1}{2}\frac{\partial_tp}{p}\partial_Zp^{-\frac{1}{2}}D_{\neq}+ 2\ma \frac{\partial_{XZ}}{p^{\frac{3}{2}}}(W^2_{\neq}-\partial_X{\de}+\nu\partial_XD)_{\neq}-\bar{\nu} p^{\frac{1}{2}}\partial_ZD_{\neq}\Bigr)\right\rangle_{H^3} dt, \\
				I^{\nxn}_1 &= \int_{0}^{T}\left\langle \ma \partial_Z(\nlt_{{\de}})_{\neq},\ma\partial_Z{\de}_{\neq}\right\rangle_{H^3} dt, \\
				I^{\nxn}_2 &= \int_{0}^{T}\left\langle \ma \partial_Zp^{-\frac{1}{2}}(\nlt_{D})_{\neq},\ma\partial_Zp^{-\frac{1}{2}}D_{\neq}\right\rangle_{H^3} dt, \\
				I^{\nxn}_3 &= \frac{1}{4}\int_{0}^{T}\Bigl[ \left\langle\frac{\partial_tp}{p^{\frac{3}{2}}}\ma\partial_Z {\de}_{\neq},\ma \partial_Zp^{-\frac{1}{2}}(\nlt_{D})_{\neq} \right\rangle_{H^3} 
				+ \left\langle\frac{\partial_tp}{p^{\frac{3}{2}}} \ma \partial_Z(\nlt_{{\de}})_{\neq},\ma\partial_Zp^{-\frac{1}{2}}D_{\neq}\right\rangle_{H^3} \Bigr] dt, \\
				I^{\nxn}_4 &= -c_1\bar{\nu}^{\frac{1}{3}}\int_{0}^{T}\Bigl[ \left\langle p^{-\frac{1}{2}}\ma\partial_Z {\de}_{\neq},\ma \partial_Zp^{-\frac{1}{2}}(\nlt_{D})_{\neq} \right\rangle_{H^3} 
				+ \left\langle p^{-\frac{1}{2}} \ma \partial_Z(\nlt_{{\de}})_{\neq},\ma\partial_Zp^{-\frac{1}{2}}D_{\neq}\right\rangle_{H^3} \Bigr] dt,
			\end{align*}
			with the nonlinear source terms $\nlt_{\de},\nlt_D$ given in \cref{t-ND}.
			
			\noindent\textbf{Estimates for $LS^{\nxn}$ and $LE^{\nxn}$.} 
			By the definition and monotonicity properties of the multipliers $m$ and $M$, we obtain
			\begin{align*}
				LS^{\nxn} &\leq \frac{1}{4}\Bigl\lVert\sqrt{-\frac{\partial_t m}{m}}\,m^{\frac{1}{4}}M\partial_Z({\de}_{\neq},p^{-\frac{1}{2}}D_{\neq})\Bigr\rVert_{L^2H^3}^2 
				+ \frac{c_1}{4}\nu^{\frac{1}{3}}\Bigl\lVert m^{\frac{1}{4}}M\partial_Z({\de}_{\neq},p^{-\frac{1}{2}}D_{\neq})\Bigr\rVert_{L^2H^3}^2, \\
				LE^{\nxn} &\leq \frac{40}{\tilde{C}}\Bigl\lVert \sqrt{-\frac{\partial_t M}{M}}\,\ma(W^2_{\neq},\nabla_{X,Z}{\de}_{\neq},\nabla_{X,Z}p^{-\frac{1}{2}}D_{\neq})\Bigr\rVert_{L^2H^3}^2 
				+ \frac{1}{20}\bar{\nu}\lVert\ma \partial_ZD_{\neq}\rVert_{L^2H^3}^2 \\
				&\quad + \frac{c_1}{20}\nu^{\frac{1}{3}}\lVert\ma\partial_Z{\de}_{\neq}\rVert_{L^2H^3}^2 
				\lesssim \varepsilon^2.
			\end{align*}
			
			\noindent\textbf{Estimate of $I^{\nxn}_1$.} 
			We split $I^{\nxn}_1$ into four components:
			\begin{align*}
				I^{\nxn}_1 &= \int_{0}^{T}\left\langle \ma \partial_Z(V\cdot \tilde{\nabla} {\de}+{\de}D)_{\neq},\ma\partial_Z{\de}_{\neq}\right\rangle_{H^3} dt \\
				&= \int_{0}^{T}\left\langle \ma \partial_Z(V^1_0\partial_X {\de}_{\neq}),\ma\partial_Z{\de}_{\neq}\right\rangle_{H^3} dt 
				+ \sum_{i=2,3}\int_{0}^{T}\left\langle \ma \partial_Z(V^i_0\tilde{\partial}_i {\de}_{\neq}),\ma\partial_Z{\de}_{\neq}\right\rangle_{H^3} dt \\
				&\quad + \int_{0}^{T}\left\langle \ma \partial_Z(V_{\neq}\cdot\tilde{\nabla} {\de}),\ma\partial_Z{\de}_{\neq}\right\rangle_{H^3} dt 
				+ \int_{0}^{T}\left\langle \ma \partial_Z({\de}D)_{\neq},\ma\partial_Z{\de}_{\neq}\right\rangle_{H^3} dt 
				= \sum_{i=1}^{4}I^{\nxn}_{1i}.
			\end{align*}
			For $I^{\nxn}_{11}$, applying the commutator estimate \cref{est:m}, \cref{est:m4}, and integration by parts yields
			\begin{align*}
				I^{\nxn}_{11} &= \int_{0}^{T}\left\langle \ma (\partial_ZV^1_0)\partial_X {\de}_{\neq},\ma\partial_Z{\de}_{\neq}\right\rangle_{H^3} dt \\
				&\quad + \int_{0}^{T}\left\langle \ma[\langle\nabla\rangle^3, V^1_0]\partial_X \partial_Z{\de}_{\neq},\langle\nabla\rangle^3\ma\partial_Z{\de}_{\neq}\right\rangle_{L^2} dt \\
				&\quad + \int_{0}^{T}\left\langle m^{\frac{1}{4}}[M, V^1_0]\partial_X\langle\nabla\rangle^3 \partial_Z{\de}_{\neq},\langle\nabla\rangle^3\ma\partial_Z{\de}_{\neq}\right\rangle_{L^2}dt \\
				&\quad + \int_{0}^{T}\left\langle [m^{\frac{1}{4}}, V^1_0]\partial_X\langle\nabla\rangle^3 M\partial_Z{\de}_{\neq},\langle\nabla\rangle^3\ma\partial_Z{\de}_{\neq}\right\rangle_{L^2}dt \\
				&\quad + \int_{0}^{T}\left\langle  V^1_0\partial_X\langle\nabla\rangle^3 \ma\partial_Z{\de}_{\neq},\langle\nabla\rangle^3\ma\partial_Z{\de}_{\neq}\right\rangle_{L^2}dt \\
				&\lesssim \bigl(\|\partial_Z V^1_0\|_{L^\infty H^4}\|m^{\frac{1}{4}}\partial_X{\de}_{\neq}\|_{L^2H^3}
				+ \|\nabla V^1_0\|_{L^\infty H^3}\|m^{\frac{1}{4}}\partial_Z{\de}_{\neq}\|_{L^2H^3}\bigr)\|\ma \partial_Z{\de}_{\neq}\|_{L^2H^3} \\
				&\lesssim \nu^{-1}\varepsilon \cdot \nu^{-1/6}\varepsilon \cdot \nu^{-1/6}\varepsilon = \nu^{-4/3}\varepsilon^3.
			\end{align*}
			
			For the remaining terms,  Proposition \ref{est:main} implies that
			\begin{align*}
				I^{\nxn}_{12}+I^{\nxn}_{13}+I^{\nxn}_{14} &\lesssim \bigl(\|(1,\partial_Z)(V^2_0,V^3_0)\|_{L^\infty H^3}\|\nabla_{X,Z}p^{\frac{1}{2}}{\de}_{\neq}\|_{L^2H^3} 
				+ \|\nabla_{X,Z}V_{\neq}\|_{L^2H^3}\|(1,\partial_Z)p^{\frac{1}{2}}{\de}\|_{L^\infty H^3} \\
				&\quad + \|(1,\partial_Z){\de}\|_{L^\infty H^3}\|\nabla_{X,Z}D_{\neq}\|_{L^2H^3} 
				+ \|\nabla_{X,Z}{\de}_{\neq}\|_{L^\infty H^3}\|(1,\partial_Z)D_0\|_{L^2H^3}\bigr)\|\ma\partial_Z{\de}_{\neq}\|_{L^2H^3} \\
				&\lesssim \bigl(\varepsilon\nu^{-2/3}\varepsilon + \nu^{-1/3}\varepsilon\nu^{-1/2}\varepsilon + \nu^{-1/6}\varepsilon\nu^{-2/3}\varepsilon + \nu^{-1/6}\varepsilon\nu^{-1/2}\varepsilon\bigr)\nu^{-1/6}\varepsilon 
				\lesssim \nu^{-1}\varepsilon^3.
			\end{align*}
			
			\noindent\textbf{Estimate of $I^{(3)}_2$.} 
			We decompose $I^{(3)}_2$ as
			\begin{align*}
				I^{\nxn}_2 &= \int_{0}^{T}\left\langle \ma \partial_Zp^{-\frac{1}{2}}(V\cdot\tilde{\nabla} D)_{\neq},\ma\partial_Zp^{-\frac{1}{2}}D_{\neq}\right\rangle_{H^3} dt \\
				&\quad + \int_{0}^{T}\left\langle \ma \partial_Zp^{-\frac{1}{2}}(\tilde{\partial}_iV^j\tilde{\partial}_jV^i)_{\neq},\ma\partial_Zp^{-\frac{1}{2}}D_{\neq}\right\rangle_{H^3} dt \\
				&\quad + \int_{0}^{T}\left\langle \ma \partial_Zp^{-\frac{1}{2}}(\dl(F({\de})\tilde{\nabla}{\de}))_{\neq},\ma\partial_Zp^{-\frac{1}{2}}D_{\neq}\right\rangle_{H^3} dt \\
				&\quad + \int_{0}^{T}\left\langle \ma \partial_Zp^{-\frac{1}{2}}(\dl(G({\de})(\nu\tilde{\Delta} V+(\nu+\nu')\tilde{\nabla}D)))_{\neq},\ma\partial_Zp^{-\frac{1}{2}}D_{\neq}\right\rangle_{H^3} dt 
				= \sum_{i=1}^{4}I^{\nxn}_{2i}.
			\end{align*}
			The leading part $I^{\nxn}_{21}$ is further split into $T^{\nxn}_1+T^{\nxn}_2+T^{\nxn}_3$. By commutator estimates \cref{est:m4},
			\begin{align*}
				T^{\nxn}_1 &= \int_{0}^{T}\left\langle \ma [p^{-\frac{1}{2}},\partial_ZV^1_0]\partial_X D_{\neq},\ma\partial_Zp^{-\frac{1}{2}}D_{\neq}\right\rangle_{H^3} dt \\
				&\quad + \int_{0}^{T}\left\langle \ma (\partial_ZV^1_0\partial_X p^{-\frac{1}{2}}D_{\neq}),\ma\partial_Zp^{-\frac{1}{2}}D_{\neq}\right\rangle_{H^3} dt \\
				&\quad + \int_{0}^{T}\left\langle ([\ma \langle \nabla\rangle^3p^{-\frac{1}{2}},V^1_0]\partial_X\partial_Z D_{\neq}),\ma\langle \nabla\rangle^3\partial_Zp^{-\frac{1}{2}}D_{\neq}\right\rangle_{L^2} dt \\
				&\lesssim \|\partial_Z\nabla V^1_0\|_{L^\infty H^3}\|\partial_Xp^{-\frac{1}{2}}D_{\neq}\|_{L^2H^3}\|p^{-\frac{1}{2}}\ma \partial_Zp^{-\frac{1}{2}}D_{\neq}\|_{L^2H^3} \\
				&\quad + \|\partial_ZV^1_0\|_{L^\infty H^4}\|\ma\partial_Xp^{-\frac{1}{2}}D_{\neq}\|_{L^2H^3}\|\ma \partial_Zp^{-\frac{1}{2}}D_{\neq}\|_{L^2H^3} \\
				&\quad + \|\nabla V^1_0\|_{L^\infty H^3}\|m^{\frac{1}{4}}\partial_Zp^{-\frac{1}{2}}D_{\neq}\|_{L^2H^3}\|\partial_Xp^{-\frac{1}{2}}\ma\partial_Zp^{-\frac{1}{2}}D_{\neq}\|_{L^2H^3} \\
				&\quad + \|\nabla V^1_0\|_{L^\infty H^3}\|m^{\frac{1}{4}}\partial_Zp^{-\frac{1}{2}}D_{\neq}\|_{L^2H^3}\|\ma\partial_Zp^{-\frac{1}{2}}D_{\neq}\|_{L^2H^3} \\
				&\lesssim \nu^{-1}\varepsilon\nu^{-1/3}\varepsilon^2 + \nu^{-1}\varepsilon(\nu^{-1/6}\varepsilon)^2 + \nu^{-1}\varepsilon\nu^{-1/6}\varepsilon^2 + \nu^{-1}\varepsilon\nu^{-1/6}\varepsilon\nu^{-1/6}\varepsilon 
				\lesssim \nu^{-4/3}\varepsilon^3,
			\end{align*}
			The remaining parts satisfy
			\begin{align*}
				T^{\nxn}_2+T^{\nxn}_3 &\lesssim \bigl(\|(V^2_0,V^3_0)\|_{L^\infty H^3} + \|V_{\neq}\|_{L^\infty H^3}\bigr)\|\tilde{\nabla}D\|_{L^2H^3}\|\ma\partial_Zp^{-\frac{1}{2}}D_{\neq}\|_{L^2H^3} \\
				&\lesssim \nu^{-1/6}\varepsilon \cdot \nu^{-1}\varepsilon \cdot \nu^{-1/6}\varepsilon \lesssim \nu^{-4/3}\varepsilon^3.
			\end{align*}
			For $I^{\nxn}_{22}$, direct H{\"o}lder's inequality and Sobolev embeddings give
			\begin{align*}
				I^{\nxn}_{22} &= \int_{0}^{T}\left\langle \ma \partial_Zp^{-\frac{1}{2}}(\tilde{\partial}_iV^j\tilde{\partial}_jV^i)_{\neq},\ma\partial_Zp^{-\frac{1}{2}}D_{\neq}\right\rangle_{H^3} dt \\
				&\lesssim \bigl(\|\partial_XV^1_{\neq}\|_{L^\infty H^3}^2\|\partial_XV^1_{\neq}\|_{L^2 H^3} 
				+ \|\tilde{\partial}_YV^1\|_{L^\infty H^3}\|\partial_XV^2_{\neq}\|_{L^2H^3} + \|\partial_ZV^1\|_{L^\infty H^3}\|\partial_XV^3_{\neq}\|_{L^2H^3} \\
				&\quad + \|\tilde{\partial}_YV^2\|_{L^\infty H^3}\|\tilde{\partial}_YV^2_{\neq}\|_{L^2H^3} + \|\tilde{\partial}_YV^3\|_{L^\infty H^3}\|\partial_ZV^2_{\neq}\|_{L^2H^3} 
				+ \|\tilde{\partial}_YV^3_{\neq}\|_{L^2 H^3}\|\partial_ZV^2_0\|_{L^\infty H^3} \\
				&\quad + \|\partial_ZV^3\|_{L^\infty H^3}\|\partial_ZV^3_{\neq}\|_{L^2H^3}\bigr)\|\ma\partial_Zp^{-\frac{1}{2}}D_{\neq}\|_{L^2 H^3} \\
				&\lesssim \bigl(\varepsilon\nu^{-1/6}\varepsilon + \nu^{-1}\varepsilon\nu^{-1/3}\varepsilon + \nu^{-1}\varepsilon\nu^{-1/6}\varepsilon + \nu^{-1/2}\varepsilon\nu^{-2/3}\varepsilon + (\nu^{-1/3}\varepsilon)^2 + \nu^{-1/2}\varepsilon^2 + \varepsilon\nu^{-1/6}\varepsilon\bigr)\nu^{-1/6}\varepsilon \\
				&\lesssim \nu^{-3/2}\varepsilon^3.
			\end{align*}
			Finally, using the divergence structure in \cref{ineq:div}, we have
			\begin{align*}
				I^{\nxn}_{23}+I^{\nxn}_{24} &= \int_{0}^{T}\left\langle \ma \partial_Zp^{-\frac{1}{2}}(\dl(F({\de})\tilde{\nabla}{\de}))_{\neq},\ma\partial_Zp^{-\frac{1}{2}}D_{\neq}\right\rangle_{H^3} dt \\
				&\quad + \int_{0}^{T}\left\langle \ma \partial_Zp^{-\frac{1}{2}}(\dl(G({\de})(\nu\tilde{\Delta} V+(\nu+\nu')\tilde{\nabla}D)))_{\neq},\ma\partial_Zp^{-\frac{1}{2}}D_{\neq}\right\rangle_{H^3} dt \\	
				&\lesssim \bigl(\|(1,\partial_Z){\de}\|_{L^\infty H^3}\|\nabla_{X,Z}\tilde{\nabla} {\de}_{\neq}\|_{L^2H^3} 
				+ \|\nabla_{X,Z}{\de}_{\neq}\|_{L^2 H^3}\|(1,\partial_Z)\nabla {\de}_0\|_{L^\infty H^3} \\
				&\quad + \|\partial_X{\de}_{\neq}\|_{L^2 H^3}\bigl(\nu\|\tilde{\Delta} V^1\|_{L^\infty H^3}+(\nu+\nu')\|\partial_XD_{\neq}\|_{L^\infty H^3}\bigr) \\
				&\quad + \|{\de}_0\|_{L^\infty H^3}\bigl(\nu\|\partial_X\tilde{\Delta} V^1_{\neq}\|_{L^2 H^3}+(\nu+\nu')\|\partial_{XX}D_{\neq}\|_{L^2 H^3}\bigr) \\
				&\quad + \|\partial_Z{\de}_{\neq}\|_{L^2 H^3}\bigl(\nu\|\tilde{\Delta} V^2\|_{L^\infty H^3}+(\nu+\nu')\|\tilde{\partial}_YD\|_{L^\infty H^3}\bigr) \\
				&\quad + \|{\de}_0\|_{L^\infty H^3}\bigl(\nu\|\partial_Z\tilde{\Delta} V^2_{\neq}\|_{L^2 H^3}+(\nu+\nu')\|\tilde{\partial}_Y\partial_ZD_{\neq}\|_{L^2 H^3}\bigr) \\
				&\quad + \|\partial_Z{\de}_{\neq}\|_{L^2 H^3}\bigl(\nu\|\tilde{\Delta} V^3\|_{L^\infty H^3}+(\nu+\nu')\|\partial_ZD\|_{L^\infty H^3}\bigr) \\
				&\quad + \|{\de}_0\|_{L^\infty H^3}\bigl(\nu\|\partial_Z\tilde{\Delta} V^3_{\neq}\|_{L^2 H^3}+(\nu+\nu')\|\partial_{ZZ}D_{\neq}\|_{L^2 H^3}\bigr)\bigr)\|\ma\partial_Zp^{-\frac{1}{2}}D_{\neq}\|_{L^2 H^3} \\
				&\lesssim \bigl(\nu^{-1/6}\varepsilon\nu^{-2/3}\varepsilon + \nu^{-1/3}\varepsilon\nu^{-1/2}\varepsilon + \nu^{-1/3}\varepsilon(\nu\nu^{-1}\varepsilon+(\nu+\nu')\nu^{-1/2}\varepsilon) + \varepsilon\bar{\nu}\nu^{-1}\varepsilon \\
				&\quad + \nu^{-1/3}\varepsilon\bar{\nu}\nu^{-5/6}\varepsilon + \varepsilon\bar{\nu}\nu^{-1}\varepsilon + \nu^{-1/3}\varepsilon(\nu\nu^{-2/3}\varepsilon+(\nu+\nu')\nu^{-1/2}\varepsilon) + \varepsilon\bar{\nu}\nu^{-1}\varepsilon\bigr)\nu^{-1/6}\varepsilon \\
				&\lesssim \nu^{-1}\varepsilon^3.
			\end{align*}
			
			\noindent\textbf{Estimates for $I^{\nxn}_3$ and $I^{\nxn}_4$.} 
			We first decompose $I^{(3)}_3$ into three parts:
			\begin{align*}
				I^{\nxn}_3 = I^{\nxn}_{31} + I^{\nxn}_{32} + I^{\nxn}_{33},
			\end{align*}
			where
			\begin{align*}
				I^{\nxn}_{31} &= \int_{0}^{T}\left\langle \frac{\partial_t p}{p^{\frac{3}{2}}} \ma (V^1_0\partial_{XZ} {\de}_{\neq}),\ma\partial_Zp^{-\frac{1}{2}}D_{\neq}\right\rangle_{H^3} 
				+ \left\langle \frac{\partial_t p}{p^{\frac{3}{2}}}\ma\partial_Z{\de}_{\neq},\ma p^{-\frac{1}{2}}(V^1_0\partial_{XZ} D_{\neq})\right\rangle_{H^3} dt, \\
				I^{\nxn}_{32} &= \int_{0}^{T}\left\langle \frac{\partial_t p}{p^{\frac{3}{2}}} \ma (\partial_ZV^1_0\partial_X {\de}_{\neq}),\ma\partial_Zp^{-\frac{1}{2}}D_{\neq}\right\rangle_{H^3} 
				+ \left\langle \frac{\partial_t p}{p^{\frac{3}{2}}}\ma\partial_Z{\de}_{\neq},\ma p^{-\frac{1}{2}}(\partial_ZV^1_0\partial_X D_{\neq})\right\rangle_{H^3} dt, \\
				I^{\nxn}_{33} &= \frac{1}{4}\int_{0}^{T}\Bigl[ \left\langle\frac{\partial_tp}{p^{\frac{3}{2}}}\ma\partial_Z {\de}_{\neq},\ma \partial_Zp^{-\frac{1}{2}}(\nlt_{D}-V^1_0\partial_XD_{\neq})_{\neq} \right\rangle_{H^3} 
				\\
				&\quad+ \left\langle\frac{\partial_tp}{p^{\frac{3}{2}}} \ma \partial_Z(\nlt_{{\de}}-V^1_0\partial_X{\de}_{\neq})_{\neq},\ma\partial_Zp^{-\frac{1}{2}}D_{\neq}\right\rangle_{H^3} \Bigr] dt.
			\end{align*}
			For $I^{\nxn}_{31}$, applying the commutator estimate yields
			\begin{align*}
				I^{\nxn}_{31} &= \int_{0}^{T}\left\langle \frac{\partial_t p}{p^{\frac{3}{2}}} \ma (V^1_0\partial_{XZ} {\de}_{\neq}),\ma\partial_Zp^{-\frac{1}{2}}D_{\neq}\right\rangle_{H^3} 
				+ \left\langle \frac{\partial_t p}{p^{\frac{3}{2}}}\ma\partial_Z{\de}_{\neq},\ma p^{-\frac{1}{2}}(V^1_0\partial_{XZ} D_{\neq})\right\rangle_{H^3} dt \\
				&= \int_{0}^{T}\left\langle ([\langle \nabla\rangle^3 \ma p^{-\frac{1}{2}},V^1_0]\partial_{XZ} {\de}_{\neq}),\langle \nabla\rangle^3\ma\frac{\partial_t p}{p}\partial_Zp^{-\frac{1}{2}}D_{\neq}\right\rangle_{L^2} dt \\
				&\quad + \int_{0}^{T}\left\langle \langle \nabla\rangle^3\ma p^{-\frac{1}{2}}\partial_Z{\de}_{\neq}, ([\langle \nabla\rangle^3\ma \frac{\partial_t p}{p}p^{-\frac{1}{2}},V^1_0]\partial_{XZ} D_{\neq})\right\rangle_{L^2} dt \\
				&\quad + \int_{0}^{T} \left\langle V^1_0\partial_{X}\langle \nabla\rangle^3m^{\frac{1}{4}}Mp^{-\frac{1}{2}} \partial_Z{\de}_{\neq},\langle \nabla\rangle^3\ma\frac{\partial_t p}{p}\partial_Zp^{-\frac{1}{2}}D_{\neq}\right\rangle_{L^2} \\
				&\quad\quad + \left\langle \langle \nabla\rangle^3\ma p^{-\frac{1}{2}}\partial_Z{\de}_{\neq}, V^1_0\partial_{X}\langle \nabla\rangle^3m^{\frac{1}{4}}M\frac{\partial_t p}{p}\partial_Zp^{-\frac{1}{2}} D_{\neq}\right\rangle_{L^2} dt.
			\end{align*}
			Using integration by parts, the last term vanishes. Therefore,
			\begin{align*}
				I^{\nxn}_{31} &\lesssim \|\nabla V^1_{0}\|_{L^\infty H^4}\|m^{\frac{1}{4}} p^{-\frac{1}{2}}\partial_Z{\de}_{\neq}\|_{L^2H^3}\|\partial_Xp^{-\frac{1}{2}}\ma\partial_Zp^{-\frac{1}{2}}D_{\neq}\|_{L^2H^3} \\
				&\quad + \|\nabla V^1_{0}\|_{L^\infty H^4}\|\ma p^{-\frac{1}{2}}\partial_Z{\de}_{\neq}\|_{L^2H^3}\|m^{\frac{1}{4}}\partial_Zp^{-\frac{1}{2}}D_{\neq}\|_{L^2H^3} 
				\lesssim \nu^{-7/6}\varepsilon^3.
			\end{align*}
			For $I^{\nxn}_{32}$, by Proposition \ref{est:com-p}, we have
			\begin{align*}
				I^{\nxn}_{32}&= \int_{0}^{T}\left\langle \frac{\partial_t p}{p^{\frac{3}{2}}} \ma (\partial_ZV^1_0\partial_X {\de}_{\neq}),\ma\partial_Zp^{-\frac{1}{2}}D_{\neq}\right\rangle_{H^3} 
				+ \left\langle \frac{\partial_t p}{p^{\frac{3}{2}}}\ma\partial_Z{\de}_{\neq},\ma ([p^{-\frac{1}{2}},\partial_ZV^1_0]\partial_X D_{\neq})\right\rangle_{H^3} \\
				&\quad + \left\langle \frac{\partial_t p}{p^{\frac{3}{2}}}\ma\partial_Z{\de}_{\neq},\ma (\partial_ZV^1_0\partial_Xp^{-\frac{1}{2}} D_{\neq})\right\rangle_{H^3} dt \\
				&\lesssim \|\partial_ZV^1_0\|_{L^\infty H^4}\|m^{\frac{1}{4}}\partial_X{\de}_{\neq}\|_{L^2H^3}\|\partial_Xp^{-\frac{1}{2}}\ma\partial_Zp^{-\frac{1}{2}}D_{\neq}\|_{L^2H^3} \\
				&\quad + \|\partial_ZV^1_0\|_{L^\infty H^4}\|\partial_Xp^{-\frac{1}{2}}D_{\neq}\|_{L^2H^3}\|\partial_Xp^{-\frac{1}{2}}\ma\partial_Z{\de}_{\neq}\|_{L^2H^3} 
				\lesssim \nu^{-4/3}\varepsilon^3.
			\end{align*}
			The estimate of $I^{\nxn}_{33}$ follows the same pattern as $I^{\nxn}_{1}$ and $I^{\nxn}_{2}$; we omit the details and obtain
			\begin{align*}
				I^{\nxn}_{33} \lesssim \nu^{-4/3}\varepsilon^3.
			\end{align*}
			The term $I^{(3)}_4$ is estimated via an argument analogous to that used for $I^{(3)}_3$, the detailed derivations being omitted. Having obtained all necessary estimates for $I_{1}^{(3)}$, $I_{2}^{(3)}$, $I_{3}^{(3)}$, $I_{4}^{(3)}$, $LS^{(3)}$, and $LE^{(3)}$, we thus establish Lemma \ref{Lem-E3}.
		\end{proof}

		\subsection{$H^3$ estimates on $(p^{\frac{1}{2}}{\de},D)$ and $\nabla_{X,Z}(p^{\frac{1}{2}}{\de},D)$}
		\label{4.7}
		In this subsection, we improve the estimates \cref{bs-yn}-\cref{bs-zyn} for $j=0$. The estimates for $(p^{\frac{1}{2}}{\de},D)$ and $\partial_X(p^{\frac{1}{2}}{\de},D)$ are similar to those for $\partial_Z(p^{\frac{1}{2}}{\de},D)$. Therefore, we only focus on the energy
		\begin{align*}
			E^{\nyn}=\frac{1}{2}\|(\partial_Zp^{\frac{1}{2}}{\de},\partial_ZD)\|^2_{H^3}-\frac{1}{2} Re\left\langle \frac{\partial_t p}{p}\partial_Z{\de},\partial_ZD\right\rangle_{H^3}-c_1\bar{\nu}^{\frac{1}{3}} Re\langle \partial_Z{\de}_{\neq},\partial_ZD_{\neq}\rangle_{H^3}-c_1\bar{\nu}  Re\langle \partial_Z{\de}_0,\partial_ZD_0\rangle_{H^3}.
		\end{align*}
		\begin{Lem}\label{Lem-E4}
			Under the same assumptions as \cref{MT} and  Proposition \ref{prop-bs}, it holds that,
			\begin{align}
				\sup_{t\in[0,T]}E^{(4)}(t)&
				+\bar{\nu}\|\tilde{\nabla}\partial_ZD\|^2_{L^2H^3}+c_1\bar{\nu}^{\frac{1}{3}}\|\partial_Zp^{\frac{1}{2}}{\de}_{\neq}\|^2_{L^2H^3}+c_1\bar{\nu}\|\partial_Zp^{\frac{1}{2}}{\de}_0\|^2_{L^2H^3}
				\lesssim E^{(4)}(0)+\nu^{-3/2}\eps^3 .
			\end{align}
		\end{Lem}
		\begin{proof}
			
			The system of $\partial_Z(p^{\frac{1}{2}}{\de},D)$ reads
			\begin{align*}
				\left\{\begin{aligned}
					&\partial_t (\partial_Zp^{\frac{1}{2}}{\de})-\frac{\partial_t p}{2p}(\partial_Zp^{\frac{1}{2}}{\de})+p^{\frac{1}{2}}(\partial_ZD)=\partial_Zp^{\frac{1}{2}}\nlt_{{\de}},\\[2mm]
					&\partial_t (\partial_ZD)-\frac{\partial_tp}{p}(\partial_ZD)-p^{\frac{1}{2}}(\partial_Zp^{\frac{1}{2}}{\de})-2\frac{\partial_{XZ}}{p}(W^2-\partial_X{\de}+\nu\partial_XD)+\bar{\nu} p(\partial_ZD)=\partial_Z\nlt_{D},
				\end{aligned}\right.
			\end{align*}
			where the nonlinear terms $\nlt_{\de},\nlt_D$ are given in \cref{t-ND}.
			Then an energy estimate gives
			\begin{align*}
				&E^{\nyn}(T)+\bar{\nu}\|\tilde{\nabla}\partial_ZD\|^2_{L^2H^3}+c_1\bar{\nu}^{\frac{1}{3}}\|\partial_Zp^{\frac{1}{2}}{\de}_{\neq}\|^2_{L^2H^3}+c_1\bar{\nu}\|\partial_Zp^{\frac{1}{2}}{\de}_0\|^2_{L^2H^3}\\
				=&E^{\nyn}(0)+LS^{\nyn}+LE^{\nyn}+I^{\nyn}_1+I^{\nyn}_2+I^{\nyn}_3+I^{\nyn}_4,
			\end{align*}
			where we denote
			\begin{align*}
				LS^{\nyn}=&\int_{0}^{T}\left\langle\frac{3}{2}\frac{\partial_t p}{p}\partial_Z D_{\neq},\partial_Z D_{\neq}\right\rangle_{H^3} dt,\\
				LE^{\nyn}=&2\int_{0}^{T}\left\langle  \frac{\partial_{XZ}}{p}(W^2_{\neq}-\partial_X{\de}+\nu\partial_XD)_{\neq}, \partial_ZD_{\neq}\right\rangle_{H^3}-\frac{1}{2} Re\left\langle \left(\frac{\partial_{tt}p}{p}-\frac{(\partial_tp)^2}{p^{2}}\right)\partial_Z{\de}_{\neq},\partial_ZD_{\neq}\right\rangle_{H^3}\\
				&- Re\frac{1}{2}\left\langle\frac{\partial_tp}{p} \partial_Z{\de}_{\neq},\frac{\partial_tp}{p}\partial_ZD_{\neq}+ 2 \frac{\partial_{XZ}}{p}(W^2_{\neq}-\partial_X{\de}+\nu\partial_XD)_{\neq}-\bar{\nu} p\partial_ZD_{\neq}\right\rangle_{H^3}\\
				&+c_1\bar{\nu}^{\frac{1}{3}}\| \partial_Z D_{\neq}\|^2_{H^3}-c_1\bar{\nu}^{\frac{1}{3}}\left\langle \partial_Z{\de}_{\neq},\frac{\partial_tp}{p}\partial_ZD_{\neq}+ 2 \frac{\partial_{XZ}}{p}(W^2_{\neq}-\partial_X{\de}+\nu\partial_XD)_{\neq}-\bar{\nu} p\partial_ZD_{\neq}\right\rangle_{H^3}\\
				&+c_1\bar{\nu}\|\partial_ZD_0\|^2_{H^3}+c_1\bar{\nu}\left\langle \partial_Z{\de}_0,\bar{\nu} p\partial_ZD_0\right\rangle_{H^3} dt,\\
				I^{\nyn}_1=&\int_{0}^{T}\left\langle \partial_Zp^{\frac{1}{2}}\nlt_{{\de}},\partial_Zp^{\frac{1}{2}}{\de}\right\rangle_{H^3} dt,\quad
				I^{\nyn}_2=\int_{0}^{T}\left\langle \partial_Z\nlt_{D},\partial_ZD\right\rangle_{H^3} dt,\\
				I^{\nyn}_3=&-\frac{1}{2}\int_{0}^{T}\left\langle\frac{\partial_tp}{p}\partial_Z {\de}_{\neq}, \partial_Z(\nlt_{D})_{\neq} \right\rangle_{H^3}+\left\langle\frac{\partial_tp}{p}  \partial_Z(\nlt_{{\de}})_{\neq},\partial_ZD_{\neq}\right\rangle_{H^3} dt,\\
				I^{\nyn}_4=&-c_1\bar{\nu}^{\frac{1}{3}}\int_{0}^{T}\left\langle \partial_Z {\de}_{\neq}, \partial_Z(\nlt_{D})_{\neq} \right\rangle_{H^3}+\left\langle \partial_Z(\nlt_{{\de}})_{\neq},\partial_ZD_{\neq}\right\rangle_{H^3} dt,\\
				I^{\nyn}_5=&-c_1\bar{\nu}\int_{0}^{T}\left\langle \partial_Z {\de}_0, \partial_Z(\nlt_{D})_0\right\rangle_{H^3}+\left\langle \partial_Z(\nlt_{{\de}})_0,\partial_ZD_0\right\rangle_{H^3} dt.
			\end{align*}
			with the nonlinear terms $\nlt_{\de},\nlt_D$ given in \cref{t-ND}.
			For $LS^{\nyn}$ and $LE^{\nyn}$, by bootstrap argument, we have
			\begin{align*}
				LS^{\nyn}\lesssim&\frac{3}{2}\|m^{\frac{1}{4}}\partial_ZD_{\neq}\|^2_{L^2H^3}\lesssim\frac{3}{2C_0^2}(C_0\nu^{-1/2}\varepsilon)^2,\\
				LE^{\nyn}\lesssim&\|(m^{\frac{1}{4}}W^2_{\neq},m^{\frac{1}{4}}\nabla_{X,Z}{\de}_{\neq})\|_{L^2H^3}\|\partial_ZD_{\neq}\|_{L^2H^3}+\nu\|\nabla_{X,Z}D_{\neq}\|^2_{L^2H^3}+\|(m^{\frac{1}{4}}W^2_{\neq},m^{\frac{1}{4}}\nabla_{X,Z}{\de}_{\neq})\|^2_{L^2H^3}\\
				&+\bar{\nu}\|\partial_{Z}p^{\frac{1}{2}}{\de}_{\neq}\|_{L^2H^3}\|p^{\frac{1}{2}}\partial_ZD_{\neq}\|_{L^2H^3}+c_1\bar{\nu}^{\frac{1}{3}}\|\partial_ZD_{\neq}\|^2_{L^2H^3}+c_1\bar{\nu}\|\partial_ZD_0\|^2_{L^2H^3}\\
				&+c_1\bar{\nu}^2\|\partial_Zp^{\frac{1}{2}}{\de}_0\|_{L^2H^3}\|p^{\frac{1}{2}}\partial_ZD_0\|_{L^2H^3}\\
				\lesssim&\nu^{-1/6}\varepsilon\nu^{-2/3}\varepsilon+\nu(\nu^{-2/3}\varepsilon)^2+\nu^{-1/3}\varepsilon^2+\bar{\nu}\nu^{-2/3}C_0\varepsilon\nu^{-1}C_0\varepsilon+c_1(\nu^{-1/2}\varepsilon)^2+c_1\varepsilon^2+\bar{\nu}^2(\nu^{-1}C_0\varepsilon)^2\\
				\lesssim&\left(\bar{\nu}^{\frac{1}{3}}+\frac{1}{C_0}\right)(C_0\nu^{-1/2}\varepsilon)^2.
			\end{align*}
			We divide $I^{\nyn}_1$ into 
			\begin{align*}
				I^{\nyn}_1=&\int_{0}^{T}\left\langle \partial_Zp^{\frac{1}{2}}(V\cdot \tilde{\nabla} {\de}+{\de}D),\partial_Zp^{\frac{1}{2}}{\de}\right\rangle_{H^3} dt\\
				=&\int_{0}^{T}\left\langle [p^{\frac{1}{2}},\partial_ZV]\cdot \tilde{\nabla} {\de},\partial_Zp^{\frac{1}{2}}{\de}\right\rangle_{H^3} dt+\int_{0}^{T}\left\langle \partial_ZV\cdot \tilde{\nabla} p^{\frac{1}{2}}{\de},\partial_Zp^{\frac{1}{2}}{\de}\right\rangle_{H^3} dt\\
				&+\int_{0}^{T}\left\langle \langle\nabla\rangle^3[p^{\frac{1}{2}},V]\cdot \tilde{\nabla} \partial_Z{\de},\langle\nabla\rangle^3\partial_Zp^{\frac{1}{2}}{\de}\right\rangle_{L^2} dt
				+\int_{0}^{T}\left\langle [\langle\nabla\rangle^3,V]\cdot \tilde{\nabla} \partial_Zp^{\frac{1}{2}}{\de},\langle\nabla\rangle^3\partial_Zp^{\frac{1}{2}}{\de}\right\rangle_{L^2} dt\\
				&+\int_{0}^{T}\left\langle D\langle\nabla\rangle^3 \partial_Zp^{\frac{1}{2}}{\de},\langle\nabla\rangle^3\partial_Zp^{\frac{1}{2}}{\de}\right\rangle_{L^2} dt+\int_{0}^{T}\left\langle \partial_Zp^{\frac{1}{2}}({\de}D),\partial_Zp^{\frac{1}{2}}{\de}\right\rangle dt=\sum_{i=1}^{6}I^{\nyn}_{1i}.
			\end{align*}
			For $I^{\nyn}_{11}$, using Proposition \ref{est:com-p}, we deduce that
			\begin{align*}
				I^{\nyn}_{11}=&\int_{0}^{T}\left\langle [p^{\frac{1}{2}},\partial_ZV^1_0] \partial_X{\de}_{\neq},\partial_Zp^{\frac{1}{2}}{\de}_{\neq}\right\rangle_{H^3} dt+\int_{0}^{T}\left\langle [p^{\frac{1}{2}},\partial_ZV^1_{\neq}]\partial_X{\de}_{\neq},\partial_Zp^{\frac{1}{2}}{\de}\right\rangle_{H^3} dt\\
				&+\int_{0}^{T}\left\langle [p^{\frac{1}{2}},\partial_ZV^2] \tilde{\partial}_Y{\de},\partial_Zp^{\frac{1}{2}}{\de}\right\rangle_{H^3} dt+\int_{0}^{T}\left\langle [p^{\frac{1}{2}},\partial_ZV^3]\partial_Z{\de},\partial_Zp^{\frac{1}{2}}{\de}\right\rangle_{H^3} dt\\
				\lesssim&\|\partial_Z\nabla V^1_0\|_{L^\infty H^3}\|\partial_X{\de}_{\neq}\|_{L^2H^3}\|\partial_Zp^{\frac{1}{2}}{\de}_{\neq}\|_{L^2H^3}+\|\partial_Zp^{\frac{1}{2}}V^1_{\neq}\|_{L^2H^3}\|\partial_X{\de}_{\neq}\|_{L^2H^3}\|\partial_Zp^{\frac{1}{2}}{\de}\|_{L^\infty H^3}\\
				&+\|\partial_Zp^{\frac{1}{2}}(V^2,V^3)\|_{L^2H^3}\|p^{\frac{1}{2}}{\de}\|_{L^2H^3}\|\partial_Zp^{\frac{1}{2}}{\de}\|_{L^\infty H^3}\\
				\lesssim&\nu^{-1}\varepsilon\nu^{-1/3}\varepsilon\nu^{-2/3}\varepsilon+\nu^{-1/2}\varepsilon\nu^{-1/3}\varepsilon\nu^{-1/2}\varepsilon+\nu^{-2/3}\varepsilon\nu^{-1}\varepsilon\nu^{-1/2}\varepsilon\lesssim\nu^{-7/6}\varepsilon(\nu^{-1/2}\varepsilon)^2.
			\end{align*}
			Turning to $I^{\nyn}_{12}$, the bootstrap assumption implies that
			\begin{align*}
				I^{\nyn}_{12}\lesssim&(\|\partial_Z V^1_0\|_{L^\infty H^3}+\|\partial_ZV^1_{\neq}\|_{L^\infty H^3})\|\nabla_{X,Z}p^{\frac{1}{2}}{\de}_{\neq}\|^2_{L^2H^3}\\
				&+\|\partial_Z(V^2,V^3)\|_{L^2H^3}\|p{\de}\|_{L^2H^3}\|\partial_Zp^{\frac{1}{2}}{\de}\|_{L^\infty H^3}\\
				\lesssim&\nu^{-1}\varepsilon(\nu^{-2/3}\varepsilon)^2+\nu^{-1/2}\varepsilon\nu^{-1/2}\nu^{-5/6}\varepsilon\nu^{-1/2}\varepsilon\lesssim\nu^{-4/3}\varepsilon(\nu^{-1/2}\varepsilon)^2.
			\end{align*}
			$I^{\nyn}_{13}$ and $I^{\nyn}_{14}$ are treated similarly to $I^{\nxn}_{11}$, $I^{\nxn}_{12}$, respectively. Then we have
			\begin{align*}
				I^{\nyn}_{13}
				\lesssim&\|\nabla V^1_0\|_{L^\infty H^3}\|\partial_X\partial_Z{\de}_{\neq}\|_{L^2H^3}\|\partial_Zp^{\frac{1}{2}}{\de}_{\neq}\|_{L^2H^3}+\|p^{\frac{1}{2}}V^1_{\neq}\|_{L^2H^3}\|\partial_X\partial_Z{\de}_{\neq}\|_{L^2H^3}\|\partial_Zp^{\frac{1}{2}}{\de}\|_{L^\infty H^3}\\
				&+\|p^{\frac{1}{2}}(V^2,V^3)\|_{L^2H^3}\|\partial_Zp^{\frac{1}{2}}{\de}\|_{L^2H^3}\|\partial_Zp^{\frac{1}{2}}{\de}\|_{L^\infty H^3}\\
				\lesssim&\nu^{-1}\varepsilon\nu^{-2/3}\varepsilon\nu^{-2/3}\varepsilon+\nu^{-1/2}\varepsilon\nu^{-2/3}\varepsilon\nu^{-1/2}\varepsilon+\nu^{-2/3}\varepsilon\nu^{-1}\varepsilon\nu^{-1/2}\varepsilon\lesssim\nu^{-4/3}\varepsilon(\nu^{-1/2}\varepsilon)^2,
			\end{align*}
			and 
			\begin{align*}
				I^{\nyn}_{14}\lesssim&(\|\nabla V^1_0\|_{L^\infty H^2}+\|V^1_{\neq}\|_{L^\infty H^3})\|\nabla_{X,Z}p^{\frac{1}{2}}{\de}_{\neq}\|^2_{L^2H^3}\\
				&+\|\nabla(V^2,V^3)\|_{L^2H^2}\|p{\de}\|_{L^2H^3}\|\partial_Zp^{\frac{1}{2}}{\de}\|_{L^\infty H^3}\\
				\lesssim&\nu^{-1}\varepsilon(\nu^{-2/3}\varepsilon)^2+\nu^{-1/2}\varepsilon\nu^{-1/2}\nu^{-5/6}\varepsilon\nu^{-1/2}\varepsilon\lesssim\nu^{-4/3}\varepsilon(\nu^{-1/2}\varepsilon)^2.
			\end{align*}
			For $I^{\nyn}_{15}$ and $I^{\nyn}_{16}$, it immediately follows that
			\begin{align*}
				I^{\nyn}_{15}+I^{\nyn}_{16}\lesssim&\|D\|_{L^2H^3}\|\partial_Zp^{\frac{1}{2}}{\de}\|_{L^2H^3}\|\partial_Zp^{\frac{1}{2}}{\de}\|_{L^\infty H^3}+\|(1,\partial_Z)p^{\frac{1}{2}}{\de}\|_{L^2H^3}\|(1,\partial_Z)D\|_{L^2H^3}\|\partial_Zp^{\frac{1}{2}}{\de}\|_{L^\infty H^3}\\
				&+\|(1,\partial_Z){\de}\|_{L^\infty H^3}\|(1,\partial_Z)p^{\frac{1}{2}}D\|_{L^2H^3}\|\partial_Zp^{\frac{1}{2}}{\de}\|_{L^2 H^3}\\
				\lesssim&\nu^{-2/3}\varepsilon\nu^{-1}\varepsilon\nu^{-1/2}\varepsilon+\nu^{-1}\varepsilon\nu^{-2/3}\varepsilon\nu^{-1/2}\varepsilon+\nu^{-1/6}\varepsilon\nu^{-1}\varepsilon\nu^{-1}\varepsilon\lesssim\nu^{-7/6}\varepsilon(\nu^{-1/2}\varepsilon)^2.
			\end{align*}
			Next, we decompose $I^{\nyn}_2$ into three parts:
			\begin{align*}
				I^{\nyn}_2=&\int_{0}^{T}\left\langle  \partial_Z(V\cdot\tilde{\nabla} D),\partial_ZD\right\rangle_{H^3} dt+\int_{0}^{T}\left\langle  \partial_Z(\tilde{\partial}_iV^j\tilde{\partial}_jV^i),\partial_ZD\right\rangle_{H^3} dt\\
				&+\int_{0}^{T}\left\langle  \partial_Z\dl(F({\de})\tilde{\nabla}{\de}+G({\de})(\nu\tilde{\Delta} V+(\nu+\nu')\tilde{\nabla}D)),\partial_ZD\right\rangle_{H^3} dt	=\sum_{i=1}^{3}I^{\nyn}_{2i}.
			\end{align*}
			For $I^{\nyn}_{21}$, the by now standard treatment implies that
			\begin{align*}
				I^{\nyn}_{21}=&\int_{0}^{T}\left\langle (\partial_ZV)\cdot\tilde{\nabla} D,\partial_ZD\right\rangle_{H^3} dt+\int_{0}^{T}\left\langle [\langle \nabla\rangle^3, V]\cdot\tilde{\nabla} \partial_ZD,\langle \nabla\rangle^3\partial_ZD\right\rangle_{L^2} dt\\
				&-\frac{1}{2}\int_{0}^{T}\left\langle D \langle \nabla\rangle^3\partial_ZD,\langle \nabla\rangle^3\partial_ZD\right\rangle_{L^2} dt\\
				\lesssim&\|\partial_ZV^1_0\|_{L^\infty H^3}\|\nabla_{X,Z}D_{\neq}\|^2_{L^2H^3}+\|\partial_ZV^1_{\neq}\|_{L^2H^3}\|\partial_XD_{\neq}\|_{L^2H^3}\|\partial_ZD\|_{L^\infty H^3}\\
				&+\|\partial_Z(V^2,V^3)\|_{L^\infty H^3}\|\tilde{\nabla} D\|_{L^2H^3}\|\partial_Z D\|_{L^2H^3}+\|\partial_ZV^1_0\|_{L^\infty H^2}\|\nabla_{X,Z}D_{\neq}\|^2_{L^2H^3}\\
				&+\|\partial_ZV^1_{\neq}\|_{L^2H^2}\|\partial_XD_{\neq}\|_{L^2H^3}\|\partial_ZD\|_{L^\infty H^3}+\|\partial_Z(V^2,V^3)\|_{L^\infty H^2}\|\tilde{\nabla} D\|_{L^2H^3}\|\partial_Z D\|_{L^2H^3}\\
				&+\|D\|_{L^\infty H^2}\|\partial_ZD\|^2_{L^2H^3}\\
				\lesssim&\nu^{-1}\varepsilon(\nu^{-2/3}\varepsilon)^2+\nu^{-1/6}\varepsilon\nu^{-2/3}\varepsilon\nu^{-1/2}\varepsilon+\nu^{-1/6}\varepsilon\nu^{-1}\varepsilon\nu^{-2/3}\varepsilon+\nu^{-1/2}\varepsilon(\nu^{-2/3}\varepsilon)^2\lesssim\nu^{-4/3}\varepsilon(\nu^{-1/2}\varepsilon)^2.
			\end{align*}
			Turning to $I^{\nyn}_{22}$, following the same manner as in $I^{(3)}_{22}$, we deduce that 
			\begin{align*}
				I^{\nyn}_{22}=	&\int_{0}^{T}\left\langle  \partial_Z(\tilde{\partial}_iV^j\tilde{\partial}_jV^i),\partial_ZD\right\rangle_{H^3} dt\\
				\lesssim&(\|\partial_XV^1_{\neq}\|^2_{L^\infty H^3}\|\partial_XV^1_{\neq}\|^2_{L^2 H^3}+\|\tilde{\partial}_YV^1\|_{L^\infty H^3}\|\partial_XV^2_{\neq}\|_{L^2H^3}+\|\partial_ZV^1\|_{L^\infty H^3}\|\partial_XV^3_{\neq}\|_{L^2H^3}\\
				&+\|\tilde{\partial}_YV^2\|_{L^\infty H^3}\|\tilde{\partial}_YV^2_{\neq}\|_{L^2H^3}+\|\tilde{\partial}_YV^3\|_{L^\infty H^3}\|\partial_ZV^2_{\neq}\|_{L^2H^3}+\|\partial_YV^3_{\neq}\|_{L^2 H^3}\|\partial_ZV^2_0\|_{L^\infty H^3}\\
				&+\|\partial_ZV^3\|_{L^\infty H^3}\|\partial_ZV^3_{\neq}\|_{L^2H^3})\|\partial_Z\partial_ZD\|_{L^2 H^3}\\
				&+(\|\partial_Y\partial_ZV^2_0\|_{L^2H^3}\|\partial_YV^2_0\|_{L^2H^3}+\|\partial_Y(1,\partial_Z)V^2_0\|_{L^2H^3}\|\partial_{ZZ}V^3_0\|_{L^2H^3}+\|\partial_{ZZ}V^3_0\|^2_{L^2H^3})\|\partial_ZD_0\|_{L^\infty H^3}\\
				\lesssim&(\varepsilon\nu^{-1/6}\varepsilon+\nu^{-1}\varepsilon\nu^{-1/3}\varepsilon+\nu^{-1}\varepsilon\nu^{-1/6}\varepsilon+\nu^{-1/2}\varepsilon\nu^{-2/3}\varepsilon+(\nu^{-1/3}\varepsilon)^2+\nu^{-1/2}\varepsilon^2+\varepsilon\nu^{-1/6}\varepsilon)\nu^{-1}\varepsilon+(\nu^{-1/2}\varepsilon)^3\\
				\lesssim&\nu^{-4/3}\varepsilon(\nu^{-1/2}\varepsilon)^2.
			\end{align*}
			For $I^{\nyn}_{23}$, it follows from \cref{ineq:div} that
			\begin{align*}
				I^{\nyn}_{23}\lesssim&(\|(1,\partial_Z){\de}\|_{L^\infty H^3}\|\nabla_{X,Z}\tilde{\nabla} {\de}_{\neq}\|_{L^2H^3}+\|\nabla_{X,Z}{\de}_{\neq}\|_{L^2 H^3}\|(1,\partial_Z)\nabla {\de}_0\|_{L^\infty H^3}\\
				&+\|\partial_X{\de}_{\neq}\|_{L^2 H^3}(\nu\|\tilde{\Delta} V^1\|_{L^\infty H^3},(\nu+\nu')\|\partial_XD_{\neq}\|_{L^\infty H^3})\\
				&+\|{\de}_0\|_{L^\infty H^3}(\nu\|\partial_X\tilde{\Delta} V^1_{\neq}\|_{L^2 H^3},(\nu+\nu')\|\partial_{XX}D_{\neq}\|_{L^2 H^3})\\
				&+\|\partial_Z{\de}_{\neq}\|_{L^2 H^3}(\nu\|\tilde{\Delta} V^2\|_{L^\infty H^3},(\nu+\nu')\|\tilde{\partial}_YD\|_{L^\infty H^3})\\
				&+\|{\de}_0\|_{L^\infty H^3}(\nu\|\partial_Z\tilde{\Delta} V^2_{\neq}\|_{L^2 H^3},(\nu+\nu')\|\tilde{\partial}_Y\partial_ZD_{\neq}\|_{L^2 H^3})\\
				&+\|\partial_Z{\de}_{\neq}\|_{L^2 H^3}(\nu\|\tilde{\Delta} V^3\|_{L^\infty H^3},(\nu+\nu')\|\partial_ZD\|_{L^\infty H^3})\\
				&+\|{\de}_0\|_{L^\infty H^3}(\nu\|\partial_Z\tilde{\Delta} V^3_{\neq}\|_{L^2 H^3},(\nu+\nu')\|\partial_{ZZ}D_{\neq}\|_{L^2 H^3}))\|\tilde{\nabla}\partial_ZD\|_{L^2 H^3}\\
				&+(\|(1,\partial_Z)p^{\frac{1}{2}}{\de}_0\|_{L^2H^3}+\nu\|(1,\partial_Z)\Delta(V^2_0,V^3_0)\|_{L^2H^3}+(\nu+\nu')\|(1,\partial_Z)\nabla D_0\|_{L^2H^3})\\
				&\times\|(1,\partial_Z){\de}_0\|_{L^\infty H^3}\|\nabla\partial_ZD_0\|_{L^2H^3}\\
				\lesssim&(\nu^{-1/6}\varepsilon\nu^{-2/3}\varepsilon+\nu^{-1/3}\varepsilon\nu^{-1/2}\varepsilon+\nu^{-1/3}\varepsilon(\nu\nu^{-1}\varepsilon+(\nu+\nu')\nu^{-1/2}\varepsilon)+\varepsilon\bar{\nu}\nu^{-1}\varepsilon+\nu^{-1/3}\varepsilon\bar{\nu}\nu^{-5/6}\varepsilon\\
				&+\varepsilon\bar{\nu}\nu^{-1}\varepsilon+\nu^{-1/3}\varepsilon(\nu\nu^{-2/3}\varepsilon+(\nu+\nu')\nu^{-1/2}\varepsilon)+\varepsilon\bar{\nu}\nu^{-1}\varepsilon)\nu^{-1}\varepsilon+(\nu^{-1}\varepsilon+\bar{\nu}\nu^{-1}\varepsilon)\nu^{-1}\varepsilon^2\\
				\lesssim&\nu^{-1}\varepsilon(\nu^{-1/2}\varepsilon)^2=\nu^{-3/2}\varepsilon^3.
			\end{align*}
			The treatment of $I^{\nyn}_3$, $I^{\nyn}_4$, and $I^{\nyn}_5$ is analogous to the previous cases and is therefore omitted. Combining all estimates for $I_i^{(4)}$ ($i=1,2,\ldots,5$), $LS^{(4)}$, and $LE^{(4)}$, we obtain Lemma \ref{Lem-E4}.
		\end{proof}
		
		\subsection{$H^3$ estimate on $(p{\de},p^{\frac{1}{2}}D)$}\label{sec:48}
		In this subsection, we improve the estimate \cref{bs-pn} for $j=0$. Recall the energy
		\begin{align*}
			E^{\npn}(t)=\frac{1}{2}\|(p{\de},p^{\frac{1}{2}}D)\|^2_{H^3}-c_1\bar{\nu}^{\frac{1}{3}} Re\langle p^{\frac{1}{2}}{\de}_{\neq},p^{\frac{1}{2}}D_{\neq}\rangle_{H^3}-c_1\bar{\nu}  Re\langle p^{\frac{1}{2}}{\de}_0,p^{\frac{1}{2}}D_0\rangle_{H^3}.
		\end{align*}
		Then we have the following estimate.
		\begin{Lem}\label{Lem-E5}
			Under the same assumptions as \cref{MT} and  Proposition \ref{prop-bs}, it holds that,
			\begin{align}
				\sup_{t\in[0,T]}E^{(5)}(t)&
				+c_1\bar{\nu}^{\frac{1}{3}}\|p{\de}_{\neq}\|^2_{L^2H^3}+c_1\bar{\nu}^{\frac{1}{3}}\|p{\de}_{\neq}\|^2_{L^2H^3}+c_1\bar{\nu}\|p{\de}_0\|^2_{L^2H^3}\lesssim E^{\npn}(0)+\nu^{-3/2}\eps^3 .
			\end{align}
		\end{Lem}
		\begin{proof}

			The system of $(p{\de},p^{\frac{1}{2}}D)$ reads
			\begin{align*}
				\left\{
				\begin{aligned}
					&\partial_t (p{\de})-\frac{\partial_t p}{p}(p{\de})+pD=p\nlt_{{\de}},\\
					&\partial_t (p^{\frac{1}{2}}D)-\frac{3}{2}\frac{\partial_tp}{p}(p^{\frac{1}{2}}D)-p^{\frac{1}{2}}(p{\de})-2\partial_{X}p^{-\frac{1}{2}}(W^2-\partial_X{\de}+\nu\partial_XD)+\bar{\nu} p(p^{\frac{1}{2}}D)=p^{\frac{1}{2}}\nlt_{D},
				\end{aligned} \right.
			\end{align*}
			where the nonlinear terms $\nlt_{\de},\nlt_D$ are given in \cref{t-ND}.
			Then the energy estimate gives
			\begin{align*}
				&E^{\npn}(T)+\bar{\nu}\|pD\|^2_{L^2H^3}+c_1\bar{\nu}^{\frac{1}{3}}\|p{\de}_{\neq}\|^2_{L^2H^3}+c_1\bar{\nu}\|p{\de}_0\|^2_{L^2H^3}\\
				=&E^{\npn}(0)+LS^{\npn}+LE^{\npn}+I^{\npn}_1+I^{\npn}_2+I^{\npn}_3+I^{\npn}_4,
			\end{align*}
			where we denote
			\begin{align*}
				LS^{\npn}=&\int_{0}^{T}\left\langle\frac{3}{2}\partial_t p D_{\neq}, D_{\neq}\right\rangle_{H^3}+\left\langle\partial_t p {\de}_{\neq}, p{\de}_{\neq}\right\rangle_{H^3} dt,\\
				LE^{\npn}=&2\int_{0}^{T}\left\langle  \partial_Xp^{-\frac{1}{2}}(W^2_{\neq}-\partial_X{\de}+\nu\partial_XD)_{\neq}, p^{\frac{1}{2}}D_{\neq}\right\rangle_{H^3}\\
				&+c_1\bar{\nu}^{\frac{1}{3}}\| p^{\frac{1}{2}} D_{\neq}\|^2_{H^3}-c_1\bar{\nu}^{\frac{1}{3}}\left\langle p{\de}_{\neq},2\frac{\partial_tp}{p}D_{\neq}+ 2 \partial_Xp^{-1}(W^2_{\neq}-\partial_X{\de}+\nu\partial_XD)_{\neq}-\bar{\nu} pD_{\neq}\right\rangle_{H^3}\\
				&+c_1\bar{\nu}\|p^{\frac{1}{2}}D_0\|^2_{H^3}+c_1\bar{\nu}\left\langle p{\de}_0,\bar{\nu} pD_0\right\rangle_{H^3} dt,\\
				I^{\npn}_1=&\int_{0}^{T}\left\langle p\nlt_{{\de}},p{\de}\right\rangle_{H^3} dt,\quad
				I^{\npn}_2=\int_{0}^{T}\left\langle p^{\frac{1}{2}}\nlt_{D},p^{\frac{1}{2}}D\right\rangle_{H^3} dt,\\
				I^{\npn}_3=&-c_1\bar{\nu}^{\frac{1}{3}}\int_{0}^{T}\left\langle p^{\frac{1}{2}} {\de}_{\neq}, p^{\frac{1}{2}}(\nlt_{D})_{\neq} \right\rangle_{H^3}+\left\langle p^{\frac{1}{2}}(\nlt_{{\de}})_{\neq},p^{\frac{1}{2}}D_{\neq}\right\rangle_{H^3} dt,\\
				I^{\npn}_4=&-c_1\bar{\nu}\int_{0}^{T}\left\langle p^{\frac{1}{2}} {\de}_0, p^{\frac{1}{2}}(\nlt_{D})_0\right\rangle_{H^3}+\left\langle p^{\frac{1}{2}}(\nlt_{{\de}})_0,p^{\frac{1}{2}}D_0\right\rangle_{H^3} dt,
			\end{align*}
			with $\nlt_{\de},\nlt_D$ given in \cref{t-ND}.
			For $LS^{\npn}$ and $LE^{\npn}$, the bootstrap argument yields
			\begin{align*}
				LS^{\npn}\leq&\frac{3}{2}\|\partial_XD_{\neq}\|_{L^2H^3}\|p^{\frac{1}{2}}D_{\neq}\|_{L^2H^3}+\|\partial_Xp^{\frac{1}{2}}{\de}_{\neq}\|_{L^2H^3}\|p{\de}_{\neq}\|_{L^2H^3}\lesssim\frac{1}{C_0}(C_0^2\nu^{-5/6}\varepsilon)^2,\\
				LE^{\npn}\lesssim&\|m^{\frac{1}{4}}(W^2_{\neq},\nabla_{X,Z}{\de}_{\neq},\partial_Xp^{-\frac{1}{2}}D_{\neq})\|_{L^2H^3}\|(p{\de}_{\neq},p^{\frac{1}{2}}D_{\neq})\|_{L^2H^3}+\nu\|\partial_XD_{\neq}\|_{L^2H^3}\|p^{\frac{1}{2}}D_{\neq}\|_{L^2H^3}\\
				&+c_1\bar{\nu}^{\frac{4}{3}}\|p{\de}_{\neq}\|_{L^2H^3}\|pD_{\neq}\|_{L^2H^3}+c_1\bar{\nu}^{\frac{1}{3}}\|p^{\frac{1}{2}}D_{\neq}\|^2_{L^2H^3}+c_1\bar{\nu}\|p^{\frac{1}{2}}D_0\|^2_{L^2H^3}\\
				&+c_1\bar{\nu}^2\|p{\de}_0\|_{L^2H^3}\|pD_0\|_{L^2H^3}\\
				\lesssim&\left(\bar{\nu}^{\frac{2}{3}}+\frac{1}{C_0}\right)(C_0^2\nu^{-5/6}\varepsilon)^2.
			\end{align*}
			We divide $I^{\npn}_1$ into
			\begin{align*}
				I^{\npn}_1=&\int_{0}^{T}\left\langle p(V\cdot \tilde{\nabla} {\de}+{\de}D),p{\de}\right\rangle_{H^3} dt\\
				=&\int_{0}^{T}\left\langle [\langle\nabla\rangle^3,V]\cdot \tilde{\nabla} p{\de},\langle\nabla\rangle^3p{\de}\right\rangle_{L^2} dt-\frac{1}{2}\int_{0}^{T}\left\langle D\langle\nabla\rangle^3 p{\de},\langle\nabla\rangle^3p{\de}\right\rangle_{L^2} dt\\
				&+\int_{0}^{T}\left\langle (pV)\cdot \tilde{\nabla} {\de},p{\de}\right\rangle_{H^3} dt+\sum_{i=1}^{3}\int_{0}^{T}\left\langle \tilde{\partial}_iV\cdot \tilde{\nabla} \tilde{\partial}_i{\de},\partial_Zp^{\frac{1}{2}}{\de}\right\rangle_{H^3} dt+\int_{0}^{T}\left\langle p({\de}D),p{\de}\right\rangle_{H^3} dt	=\sum_{i=1}^{5}I^{\npn}_{1i}.
			\end{align*}
			For $I^{\npn}_{11}$, Lemma \ref{lemma:commutator} implies that
			\begin{align*}
				I^{\npn}_{11}=&\int_{0}^{T}\left\langle [\langle\nabla\rangle^3,V^i] \tilde{\partial}_i p{\de},\langle\nabla\rangle^3p{\de}\right\rangle_{L^2} dt\\
				\lesssim&\|\nabla V^1_0\|_{L^\infty H^2}\|\partial_Xp{\de}\|_{L^2H^2}\|p{\de}_{\neq}\|_{L^2H^3}+\|\nabla V^1_{\neq}\|_{L^2 H^2}\|\partial_Xp{\de}\|_{L^2H^2}\|p{\de}_{\neq}\|_{L^\infty H^3}\\
				&+(\|V^2\|_{L^\infty H^3}{\|p^{\frac{3}{2}}{\de}\|_{L^2H^2}}+\|V^3\|_{L^\infty H^3}\|\partial_Zp{\de}\|_{L^2H^2})\|p{\de}\|_{L^2H^3}\\
				\lesssim&(\nu^{-1}\varepsilon)^3+\nu^{-1/6}\varepsilon\nu^{-1}\varepsilon\nu^{-5/6}\varepsilon+\nu^{-1/6}\varepsilon\nu^{-5/3}\varepsilon\nu^{-4/3}\varepsilon+\varepsilon(\nu^{-4/3}\varepsilon)^2\lesssim\nu^{-3/2}\varepsilon(\nu^{-5/6}\varepsilon)^2.
			\end{align*}
			where we used the bootstrap assumption $\|p^{\frac{3}{2}}{\de}\|_{L^2H^2}\lesssim\nu^{-5/3}\eps.$ It also follows from the bootstrap assumption that
			\begin{align*}
				\sum_{i=2}^{5}I^{\npn}_{1i}\lesssim&\|D\|_{L^2H^3}\|p{\de}\|_{L^2H^3}\|p{\de}\|_{L^\infty H^3}\\
				&+\|pV^1_0\|_{L^\infty H^3}\|\partial_X{\de}_{\neq}\|_{L^2H^3}\|p{\de}_{\neq}\|_{L^2 H^3}+\|pV^1_{\neq}\|_{L^2H^3}\|\partial_X{\de}_{\neq}\|_{L^2H^3}\|p{\de}\|_{L^\infty H^3}\\
				&+(\|pV^2\|_{L^2H^3}\|p^{\frac{1}{2}}{\de}\|_{L^\infty H^3}+\|pV^3\|_{L^2H^3}\|\partial_Z{\de}\|_{L^\infty H^3})\|p{\de}\|_{L^2H^3}\\
				&+\|p^{\frac{1}{2}}V^1_0\|_{L^\infty H^3}\|\partial_Xp^{\frac{1}{2}}{\de}_{\neq}\|_{L^2H^3}\|p{\de}_{\neq}\|_{L^2 H^3}+\|p^{\frac{1}{2}}V^1_{\neq}\|_{L^2H^3}\|\partial_Xp^{\frac{1}{2}}{\de}_{\neq}\|_{L^2H^3}\|p{\de}\|_{L^\infty H^3}\\
				&+(\|p^{\frac{1}{2}}V^2\|_{L^2H^3}\|p{\de}\|_{L^\infty H^3}+\|p^{\frac{1}{2}}V^3\|_{L^2H^3}\|\partial_Zp^{\frac{1}{2}}{\de}\|_{L^\infty H^3})\|p{\de}\|_{L^2H^3}\\
				&+(\|D\|_{L^2H^3}\|p{\de}\|_{L^\infty H^3}+\|pD\|_{L^2H^3}\|{\de}\|_{L^\infty H^3})\|p{\de}\|_{L^2 H^3}\\
				\lesssim&\nu^{-2/3}\varepsilon\nu^{-1/2}(\nu^{-5/6}\varepsilon)^2+\nu^{-1}\nu^{-1/3}\varepsilon\nu^{-1}\varepsilon+\nu^{-5/6}\varepsilon\nu^{-1/3}\varepsilon\nu^{-5/6}\varepsilon\\
				&+(\nu^{-1}\varepsilon\nu^{-1/2}\varepsilon+\nu^{-5/6}\varepsilon\nu^{-1/6}\varepsilon)\nu^{-4/3}\varepsilon+\nu^{-1}\nu^{-2/3}\varepsilon\nu^{-1}\varepsilon+\nu^{-1/2}\varepsilon\nu^{-2/3}\varepsilon\nu^{-5/6}\varepsilon\\
				&+(\nu^{-2/3}\varepsilon\nu^{-5/6}\varepsilon+\nu^{-1/2}\varepsilon\nu^{-1/2}\varepsilon)\nu^{-4/3}\varepsilon+(\nu^{-2/3}\varepsilon\nu^{-5/6}\varepsilon+\nu^{-1/2}\nu^{-5/6}\varepsilon\nu^{-1/6}\varepsilon)\nu^{-4/3}\varepsilon\\
				\lesssim&\nu^{-7/6}\varepsilon(\nu^{-5/6}\varepsilon)^2.
			\end{align*}
			The estimates of $I^{\npn}_2$ are similar to the above. In fact, we divide $I^{\npn}_2$ into
			\begin{align*}
				I^{\npn}_2=\int_{0}^{T}\left \langle V\cdot\tilde{\nabla} D+\tilde{\partial}_iV^j\tilde{\partial}_jV^i+\dl(F({\de})\tilde{\nabla}{\de}+G({\de})(\nu\tilde{\Delta} 	V+(\nu+\nu')\tilde{\nabla}D)),pD\right\rangle_{H^3} dt=\sum_{i=1}^{3}I^{\npn}_{2i},
			\end{align*}
			\cref{2026-3-21-2} and  Proposition \ref{est:main} imply 
			\begin{align*}
				I^{\npn}_{21}=&\int_{0}^{T}\left\langle  V\cdot\tilde{\nabla} D,pD\right\rangle_{H^3} dt\\
				\lesssim&{(\|V^1\|_{L^\infty L^\infty}+\|\tilde{\nabla}V^1\|_{L^\infty H^2})}\|\partial_XD_{\neq}\|_{L^2H^3}\|pD\|_{L^2H^3}
				+\|(V^2,V^3)\|_{L^\infty H^3}\|\tilde{\nabla} D\|_{L^2H^3}\|p D\|_{L^2H^3}\\
				\lesssim&\nu^{-1}\varepsilon\nu^{-2/3}\varepsilon\nu^{-4/3}\varepsilon+\nu^{-1/6}\varepsilon\nu^{-1}\varepsilon\nu^{-4/3}\varepsilon\lesssim\nu^{-4/3}\varepsilon(\nu^{-5/6}\varepsilon)^2,
			\end{align*}
			and 
			\begin{align*}
				I^{\npn}_{22}=	&\int_{0}^{T}\left\langle  \tilde{\partial}_iV^j\tilde{\partial}_jV^i,pD\right\rangle_{H^3} dt\\
				\lesssim&(\|\partial_XV^1_{\neq}\|^2_{L^\infty H^3}\|\partial_XV^1_{\neq}\|^2_{L^2 H^3}+\|\tilde{\partial}_YV^1\|_{L^\infty H^3}\|\partial_XV^2_{\neq}\|_{L^2H^3}+\|\partial_ZV^1\|_{L^\infty H^3}\|\partial_XV^3_{\neq}\|_{L^2H^3}\\
				&+\|\tilde{\partial}_YV^2\|_{L^\infty H^3}\|\tilde{\partial}_YV^2_{\neq}\|_{L^2H^3}+\|\tilde{\partial}_YV^3\|_{L^\infty H^3}\|\partial_ZV^2_{\neq}\|_{L^2H^3}+\|\tilde{\partial}_YV^3_{\neq}\|_{L^2 H^3}\|\partial_ZV^2_0\|_{L^\infty H^3}\\
				&+\|\partial_ZV^3\|_{L^\infty H^3}\|\partial_ZV^3_{\neq}\|_{L^2H^3})\|pD\|_{L^2 H^3}\\
				&+(\|\partial_YV^2_0\|_{L^2H^3}\|\partial_YV^2_0\|_{L^\infty H^3}+\|\partial_YV^2_0\|_{L^2H^3}\|\partial_{Z}V^3_0\|_{L^\infty H^3}+\|\partial_{Z}V^3_0\|_{L^2H^3}\|\partial_{Z}V^3_0\|_{L^\infty H^3})\|pD_0\|_{L^2 H^3}\\
				\lesssim&(\varepsilon\nu^{-1/6}\varepsilon+\nu^{-1}\varepsilon\nu^{-1/3}\varepsilon+\nu^{-1}\varepsilon\nu^{-1/6}\varepsilon+\nu^{-1/2}\varepsilon\nu^{-2/3}\varepsilon+(\nu^{-1/3}\varepsilon)^2+\nu^{-1/2}\varepsilon^2+\varepsilon\nu^{-1/6}\varepsilon+(\nu^{-1/2}\varepsilon)^2)\nu^{-4/3}\varepsilon\\
				\lesssim&\nu^{-2/3}\varepsilon(\nu^{-1/2}\varepsilon)^2.
			\end{align*}
			For $I^{\npn}_{23}$, using \cref{ineq:div}, we have
			\begin{align*}
				I^{\npn}_{23}\lesssim&(\|p^{\frac{1}{2}}{\de}\|_{L^\infty H^3}\|p^{\frac{1}{2}} {\de}\|_{L^2H^3}+\|{\de}\|_{L^\infty H^3}\|p {\de}\|_{L^2 H^3}+\|{\de}\|_{L^\infty H^3}\bar{\nu}\|pD\|_{L^2H^3}\\
				&+\|\partial_X{\de}_{\neq}\|_{L^2 H^3}(\nu\|\tilde{\Delta} V^1\|_{L^\infty H^3}+(\nu+\nu')\|\partial_XD_{\neq}\|_{L^\infty H^3})\\
				&+\|\partial_Y{\de}\|_{L^2 H^3}(\nu\|\tilde{\Delta} V^2\|_{L^\infty H^3}+(\nu+\nu')\|\tilde{\partial}_YD\|_{L^\infty H^3})\\
				&+\|\partial_Z{\de}\|_{L^2 H^3}(\nu\|\tilde{\Delta} V^3\|_{L^\infty H^3}+(\nu+\nu')\|\partial_ZD\|_{L^\infty H^3}))\|pD\|_{L^2H^3}\\
				\lesssim&(\nu^{-1/2}\varepsilon\nu^{-1}\varepsilon+\nu^{-1/6}\varepsilon\nu^{-4/3}\varepsilon+\nu^{-1/6}\varepsilon\bar{\nu}\nu^{-4/3}\varepsilon+\nu^{-1/3}\varepsilon(\nu\nu^{-1}\varepsilon+(\nu+\nu')\nu^{-1/2}\varepsilon)\\
				&+\nu^{-1/2}\varepsilon\bar{\nu}\nu^{-1}\varepsilon+\nu^{-1/2}\varepsilon\bar{\nu}\nu^{-2/3}\varepsilon)\nu^{-4/3}\varepsilon\\
				\lesssim&\nu^{-7/6}\varepsilon(\nu^{-5/6}\varepsilon)^2.
			\end{align*}
			The treatment of $I^{\npn}_3$ and $I^{\npn}_4$ is analogous to the above argument and is therefore omitted.
		\end{proof}
		
		\section{Energy estimates for the incompressible part at low regularity}
		In this section, we prove the incompressible part of  Proposition \ref{prop-bs} for $j=0$.
		\subsection{$H^3$ estimate on  $W^2$}
		In this subsection, we improve \cref{bs-w2} for $j=0$ by providing estimates for $\norm{W^2_{\neq}}_{H^3}$ and $\norm{W^2_{0\neq}}_{H^3}$ separately.
		\subsubsection{$H^3$ estimate on  $W^2_{\neq}$}\label{sec:511}
		\begin{Lem}\label{Lem-E6}
			Under the same assumptions as \cref{MT} and  Proposition \ref{prop-bs}, it holds that
			\begin{align}
				\sup_{0\leq t\leq T}&\|\ma W^2_{\neq}(t)\|^2_{H^3}+\nu\|\tilde{\nabla}\ma W^2_{\neq}\|^2_{L^2H^3}\nonumber\\
				&+\left\|\sqrt{-\frac{\partial_t M}{M}}\ma W^2_{\neq}\right\|^2_{L^2H^3}+\frac{1}{4}\left\|\sqrt{-\frac{\partial_t m}{m}}\ma W^2_{\neq}\right\|^2_{L^2H^3}\lesssim\|W^2_{\neq}(0)\|^2_{H^3}+\nu^{-\frac{3}{2}}\varepsilon^3
			\end{align}
		\end{Lem}
		\begin{proof}
			
			The energy estimate gives
			\begin{align*}
				&\|\ma W^2_{\neq}(T)\|^2_{H^3}+\nu\|\tilde{\nabla}\ma W^2_{\neq}\|^2_{L^2H^3}+\left\|\sqrt{-\frac{\partial_t M}{M}}\ma W^2_{\neq}\right\|^2_{L^2H^3}+\frac{1}{4}\left\|\sqrt{-\frac{\partial_t m}{m}}\ma W^2_{\neq}\right\|^2_{L^2H^3}\\
				=&\|W^2_{\neq}(0)\|^2_{H^3}+LE^{\nwii}+\sum_{i=1}^{6}I^{\nwii}_i,
			\end{align*}
			where the lower-order and error terms are given by
			\begin{align*}
				LE^{\nwii}=&-\int_{0}^{T}\left\langle \ma\left(2\nu\partial_X^2p^{-1}(W^2-\partial_X{\de}+\nu\partial_XD)+\nu\frac{\partial_tp}{p}\partial_XD-\nu(\nu+\nu')p\partial_XD\right)_{\neq},\ma W^2_{\neq}\right\rangle_{H^3} dt,\\
				I_1^{\nwii}=&-\int_{0}^{T}\left\langle \ma(V\cdot\tilde{\nabla}W^2)_{\neq},\ma W^2_{\neq}\right\rangle_{H^3} dt,\\
				I_2^{\nwii}=&\int_{0}^{T}\left\langle \ma\left(\sum_{i=1,3}(\tilde{\partial}_YV^i\partial_iD-Q^i\partial_iV^2)-\Omega^2\tilde{\partial}_YV^2\right)_{\neq},\ma W^2_{\neq}\right\rangle_{H^3} dt,\\
				I_3^{\nwii}=&\int_{0}^{T}\left\langle \ma\left(2\sum_{i=1,3}(\partial_iV^2\tilde{\p}_{Y}^{2} V^i-\tilde{\partial}_iV^j\tilde{\partial}_{ij}V^2)\right)_{\neq},\ma W^2_{\neq}\right\rangle_{H^3} dt,\\
				I_4^{\nwii}=&\int_{0}^{T}\left\langle \ma\left(\sum_{i,j=1,3}\tilde{\partial}_Y(\tilde{\partial}_iV^j\tilde{\partial}_jV^i)\right)_{\neq},\ma W^2_{\neq}\right\rangle_{H^3} dt,\\
				I_5^{\nwii}=&\int_{0}^{T}\left\langle \ma\left(-\partial_XV\cdot\tilde{\nabla}{\de}-\partial_X({\de}D)+\nu\partial_XV\cdot\tilde{\nabla}D+\nu\partial_X(\tilde{\partial}_iV^j\tilde{\partial}_jV^i)\right)_{\neq},\ma W^2_{\neq}\right\rangle_{H^3} dt,\\
				I_6^{\nwii}=&\int_{0}^{T}\big\langle \ma\big((\tilde{\partial}_Y\tilde{\textrm{div}})\left(G({\de})(\nu\tilde{\Delta} V+(\nu+\nu')\tilde{\nabla} D)\right)-\tilde{\Delta}\left(G({\de})(\nu\tilde{\Delta} V^2+(\nu+\nu')\tilde{\partial}_Y D)\right)\\
				&+\nu\partial_X\tilde{\textrm{div}} \left(F({\de})\tilde{\nabla} {\de}+G({\de})(\nu\tilde{\Delta} V+(\nu+\nu')\tilde{\nabla} D)\right)\big),\ma W^2_{\neq}\big\rangle_{H^3} dt.
			\end{align*}
			For $LE^{\nwii}$, the bootstrap argument implies that
			\begin{align*}
				LE\lesssim& \nu\|\partial_Xp^{-\frac{1}{2}}m^{\frac{1}{4}}(W^2_{\neq},\partial_X{\de}_{\neq})\|^2_{L^2H^3}+\nu\|\partial_Xp^{-\frac{1}{2}}m^{\frac{1}{4}}W^2_{\neq}\|_{L^2H^3}\|m^{\frac{1}{4}}\partial_XD_{\neq}\|_{L^2H^3}\\
				&+\nu(\nu+\nu')\|p^{\frac{1}{2}}m^{\frac{1}{4}}W^2_{\neq}\|_{L^2H^3}\|p^{\frac{1}{2}}\partial_XD_{\neq}\|_{L^2H^3}\\
				\lesssim&\nu\varepsilon^2+\nu\varepsilon\nu^{-1/2}\varepsilon+\nu^2\nu^{-1/2}\varepsilon\nu^{-1}C_0\varepsilon\lesssim\nu^{1/2}C_0\varepsilon^2.
			\end{align*}
			Using a commutator argument similar to ${\de}_{\neq}$,	we divide $I^{\nwii}_1$ into 
			\begin{align*}
				I^{\nwii}_1=&\int_{0}^{T}\left\langle \ma (V\cdot \tilde{\nabla} W^2)_{\neq},\ma W^2_{\neq}\right\rangle_{H^3} dt,\\
				=&\int_{0}^{T}\left\langle \ma (V^1_0\partial_X W^2_{\neq}),\ma W^2_{\neq}\right\rangle_{H^3} dt+\sum_{i=2,3}\int_{0}^{T}\left\langle \ma (V^i_0\tilde{\partial}_i W^2_{\neq}),\ma W^2_{\neq}\right\rangle_{H^3} dt\\
				&+\int_{0}^{T}\left\langle \ma (V_{\neq}\cdot\tilde{\nabla} W^2),\ma W^2_{\neq}\right\rangle_{H^3} dt=\sum_{i=1}^{3}I^{\nwii}_{1i}.
			\end{align*}
			For $I^{\nwii}_{11}$, applying the commutator estimates \cref{est:m}, $\cref{est:m4}$, and integration by parts, we obtain
			\begin{align*}
				I^{\nwii}_{11}=&\int_{0}^{T}\left\langle [\langle\nabla\rangle^3\ma, V^1_0]\partial_X W^2_{\neq},\langle\nabla\rangle^3\ma W^2_{\neq}\right\rangle_{L^2} dt\\
				&+\int_{0}^{T}\left\langle  V^1_0\partial_X\langle\nabla\rangle^3 \ma W^2_{\neq},\langle\nabla\rangle^3\ma W^2_{\neq}\right\rangle_{L^2}dt\\
				\lesssim&(\|\partial_Z V^1_0\|_{L^\infty H^4}\|m^{\frac{1}{4}} W^2_{\neq}\|_{L^2H^3}+\|\nabla V^1_0\|_{L^\infty H^3}\|m^{\frac{1}{4}} W^2_{\neq}\|_{L^2H^3})\|\ma W^2_{\neq}\|_{L^2H^3}\\
				\lesssim&\nu^{-4/3}\varepsilon^3.
			\end{align*}
			For the other terms,  Proposition \ref{est:main} implies that
			\begin{align*}
				I^{\nwii}_{12}+I^{\nwii}_{13}\lesssim&\|(V^2_0,V^3_0)\|_{L^\infty H^3}\|\tilde{\nabla}W^2_{\neq}\|_{L^2H^3}\|\ma W^2_{\neq}\|_{L^2H^3}\\
				&+\|V_{\neq}\|_{L^\infty H^3}\|\tilde{\nabla} W^2\|_{L^2H^3}\|\ma W^2_{\neq}\|_{L^2H^3}\\
				\lesssim&\nu^{-1/6}\varepsilon\nu^{-2/3}\varepsilon\nu^{-1/6}\varepsilon\lesssim\nu^{-1}\varepsilon^3.		
			\end{align*}
			Next, we turn to the $I^{\nwii}_2$ term which includes one of the interactions between lift-up effect and compressibility. We first divide it into
			\begin{align*}
				I^{\nwii}_2=&\int_{0}^{T}\left\langle \ma(\tilde{\partial}_YV^1_0\partial_XD_{\neq}),\ma W^2_{\neq}\right\rangle_{H^3} dt\\
				&+\int_{0}^{T}\left\langle \ma\left(\tilde{\partial}_YV^1_{\neq}\partial_XD-Q^1\partial_XV^2\right)_{\neq},\ma W^2_{\neq}\right\rangle_{H^3} dt\\
				&+\int_{0}^{T}\left\langle \ma\left(\tilde{\partial}_YV^3\partial_ZD-Q^3\partial_ZV^2-\Omega^2\tilde{\partial}_YV^2\right)_{\neq},\ma W^2_{\neq}\right\rangle_{H^3} dt=I^{\nwii}_{21}+I^{\nwii}_{22}+I^{\nwii}_{23}.
			\end{align*}
			Recalling the equation $\partial_X{\de}_{\neq}$ :
			\begin{align*}
				&\partial_t \partial_X{\de}+\partial_XD=-\partial_X(V\cdot \tilde{\nabla} {\de})-\partial_X({\de}D).
			\end{align*}
			Invoking the above equation in $I^{\nwii}_{21}$ and using integration by parts in time, we arrive at the following
			\begin{equation}\label{est:I21}
				\begin{aligned}
					I^{\nwii}_{21}=&\int_{0}^{T}\langle \ma(\partial_yV^{1}_{0}\partial_{X}D_{\neq}), \ma W^2\rangle_{H^3}dt\\
					=&\int_{0}^{T}\langle\ma(\partial_yV^{1}_{0}(-\partial_t \partial_X{\de}-\partial_X(V\cdot \tilde{\nabla} {\de})-\partial_X({\de}D))_{\neq}), \ma W^2_{\neq}\rangle_{H^3}dt\\
					=&\left\langle \ma(\partial_yV^{1}_{0}(t)(- \partial_X{\de}(t))_{\neq}), \ma W^2_{\neq}(t)\right\rangle_{H^3}\Big|_{t=0}^{t=T}\\
					&+\int_{0}^{T}\langle \ma(\partial_t(\partial_yV^{1}_{0})\partial_X{\de}_{\neq}), \ma W^2_{\neq}\rangle_{H^3} dt+\int_{0}^{T}\langle \ma((\partial_yV^{1}_{0})\partial_X{\de})_{\neq}, \ma \partial_tW^2_{\neq}\rangle_{H^3} dt\\
					&+2\int_{0}^{T}\langle \partial_t(\ma)((\partial_yV^{1}_{0})\partial_X{\de})_{\neq}, \ma W^2_{\neq}\rangle_{H^3} dt\\
					&-\int_{0}^{T}\langle \ma(\partial_yV^{1}_{0}(\partial_X(V\cdot \tilde{\nabla} {\de})+\partial_X({\de}D))_{\neq}), \ma 	W^2_{\neq}\rangle_{H^3} dt	=\sum_{k=1}^{5}I_{21k}^{(6)}.
				\end{aligned}
			\end{equation}
			The bootstrap assumption implies that
			\begin{align*}
				I^{\nwii}_{211}=&\langle \ma(\partial_yV^{1}_{0}(t)(- \partial_X{\de}(t))_{\neq}), \ma W^2_{\neq}(t)\rangle_{H^3}\Big|_{t=0}^{t=T}\\
				\lesssim&\|\partial_yV^1_0\|_{L^\infty H^4}\|m^{\frac{1}{4}}\partial_X{\de}_{\neq}\|_{L^\infty H^3}\|\ma W^2_{\neq}\|_{L^\infty H^3}
				\lesssim\nu^{-1}\varepsilon^3\lesssim\nu^{1/2}\varepsilon^2.
			\end{align*}
			For $I^{\nwii}_{212}$, which is not the main problem, we give a sketch
			\begin{align*}
				I^{\nwii}_{212}=&\int_{0}^{T}\langle \ma(\partial_t(\partial_yV^{1}_{0})\partial_X{\de}_{\neq}), \ma W^2_{\neq}\rangle_{H^3}dt\\
				=&\int_{0}^{T}\langle \ma((-\partial_yV^2_0+\nu\Delta\partial_yV^1_0+(\nlt_{\partial_yV^1_0})_0)\partial_X{\de}_{\neq}), \ma W^2_{\neq}\rangle_{H^3} dt\\
				\lesssim&\nu^{-1/2}\varepsilon\nu^{-1/3}\varepsilon^2+\nu^{-2}\varepsilon^4\lesssim\varepsilon^2,
			\end{align*}
			where we used that
			\begin{align*}
				&\int_{0}^{T}\langle \ma((\nlt_{\partial_yV^1_0})_0\partial_X{\de}_{\neq}), \ma W^2_{\neq}\rangle_{H^3}dt\\
				\lesssim&\|\tilde{\partial}_Y(V\cdot \tilde{\nabla}V^1)_0+\tilde{\partial}_Y(F({\de})\partial_X{\de})_0\|_{L^\infty H^3}\|\partial_X{\de}_{\neq}\|_{L^2H^3}\|\ma W^2_{\neq}\|_{L^2H^3}\\
				&+\|\tilde{\partial}_Y(G({\de})(\nu\tilde{\Delta} V^1+(\nu+\nu')\partial_XD))_0\|_{L^2 H^3}\|\partial_X{\de}_{\neq}\|_{L^\infty H^3}\|\ma W^2_{\neq}\|_{L^2H^3}\\
				\lesssim&\nu^{-2}\varepsilon^4.
			\end{align*}
			Turning to $I^{\nwii}_{213}$, we decompose it into
			\begin{align*}
				I^{\nwii}_{213}=&\int_{0}^{T}\langle \ma((\partial_YV^{1}_{0})\partial_X{\de}_{\neq}), \ma\partial_t(W^2_{\neq})\rangle_{H^3}dt\\
				=&\int_{0}^{T}\langle \ma((\partial_YV^{1}_{0})\partial_X{\de}_{\neq}), (\ma(V^1_0\partial_XW^2_{\neq}))\rangle_{H^3}dt\\
				&+\int_{0}^{T}\langle \ma((\partial_YV^{1}_{0})\partial_X{\de}_{\neq}), (\ma(\partial_YV^1_0\partial_XD_{\neq}))\rangle_{H^3}dt\\
				&-\int_{0}^{T}\langle \ma((\partial_YV^{1}_{0})\partial_X{\de})_{\neq}, \ma( \mathcal{L}_{W^2}+\nu pW^2)\rangle_{H^3}dt\\
				&+\int_{0}^{T}\langle \ma((\partial_yV^{1}_{0})\partial_X{\de}_{\neq}), (\ma(\nlt_{W^2_{\neq}}-V^1_0\partial_XW^2_{\neq}-\partial_YV^1_0\partial_XD_{\neq}))\rangle_{H^3}dt=\sum_{j=1}^{4}I^{\nwii}_{213j},
			\end{align*}
			where the linear and nonlinear terms are given in \cref{l-W2} and \cref{t-W2}.
			It follows from \cref{2026-3-21-2} that
			\begin{align*}
				I^{\nwii}_{2131}\lesssim&\|\partial_YV^1_0\|_{L^\infty H^4}\|\ma\partial_X{\de}_{\neq}\|_{L^2H^3}(\|V^1_0\|_{L^\infty L^\infty}+\|\nabla V^1_0\|_{L^\infty H^2})\|\tilde{\nabla} \ma W^2_{\neq}\|_{L^2H^3}\\
				\lesssim&\nu^{-1}\varepsilon\nu^{-1/6}\varepsilon\nu^{-1}\varepsilon\nu^{-1/2}\varepsilon=(\nu^{-8/3}\varepsilon^2)\varepsilon^2.
			\end{align*}
			For $I^{\nwii}_{2132}$, direct H{\"o}lder's inequality and Sobolev embeddings give
			\begin{align*}
				I^{\nwii}_{2132}\lesssim&\|\partial_YV^1_0\|_{L^\infty H^4}\|\ma\partial_X{\de}_{\neq}\|_{L^2H^3}\|\partial_Y V^1_0\|_{L^\infty H^4}\|m^{\frac{1}{4}}\partial_X D_{\neq}\|_{L^2H^3}\\
				\lesssim&\nu^{-1}\varepsilon\nu^{-1/6}\varepsilon\nu^{-1}\varepsilon\nu^{-1/2}\varepsilon\lesssim\varepsilon^2.
			\end{align*}
			Noting that $I^{\nwii}_{2133}$ and $I^{\nwii}_{2134}$ are the same as the corresponding terms where $W^2_{\neq}$ is replaced with $(\partial_YV^1_0)\partial_X{\de}_{\neq}$, and both have the same bound. Therefore, we omit them and finish the estimate of $I^{\nwii}_{213}$.
			
			Recalling the definition of $m,M$, we obtain
			\begin{align*}
				I^{\nwii}_{214}\lesssim&\|\partial_YV^1_0\|_{L^\infty H^4}\|m^{\frac{1}{4}}\partial_X{\de}_{\neq}\|_{L^\infty H^3}\|\ma W^2_{\neq}\|_{L^2H^3}\\
				\lesssim&\nu^{-1}\varepsilon(\nu^{-1/6}\varepsilon)^2=\nu^{-4/3}\varepsilon^3.
			\end{align*}
			For $I^{\nwii}_{215}$, we also use \cref{2026-3-21-2} to get
			\begin{align*}
				I^{\nwii}_{215}=&\int_{0}^{T}\langle \ma\partial_yV^{1}_{0}(V^1_0\partial_{XX}{\de}_{\neq}), \ma W^2_{\neq}\rangle_{H^3}dt+...\\
				=&\int_{0}^{T}\langle \ma[p^{-1/2},\partial_y(V^1_0)^2]\partial_{X}p^{1/2}{\de}_{\neq},\ma \partial_XW^2_{\neq}\rangle_{H^3}dt+\\
				&+\int_{0}^{T}\langle m^{\frac{1}{4}}\partial_yV^{1}_{0}(V^1_0\partial_{X}p^{\frac{1}{2}}{\de}_{\neq}), m^{\frac{1}{4}}\partial_Xp^{-\frac{1}{2}}W^2_{\neq}\rangle_{H^3}dt+...\\
				\lesssim&\|\nabla\partial_y (V^1_0)^2\|_{L^\infty H^3}\|\partial_Xp^{\frac{1}{2}}{\de}_{\neq}\|_{L^2H^3}\|m^{\frac{1}{4}}\partial_Xp^{-\frac{1}{2}}W^2\|_{L^2H^3}\\
				\lesssim&(\nu^{-1}\varepsilon)^2\nu^{-2/3}\varepsilon^2=\nu^{-8/3}\varepsilon^4.
			\end{align*}
			Turning to $I^{\nwii}_{22}$ and $I^{\nwii}_{23}$,  Proposition \ref{est:main} implies that
			\begin{align*}
				I^{\nwii}_{22}+I^{\nwii}_{23}=&\int_{0}^{T}\left\langle \ma\left(\tilde{\partial}_YV^1_{\neq}\partial_XD-Q^1\partial_XV^2\right)_{\neq},\ma W^2_{\neq}\right\rangle_{H^3} dt\\
				&+\int_{0}^{T}\left\langle \ma\left(\tilde{\partial}_YV^3\partial_ZD-Q^3\partial_ZV^2-\Omega^2\tilde{\partial}_YV^2\right)_{\neq},\ma W^2_{\neq}\right\rangle_{H^3} dt\\
				\lesssim&\|\partial_YV^1_{\neq}\|_{L^2H^3}\|\partial_XD_{\neq}\|_{L^\infty H^3}\|\ma W^2_{\neq}\|_{L^2H^3}+{\|Q^1\|_{L^\infty H^3}\|\partial_XV^2_{\neq}\|_{L^2 H^3}\|\ma W^2_{\neq}\|_{L^2H^3}}\\ 
				&+\|\tilde{\partial}_Y V^3\|_{L^2H^3}\|\partial_ZD\|_{L^2H^3}\|\ma W^2_{\neq}\|_{L^\infty H^3}+\|Q^3_{\neq}\|_{L^2H^3}\|\partial_ZV^2\|_{L^\infty H^3}\|\ma W^2_{\neq}\|_{L^2H^3}\\
				&+\|Q^3_{0}\|_{L^2H^3}\|\partial_ZV^2_{\neq}\|_{L^2 H^3}\|\ma W^2_{\neq}\|_{L^\infty H^3}+\|\Omega^2\|_{L^\infty H^3}\|\partial_YV^2\|_{L^2H^3}\|\ma W^2_{\neq}\|_{L^2H^3}\\
				\lesssim&(\nu^{-1/2}\varepsilon)^2\nu^{-1/6}\varepsilon+\nu^{-1}\varepsilon\nu^{-1/3}\varepsilon\nu^{-1/6}\varepsilon+\nu^{-1/2}\varepsilon\nu^{-2/3}\varepsilon^2\\
				&+\nu^{-5/6}\varepsilon(\nu^{-1/6}\varepsilon)^2+\nu^{-1/2}\varepsilon\nu^{-1/3}\varepsilon^2+\nu^{-1/6}\varepsilon\nu^{-2/3}\varepsilon\nu^{-1/6}\varepsilon\lesssim\nu^{-3/2}\varepsilon^3.
			\end{align*}
			For $I^{\nwii}_3$, we treat $i=1,3$ respectively. It holds for $i=1$ that
			\begin{align*}
				&\int_{0}^{T}\langle\ma (\partial_XV^2_{\neq}\tilde{\p}_{Y}^{2} V^1-(\partial_XV)\cdot\tilde{\nabla}(\partial_XV^2))_{\neq},\ma W^2_{\neq}\rangle_{H^3} dt\\
				\lesssim&{\|\partial_XV^2_{\neq}\|_{L^2H^3}\|\partial^2_{Y}V^1\|_{L^\infty H^3}\|\ma W^2_{\neq}\|_{L^2H^3}}+\|\partial_XV_{\neq}\|_{L^2H^3}\|\partial_X\tilde{\nabla}V^2\|_{L^2H^3}\|\ma W^2_{\neq}\|_{L^\infty H^3}\\
				\lesssim&\nu^{-1/3}\varepsilon\nu^{-1}\varepsilon\nu^{-1/6}\varepsilon+\nu^{-1/3}\varepsilon\nu^{-2/3}\varepsilon^2\lesssim\nu^{-3/2}\varepsilon^3.
			\end{align*}
			When $i=3$, we divide it into
			\begin{align*}
				&\int_{0}^{T}\langle \ma(\partial_ZV\cdot\tilde{\nabla}\partial_{Z}V^2+\partial_ZV^2\tilde{\p}_{Y}^{2} V^3)_{\neq},\ma W^2_{\neq}\rangle_{H^3} dt\\
				=&\int_{0}^{T}\langle \ma(\partial_ZV^1_{0}\partial_{XZ}(-\tilde{\partial}_Yp^{-1}D)_{\neq}),\ma W^2_{\neq}\rangle_{H^3} dt+\int_{0}^{T}\langle \ma(\partial_ZV^1_{0}\partial_{XZ}\Omega^2_{\neq}),\ma W^2_{\neq}\rangle_{H^3} dt\\
				&+\int_{0}^{T}\langle \ma(\partial_ZV^1_{\neq}\partial_{XZ}V^2_{\neq}),\ma W^2_{\neq}\rangle_{H^3} dt+\int_{0}^{T}\langle \ma(\partial_ZV^2\tilde{\partial}_{YZ}V^2)_{\neq},\ma W^2_{\neq}\rangle_{H^3} dt\\
				&+\int_{0}^{T}\langle \ma(\partial_ZV^3\partial_{ZZ}V^2)_{\neq},\ma W^2_{\neq}\rangle_{H^3} dt+\int_{0}^{T}\langle \ma(\partial_ZV^2\tilde{\p}_{Y}^{2} V^3)_{\neq},\ma W^2_{\neq}\rangle_{H^3} dt
				=\sum_{k=1}^{6}I_{3k}.
			\end{align*}
			For  $I^{\nwii}_{31}$, we use Proposition \ref{est:com-p} and \cref{est:m} to obtain
			\begin{equation}\label{est:I31}
				\begin{aligned}
					I^{\nwii}_{31}=&\int_{0}^{T}\langle \ma(\partial_ZV^1_{0}\partial_{Z}(\tilde{\partial}_Yp^{-1}D)_{\neq}),\ma \partial_XW^2_{\neq}\rangle_{H^3} dt\\
					=&\int_{0}^{T}\langle \ma([p^{-\frac{1}{2}},\partial_ZV^1_{0}]\partial_{Z}(\tilde{\partial}_Yp^{-\frac{1}{2}}D)_{\neq}),\ma \partial_XW^2_{\neq}\rangle_{H^3} dt\\
					&+\int_{0}^{T}\langle \ma(\partial_ZV^1_{0}\partial_{Z}(\tilde{\partial}_Yp^{-\frac{1}{2}}D)_{\neq}),\ma \partial_Xp^{-\frac{1}{2}}W^2_{\neq}\rangle_{H^3} dt\\
					\lesssim&\|\partial_ZV^1_0\|_{L^\infty H^4}(\|\partial_Zp^{-\frac{1}{2}}D_{\neq}\|_{L^2H^3}+{\|m^{\frac{1}{4}}\partial_ZD_{\neq}\|_{L^2H^3}})\|\partial_Xp^{-\frac{1}{2}}\ma W^2_{\neq}\|_{L^2H^3}\\
					\lesssim&\nu^{-1}\varepsilon\nu^{-1/2}\varepsilon^2=\nu^{-3/2}\varepsilon^3.
				\end{aligned}
			\end{equation}
			Using  Proposition \ref{est:main}, we conclude that
			\begin{align*}
				\sum_{k=2}^{6}I^{\nwii}_{3k}\lesssim&\|\partial_ZV^1_0\|_{L^\infty H^3}\|\partial_{XZ}p^{-1}(W^2_{\neq}-\partial_X{\de}_{\neq}+\nu\partial_XD)\|_{L^2H^3}\|\ma W^2_{\neq}\|_{L^2H^3}\\
				&+\|\partial_ZV^1_{\neq}\|_{L^2H^3}\|\partial_{XZ}V^2_{\neq}\|_{L^2H^3}\|\ma W^2_{\neq}\|_{L^\infty H^3}\\
				&+\|\partial_Z(V^2,V^3)\|_{L^\infty H^3}\|\partial_{Z}p^{\frac{1}{2}}V^2\|_{L^2H^3}\|_{L^2H^3}\|\ma W^2_{\neq}\|_{L^2 H^3}\\
				&+\|\partial_ZV^2\|_{L^\infty H^3}\|\tilde{\p}_{Y}^{2} V^3\|_{L^2H^3}\|\ma W^2_{\neq}\|_{L^2 H^3}\\
				\lesssim&\nu^{-1}\varepsilon^2\nu^{-1/6}\varepsilon+\nu^{-1/6}\varepsilon\nu^{-1/2}\varepsilon^2+\nu^{-1/6}\varepsilon\nu^{-2/3}\varepsilon\nu^{-1/6}\varepsilon+\nu^{-1/6}\varepsilon\nu^{-5/6}\varepsilon\nu^{-1/6}\varepsilon \lesssim\nu^{-7/6}\varepsilon^3.
			\end{align*}
			For the component of $I^{\nwii}_4$ within $\partial_ZV^1_0$, a combination of \cref{est:m} and integration by parts yields
			\begin{align*}
				&\int_{0}^{T}\langle \ma\tilde{\partial}_Y(\partial_XV^3_{\neq}\partial_ZV^1_{0}),\ma W^2_{\neq}\rangle_{H^3}dt\\
				=&\int_{0}^{T}-\langle \ma(\partial_{XZ}p^{-1}D_{\neq}\partial_ZV^1_0),\ma\tilde{\partial}_YW^2_{\neq}\rangle_{H^3} +
				\langle \ma(\tilde{\partial}_{XY}p^{-1}\Omega^3_{\neq}\partial_ZV^1_0),\ma W^2_{\neq}\rangle_{H^3} dt\\
				&+\int_{0}^{T}
				\langle \ma(\partial_{X}p^{-1}\Omega^3_{\neq}\tilde{\partial}_{YZ}V^1_0),\ma W^2_{\neq}\rangle_{H^3} dt\\
				\lesssim&{\|\partial_Xp^{-\frac{1}{2}}m^{\frac{1}{4}}\partial_Zp^{-\frac{1}{2}}D_{\neq}\|_{L^2H^3}\|\partial_ZV^1_0\|_{L^\infty H^4}\|\tilde{\nabla} \ma W^2_{\neq}\|_{L^2H^3}+\|m^{\frac{3}{4}}\Omega^3_{\neq}\|_{L^2H^3}\|\partial_ZV^1_{0}\|_{L^\infty H^4}\|\ma W^2_{\neq}\|_{L^2H^3}}\\
				&+\|m\Omega^3_{\neq}\|_{L^2H^3}\|\partial_ZV^1_0\|_{L^\infty H^4}\|\ma W^2_{\neq}\|_{L^2H^3}\\
				\lesssim&\varepsilon\nu^{-1}\varepsilon\nu^{-1/2}\varepsilon+\nu^{-1/3}\varepsilon\nu^{-1}\varepsilon\nu^{-1/6}\varepsilon+\nu^{-1/6}\varepsilon\nu^{-1}\varepsilon\nu^{-1/6}\varepsilon\lesssim\nu^{-3/2}\varepsilon^3.
			\end{align*}
			Considering the other components of $I^{\nwii}_4$, we have
			\begin{align*}
				&\int_{0}^{T}\langle \ma((\partial_XV^1_{\neq})^2+\partial_XV^3_{\neq}\partial_ZV^1_{\neq}+(\partial_ZV^3)^2_{\neq}),\ma \tilde{\partial}_YW^2_{\neq}\rangle_{H^3} dt\\
				\lesssim&\|\nabla_{X,Z}(V^1_{\neq}, V^3)\|_{L^\infty H^3}\|\nabla_{X,Z}(V^1_{\neq},V^3_{\neq})\|_{L^2 H^3}\|\tilde{\nabla}\ma W^2_{\neq}\|_{L^2H^3}\\
				\lesssim&\varepsilon\nu^{-1/6}\varepsilon\nu^{-1/2}\varepsilon\lesssim\nu^{-2/3}\varepsilon^3.
			\end{align*}
			For $I^{\nwii}_5$ and $I^{\nwii}_6$, applying the bootstrap assumption, we conclude that
			\begin{align*}
				I^{\nwii}_5\lesssim&\|\partial_XV_{\neq}\|_{L^2H^3}\|\tilde{\nabla} {\de}\|_{L^\infty H^3}\|\ma W^2_{\neq}\|_{L^2H^3}+\|\partial_X{\de}_{\neq}\|_{L^2H^3}\|D\|_{L^\infty H^3}\|\ma W^2_{\neq}\|_{L^2H^3}\\
				&+\|{\de}\|_{L^\infty H^3}\|\partial_X D_{\neq}\|_{L^2H^3}\|\ma W^2_{\neq}\|_{L^2H^3}+\nu\|\partial_XV_{\neq}\|_{L^2H^3}\|\tilde{\nabla} D\|_{L^2 H^3}\|\ma W^2_{\neq}\|_{L^\infty H^3}\\
				&+\nu\|\partial_iV^j\partial_jV^i\|_{L^2H^3}\|\partial_X\ma W^2_{\neq}\|_{L^2H^3}\\
				\lesssim&\nu^{-1/3}\varepsilon\nu^{-1/2}\varepsilon\nu^{-1/6}\varepsilon+\nu^{-1/6}\varepsilon\nu^{-2/3}\varepsilon\nu^{-1/6}\varepsilon+\nu\nu^{-1/3}\varepsilon\nu^{-1}\varepsilon^2+\nu\nu^{-4/3}\varepsilon^2\nu^{-1/2}\varepsilon\lesssim\nu^{-1}\varepsilon^3,
			\end{align*}
			and
			\begin{align*}
				I^{\nwii}_6\lesssim&(\|p^{\frac{1}{2}}{\de}\|_{L^\infty H^3}(\nu\|\tilde{\Delta} V^2\|_{L^2H^3}+(\nu+\nu')\|\tilde{\partial}_YD\|_{L^2H^3})\\
				&+\|{\de}\|_{L^\infty H^3}(\nu\|p^{\frac{1}{2}}\tilde{\Delta} V^2\|_{L^2H^3}+(\nu+\nu')\|p^{\frac{1}{2}}\tilde{\partial}_YD\|_{L^2H^3})\\
				&+\|\partial_X{\de}\|_{L^2H^3}(\nu\|\tilde{\Delta}V^1\|_{L^\infty H^3}+(\nu+\nu')\|\partial_XD\|_{L^\infty H^3})\\
				&+\|\tilde{\partial}_Y{\de}\|_{L^\infty H^3}(\nu\|\tilde{\Delta}V^2\|_{L^2 H^3}+(\nu+\nu')\|\tilde{\partial}_YD\|_{L^2 H^3})\\
				&+\|\partial_Z{\de}\|_{L^\infty H^3}(\nu\|\tilde{\Delta}V^3\|_{L^2 H^3}+(\nu+\nu')\|\partial_ZD\|_{L^2 H^3})+\|{\de}\|_{L^\infty H^3}\bar{\nu}\|pD\|_{L^2H^3}\\
				&+\nu\|\partial_X{\de}\|_{L^2H^3}\|\tilde{\nabla}{\de}\|_{L^\infty H^3}+\nu\|{\de}\|_{L^\infty H^3}\|\partial_X\tilde{\nabla}{\de}\|_{L^2 H^3})\|\ma p^{\frac{1}{2}}W^2_{\neq}\|_{L^2H^3}\\
				\lesssim&\nu^{-1/2}\varepsilon\bar{\nu}\nu^{-1}\varepsilon\nu^{-1/2}\varepsilon+\nu^{-1/6}\varepsilon\bar{\nu}\nu^{-4/3}\varepsilon\nu^{-1/2}\varepsilon+\nu^{-1/3}\varepsilon\nu^{-1}\varepsilon\nu^{-1/2}\varepsilon+\nu^{-1/2}\varepsilon\bar{\nu}\nu^{-1}\varepsilon\nu^{-1/2}\varepsilon\\
				&+\nu^{-1/6}\varepsilon\bar{\nu}\nu^{-5/6}\varepsilon\nu^{-1/2}\varepsilon+\nu^{-1/6}\varepsilon\bar{\nu}\nu^{-4/3}\varepsilon\nu^{-1/2}\varepsilon+\nu\nu^{-1/6}\varepsilon(\nu^{-1/2}\varepsilon)^2+\nu\nu^{-1/6}\varepsilon\nu^{-2/3}\varepsilon\nu^{-1/2}\varepsilon\\
				\lesssim&\nu^{-1}\varepsilon^3.
			\end{align*} 
			Collecting all the estimates for $I_i^{(6)}$ ($i = 1, \ldots, 6$), and $LS^{(6)}$, $LE^{(6)}$ we thus obtain Lemma \ref{Lem-E6}.
		\end{proof}
		
		\subsubsection{$H^3$ estimate on  $W^2_{0\neq}$}
		\begin{Lem}\label{Lem-E61}
			Under the same assumptions as \cref{MT} and  Proposition \ref{prop-bs}, it holds that
			\begin{align}
				\sup_{0\leq t\leq T}\|W^2_0(t)\|^2_{H^3}+\nu\|\nabla W^2_0\|^2_{L^2H^3}\lesssim\|W^2_0(0)\|^2_{H^3}+\nu^{-\frac{3}{2}}\varepsilon^3.
			\end{align}
		\end{Lem}
		\begin{proof}
			
			The energy estimate gives
			\begin{align*}
				\|W^2_0(T)\|^2_{H^3}+\nu\|\nabla W^2_0\|^2_{L^2H^3}=\|W^2_0(0)\|^2_{H^3}+K^{(1)}_1+K^{(1)}_2,
			\end{align*}
			where we denote
			\begin{align*}
				K^{(1)}_1=&\int_{0}^{T}\langle (\nlt_{W^2_0})_{1}, W^2_0\rangle_{H^3} dt,\\
				K^{(1)}_2=&\int_{0}^{T}\langle (\nlt_{W^2_0})_{2}, W^2_0\rangle_{H^3} dt,
			\end{align*}
			with $\nlt_{W^2_0}$ given in \cref{bs-w2}, and
			\begin{align*}
				(\nlt_{W^2_0})_{1}=&-V_0\cdot\nabla W^2_0+\partial_YV^3_0\partial_ZD_0-Q^3_0\partial_ZV^2_0-W^2_0\partial_YV^2_0\\
				&+2\partial_ZV^2_0\partial^2_{Y}V^3_0-\partial_ZV^2_0\partial_{ZY}V^2_0-\partial_ZV^3_0\partial_{ZZ}V^2_0+\partial_Y(\partial_ZV^3_0)^2\\
				&-\Delta(G({\de})_0(\nu\Delta V^2_0+(\nu+\nu')\partial_YD_0))\\
				&+\partial_Y(\partial_iG({\de})_0(\nu\Delta V^i_0+(\nu+\nu')\partial_i D_0))+\partial_Y(G({\de})_0\bar{\nu}\Delta D_0),\\
				(\nlt_{W^2_0})_2=&\nlt_{W^2_0}-(\nlt_{W^2_0})_1.
			\end{align*}	
			For $K^{(1)}_1$, noting that incompressible part $W^2_0=\Omega^2_0$ has the nonzero $l$-mode and does not have interaction within $V^1_0$, we deduce that
			\begin{align*}
				K^{(1)}_1\lesssim&\|(V^2_0,V^3_0)\|_{L^\infty H^3}\|\nabla W^2_0\|_{L^2H^3}\|W^2_0\|_{L^2H^3}+\|\partial_YV^3_0\|_{L^2H^3}\|\partial_ZD_0\|_{L^2H^3}\|W^2_0\|_{L^\infty H^3}\\
				&+\|W^2_0\|_{L^2H^3}\|\partial_YV^2_0\|_{L^2H^3}\|W^2_0\|_{L^\infty H^3}+\|\partial_Z(V^2_0,V^3_0)\|_{L^2H^3}\|(\partial_Z\nabla V^2_0,\Delta V^3_0)\|_{L^2H^3}\|W^2_0\|_{L^\infty H^3}\\
				&+(\|p^{\frac{1}{2}}{\de}_0\|_{L^\infty H^3}\|(\nu\Delta V^2_0+(\nu+\nu')\partial_YD_0)\|_{L^2H^3}+\|{\de}_0\|_{L^\infty H^3}\|p^{\frac{1}{2}}(\nu\Delta V^2_0+(\nu+\nu')\partial_YD_0)\|_{L^2H^3}\\
				&+\|\nabla {\de}\|_{L^\infty H^3}\|(\nu\partial_Y\Delta V^2_0+\nu\partial_Z\Delta V^3_0+(\nu+\nu')p^{\frac{1}{2}} D_0)\|_{L^2H^3}+\|{\de}_0\|_{L^\infty H^3}\bar{\nu}\|pD\|_{L^2H^3})\|\nabla W^2_0\|_{L^2H^3}\\
				\lesssim&\nu^{-1}\varepsilon^3.
			\end{align*}
			Through a process similar to $W^2_{\neq}$, we conclude that
			\begin{align*}
				K^{(1)}_2\lesssim\nu^{-7/6}\varepsilon^3.
			\end{align*}
			Then Lemma \ref{Lem-E61} follows.
		\end{proof}
		\subsection{$H^3$ estimate on $\Omega^3$}
		In this subsection, we improve \cref{bs-o3} for $j=0$ by providing estimates for $\norm{\Omega^3_{\neq}}_{H^3}$ and $\norm{\Omega^3_{0}}_{H^3}$ separately.
		\subsubsection{$H^3$ estimate on  $\Omega^3_{\neq}$}
		\begin{Lem}\label{Lem-E7}
			Under the same assumptions as \cref{MT} and  Proposition \ref{prop-bs}, it holds that
			\begin{align}
				\sup_{0\leq t\leq T}&\|mM \Omega^3_{\neq}(T)\|^2_{H^3}+\nu\|\tilde{\nabla}mM \Omega^3_{\neq}\|^2_{L^2H^3}\nonumber\\&+\left\|\sqrt{-\frac{\partial_t M}{M}}mM \Omega^3_{\neq}\right\|^2_{L^2H^3}+\left\|\sqrt{-\frac{\partial_t m}{m}}mM \Omega^3_{\neq}\right\|^2_{L^2H^3}
				\lesssim\|\Omega^3_{\neq}(0)\|_{H^3}+\nu^{-\frac{3}{2}}\varepsilon^3.
			\end{align}
		\end{Lem}
		\begin{proof}
			
			The energy estimate gives
			\begin{align*}
				&\|mM \Omega^3_{\neq}(T)\|^2_{H^3}+\nu\|\tilde{\nabla}mM \Omega^3_{\neq}\|^2_{L^2H^3}+\left\|\sqrt{-\frac{\partial_t M}{M}}mM \Omega^3_{\neq}\right\|^2_{L^2H^3}+\left\|\sqrt{-\frac{\partial_t m}{m}}mM \Omega^3_{\neq}\right\|^2_{L^2H^3}\\
				=&\|\Omega^3_{\neq}(0)\|^2_{H^3}+LS^{\noiii}+LE^{\noiii}+\sum_{i=1}^{5}I^{\noiii}_i,
			\end{align*}
			where linear and nonlinear terms are defined by
			\begin{align*}
				LS^{\noiii}=&\int_{0}^{T}\left\langle \frac{\partial_t p}{p}mM\Omega^3_{\neq},mM\Omega^3_{\neq}\right\rangle_{H^3} dt,\\
				LE^{\noiii}=&-\int_{0}^{T}\left\langle mM\left(2\frac{\partial_{XZ}}{p}(W^2-2\partial_X{\de}+\nu\partial_XD)\right)_{\neq},mM \Omega^3_{\neq}\right\rangle_{H^3} dt,\\
				I_1^{\noiii}=&-\int_{0}^{T}\left\langle mM(V\cdot\tilde{\nabla}\Omega^3)_{\neq},mM \Omega^3_{\neq}\right\rangle_{H^3} dt,\\
				I_2^{\noiii}=&\int_{0}^{T}\left\langle mM\left(-Q^1\partial_XV^3-Q^2\tilde{\partial}_YV^3-\Omega^3\partial_ZV^3+\partial_ZV^1\partial_XD+\partial_ZV^2\tilde{\partial}_YD\right)_{\neq},mM \Omega^3_{\neq}\right\rangle_{H^3} dt,\\
				I_3^{\noiii}=&\int_{0}^{T}\left\langle mM\left(2\sum_{i=1,2}(\tilde{\partial}_iV^3\partial_{ZZ}V^i-\tilde{\partial}_iV^j\tilde{\partial}_{ij}V^3)\right)_{\neq},mM \Omega^3_{\neq}\right\rangle_{H^3} dt,\\
				I_4^{\noiii}=&\int_{0}^{T}\left\langle mM\left(\sum_{i,j=1,2}\partial_Z(\tilde{\partial}_iV^j\tilde{\partial}_jV^i)\right)_{\neq},mM \Omega^3_{\neq}\right\rangle_{H^3} dt,\\
				I_5^{\noiii}=&\int_{0}^{T}\big\langle mM\left(\partial_Z\tilde{\textrm{div}}\left(G({\de})(\nu\tilde{\Delta} V+(\nu+\nu')\tilde{\nabla} D)\right)-\tilde{\Delta}\left(G({\de})(\nu\tilde{\Delta} V^3+(\nu+\nu')\partial_ZD)\right)\right)_{\neq},mM \Omega^3_{\neq}\rangle_{H^3} dt.
			\end{align*}
			For $LS^{\noiii}$ and $LE^{\noiii}$, by the definition and monotonicity properties of the multipliers $m$ and $M$, we have
			\begin{align*}
				LS^{\noiii}\leq& \left\|\sqrt{-\frac{\partial_t m}{m}}mM \Omega^3_{\neq}\right\|^2_{L^2H^3}+\frac{c_1}{4}\nu^{1/3}\|mM\Omega^3_{\neq}\|^2_{L^2H^3},\\
				LE^{\noiii}\lesssim&\frac{1}{\tilde{C}}\left\|\sqrt{-\frac{\partial_t M_1}{M_1}}mM \Omega^3_{\neq}\right\|_{L^2H^3}\left(\left\|\sqrt{-\frac{\partial_t M_1}{M_1}}(\ma W^2_{\neq},\ma \partial_X{\de}_{\neq})\right\|_{L^2H^3}+\nu\|\tilde{\nabla} \nabla_{X,Z}p^{-\frac{1}{2}}D_{\neq}\|_{L^2H^3}\right).
			\end{align*}
			Turning to $I^{\noiii}_1$, we divide it into
			\begin{equation}\label{est:I1}
				\begin{aligned}
					I_1^{\noiii}=&\int_{0}^{T}\left\langle mM(V^1_0\partial_X\Omega^3_{\neq}),mM \Omega^3_{\neq}\right\rangle_{H^3} dt+\sum_{i=2,3}\int_{0}^{T}\left\langle mM(V^i_0\tilde{\partial}_i\Omega^3_{\neq}),mM \Omega^3_{\neq}\right\rangle_{H^3} dt\\
					&+\int_{0}^{T}\left\langle mM(V_{\neq}\cdot\tilde{\nabla}\Omega^3),mM \Omega^3_{\neq}\right\rangle_{H^3} dt=\sum_{i=1}^{3}I^{\noiii}_{1i}.
				\end{aligned}
			\end{equation}
			It follows from Lemma \ref{lemma:commutator} and Proposition \ref{est:com} that
			\begin{align*}
				I^{\noiii}_{11}=&\int_{0}^{T}\left\langle ([mM\langle\nabla\rangle^3,V^1_0]\partial_X\Omega^3_{\neq}),mM\langle\nabla\rangle^3\Omega^3_{\neq}\right\rangle_{L^2} dt\\
				=&\int_{0}^{T}\left\langle M\langle\nabla\rangle^3([m,V^1_0]\partial_X\Omega^3_{\neq}),mM\langle\nabla\rangle^3\Omega^3_{\neq}\right\rangle_{L^2} dt+\int_{0}^{T}\left\langle ([M\langle\nabla\rangle^3,V^1_0]m\partial_X\Omega^3_{\neq}),mM\langle\nabla\rangle^3\Omega^3_{\neq}\right\rangle_{L^2} dt\\
				\lesssim&\|\nabla V^1_{0}\|_{L^\infty H^4}\|mM\Omega^3_{\neq}\|^2_{L^2H^3}\lesssim\nu^{-4/3}\varepsilon^3.
			\end{align*}
			For $I^{\noiii}_{12}$ and $I^{\noiii}_{13}$,  Proposition \ref{est:main} implies that
			\begin{align*}
				I^{\noiii}_{12}+I^{\noiii}_{13}\lesssim&\|(V^2_0,V^3_0)\|_{L^\infty H^3}\|\tilde{\nabla} \Omega^3_{\neq}\|_{L^2H^3}\|mM\Omega^3_{\neq}\|_{L^2H^3}+{\|V_{\neq}\|_{L^2 H^3}\|\tilde{\nabla} \Omega^3\|_{L^2H^3}\|mM\Omega^3_{\neq}\|_{L^\infty H^3}}\\
				\lesssim&\varepsilon\nu^{-7/6}\varepsilon\nu^{-1/6}\varepsilon+\nu^{-1/3}\varepsilon\nu^{-7/6}\varepsilon^2\lesssim\nu^{-3/2}\varepsilon^3.
			\end{align*}
			Turning to $I^{\noiii}_2$, we divide it into
			\begin{align*}
				I^{\noiii}_2=&\int_{0}^{T}\left\langle mM(\partial_ZV^1_0\partial_XD_{\neq}),mM\Omega^3_{\neq}\right\rangle_{H^3} dt\\
				&+\int_{0}^{T}\left\langle mM\left(\partial_ZV^1_{\neq}\partial_XD-Q^1\partial_XV^3\right)_{\neq},mM \Omega^3_{\neq}\right\rangle_{H^3} dt\\
				&+\int_{0}^{T}\left\langle mM\left(\partial_ZV^2\tilde{\partial}_YD-Q^2\tilde{\partial}_YV^3-\Omega^3\partial_ZV^3\right)_{\neq},mM \Omega^3_{\neq}\right\rangle_{H^3} dt=I^{\noiii}_{21}+I^{\noiii}_{22}+I^{\noiii}_{23}.
			\end{align*}
			The estimates of $I^{\noiii}_{21}$ are similar to those of \cref{est:I21} line by line, {we omit it here.}

			For $I^{\noiii}_3$, we treat $i=1,2$ respectively. It holds for $i=1$ that
			\begin{align*}
				&\int_{0}^{T}\langle mM (\partial_XV^3_{\neq}\partial_{ZZ}V^1-(\partial_XV)\cdot\tilde{\nabla}(\partial_XV^3))_{\neq},mM \Omega^3_{\neq}\rangle_{H^3} dt\\
				\lesssim&\|\partial_XV^3_{\neq}\|_{L^2H^3}\|\partial_{ZZ}V^1\|_{L^\infty H^3}\|mM \Omega^3_{\neq}\|_{L^2H^3}+\|\partial_XV_{\neq}\|_{L^2H^3}\|\partial_X\tilde{\nabla}V^3\|_{L^2}\|mM\Omega^3_{\neq}\|_{L^\infty H^3}\\
				\lesssim&\nu^{-1/6}\varepsilon\nu^{-1}\varepsilon\nu^{-1/6}\varepsilon+\nu^{-1/3}\varepsilon\nu^{-1/2}\varepsilon^2\lesssim\nu^{-4/3}\varepsilon^3.
			\end{align*}
			When $i=2$, we divide it into
			\begin{align*}
				&\int_{0}^{T}\langle mM(\tilde{\partial}_{Y}V\cdot\tilde{\nabla}\tilde{\partial}_{Y}V^3+\tilde{\partial}_YV^3\partial_{ZZ}V^2)_{\neq},mM \Omega^3_{\neq}\rangle_{H^3} dt\\
				=&\int_{0}^{T}\langle mM(\partial_YV^1_{0}\tilde{\partial}_{XY}(-\partial_Zp^{-1}D)_{\neq}),mM \Omega^3_{\neq}\rangle_{H^3} dt+\int_{0}^{T}\langle mM(\partial_YV^1_{0}\tilde{\partial}_{XY}p^{-1}\Omega^3_{\neq}),mM \Omega^3_{\neq}\rangle_{H^3} dt\\
				&+\int_{0}^{T}\langle mM(\tilde{\partial}_YV^1_{\neq}\tilde{\partial}_{XY}V^3_{\neq}),mM \Omega^3_{\neq}\rangle_{H^3} dt+\int_{0}^{T}\langle mM(\tilde{\partial}_YV^2\tilde{\p}_{Y}^{2} V^3)_{\neq},mM \Omega^3_{\neq}\rangle_{H^3} dt\\
				&+\int_{0}^{T}\langle mM(\tilde{\partial}_YV^3\tilde{\partial}_{YZ}V^3)_{\neq},mM \Omega^3_{\neq}\rangle_{H^3} dt+\int_{0}^{T}\langle mM(\tilde{\partial}_YV^3\partial_{ZZ}V^2)_{\neq},mM \Omega^3_{\neq}\rangle_{H^3} dt		=\sum_{k=1}^{6}I^{\noiii}_{3k}.
			\end{align*}
			By the same procedure as in \cref{est:I31}, we obtain
			\begin{align*}
				I^{\noiii}_{31}=&\int_{0}^{T}\langle mM([p^{-\frac{1}{2}},\partial_YV^1_{0}]\partial_{Z}(\tilde{\partial}_Yp^{-\frac{1}{2}}D)_{\neq}),mM \partial_X\Omega^3_{\neq}\rangle_{H^3} dt\\
				&+\int_{0}^{T}\langle mM(\partial_YV^1_{0}\partial_{Z}(\tilde{\partial}_Yp^{-\frac{1}{2}}D)_{\neq}),mM \partial_Xp^{-\frac{1}{2}}\Omega^3_{\neq}\rangle_{H^3} dt\\
				\lesssim&\|\partial_YV^1_0\|_{L^\infty H^4}(\|\partial_Zp^{-\frac{1}{2}}D_{\neq}\|_{L^2H^3}+{\|m^{\frac{1}{4}}\partial_ZD_{\neq}\|_{L^2H^3}})\|\partial_Xp^{-\frac{1}{2}}mM \Omega^3_{\neq}\|_{L^2H^3}\\
				\lesssim&\nu^{-1}\varepsilon\nu^{-1/2}\varepsilon^2=\nu^{-3/2}\varepsilon^3.
			\end{align*}
			The remaining terms can be estimated by  Proposition \ref{est:main}.
			\begin{align*}
				\sum_{k=2}^{6}I^{\noiii}_{3k}\lesssim&\|\partial_YV^1_0\|_{L^\infty H^4}\|m^{\frac{1}{2}}\tilde{\partial}_{XY}p^{-1}\Omega^3_{\neq}\|_{L^2H^3}\|mM \Omega^3_{\neq}\|_{L^2H^3}\\
				&+\|\tilde{\partial}_YV^1_{\neq}\|_{L^2H^3}\|\tilde{\partial}_{XY}V^3_{\neq}\|_{L^2H^3}\|mM \Omega^3_{\neq}\|_{L^\infty H^3}\\
				&+{\|\tilde{\partial}_Y(V^2,V^3)\|_{L^2 H^3}\|pV^3\|_{L^2H^3}\|_{L^2H^3}\|mM\Omega^3_{\neq}\|_{L^\infty H^3}}\\
				&+\|\tilde{\partial}_YV^3\|_{L^2 H^3}\|\partial_{ZZ}V^2\|_{L^2H^3}\|mM \Omega^3_{\neq}\|_{L^\infty H^3}\\
				\lesssim&\nu^{-1}\varepsilon(\nu^{-1/6}\varepsilon)^2+\nu^{-1/2}\varepsilon\nu^{-1/2}\varepsilon^2+\nu^{-2/3}\varepsilon\nu^{-5/6}\varepsilon^2+\nu^{-1/2}\varepsilon\nu^{-1/2}\varepsilon^2 \lesssim\nu^{-3/2}\varepsilon^3.
			\end{align*}
			The component of $I^{\noiii}_4$ within $\partial_YV^1_0$ is similar to that of \cref{est:I31}, we use commutator estimate to get
			\begin{align*}
				&\int_{0}^{T}\langle mM\partial_Z(\partial_XV^2_{\neq}\partial_YV^1_{0}),mM \Omega^3_{\neq}\rangle_{H^3} dt\\
				=&-\int_{0}^{T}\langle mM(\tilde{\partial}_{XYZ}p^{-1}D_{\neq}\partial_YV^1_0), mM \Omega^3_{\neq}\rangle_{H^3} +
				\langle mM(\partial_{XZ}p^{-1}(W^2-\partial_X{\de}+\nu\partial_XD)_{\neq}\partial_YV^1_0),mM \Omega^3_{\neq}\rangle_{H^3} dt\\
				&+\int_{0}^{T}
				\langle mM(\partial_XV^2_{\neq}\partial_{YZ}V^1_0),mM \Omega^3_{\neq}\rangle_{H^3} dt\\
				\lesssim&\nu^{-3/2}\varepsilon^3.
			\end{align*}
			For the other component of $I^{\noiii}_4$, it holds that
			\begin{align*}
				&\int_{0}^{T}\langle mM\partial_Z((\partial_XV^1_{\neq})^2+\partial_XV^2_{\neq}\partial_YV^1_{\neq}+(\partial_YV^2)^2_{\neq}),mM \Omega^3_{\neq}\rangle_{H^3} dt\\
				\lesssim&(\|\partial_{XZ}V^1_{\neq}\|_{L^2H^3}\|\partial_XV^1_{\neq}\|_{L^2H^3}+\|\tilde{\partial}_{YZ}V^1_{\neq}\|_{L^2H^3}\|\partial_XV^2_{\neq}\|_{L^2H^3}+\|\tilde{\partial}_{Y}V^1_{\neq}\|_{L^2H^3}\|\partial_{XZ}V^2_{\neq}\|_{L^2H^3}\\
				&+\|\tilde{\partial}_YV^2\|_{L^2H^3}\|\tilde{\partial}_{YZ}V^2\|_{L^2H^3})\|mM \Omega^3_{\neq}\|_{L^\infty H^3}
				\lesssim\nu^{-4/3}\varepsilon^3.
			\end{align*}
			For $I^{\noiii}_5$, by bootstrap argument, we have
			\begin{align*}
				I^{\noiii}_5\lesssim&(\|p^{\frac{1}{2}}{\de}\|_{L^\infty H^3}(\nu\|\tilde{\Delta} V^3\|_{L^2H^3}+(\nu+\nu')\|\partial_ZD\|_{L^2H^3})\\
				&+\|{\de}\|_{L^\infty H^3}(\nu\|p^{\frac{1}{2}}\tilde{\Delta} V^3\|_{L^2H^3}+(\nu+\nu')\|p^{\frac{1}{2}}\partial_ZD\|_{L^2H^3})\\
				&+\|\partial_X{\de}\|_{L^2H^3}(\nu\|\tilde{\Delta}V^1\|_{L^\infty H^3}+(\nu+\nu')\|\partial_XD\|_{L^\infty H^3})\\
				&+\|\tilde{\partial}_Y{\de}\|_{L^\infty H^3}(\nu\|\tilde{\Delta}V^2\|_{L^2 H^3}+(\nu+\nu')\|\tilde{\partial}_YD\|_{L^2 H^3})\\
				&+\|\partial_Z{\de}\|_{L^\infty H^3}(\nu\|\tilde{\Delta}V^3\|_{L^2 H^3}+(\nu+\nu')\|\partial_ZD\|_{L^2 H^3})+\|{\de}\|_{L^\infty H^3}\bar{\nu}\|pD\|_{L^2H^3})\|mM p^{\frac{1}{2}}\Omega^3_{\neq}\|_{L^2H^3}\\
				\lesssim&\nu^{-1/2}\varepsilon\bar{\nu}\nu^{-1}\varepsilon\nu^{-1/2}\varepsilon+\nu^{-1/6}\varepsilon\bar{\nu}\nu^{-7/6}\varepsilon\nu^{-1/2}\varepsilon+\nu^{-1/3}\varepsilon\nu^{-1}\varepsilon\nu^{-1/2}\varepsilon+\nu^{-1/2}\varepsilon\bar{\nu}\nu^{-1}\varepsilon\nu^{-1/2}\varepsilon\\
				&+\nu^{-1/6}\varepsilon\bar{\nu}\nu^{-5/6}\varepsilon\nu^{-1/2}\varepsilon+\nu^{-1/6}\varepsilon\bar{\nu}\nu^{-4/3}\varepsilon\nu^{-1/2}\varepsilon
				\lesssim\nu^{-1}\varepsilon^3.
			\end{align*} 
			Putting together the estimates for $LS^{(7)}$, $LE^{(7)}$, and $I_{i}^{(7)}$ ($i=1,\dots,5$) yields Lemma \ref{Lem-E7}.
		\end{proof}
		
		\subsubsection{$H^3$ estimate on  $\Omega^3_0$}
		\begin{Lem}\label{Lem-E71}
			Under the same assumptions as \cref{MT} and  Proposition \ref{prop-bs}, it holds that
			\begin{align}
				\sup_{0\leq t\leq T}\|\Omega^3_0(t)\|^2_{H^3}+\nu\|\nabla \Omega^3_0\|^2_{L^2H^3}\lesssim\|\Omega^3_0(0)\|^2_{H^3}+\nu^{-\frac{3}{2}}\varepsilon^3.
			\end{align}
		\end{Lem}
		\begin{proof}
			
			The energy estimate gives
			\begin{align*}
				\|\Omega^3_0(T)\|^2_{H^3}+\nu\|\nabla \Omega^3_0\|^2_{L^2H^3}=\|\Omega^3_0(0)\|^2_{H^3}+K^{(2)}_1+K^{(2)}_2,
			\end{align*}
			where we denote
			\begin{align*}
				K^{(2)}_1=&\int_{0}^{T}\langle (\nlt_{\Omega^3_0})_{1}, \Omega^3_0\rangle_{H^3} dt,\\
				K^{(2)}_2=&\int_{0}^{T}\langle (\nlt_{\Omega^3_0})_{2}, \Omega^3_0\rangle_{H^3} dt,
			\end{align*}
			with $\nlt_{\Omega^3_0}$  given in \cref{t-O3}, and 
			\begin{align*}
				(\nlt_{\Omega^3_0})_{1}=&-V_0\cdot\nabla \Omega^3_0+\partial_ZV^2_0\partial_YD_0-Q^2_0\partial_YV^3_0-\Omega^3_0\partial_ZV^3_0\\
				&+2\partial_YV^3_0\partial_{ZZ}V^3_0-2\partial_YV^2_0\partial^2_{Y}V^3_0-2\partial_YV^3_0\partial_{ZY}V^2_0+\partial_Z(\partial_YV^2_0)^2\\
				&-\Delta(G({\de})_0(\nu\Delta V^3_0+(\nu+\nu')\partial_ZD_0))\\
				&+\partial_Z(\partial_iG({\de})_0(\nu\Delta V^i_0+(\nu+\nu')\partial_i D_0))+\partial_Z(G({\de})_0\bar{\nu}\Delta D_0),\\
				(\nlt_{\Omega^3_0})_{2}=&\nlt_{\Omega^3_0}-(\nlt_{\Omega^3_0})_{1}.
			\end{align*}	
			For the component of $K^{(2)}_1$ within $\partial_YD_0$, integration by parts implies that
			\begin{align*}
				\int_{0}^{T}\langle (\partial_ZV^2_0\partial_YD_0-Q^2_0\partial_YV^3_0),\Omega^3_{0}\rangle_{H^3} dt
				=&\int_{0}^{T}-\langle W^2_0\partial_YV^3_0,\Omega^3_{0}\rangle_{H^3}+\langle D_0(\partial^2_{Y}V^3_0+\partial_{YZ}V^2_0),\Omega^3_{0}\rangle_{H^3}\\
				&+\langle D_0(\partial_{Y}V^3_0+\partial_{Z}V^2_0),\partial_Y\Omega^3_{0}\rangle_{H^3} dt.
			\end{align*}
			Therefore, by  Proposition \ref{est:main}, we have
			\begin{align*}
				K^{(2)}_1\lesssim&\|(V^2_0,V^3_0)\|_{L^\infty H^3}\|\nabla \Omega^3_0\|_{L^2H^3}\|\Omega^3_0\|_{L^2H^3}+\|W^2_0\|_{L^2H^3}\|\partial_YV^3_0\|_{L^2H^3}\|\Omega^3_0\|_{L^\infty H^3}\\
				&+\|D_0\|_{L^2H^3}\|(\partial_Y V^3_0,\partial_ZV^2_0)\|_{L^\infty H^3}\|\partial_Y\Omega^3_0\|_{L^2 H^3}+\|D_0\|_{L^2H^3}\|\partial_Y(\partial_Y V^3_0,\partial_ZV^2_0)\|_{L^2 H^3}\|\Omega^3_0\|_{L^\infty H^3}\\
				&+\|\partial_ZV^2_0\|_{L^\infty H^3}\|\Omega^3_0\|^2_{L^2H^3}+
				\|\partial_Y(V^2_0,V^3_0)\|_{L^2H^3}\|(\partial_Z\nabla V^2_0,\Delta V^3_0)\|_{L^2H^3}\|\Omega^3_0\|_{L^\infty H^3}\\
				&+(\|p^{\frac{1}{2}}{\de}_0\|_{L^\infty H^3}\|(\nu\Delta V^3_0+(\nu+\nu')\partial_ZD_0)\|_{L^2H^3}+\|{\de}_0\|_{L^\infty H^3}\|p^{\frac{1}{2}}(\nu\Delta V^3_0+(\nu+\nu')\partial_ZD_0)\|_{L^2H^3}\\
				&+\|\nabla {\de}\|_{L^\infty H^3}\|(\nu\partial_Y\Delta V^2_0+\nu\partial_Z\Delta V^3_0+(\nu+\nu')p^{\frac{1}{2}} D_0)\|_{L^2H^3}+\|{\de}_0\|_{L^\infty H^3}\bar{\nu}\|pD\|_{L^2H^3})\|\nabla \Omega^3_0\|_{L^2H^3}\\
				\lesssim&\nu^{-1}\varepsilon^3.
			\end{align*}
			Through a process similar to $\Omega^3_{\neq}$, we obtain
			\begin{align*}
				K^{(2)}_2\lesssim\nu^{-3/2}\varepsilon^3.
			\end{align*}
			Combining estimates on $K^{(2)}_1$ and $K^{(2)}_2$, Lemma \ref{Lem-E71} follows.
		\end{proof}
		
		\subsection{$H^3$ estimate on  $W^1_{\neq}$}
		In this subsection, we improve \cref{bs-w1} for $j=0$. The energy estimates on $W^1_{\neq}$ are generally much simpler than those on $\Omega^3_{\neq}$ as the additional growth of $\nu^{-1/3}$ arising by lift-up effect.
		\begin{Lem}\label{Lem-E8}
			Under the same assumptions as Theorem \ref{MT} and  Proposition \ref{prop-bs}, it holds that
			\begin{align}
				\sup_{0\leq t\leq T}&\|mM W^1_{\neq}(t)\|^2_{H^3}+\nu\|\tilde{\nabla}mM W^1_{\neq}\|^2_{L^2H^3}\nonumber\\
				&+\left\|\sqrt{-\frac{\partial_t M}{M}}mM W^1_{\neq}\right\|^2_{L^2H^3}+\left\|\sqrt{-\frac{\partial_t m}{m}}mM W^1_{\neq}\right\|^2_{L^2H^3}\lesssim\|mM W^1_{\neq}(0)\|^2_{H^3}+\nu^{-\frac{3}{2}}\varepsilon(\nu^{-1/3}\eps)^2.
			\end{align}
		\end{Lem}
		\begin{proof}
			
			The energy estimate gives
			\begin{align*}
				&\|mM W^1_{\neq}(T)\|^2_{H^3}+\nu\|\tilde{\nabla}mM W^1_{\neq}\|^2_{L^2H^3}+\left\|\sqrt{-\frac{\partial_t M}{M}}mM W^1_{\neq}\right\|^2_{L^2H^3}+\left\|\sqrt{-\frac{\partial_t m}{m}}mM W^1_{\neq}\right\|^2_{L^2H^3}\\
				=&\|W^1_{\neq}(0)\|^2_{H^3}+LU^{\nwi}+LS^{\nwi}+LE^{\nwi}+\sum_{i=1}^{6}I^{\nwi}_i,
			\end{align*}
			where the linear and nonlinear terms are given by
			\begin{align*}
				LU^{\nwi}=&\int_{0}^{T}\big\langle mM(-W^2+2\partial_X{\de}-2\nu\partial_XD-2(\partial_{XX}+\tilde{\p}_{Y}^{2} )p^{-1}\partial_X{\de}\\
				&+2\partial_{XX}p^{-1}W^2+2\nu\partial_{XX}p^{-1}\partial_XD)_{\neq},mMW^1_{\neq}\big\rangle_{H^3} dt,\\
				LS^{\nwi}=&\int_{0}^{T}\left\langle \frac{\partial_tp}{p}mMW^1_{\neq},mMW^1_{\neq}\right\rangle_{H^3} dt,\\
				LE^{\nwi}=&-\int_{0}^{T}\left\langle mM\left(2\nu\tilde{\partial}_{XY}\tilde{\Delta}^{-1}W^2-2\nu\tilde{\partial}_{XY}\tilde{\Delta}^{-1}\partial_X{\de}+2\nu^2\tilde{\partial}_{XY}\tilde{\Delta}^{-1}\partial_XD-\nu(\nu+\nu')\tilde{\Delta}\tilde{\partial}_YD\right)_{\neq},mM W^1_{\neq}\right\rangle_{H^3} dt,\\
				I^{\nwi}_1=&-\int_{0}^{T}\left\langle mM(V\cdot\tilde{\nabla}W^1)_{\neq},mM W^1_{\neq}\right\rangle_{H^3} dt,\\
				I^{\nwi}_2=&\int_{0}^{T}\left\langle mM\left(\partial_XV^2\tilde{\partial}_YD+\partial_XV^3\partial_ZD-Q^2\tilde{\partial}_YV^1-Q^3\partial_ZV^1-\Omega^1\partial_XV^1\right)_{\neq},mM W^1_{\neq}\right\rangle_{H^3} dt,\\
				I^{\nwi}_3=&\int_{0}^{T}\left\langle mM\left(2\sum_{i=2,3}(\partial_iV^1\partial_{XX}V^i-\tilde{\partial}_iV^j\tilde{\partial}_{ij}V^1)\right)_{\neq},mM W^1_{\neq}\right\rangle_{H^3} dt,\\
				I^{\nwi}_4=&\int_{0}^{T}\left\langle mM\left(\sum_{i,j=2,3}\partial_X(\tilde{\partial}_iV^j\tilde{\partial}_jV^i)\right)_{\neq},mM W^1_{\neq}\right\rangle_{H^3} dt,\\
				I^{\nwi}_5=&\int_{0}^{T}\left\langle -mM\left((\tilde{\partial}_YV)\cdot\tilde{\nabla}{\de}-\tilde{\partial}_Y({\de}D)+\nu(\tilde{\partial}_YV)\cdot\tilde{\nabla}D+\nu\tilde{\partial}_Y(\tilde{\partial}_iV^j\tilde{\partial}_jV^i)\right)_{\neq},mM W^1_{\neq}\right\rangle_{H^3} dt,\\
				I^{\nwi}_6=&\int_{0}^{T}\big\langle mM\big(\partial_X\tilde{\textrm{div}}\left(G({\de})(\nu\tilde{\Delta} V+(\nu+\nu')\tilde{\nabla} D)\right)-\tilde{\Delta}\left(G({\de})(\nu\tilde{\Delta} V^1+(\nu+\nu')\partial_X D)\right)\\
				&-\nu\tilde{\partial}_Y\tilde{\textrm{div}} \left(F({\de})\tilde{\nabla} {\de}+G({\de})(\nu\tilde{\Delta} V+(\nu+\nu')\tilde{\nabla} D)\right)\big),mM W^1_{\neq}\big\rangle_{H^3} dt.
			\end{align*}
			For $LU^{\nwi}$, $LS^{\nwi}$, and $LE^{\nwi}$, a combination of properties of the multipliers $m,M$ and bootstrap assumptions implies that
			\begin{align*}
				LU^{\nwi}\lesssim&\|m^{\frac{1}{4}}(W^2_{\neq},\partial_X{\de}_{\neq})\|_{L^2H^3}\|mMW^1_{\neq}\|_{L^2H^3}+\nu\|m^{\frac{1}{4}}\partial_Xp^{-\frac{1}{2}}D_{\neq}\|_{L^2H^3}\|p^{\frac{1}{2}}mMW^1_{\neq}\|_{L^2H^3}\\
				\lesssim&\nu^{-1/6}\varepsilon\nu^{-1/2}C_0\varepsilon+\nu\nu^{-1/6}\varepsilon\nu^{-5/6}C_0\varepsilon\lesssim\frac{1}{C_0}(\nu^{-1/3}C_0\varepsilon)^2,\\
				LS^{\nwi}\leq&\left\|\sqrt{-\frac{\partial_t m}{m}}mM W^1_{\neq}\right\|^2_{L^2H^3}+c_1\nu^{1/3}\|mMW^1_{\neq}\|^2_{L^2H^3},
			\end{align*}
			and 
			\begin{align*}
				LE^{\nwi}\lesssim&\nu\|m^{\frac{1}{4}}\partial_Xp^{-\frac{1}{2}}(W^2,\partial_X{\de})\|_{L^2H^3}\|mMW^1_{\neq}\|_{L^2H^3}+\nu^2\|m^{\frac{1}{4}}\partial_Xp^{-\frac{1}{2}}D\|_{L^2H^3}\|\partial_XmMW^1_{\neq}\|_{L^2H^3}\\
				&+\nu(\nu+\nu')\|pD_{\neq}\|_{L^2H^3}\|\tilde{\partial}_YmMW^1_{\neq}\|_{L^2H^3}\\
				\lesssim&(\nu^{\frac{1}{2}}C_0+\frac{1}{C_0})(\nu^{-1/3}C_0\varepsilon)^2.
			\end{align*}
			For $I^{\nwi}_1$, arguing as in the estimate of \cref{est:I1}, we have
			\begin{align*}
				I^{\nwi}_1\lesssim&\|\nabla V^1_0\|_{L^\infty H^4}\|mMW^1_{\neq}\|_{L^2H^3}+\|(V^2_0,V^3_0)\|_{L^\infty H^3}\|\tilde{\nabla} W^1_{\neq}\|_{L^2H^3}\|W^1_{\neq}\|_{L^2H^3}\\
				&+{(\|V_{\neq}\|_{L^2H^3}\|\tilde{\nabla} W^1_{\neq}\|_{L^2H^3}+\|V_{\neq}\|_{L^2H^3}\|\nabla(\Delta V^1_0-\partial_Y{\de}_0+\nu\partial_YD_0)\|_{L^2H^3})\|mMW^1_{\neq}\|_{L^\infty H^3}}\\
				\lesssim&\nu^{-1}\varepsilon(\nu^{-1/2}\varepsilon)^2+\nu^{-1/3}\varepsilon\nu^{-3/2}\varepsilon\nu^{-1/3}\varepsilon\lesssim\nu^{-3/2}\varepsilon(\nu^{-1/3}\varepsilon)^2.
			\end{align*}
			Turning to $I^{\nwi}_2$, we divide it into
			\begin{align*}
				I^{\nwi}_2=&-\int_{0}^{T}\left\langle mM(\tilde{\partial}_YD_{\neq}\partial_YV^1_0),mMW^1_{\neq}\right\rangle_{H^3} dt\\
				&-\int_{0}^{T}\left\langle mM((W^2-\partial_X{\de}+\nu\partial_XD)_{\neq}\partial_YV^1_0),mMW^1_{\neq}\right\rangle _{H^3}dt\\
				&-\int_{0}^{T}\left\langle mM\left(Q^2\tilde{\partial}_YV^1_{\neq}+Q^3\partial_ZV^1\right)_{\neq},mM W^1_{\neq}\right\rangle_{H^3} dt\\
				&+\int_{0}^{T}\left\langle mM\left(\partial_XV^2\tilde{\partial}_YD+\partial_XV^3\partial_ZD-\Omega^1\partial_XV^1\right)_{\neq},mM W^1_{\neq}\right\rangle_{H^3} dt=I^{\nwi}_{21}+I^{\nwi}_{22}+I^{\nwi}_{23}+I^{\nwi}_{24}.
			\end{align*}
			It is worth noting that although $\tilde{\partial}_Y D_{\neq}$ has an additional $\nu^{-1/3}$ growth relative to ${\partial}_X D_{\neq}$, $mW^1_{\neq}$ also has an additional $\nu^{-1/3}$ growth relative to $m^{\frac{1}{4}}W^2_{\neq}$. Therefore, the estimate of $I^{\nwi}_{21}$ is similar to that of \cref{est:I21}. In fact, invoking the equation
			\begin{align*}
				&\partial_t (\tilde{\partial}_Y{\de})+\partial_X{\de}+\tilde{\partial}_YD=-\tilde{\partial}_Y(V\cdot \tilde{\nabla} {\de})-\tilde{\partial}_Y({\de}D)
			\end{align*}
			in $I^{\nwi}_{21}$ and using integration by parts in time, we deduce
			\begin{align*}
				-I^{\nwi}_{21}=&\int_{0}^{T}\langle mM(\partial_YV^{1}_{0}\tilde{\partial}_{Y}D_{\neq}), mM W^1_{\neq}\rangle_{H^3} dt\\
				=&\int_{0}^{T}\langle mM(\partial_YV^{1}_{0}(-\partial_t (\tilde{\partial}_Y{\de})-\partial_X{\de}-\tilde{\partial}_Y(V\cdot \tilde{\nabla} {\de})-\tilde{\partial}_Y({\de}D))_{\neq}), mM W^1_{\neq}\rangle_{H^3}dt\\
				=&\left\langle mM(\partial_YV^{1}_{0}(t)(- \tilde{\partial}_Y{\de}(t))_{\neq}), mM W^1_{\neq}(t)\right\rangle_{H^3}\Big|_{t=0}^{t=T}\\
				&+\int_{0}^{T}\langle mM(\partial_t(\partial_YV^{1}_{0})\tilde{\partial}_Y{\de}_{\neq}), mM W^1_{\neq}\rangle_{H^3} dt+\int_{0}^{T}\langle mM((\partial_YV^{1}_{0})\tilde{\partial}_Y{\de}_{\neq}), mM \partial_t W^1_{\neq}\rangle_{H^3} dt\\
				&+2\int_{0}^{T}\langle \partial_t(mM)((\partial_YV^{1}_{0})\tilde{\partial}_Y{\de})_{\neq}, mM W^1_{\neq}\rangle_{H^3} dt-\int_{0}^{T}\langle mM((\partial_YV^{1}_{0})\partial_X{\de}_{\neq}), mM W^1_{\neq}\rangle_{H^3} dt\\
				&-\int_{0}^{T}\langle mM(\partial_YV^{1}_{0}(\tilde{\partial}_Y(V\cdot \tilde{\nabla} {\de})+\tilde{\partial}_Y({\de}D))_{\neq}), mM 	W^1_{\neq}\rangle_{H^3} dt	=\sum_{k=1}^{6}I^{\nwi}_{21k}.
			\end{align*}
			The bootstrap assumption yields that
			\begin{align*}
				I^{\nwi}_{211}=&\langle mM(\partial_YV^{1}_{0}(t)(- \tilde{\partial}_Y{\de}(t))_{\neq}), mMW^1_{\neq}(t)\rangle_{H^3}\Big|_{t=0}^{t=T}\\
				\lesssim&\|\partial_YV^1_0\|_{L^\infty H^3}\|\tilde{\partial}_Y{\de}_{\neq}\|_{L^\infty H^3}\|mM W^1_{\neq}\|_{L^\infty H^3}
				\lesssim\nu^{-1}\varepsilon\nu^{-1/2}\varepsilon\nu^{-1/3}\varepsilon\lesssim\nu^{-7/6}\varepsilon(\nu^{-1/3}\varepsilon)^2.
			\end{align*}
			For $I^{\nwi}_{212}$, which is not the main problem, we give a sketch
			\begin{align*}
				I^{\nwi}_{212}=&\int_{0}^{T}\langle mM(\partial_t(\partial_YV^{1}_{0})\tilde{\partial}_Y{\de}_{\neq}), mM W^1_{\neq}\rangle_{H^3} dt\\
				=&\int_{0}^{T}\langle mM((-\partial_YV^2_0+\nu\Delta\partial_YV^1_0+(\nlt_{\partial_YV^1_0})_0)\tilde{\partial}_Y{\de}_{\neq}), mMW^1_{\neq}\rangle_{H^3} dt\\
				\lesssim&\nu^{-1/2}\varepsilon\nu^{-2/3}\varepsilon\nu^{-1/3}\varepsilon+\nu^{-2}\varepsilon^2(\nu^{-1/3}\varepsilon)^2\lesssim(\nu^{-1/3}\varepsilon)^2,
			\end{align*}
			where we used that
			\begin{align*}
				&\int_{0}^{T}\langle mM((\nlt_{\partial_yV^1_0})_0\tilde{\partial}_Y{\de}_{\neq}), mM W^1_{\neq}\rangle_{H^3}dt\\
				\lesssim&\|\partial_Y(V\cdot \tilde{\nabla}V^1)_0+\partial_Y(F({\de})\partial_X{\de})_0\|_{L^\infty H^3}\|\tilde{\partial}_Y{\de}_{\neq}\|_{L^2H^3}\|mM W^1_{\neq}\|_{L^2H^3}\\
				&+\|\partial_Y(G({\de})(\nu\tilde{\Delta} V^1+(\nu+\nu')\partial_XD))_0\|_{L^2 H^3}\|\tilde{\partial}_Y{\de}_{\neq}\|_{L^\infty H^3}\|mM W^1_{\neq}\|_{L^2H^3}\\
				\lesssim&\nu^{-2}\varepsilon^2(\nu^{-1/3}\varepsilon)^2.
			\end{align*}
			For $I^{\nwi}_{213}$, we just address newly emerging terms 
			\begin{align*}
				I^{\nwi}_{213}=-&\int_{0}^{T}\langle mM((\partial_YV^{1}_{0})\tilde{\partial}_Y{\de}_{\neq}), mM\partial_t(W^1_{\neq})\rangle_{H^3} dt\\
				=&\int_{0}^{T}\langle mM((\partial_YV^{1}_{0})\tilde{\partial}_Y{\de})_{\neq}, mM(\mathcal{L}_{W^1}+\nu pW^1)\rangle_{H^3}dt\\
				&+\underbrace{\int_{0}^{T}\langle mM((\partial_YV^{1}_{0})\tilde{\partial}_Y{\de}_{\neq}), (mM(V^1_0\partial_XW^1_{\neq}))\rangle_{H^3}dt}_{=:I_{2131}}\\
				&+\underbrace{\int_{0}^{T}\langle mM((\partial_YV^{1}_{0})\tilde{\partial}_Y{\de}_{\neq}), (mM(\partial_YV^1_0\tilde{\partial}_YD_{\neq}))\rangle_{H^3}dt}_{=:I_{2132}}\\
				&+\int_{0}^{T}\langle mM((\partial_YV^{1}_{0})\tilde{\partial}_Y{\de}_{\neq}), (mM((\nlt_{W^1})_{\neq}-V^1_0\partial_XW^1_{\neq}-\partial_YV^1_0\tilde{\partial}_YD_{\neq}))\rangle_{H^3}dt,
			\end{align*}
			where $\lt_{W^1},\nlt_{W^1}$ are given in \cref{bs-w1}.
			For $I^{\nwi}_{2131}$, it follows from Theorem \ref{2026-3-21-2} that
			\begin{align*}
				I^{\nwi}_{2131}\lesssim&\|\partial_YV^1_0\|_{L^\infty H^3}\|\tilde{\partial}_Y{\de}_{\neq}\|_{L^2H^3}(\|V^1_0\|_{L^\infty L^\infty}+\|\nabla V^1_0\|_{L^\infty H^3})\|\tilde{\nabla} m W^1_{\neq}\|_{L^2H^3}\\
				\lesssim&\nu^{-1}\varepsilon\nu^{-2/3}\varepsilon\nu^{-1}\varepsilon\nu^{-5/6}\varepsilon=(\nu^{-17/6}\varepsilon^2)(\nu^{-1/3}\varepsilon)^2
			\end{align*}
			For $I^{\nwi}_{2132}$, by Proposition \ref{est:main}, we have
			\begin{align*}
				I^{\nwi}_{2132}\lesssim&{\|\partial_YV^1_0\|_{L^\infty H^3}\|\tilde{\partial}_Y{\de}_{\neq}\|_{L^2H^3}\|\partial_Y V^1_0\|_{L^\infty H^3}\|\tilde{\partial}_Y D_{\neq}\|_{L^2H^3}}\\
				\lesssim&\nu^{-1}\varepsilon\nu^{-2/3}\varepsilon\nu^{-1}\varepsilon\nu^{-1}\varepsilon\lesssim\nu^{-3}\varepsilon^2(\nu^{-1/3}\varepsilon)^2.
			\end{align*}
		For the same reason of \cref{est:I21}, we omit the estimates of remaining terms. Recalling the definition of $m,M$, we have
		\begin{align*}
			I^{\nwi}_{214}+I^{\nwi}_{215}\lesssim&\|\partial_YV^1_0\|_{L^\infty H^3}\|p^{\frac{1}{2}}{\de}_{\neq}\|_{L^2 H^3}\|mM W^1_{\neq}\|_{L^2H^3}\\
			\lesssim&\nu^{-1}\nu^{-2/3}\varepsilon\nu^{-1/3}\varepsilon=\nu^{-4/3}\varepsilon(\nu^{-1/3}\varepsilon)^2.
		\end{align*}
		For $I^{\nwi}_{216}$, we also use Theorem \ref{2026-3-21-2} to get
		\begin{align*}
			I^{\nwi}_{216}=&\int_{0}^{T}\langle mM(\partial_YV^{1}_{0}\tilde{\partial}_Y(V^1_0\partial_{X}{\de}_{\neq})), mM W^1_{\neq}\rangle_{H^3}dt+...\\
			=&\int_{0}^{T}\langle mM((\partial^2_{Y}V^1_0)(V^1_0\partial_{X}{\de}_{\neq})),mM W^1_{\neq}\rangle_{H^3}dt+\int_{0}^{T}\langle mM(\partial_YV^{1}_{0}(V^1_0\partial_{X}{\de}_{\neq})), \tilde{\partial}_YmMW^1_{\neq}\rangle_{H^3}dt+...\\
			\lesssim&(\nu^{-1}\varepsilon)^2(\nu^{-1/3}\varepsilon\nu^{-5/6}\varepsilon)\lesssim\nu^{-5/2}\varepsilon^2(\nu^{-1/3}\varepsilon)^2.
		\end{align*}
		Turning to $I^{\nwi}_{22}$, $I^{\nwi}_{23}$, and $I^{\nwi}_{24}$, it follows from Proposition \ref{est:main} that
		\begin{align*}
			\sum_{k=2}^{4}I^{\nwi}_{2k}\lesssim&\|m^{\frac{1}{4}}(W^2_{\neq},\partial_X{\de}_{\neq},\nu\partial_XD_{\neq})\|_{L^2H^3}\|\partial_YV^1_0\|_{L^\infty H^4}\|mMW^1_{\neq}\|_{L^2H^3}\\
			&+\|Q^2\|_{L^2H^3}\|\partial_YV^1_{\neq}\|_{L^2H^3}\|mMW^1_{\neq}\|_{L^\infty H^3}+\|\tilde{\Delta}V^3\|_{L^2H^3}\|\partial_ZV^1_{\neq}\|_{L^2H^3}\|mMW^1_{\neq}\|_{L^\infty H^3}\\
			&+\|m^{\frac{1}{2}}\Omega^3_{\neq}\|_{L^2H^3}\|\partial_ZV^1_0\|_{L^\infty H^4}\|mMW^1_{\neq}||_{L^2H^3}+\|m^{\frac{1}{4}}\partial_ZD_{\neq}\|_{L^2H^3}\|\partial_ZV^1_0\|_{L^\infty H^4}\|mMW^1_{\neq}||_{L^2H^3}\\
			&+\|\partial_X(V^2,V^3)\|_{L^2H^3}\|\tilde{\nabla} D\|_{L^2H^3}\|mMW^1_{\neq}\|_{L^\infty H^3}+\|\Omega^1\|_{L^\infty H^3}\|\partial_XV^1_{\neq}\|_{L^2H^3}\|mMW^1_{\neq}\|_{L^2H^3}\\
			\lesssim&\nu^{-1/6}\varepsilon\nu^{-1}\varepsilon\nu^{-1/2}\varepsilon+\nu^{-1}\varepsilon\nu^{-1/2}\varepsilon\nu^{-1/3}\varepsilon+\nu^{-5/6}\varepsilon\nu^{-1/6}\varepsilon\nu^{-1/3}\varepsilon+\nu^{-1/2}\varepsilon\nu^{-1}\varepsilon\nu^{-1/2}\varepsilon\\
			&+\nu^{-1/3}\varepsilon\nu^{-1}\varepsilon\nu^{-1/3}\varepsilon+\nu^{-1}\varepsilon\nu^{-1/6}\varepsilon\nu^{-1/2}\varepsilon\lesssim\nu^{-4/3}\varepsilon(\nu^{-1/3}\varepsilon)^2.
		\end{align*}
		For $I^{\nwi}_3$, we have
		\begin{align*}
			I^{\nwi}_3\lesssim&\|\nabla V^1\|_{L^\infty H^3}\|\partial_{XX}(V^2,V^3)\|_{L^2H^3}\|mMW^1_{\neq}\|_{L^2H^3}\\
			&+\|\tilde{\nabla} V^1\|_{L^\infty H^3}\|\partial_X\tilde{\nabla} V^1_{\neq}\|_{L^2H^3}\|mMW^1_{\neq}\|_{L^2H^3}\\
			&+\|\tilde{\nabla} V^2\|_{L^\infty H^3}\|\tilde{\partial}_Y\tilde{\nabla} V^1_{\neq}\|_{L^2H^3}\|mMW^1_{\neq}\|_{L^2H^3}\\
			&+{\|\tilde{\nabla} V^2_{\neq}\|_{L^2 H^3}\|\partial_Y\nabla V^1_{0}\|_{L^\infty H^3}\|mMW^1_{\neq}\|_{L^2H^3}}\\
			&+\|\tilde{\nabla} V^3\|_{L^\infty H^3}\|\partial_Z\tilde{\nabla} V^1_{\neq}\|_{L^2H^3}\|mMW^1_{\neq}\|_{L^2H^3}\\
			&+\|\tilde{\nabla} V^3_{\neq}\|_{L^2 H^3}\|\partial_Z\nabla V^1_{0}\|_{L^\infty H^3}\|mMW^1_{\neq}\|_{L^2H^3}\\
			\lesssim&\nu^{-1}\varepsilon(\nu^{-1/2}\varepsilon)^2+\nu^{-1/2}\varepsilon\nu^{-5/6}\varepsilon\nu^{-1/2}\varepsilon+\nu^{-2/3}\varepsilon\nu^{-1}\varepsilon\nu^{-1/2}\varepsilon+\nu^{-1/3}\varepsilon(\nu^{-1/2}\varepsilon)^2\\
			&+\nu^{-1/2}\varepsilon\nu^{-1}\varepsilon\nu^{-1/2}\varepsilon\lesssim\nu^{-3/2}\varepsilon(\nu^{-1/3}\varepsilon)^2.
		\end{align*}
		The $I^{\nwi}_4$, $I^{\nwi}_5$, and $I^{\nwi}_6$ terms can be treated analogously to that of $I^{\nwi}_3$. By bootstrap assumption, we get
		\begin{align*}
			I^{\nwi}_4\lesssim&\|\tilde{\partial}_{XY}V^2_{\neq}\|_{L^2H^3}\|\tilde{\partial}_YV^2\|_{L^2H^3}\|mMW^1_{\neq}\|_{L^\infty H^3}+\|\partial_{XZ}V^2_{\neq}\|_{L^2H^3}\|\tilde{\partial}_YV^3\|_{L^2H^3}\|mMW^1_{\neq}\|_{L^\infty H^3}\\
			&+\|\partial_{Z}V^2\|_{L^\infty H^3}\|\tilde{\partial}_{XY}V^3_{\neq}\|_{L^2H^3}\|mMW^1_{\neq}\|_{L^2 H^3}+\|\partial_{XZ}V^3_{\neq}\|_{L^2H^3}\|\partial_ZV^3\|_{L^\infty H^3}\|mMW^1_{\neq}\|_{L^2 H^3}\\
			\lesssim&(\nu^{-2/3}\varepsilon)^2\nu^{-1/3}\varepsilon+(\nu^{-1/2}\varepsilon)^2\nu^{-1/3}\varepsilon+\nu^{-1/6}\varepsilon(\nu^{-1/2}\varepsilon)^2+\nu^{-1/2}\varepsilon^2\nu^{-1/2}\varepsilon\lesssim\nu^{-2/3}\varepsilon(\nu^{-1/3}\varepsilon)^2,
		\end{align*}
		and
		\begin{align*}
			I^{\nwi}_5\lesssim&\|\tilde{\partial}_YV_{\neq}\|_{L^2H^3}\|\tilde{\nabla} {\de}\|_{L^\infty H^3}\|mM W^1_{\neq}\|_{L^2H^3}+\|\partial_YV^i_{0}\|_{L^\infty H^3}\|\tilde{\partial}_i {\de}_{\neq}\|_{L^2 H^3}\|mM W^1_{\neq}\|_{L^2H^3}\\
			&+\|{\de}\|_{L^\infty H^3}\|D\|_{L^2H^3}\|\tilde{\partial}_YmM W^1_{\neq}\|_{L^2H^3}+\nu\|\tilde{\partial}_YV\|_{L^\infty H^3}\|\tilde{\nabla} D\|_{L^2 H^3}\|mM W^1_{\neq}\|_{L^2 H^3}\\
			&+\nu\|\partial_iV^j\partial_jV^i\|_{L^2H^3}\|\tilde{\partial}_Y\ma W^2_{\neq}\|_{L^2H^3}\\
			\lesssim&\nu^{-2/3}\varepsilon\nu^{-1/2}\varepsilon\nu^{-1/2}\varepsilon+\nu^{-1}\varepsilon\nu^{-1/3}\varepsilon\nu^{-1/2}\varepsilon+\nu^{-1/6}\varepsilon\nu^{-2/3}\varepsilon\nu^{-2/3}\varepsilon+\nu\nu^{-1}\varepsilon\nu^{-1}\varepsilon\nu^{-1/2}\varepsilon\\
			&+\nu\nu^{-4/3}\varepsilon^2\nu^{-5/6}\varepsilon\lesssim\nu^{-7/6}\varepsilon(\nu^{-1/3}\varepsilon)^2,
		\end{align*}
		and
		\begin{align*}
			I^{\nwi}_6\lesssim&(\|p^{\frac{1}{2}}{\de}\|_{L^\infty H^3}(\nu\|\tilde{\Delta} V^1\|_{L^2H^3}+(\nu+\nu')\|\partial_XD\|_{L^2H^3})\\
			&+\|{\de}\|_{L^\infty H^3}(\nu\|p^{\frac{1}{2}}\tilde{\Delta} V^1\|_{L^2H^3}+(\nu+\nu')\|p^{\frac{1}{2}}\partial_XD\|_{L^2H^3})\\
			&+\|\partial_X{\de}\|_{L^2H^3}(\nu\|\tilde{\Delta}V^1\|_{L^\infty H^3}+(\nu+\nu')\|\partial_XD\|_{L^\infty H^3})\\
			&+\|\tilde{\partial}_Y{\de}\|_{L^\infty H^3}(\nu\|\tilde{\Delta}V^2\|_{L^2 H^3}+(\nu+\nu')\|\tilde{\partial}_YD\|_{L^2 H^3})\\
			&+\|\partial_Z{\de}\|_{L^\infty H^3}(\nu\|\tilde{\Delta}V^3\|_{L^2 H^3}+(\nu+\nu')\|\partial_ZD\|_{L^2 H^3})+\|{\de}\|_{L^\infty H^3}\bar{\nu}\|pD\|_{L^2H^3}\\
			&+\nu\|\tilde{\partial}_Y{\de}\|_{L^2H^3}\|\tilde{\nabla}{\de}\|_{L^\infty H^3}+\nu\|{\de}\|_{L^\infty H^3}\|p{\de}\|_{L^2 H^3})\|\ma p^{\frac{1}{2}}W^1_{\neq}\|_{L^2H^3}\\
			\lesssim&(\nu^{-1/2}\varepsilon\bar{\nu}\nu^{-1}\varepsilon+\nu^{-1/6}\varepsilon\bar{\nu}\nu^{-4/3}\varepsilon+\nu^{-1/3}\varepsilon\nu^{-1}\varepsilon+\nu^{-1/2}\varepsilon\bar{\nu}\nu^{-1}\varepsilon+\nu^{-1/6}\varepsilon\bar{\nu}\nu^{-5/6}\varepsilon+\nu^{-1/6}\varepsilon\bar{\nu}\nu^{-4/3}\varepsilon\\
			&+\nu\nu^{-1}\varepsilon\nu^{-1/2}\varepsilon+\nu\nu^{-1/6}\varepsilon\nu^{-4/3}\varepsilon)\nu^{-5/6}\varepsilon
			\lesssim\nu^{-1}\varepsilon(\nu^{-1/3}\varepsilon)^2\leq \nu^{-\frac{3}{2}}\varepsilon^3.
		\end{align*}
		
		From the estimates for $LU^{(8)}$, $LS^{(8)}$, $LE^{(8)}$, and $I_{i}^{(8)}$ ($i=1,\ldots,6$), we obtain Lemma \ref{Lem-E8}.
	\end{proof}
	
	\section{Energy estimates on zero mode}
	\subsection{$H^4$ estimate on  $U^1_{00}$}
	In this subsection, we improve \cref{bs-u100}. Due to the absence of $L^2_t$-estimates for $V^2_{00},U^1_{00}$ in low order, we need to treat $\|U^1_{00}\|_{L^2}$ and $\|\nabla^sU^1_{00}\|_{L^2}$, respectively.
	\subsubsection{$L^2$ estimate on $U^1_{00}$}
	\begin{Lem}\label{Lem-E9}
		Under the same assumptions as Theorem \ref{MT} and  Proposition \ref{prop-bs}, it holds that
		\begin{align}
			\sup_{0\leq t\leq T}\frac{1}{2}\|U^1_{00}(t)\|^2_{L^2}+\nu\|\partial_YU^1_{00}\|^2_{L^2L^2}\lesssim\frac{1}{2}\|U^1_{00}(0)\|^2_{L^2}+\nu^{-\frac{3}{2}}\varepsilon(\nu^{-1/2}\eps)^2.
		\end{align}
	\end{Lem}
	\begin{proof}

		An energy estimate gives
		\begin{align*}
			\frac{1}{2}\|U^1_{00}(T)\|^2_{L^2}+\nu\|\partial_YU^1_{00}\|^2_{L^2L^2}=\frac{1}{2}\|U^1_{00}(0)\|^2_{L^2}+K^{(3)}_1+K^{(3)}_2+K^{(3)}_3,
		\end{align*}
		where the nonlinear terms are given by
		\begin{align*}
			K^{(3)}_1=&-\int_{0}^{T}\left\langle \partial_Y(V^2_{00}U^1_{00}),U^1_{00}\right\rangle_{L^2} dt,\\
			K^{(3)}_2=&-\int_{0}^{T}\left\langle \partial_Y(V^2_{0\neq}\partial_YV^1_{0\neq}+V^3_{0\neq}\partial_ZV^1_{0\neq}-{\de}_{0\neq}V^2_{0\neq}+\nu V_{0\neq}\cdot\nabla V^2_{0\neq}),U^1_{00}\right\rangle_{L^2} dt,\\
			K^{(3)}_3=&\int_{0}^{T}\left\langle \nlt_{U^1_{00}},U^1_{00}\right\rangle_{L^2}dt,
		\end{align*}
		with
		\begin{align*}
			\nlt_{U^1_{00}}=&-\tilde{\partial}_Y(V_{\neq}\cdot\tilde{\nabla}V^1_{\neq})+\tilde{\partial}_Y(V^2_{\neq}{\de}_{\neq})-\nu\tilde{\partial}_Y(V_{\neq}\cdot\tilde{\nabla}V^2_{\neq})\\
			&-\tilde{\partial}_Y(F({\de})\partial_X{\de}-G({\de})(\nu\tilde{\Delta}V^1+(\nu+\nu')\partial_XD))_{00}\\
			&-\nu\tilde{\partial}_Y(F({\de})\tilde{\partial}_Y{\de}-G({\de})(\nu\tilde{\Delta}V^2+(\nu+\nu')\tilde{\partial}_YD))_{00}.
		\end{align*}
		Recalling the equation of ${\de}_{00}$,
		\begin{equation*}
			\partial_t {\de}_{00}+\partial_YV^2_{00}=-(V^2_{00}\partial_Y{\de}_{00})-({\de}_{00}D_{00})-(V_{0\neq}\cdot\nabla {\de}_{0\neq})-({\de}_{0\neq}D_{0\neq})-(V_{\neq}\cdot\tilde{\nabla}{\de}_{\neq})-({\de}_{\neq}D_{\neq}),
		\end{equation*}
		a direct calculation for $I_1$ finds that
		\begin{align*}
			2K^{(3)}_{1}=&-\int_{0}^{T}\langle\partial_YV^2_{00},|U^1_{00}|^2\rangle_{L^2} dt\\
			=&\int_{0}^{T}\langle \partial_t {\de}_{00},|U^1_{00}|^2\rangle_{L^2} dt+\int_{0}^{T}\langle (V^2_{00}\partial_Y{\de}_{00})+({\de}_{00}D_{00}),|U^1_{00}|^2\rangle_{L^2} dt\\
			&+\int_{0}^{T}\langle (V_{0\neq}\cdot\nabla {\de}_{0\neq})+({\de}_{0\neq}D_{0\neq}),|U^1_{00}|^2\rangle dt+\int_{0}^{T}\langle (V_{\neq}\cdot\tilde{\nabla}{\de}_{\neq})+({\de}_{\neq}D_{\neq}),|U^1_{00}|^2\rangle_{L^2} dt\\
			=&\sum_{k=1}^{4}K^{(3)}_{1k}.
		\end{align*}
		For $K^{(3)}_{11}$, integration by parts in time implies that
		\begin{align*}
			K^{(3)}_{11}=&-\|{\de}_{00}\|_{L^\infty L^2}\|U^1_{00}\|^2_{L^\infty L^2}-2\int_{0}^{T}\langle  {\de}_{00}U^1_{00},\partial_tU^1_{00}\rangle_{L^2} dt\\
			\lesssim&\varepsilon(\nu^{-1/2}\varepsilon)^2+\nu^{-1/2}\varepsilon(\nu^{-1/2}\varepsilon)^2,
		\end{align*}
		where we used that
		\begin{equation*}
			\int_{0}^{T}\langle  {\de}_{00}U^1_{00},\partial_tU^1_{00}\rangle_{L^2} dt\lesssim\nu^{-3/2}\varepsilon^3+\nu^{-5/2}\varepsilon^4.
		\end{equation*}
		In fact, due to the lift up effect of $V^1_{0\neq}$, it holds that
		\begin{align*}
			\sum_{i=2}^{3}\int_{0}^{T}\langle  {\de}_{00}U^1_{00},\partial_Y(V^i_{0\neq}\partial_iV^1_{0\neq})\rangle_{L^2} dt\lesssim&\|{\de}_{00}\|_{L^\infty H^2}\|U^1_{00}\|_{L^\infty H^2}\|(V^2_{0\neq},V^3_{0\neq})\|_{L^2L^2}\|\nabla V^1_{0\neq}\|_{L^2L^2}\\
			\lesssim&\varepsilon(\nu^{-1/2}\varepsilon)^2(\nu^{-3/2}\varepsilon)\lesssim\nu^{-5/2}\varepsilon^4.
		\end{align*}
		For $K^{(3)}_{12}$, $K^{(3)}_{13}$, and $K^{(3)}_{14}$, by bootstrap assumption, we have
		\begin{align*}
			\sum_{k=2}^{4}K^{(3)}_{1k}\lesssim&(\|(V^2_{00},{\de}_{00})\|_{L^\infty L^2}\|(\partial_Y{\de}_{00},D_{00})\|_{L^2L^2}+\|(V^2,V^3,{\de})_{0\neq}\|_{L^\infty L^2}\|\nabla {\de}_{0\neq},D_{0\neq}\|_{L^2H^3}\\
			&+\|(V_{\neq},{\de}_{\neq})\|_{L^\infty L^2}\|\tilde{\nabla}{\de}_{\neq},D_{\neq}\|_{L^2L^2})\|U^1_{00}\|_{L^\infty L^2}\|\partial_YU^1_{00}\|_{L^2L^2}\\
			\lesssim& (\varepsilon\nu^{-1}\varepsilon+\varepsilon\nu^{-1}\varepsilon+\nu^{-1/6}\varepsilon\nu^{-2/3}\varepsilon)\nu^{-1/2}\varepsilon\nu^{-1}\varepsilon\lesssim\nu^{-3/2}\varepsilon^2(\nu^{-1/2}\varepsilon)^2.
		\end{align*}
		It follows from Proposition \ref{est:main} that
		\begin{align*}
			K^{(3)}_2\lesssim(&{\|(V^2_0,V^3_0)\|_{L^\infty L^2}\|\nabla V^1_{0\neq}\|_{L^2 L^2}}+\|(V^2_0,V^3_0,{\de}_{0})\|_{L^\infty H^3}\|\nabla V^2_{0\neq}\|_{L^2L^2})\|\partial_YU^1_{00}\|_{L^2L^2}\\
			\lesssim&\varepsilon\nu^{-3/2}\varepsilon\nu^{-1}\varepsilon=\nu^{-3/2}\varepsilon(\nu^{-1/2}\varepsilon)^2.
		\end{align*}
		The estimate of $K^{(3)}_3$ is similar to below. Therefore, we omit it here.
	\end{proof}
	\subsubsection{$H^3$ estimate on  $\partial_YU^1_{00}$}
	An energy estimate gives
	\begin{align*}
		\frac{1}{2}\|\partial_YU^1_{00}(T)\|^2_{H^3}+\nu\|\partial^2_{Y}U^1_{00}\|^2_{L^2H^3}=\frac{1}{2}\|\partial_YU^1_{00}(0)\|^2_{H^3}+K^{(4)}_1+K^{(4)}_2+K^{(4)}_3,
	\end{align*}
	where the nonlinear terms are given by
	\begin{align*}
		K^{(4)}_1=&-\int_{0}^{T}\left\langle \partial^2_{Y}(V^2_{00}U^1_{00}),\partial_YU^1_{00}\right\rangle_{H^3} dt,\\
		K^{(4)}_2=&-\int_{0}^{T}\left\langle \partial^2_{Y}(V^2_{0\neq}\partial_YV^1_{0\neq}+V^3_{0\neq}\partial_ZV^1_{0\neq}-{\de}_{0\neq}V^2_{0\neq}+\nu V_{0\neq}\cdot\nabla V^2_{0\neq}),\partial_YU^1_{00}\right\rangle_{H^3} dt,\\
		K^{(4)}_3=&-\int_{0}^{T}\left\langle \tilde{\p}_{Y}^{2} (V_{\neq}\cdot\tilde{\nabla}V^1_{\neq}-V^2_{\neq}{\de}_{\neq}+\nu V_{\neq}\cdot\tilde{\nabla}V^2_{\neq}),\partial_YU^1_{00}\right\rangle_{H^3}dt,\\
		K^{(4)}_4=&-\int_{0}^{T}\left\langle \tilde{\p}_{Y}^{2} (F({\de})\partial_X{\de}-G({\de})(\nu\tilde{\Delta}V^1+(\nu+\nu')\partial_XD))_{00},\partial_YU^1_{00}\right\rangle_{H^3}dt\\
		&-\int_{0}^{T}\left\langle \nu\tilde{\p}_{Y}^{2} (F({\de})\tilde{\partial}_Y{\de}-G({\de})(\nu\tilde{\Delta}V^2+(\nu+\nu')\tilde{\partial}_YD))_{00},\partial_YU^1_{00}\right\rangle_{H^3}dt.
	\end{align*}
	It follows from Proposition \ref{est:main} that
	\begin{align*}
		K^{(4)}_1\lesssim&\|V^2_{00}\|_{L^\infty H^3}\|\partial_YU^1_{00}\|_{L^2H^3}\|\partial^2_{Y}U^1_{00}\|_{L^2H^3}+\|\partial_YV^2_{00}\|_{L^2 H^3}\|U^1_{00}\|_{L^\infty H^3}\|\partial^2_{Y}U^1_{00}\|_{L^2H^3}\\
		\lesssim&\varepsilon(\nu^{-1}\varepsilon)^2\lesssim\nu^{-1}\varepsilon(\nu^{-1/2}\varepsilon)^2,
	\end{align*}
	and
	\begin{align*}
		K^{(4)}_2\lesssim&({\|\partial_Y(V^2_{0\neq},V^3_{0\neq})\|_{L^2H^3}\|\nabla V^1_{0\neq}\|_{L^\infty H^3}+\|(V^2_{0\neq},V^3_{0\neq})\|_{L^2H^3}\|\partial_Y\nabla V^1_{0\neq}\|_{L^\infty H^3}}\\
		&+\|\partial_Y{\de}_{0\neq}\|_{L^\infty H^3}\|V^2_{0\neq}\|_{L^2H^3}+\nu\|\partial_Y(V^2_{0\neq},V^3_{0\neq})\|_{L^\infty H^3}\|\nabla V^2_{0\neq}\|_{L^2H^3}\\
		&+\|{\de}_{0\neq}\|_{L^\infty H^3}\|\partial_YV^2_{0\neq}\|_{L^2H^3}+\nu\|(V^2_{0\neq},V^3_{0\neq})\|_{L^\infty H^3}\|\partial_Y\nabla V^2_{0\neq}\|_{L^2H^3})\|\partial^2_{Y}U^1_{00}\|_{L^2H^3}\\
		\lesssim&\nu^{-1/2}\varepsilon\nu^{-1}\varepsilon\nu^{-1}\varepsilon\lesssim\nu^{-3/2}\varepsilon(\nu^{-1/2}\varepsilon)^2,
	\end{align*}
	and 
	\begin{align*}
		K^{(4)}_3\lesssim&(\|\tilde{\partial}_Y V^i_{\neq}\|_{L^\infty H^3}\|\tilde{\partial}_iV^1_{\neq}\|_{L^2H^3}+\| V_{\neq}\|_{L^\infty H^3}\|\tilde{\partial}_Y\tilde{\nabla}V^1_{\neq}\|_{L^2H^3}\\
		&+\|\tilde{\partial}_Y{\de}_{\neq}\|_{L^\infty H^3}\|V^2_{\neq}\|_{L^2H^3}+\|{\de}_{\neq}\|_{L^\infty H^3}\|\tilde{\partial}_YV^2_{\neq}\|_{L^2H^3}\\
		&+\nu\|\tilde{\partial}_Y V^i_{\neq}\|_{L^\infty H^3}\|\tilde{\partial}_iV^2_{\neq}\|_{L^2H^3}+\nu\| V_{\neq}\|_{L^\infty H^3}\|\tilde{\partial}_Y\tilde{\nabla}V^2_{\neq}\|_{L^2H^3})\|\partial^2_{Y}U^1_{00}\|_{L^2H^3}\\
		\lesssim&(\nu^{-1/2}\varepsilon\nu^{-1/2}\varepsilon+\nu^{-1/6}\varepsilon\nu^{-5/6}\varepsilon)\nu^{-1}\varepsilon\lesssim\nu^{-1}\varepsilon(\nu^{-1/2}\varepsilon)^2.
	\end{align*}
	By bootstrap assumption, we conclude that
	\begin{align*}
		K^{(4)}_4\lesssim&(\|\tilde{\partial}_Y{\de}_{\neq}\|_{L^2H^3}\|\partial_X{\de}_{\neq}\|_{L^\infty}+\|{\de}_{\neq}\|_{L^2H^3}\|\tilde{\partial}_{XY}{\de} _{\neq}\|_{L^\infty}\\
		&+\|\partial_Y{\de}_0\|_{L^2H^3}\nu\|\Delta V^1_0\|_{L^\infty H^3}+\|{\de}_0\|_{L^\infty H^3}\nu\|\partial_Y\Delta V^1_0\|_{L^2 H^3}\\
		&+\|\tilde{\partial}_Y{\de}_{\neq}\|_{L^2H^3}(\nu\|(W^1+\tilde{\partial}_Y{\de}-\nu\tilde{\partial}_YD)_{\neq}\|_{L^\infty H^3}+\bar{\nu}\|\partial_XD_{\neq}\|_{L^\infty H^3})\\
		&+\|{\de}_{\neq}\|_{L^\infty H^3}({\nu\|\tilde{\partial}_Y(W^1+\tilde{\partial}_Y{\de}-\nu\tilde{\partial}_YD)_{\neq}\|_{L^2 H^3}}+\bar{\nu}\|\tilde{\partial}_{XY}D_{\neq}\|_{L^2 H^3})\\
		&+\nu(\|\tilde{\partial}_Y{\de}_{0}\|_{L^2H^3}\|\tilde{\partial}_Y{\de}_{0}\|_{L^\infty H^3}+\|{\de}_{0}\|_{L^2H^3}\|\tilde{\p}_{Y}^{2} {\de} _{0}\|_{L^\infty H^3}\\
		&+\|\tilde{\partial}_Y{\de}_{\neq}\|_{L^2H^3}\|\tilde{\partial}_Y{\de}_{\neq}\|_{L^\infty H^3}+\|{\de}_{\neq}\|_{L^2H^3}\|\tilde{\p}_{Y}^{2} {\de} _{\neq}\|_{L^\infty H^3}\\
		&+\|\partial_Y{\de}_0\|_{L^2H^3}\nu\|\Delta V^2_0\|_{L^\infty H^3}+\|{\de}_0\|_{L^\infty H^3}\nu\|\partial_Y\Delta V^2_0\|_{L^2 H^3}\\
		&+\|\partial_Y{\de}_0\|_{L^2H^3}(\nu+\nu')\|\tilde{\partial}_Y D_0\|_{L^\infty H^3}+\|{\de}_0\|_{L^\infty H^3}(\nu+\nu')\|\partial^2_{Y}D_0\|_{L^2 H^3}\\
		&+\|\tilde{\partial}_Y{\de}_{\neq}\|_{L^2H^3}(\nu\|(W^2-\partial_X{\de}+\nu\partial_XD)_{\neq}\|_{L^\infty H^3}+\bar{\nu}\|\tilde{\partial}_YD_{\neq}\|_{L^\infty H^3})\\
		&+\|{\de}_{\neq}\|_{L^\infty H^3}(\nu\|\tilde{\partial}_Y(W^2-\partial_X{\de}+\nu\partial_XD)_{\neq}\|_{L^2 H^3}+\bar{\nu}\|\tilde{\p}_{Y}^{2} D_{\neq}\|_{L^2 H^3})))\|\partial^2_{Y}U^1_{00}\|_{L^2H^3}\\
		\lesssim&\nu^{-1}\varepsilon(\nu^{-1/2}\varepsilon)^2.
	\end{align*}
	This finishes the improvement of \cref{bs-u100}.     
	
	\subsection{$H^3$ estimate on  $\Delta V^1_{0\neq}$}
	In this subsection, we improve \cref{bs-pv10}. An energy estimate gives
	\begin{align*}
		\frac{1}{2}\|\Delta V^1_{0\neq}(T)\|^2_{H^3}+\nu\|\nabla \Delta V^1_{0\neq}\|^2_{L^2H^3}=\frac{1}{2}\|\Delta V^1_{0\neq}(0)\|^2_{H^3}+LU+K^{(5)}_1+K^{(5)}_2,
	\end{align*}
	where
	\begin{align*}
		LU=&\int_{0}^{T}\langle \nabla V^2_{0\neq},\nabla \Delta V^1_{0\neq}\rangle_{H^3} dt,\\
		K^{(5)}_1=&\int_{0}^{T}\langle \nabla (V^2_0\partial_YV^1_0+V^3_0\partial_ZV^1_0),\nabla \Delta V^1_{0\neq}\rangle_{H^3} dt,\\
		K^{(5)}_2=&\int_{0}^{T}\langle \tilde{\nabla} (V_{\neq}\cdot \tilde{\nabla} V^1_{\neq}+F({\de})\partial_X{\de}-G({\de})(\nu\tilde{\Delta}V^1+(\nu+\nu')\partial_XD))_{0\neq},\nabla \Delta V^1_{0\neq}\rangle_{H^3} dt.
	\end{align*}
	 Proposition \ref{est:main} implies that
	\begin{align*}
		LU\lesssim\|\nabla V^2_{0\neq}\|_{L^2H^3}\|\nabla \Delta V^1_{0\neq}\|_{L^2H^3}\lesssim\frac{1}{C_0}(\nu^{-1}C_0\varepsilon)^2, 
	\end{align*}
	and
	\begin{align*}
		K^{(5)}_1\lesssim&(\|\nabla(V^2_0,V^3_0)\|_{L^2H^3}\|\nabla V^1_0\|_{L^\infty H^3}+\|(V^2_0,V^3_0)\|_{L^\infty H^3}\|p V^1_0\|_{L^2 H^3})\|\nabla \Delta V^1_{0\neq}\|_{L^2H^3}\\
		\lesssim&\nu^{-1}\varepsilon(\nu^{-1}\varepsilon)^2.
	\end{align*}
	The estimate of $K^{(5)}_2$ is similar to above. We omit it and conclude that
	\begin{align*}
		K^{(5)}_2\lesssim\nu^{-1/2}\varepsilon(\nu^{-1}\varepsilon)^2.
	\end{align*}
	This completes the improvement of \cref{bs-pv10}.    
	\subsection{$H^3$ estimate on  $V^3_{00}$}
	Noting the improvement of $\|V^3_{00}\|_{L^2}$ already obtained in \cref{sec411}. We only show the treatment of $\|\partial_YV^3_{00}\|_{H^2}$.
	The energy estimate give
	\begin{align*}
		\frac{1}{2}\|\partial_YV^3_{00}(T)\|^2_{H^2}+\nu\|\partial^2_{Y}V^3_{00}\|_{L^2H^2}=\frac{1}{2}\|\partial_YV^3_{00}(0)\|^2_{H^2}+K^{(6)}_1+K^{(6)}_2,
	\end{align*}
	where
	\begin{align*}
		K^{(6)}_1=&\int_{0}^{T}\langle (V^2_0\partial_YV^3_0+V^3_0\partial_ZV^3_0)_{00},\partial^2_{Y}V^3_{00}\rangle_{H^2} dt,\\
		K^{(6)}_2=&\int_{0}^{T}\langle (V_{\neq}\cdot \tilde{\nabla} V^3_{\neq})_{00}+(F({\de})\partial_Z{\de}-G({\de})(\nu\tilde{\Delta}V^3+(\nu+\nu')\partial_ZD))_{00},\partial^2_{Y}V^3_{00}\rangle_{H^2} dt.
	\end{align*}
	It follows from bootstrap argument that
	\begin{align*}
		K^{(6)}_1\lesssim\|(V^2_0,V^3_0)\|_{L^\infty H^2}\|\nabla V^3_0\|_{L^2H^2}\|\partial^2_{Y}V^3_{00}\|_{L^2H^2}\lesssim\nu^{-1}\varepsilon^3,
	\end{align*}
	and 
	\begin{align*}
		K^{(6)}_2\lesssim&\|V_{\neq}\|_{L^2H^3}\|\tilde{\nabla} V^3_{\neq}\|_{L^2H^3}\|V^3_{00}\|_{L^\infty H^4}+\|{\de}_{0\neq}\|_{L^\infty H^2}\|\partial_Z{\de}_{0\neq}\|_{L^2H^2}\|\partial^2_{Y}V^3_{00}\|_{L^2H^2}\\
		&+\|{\de}_{\neq}\|_{L^2 H^2}\|\partial_Z{\de}_{\neq}\|_{L^2H^3}\|V^3_{00}\|_{L^\infty H^4}\\
		&+\|{\de}_{0}\|_{L^\infty H^2}(\nu\|\Delta V^3_{0}\|_{L^2H^2}+(\nu+\nu')\|\partial_ZD_0\|_{L^2H^2})\|\partial^2_{Y}V^3_{00}\|_{L^2H^2}\\
		&+\|{\de}_{\neq}\|_{L^2 H^3}(\nu\|\tilde{\Delta} V^3_{\neq}\|_{L^2H^3}+(\nu+\nu')\|\partial_ZD_{\neq}\|_{L^2H^3})\|V^3_{00}\|_{L^\infty H^4}\\
		\lesssim&\nu^{-1}\varepsilon^3.
	\end{align*}
	This completes all of zero mode estimates.
	\section{High regularity energy estimates}
	In this section, we prove   Proposition \ref{prop-bs} for $j=1,2,3$.
	\subsection{Estimates on compressible part}
	\subsubsection{$H^{3-j}$ estimates on $\nabla_{X,Z}(p^{\frac{j}{2}}R_{\neq},p^{\frac{j-1}{2}}D_{\neq})$ }
	In this subsection, we improve the energy estimates \cref{bs-xn} for $j=1,2,3$.
	We only consider $\|m^{\frac{1}{4}}Mp^{\frac{1}{2}}\partial_X({\de},p^{-\frac{1}{2}}D)_{\neq}\|_{H^{2}}$. The treatment of the other quantities is similar. 
	The energy functional is given by
	\begin{align*}
		E^{\nxnh}=\frac{1}{2}\|\ma\partial_X(p^{\frac{1}{2}}{\de},D)_{\neq}\|^2_{H^2}-c_1\bar{\nu}^{\frac{1}{3}} Re\langle \ma \partial_X{\de}_{\neq},\ma \partial_XD_{\neq}\rangle_{H^2}.
	\end{align*}
	The energy estimate gives
	\begin{align*}
		&E^{\nxnh}(T)+\bar{\nu}\|\tilde{\nabla}\ma \partial_XD_{\neq}\|^2_{L^2H^2}+\frac{1}{4}\left\|\sqrt{-\frac{\partial_t m}{m}}m^{\frac{1}{4}}M\partial_X(p^{\frac{1}{2}}{\de},D)_{\neq}\right\|^2_{L^2H^2}+\left\|\sqrt{-\frac{\partial_t M}{M}}m^{\frac{1}{4}}M\partial_X(p^{\frac{1}{2}}{\de},D)_{\neq}\right\|^2_{L^2H^2}\\
		&+c_1\bar{\nu}^{\frac{1}{3}}\|\ma p^{\frac{1}{2}} \partial_X{\de}_{\neq}\|^2_{L^2H^2}\\
		=&E^{\nxnh}(0)+LS^{\nxnh}+LE^{\nxnh}+I^{\nxnh}_1+I^{\nxnh}_2+I^{\nxnh}_3,
	\end{align*}
	where 
	\begin{align*}
		LS^{\nxnh}=&\int_{0}^{T}\left\langle\frac{1}{2}\partial_t p \ma \partial_X{\de}_{\neq}, \ma \partial_X{\de}_{\neq}\right\rangle_{H^2} +\left\langle\frac{\partial_t p}{p} \ma \partial_XD_{\neq},\ma \partial_XD_{\neq}\right\rangle_{H^2} dt\\
		LE^{\nxnh}=&2\int_{0}^{T}\left\langle \ma \frac{ \partial_{XX}}{p}(W^2_{\neq}-\partial_X{\de}+\nu\partial_XD)_{\neq}, \ma\partial_X D_{\neq}\right\rangle_{H^2}+c_1\bar{\nu}^{\frac{1}{3}}\| \ma \partial_XD_{\neq}\|^2_{H^2}\\
		&-c_1\bar{\nu}^{\frac{1}{3}}\left\langle \ma \partial_X{\de}_{\neq},\ma \partial_X\left(\frac{\partial_tp}{p}D+ 2 \frac{\partial_{X}}{p}(W^2-\partial_X{\de}+\nu\partial_XD)-\bar{\nu} pD\right)_{\neq}\right\rangle_{H^2} dt\\
		I^{\nxnh}_1=&\int_{0}^{T}\left\langle \ma p^{\frac{1}{2}}\partial_X(\nlt_{{\de}})_{\neq},\ma p^{\frac{1}{2}}\partial_X{\de}_{\neq}\right\rangle_{H^2} dt,\quad
		I^{\nxnh}_2=\int_{0}^{T}\left\langle \ma\partial_X(\nlt_{D})_{\neq},\ma\partial_X D_{\neq}\right\rangle_{H^2} dt,\\
		I^{\nxnh}_3=&-c_1\bar{\nu}^{\frac{1}{3}}\int_{0}^{T}\left\langle \ma\partial_X {\de}_{\neq}, \ma\partial_X(\nlt_{D})_{\neq} \right\rangle_{H^2}+\left\langle \ma\partial_X(\nlt_{{\de}})_{\neq},\ma \partial_XD_{\neq}\right\rangle_{H^2} dt,
	\end{align*}
	with $\nlt_{\de},\nlt_D$ given in \cref{t-ND}.
	For $LS$ and $LE$, by bootstrap argument, we have
	\begin{align*}
		LS^{\nxnh}\lesssim&\|m^{\frac{1}{4}}\partial_Xp^{-\frac{1}{2}}\partial_XD_{\neq}\|_{L^2H^2}\|m^{\frac{1}{4}}\partial_XD_{\neq}\|_{L^2H^2}+\frac{1}{2}\|m^{\frac{1}{4}}\partial_{XX}{\de}_{\neq}\|_{L^2H^2}\|m^{\frac{1}{4}}p^{\frac{1}{2}}\partial_X{\de}_{\neq}\|_{L^2H^2}\lesssim\frac{1}{C_1}(B_1\nu^{-1/3}\varepsilon)^2,\\
		LE^{\nxnh}\lesssim&\|\ma\partial_X(W^2_{\neq},\partial_X{\de}_{\neq})\|_{L^2H^2}\|\ma \partial_XD_{\neq}\|_{L^2H^2}+\nu\|\ma \partial_XD_{\neq}\|^2_{L^2H^2}+\|\ma\partial_X(W^2_{\neq},\partial_X{\de}_{\neq})\|^2_{L^2H^2}\\
		&+\bar{\nu}^{\frac{4}{3}}\|\ma p^{\frac{1}{2}}\partial_X{\de}_{\neq}\|_{L^2H^2}\|p^{\frac{1}{2}}\partial_XD_{\neq}\|_{L^2H^2}+c_1\bar{\nu}^{\frac{1}{3}}\|\ma \partial_XD_{\neq}\|^2_{L^2H^2}\\
		\lesssim&\nu^{-1/6}\varepsilon B_1\nu^{-1/2}\varepsilon+\nu(B_1\nu^{-1/2}\varepsilon)^2+\nu^{-1/3}\varepsilon^2+\bar{\nu}^{\frac{4}{3}}\nu^{-1/3}B_1\varepsilon\nu^{-2/3}B_1\varepsilon+c_1(B_1\nu^{-1/3}\varepsilon)^2\\
		\lesssim&\left(\bar{\nu}^{\frac{2}{3}}+\frac{1}{C_1}\right)(B_1\nu^{-1/3}\varepsilon)^2.
	\end{align*}
	The estimates of $I^{\nxnh}_1$, $I^{\nxnh}_2$, and $I^{\nxnh}_3$ are exact the same to that in \cref{4.4}. Therefore, we omit them and arrive at
	\begin{align*}
		I^{\nxnh}_1+I^{\nxnh}_2+I^{\nxnh}_3\lesssim\nu^{-3/2}\varepsilon(\nu^{-1/3}\varepsilon)^2.
	\end{align*}
	\subsubsection{$H^{3-j}$ estimates on $(p^{\frac{j+1}{2}}R,p^\frac{j}{2}D)$, $\nabla_{X,Z}(p^{\frac{j+1}{2}}R,p^\frac{j}{2}D)$, and $(p^{\frac{j+2}{2}}R,p^{\frac{j+1}{2}}D)$}
	In this subsection, we improve the energy estimates \cref{bs-yn}-\cref{bs-pn} for $j=1,2,3$.
	We only consider $\|p^\frac{3}{2}(p{\de},p^{\frac{1}{2}}D)\|_{L^2}$. The treatment of the other quantities is similar. 
	The energy functional is given by
	\begin{align*}
		E^{\pnh}=\frac{1}{2}\|(p^{\frac{5}{2}}{\de},p^2D)\|^2_{L^2}-c_1\bar{\nu}^{\frac{1}{3}} Re\langle p^{2}{\de}_{\neq},p^{2}D_{\neq}\rangle_{L^2}.
	\end{align*}
	The energy estimate read
	\begin{align*}
		&E^{\pnh}(T)+\bar{\nu}\|p^{\frac{5}{2}}D\|^2_{L^2L^2}+c_1\bar{\nu}^{\frac{1}{3}}\|p^{\frac{5}{2}}{\de}_{\neq}\|^2_{L^2L^2}
		=E^{\pnh}(0)+LS^{\pnh}+LE^{\pnh}+I^{\pnh}_1+I^{\pnh}_2+I^{\pnh}_3+I^{\pnh}_4,
	\end{align*}
	where 
	\begin{align*}
		LS^{\pnh}=&\int_{0}^{T}3\left\langle(\partial_t p) pD_{\neq}, p^2D_{\neq}\right\rangle_{L^2}+\frac{5}{2}\left\langle(\partial_t p)p^{\frac{3}{2}} {\de}_{\neq}, p^{\frac{5}{2}}{\de}_{\neq}\right\rangle_{L^2} dt\\
		LE^{\pnh}=&2\int_{0}^{T}\left\langle  \partial_Xp(W^2-\partial_X{\de}+\nu\partial_XD)_{\neq}, p^{2}D_{\neq}\right\rangle_{L^2}\\
		&+c_1\bar{\nu}^{\frac{1}{3}}\| p^{2} D_{\neq}\|^2_{L^2}-c_1\bar{\nu}^{\frac{1}{3}}\left\langle p^{\frac{5}{2}}{\de}_{\neq},5(\partial_tp)p^{\frac{1}{2}}D_{\neq}+ 2 \partial_Xp^{\frac{1}{2}}(W^2-\partial_X{\de}+\nu\partial_XD)_{\neq}-\bar{\nu} p^{\frac{5}{2}}D_{\neq}\right\rangle_{L^2} dt\\
		I^{\pnh}_1=&\int_{0}^{T}\left\langle p^{\frac{5}{2}}\nlt_{{\de}},p^{\frac{5}{2}}{\de}\right\rangle_{L^2} dt,\quad
		I^{\pnh}_2=\int_{0}^{T}\left\langle p^{2}\nlt_{D},p^{2}D\right\rangle_{L^2} dt,\\
		I^{\pnh}_3=&-c_1\bar{\nu}^{\frac{1}{3}}\int_{0}^{T}\left\langle p^{2} {\de}_{\neq}, p^{2}(\nlt_{D})_{\neq} \right\rangle+\left\langle p^{2}(\nlt_{{\de}})_{\neq},p^{2}D_{\neq}\right\rangle_{L^2} dt.
	\end{align*}
	For $LS$ and $LE$, by bootstrap argument, we have
	\begin{align*}
		LS^{\pnh}\leq&2\|\partial_Xp^{\frac{3}{2}}D_{\neq}\|_{L^2L^2}\|p^{2}D_{\neq}\|_{L^2L^2}+\frac{5}{2}\|\partial_Xp^{2}{\de}_{\neq}\|_{L^2L^2}\|p^{\frac{5}{2}}{\de}_{\neq}\|_{L^2L^2}\lesssim\frac{1}{C_0}(C_0^2B_3\nu^{-11/6}\varepsilon)^2,\\
		LE^{\pnh}\lesssim&\|m^{\frac{1}{4}}p^{\frac{3}{2}}(W^2_{\neq},{\de}_{\neq},\partial_Xp^{-\frac{1}{2}}D_{\neq})\|_{L^2H^1}\|(p^{\frac{5}{2}}{\de}_{\neq},p^{2}D_{\neq})\|_{L^2L^2}+\nu\|\partial_XpD_{\neq}\|_{L^2H^1}\|p^2D_{\neq}\|_{L^2L^2}\\
		&+c_1\bar{\nu}^{\frac{4}{3}}\|p^{\frac{5}{2}}{\de}_{\neq}\|_{L^2L^2}\|p^{\frac{5}{2}}D_{\neq}\|_{L^2L^2}+c_1\bar{\nu}^{\frac{1}{3}}\|p^{2}D_{\neq}\|^2_{L^2L^2}\\
		\lesssim&\left(\bar{\nu}^{\frac{2}{3}}+\frac{1}{C_0}\right)(C_0^2B_3\nu^{-11/6}\varepsilon)^2.
	\end{align*}
	For transport term of $p^{\frac{5}{2}}{\de}$, we use integration by parts to obtain
	\begin{align*}
		&\int_{0}^{T}\left\langle p^{\frac{5}{2}}(V\cdot\tilde{\nabla} {\de}),p^{\frac{5}{2}}{\de}\right\rangle_{L^2} dt\\=&\int_{0}^{T}\left\langle [p^{\frac{5}{2}},V]\cdot\tilde{\nabla} {\de},p^{\frac{5}{2}}{\de}\right\rangle_{L^2} dt+\int_{0}^{T}\left\langle V\cdot\tilde{\nabla} p^{\frac{5}{2}}{\de},p^{\frac{5}{2}}{\de}\right\rangle_{L^2} dt\\
		\lesssim&\|p^{\frac{5}{2}} V\|_{L^2L^2}\|p^{\frac{1}{2}}{\de}\|_{L^\infty H^2}\|p^{\frac{5}{2}}{\de}\|_{L^2L^2}+\|p^{\frac{1}{2}}V^1_0\|_{L^\infty H^2}\|p^{\frac{5}{2}}{\de}_{\neq}\|^2_{L^2L^2}\\
		&+\|p^{\frac{1}{2}}(V^2_0,V^3_0)\|_{L^\infty H^2}\|p^{\frac{5}{2}}{\de}\|^2_{L^2L^2}+\|p^{\frac{1}{2}}V_{\neq}\|_{L^2H^2}\|p^{\frac{5}{2}}{\de}\|_{L^2L^2}\|p^{\frac{5}{2}}{\de}\|_{L^\infty L^2}\\
		&+\|D\|_{L^2H^2}\|p^{\frac{5}{2}}{\de}\|_{L^2L^2}\|p^{\frac{5}{2}}{\de}\|_{L^\infty L^2}\\
		\lesssim&\nu^{-2}\varepsilon\nu^{-1/2}\varepsilon\nu^{-7/3}\varepsilon+\nu^{-1}\varepsilon(\nu^{-2}\varepsilon)^2+\varepsilon(\nu^{-7/3}\varepsilon)^2+\nu^{-2/3}\varepsilon\nu^{-7/3}\varepsilon\nu^{-11/6}\varepsilon\\
		\lesssim&\nu^{-4/3}\varepsilon(\nu^{-11/6}\varepsilon)^2.
	\end{align*}
	The estimates of the other term of $I^{\pnh}_1$, $I^{\pnh}_2$, and $I^{\pnh}_3$ are exact the same to that in Section \ref{sec:48}. Therefore, we omit them and arrive at
	\begin{align*}
		I^{\pnh}_1+I^{\pnh}_2+I^{\pnh}_3\lesssim\nu^{-3/2}\varepsilon(\nu^{-11/6}\varepsilon)^2.
	\end{align*}
	\subsection{Estimates on incompressible part}
	In this subsection, we improve the energy estimates \eqref{bs-w1}-\eqref{bs-o3} for $j=1,2,3$. We only consider $\|p^{\frac{1}{2}}W^2_{\neq}\|_{H^{2}}$. The treatment of the other quantities is similar. 
	The energy estimate gives
	\begin{align*}
		&\|\ma p^{\frac{1}{2}}W^2_{\neq}(T)\|^2_{H^2}+\nu\|\tilde{\nabla}\ma p^{\frac{1}{2}} W^2_{\neq}\|^2_{L^2H^2}+\left\|\sqrt{-\frac{\partial_t M}{M}}\ma p^{\frac{1}{2}} W^2_{\neq}\right\|^2_{L^2H^2}\\
		&+\frac{1}{4}\left\|\sqrt{-\frac{\partial_t m}{m}}\ma  p^{\frac{1}{2}}W^2_{\neq}\right\|^2_{L^2H^2}
		=\|p^{\frac{1}{2}}W^2_{\neq}(0)\|^2_{H^3}+LS^{\nwh}+LE^{\nwh}+\sum_{i=1}^{6}I^{\nwh}_i,
	\end{align*}
	where
	\begin{align*}
		LS^{\nwh}=&\int_{0}^{T}\left\langle \frac{1}{2}\frac{\partial_t p}{p^{\frac{1}{2}}}\ma W^2_{\neq},\ma p^{\frac{1}{2}}W^2_{\neq}\right\rangle_{H^2}dt,\\
		LE^{\nwh}=&-\int_{0}^{T}\left\langle \ma p^{\frac{1}{2}}\left(2\nu\partial_X^2p^{-1}(W^2-\partial_X{\de}+\nu\partial_XD)+\nu\frac{\partial_tp}{p}\partial_XD-\nu(\nu+\nu')p\partial_XD\right)_{\neq},\ma p^{\frac{1}{2}} W^2_{\neq}\right\rangle_{H^2} dt,\\
		I^{\nwh}_1=&-\int_{0}^{T}\left\langle \ma p^{\frac{1}{2}}(V\cdot\tilde{\nabla}W^2)_{\neq},\ma p^{\frac{1}{2}} W^2_{\neq}\right\rangle_{H^2} dt,\\
		I^{\nwh}_2=&\int_{0}^{T}\left\langle \ma p^{\frac{1}{2}}\left(\sum_{i=1,3}(\tilde{\partial}_YV^i\partial_iD-Q^i\partial_iV^2)-\Omega^2\tilde{\partial}_YV^2\right)_{\neq},\ma p^{\frac{1}{2}}W^2_{\neq}\right\rangle_{H^2} dt,\\
		I^{\nwh}_3=&\int_{0}^{T}\left\langle \ma p^{\frac{1}{2}}\left(2\sum_{i=1,3}(\partial_iV^2\tilde{\p}_{Y}^{2} V^i-\tilde{\partial}_iV^j\tilde{\partial}_{ij}V^2)\right)_{\neq},\ma p^{\frac{1}{2}}W^2_{\neq}\right\rangle_{H^2} dt,\\
		I^{\nwh}_4=&\int_{0}^{T}\left\langle \ma p^{\frac{1}{2}}\left(\sum_{i,j=1,3}\tilde{\partial}_Y(\partial_iV^j\partial_jV^i)\right)_{\neq},\ma p^{\frac{1}{2}}W^2_{\neq}\right\rangle_{H^2} dt,\\
		I^{\nwh}_5=&\int_{0}^{T}\left\langle \ma p^{\frac{1}{2}}\left(-\partial_XV\cdot\tilde{\nabla}{\de}-\partial_X({\de}D)+\nu\partial_XV\cdot\tilde{\nabla}D+\nu\partial_X(\tilde{\partial}_iV^j\tilde{\partial}_jV^i)\right)_{\neq},\ma p^{\frac{1}{2}} W^2_{\neq}\right\rangle_{H^2} dt,\\
		I^{\nwh}_6=&\int_{0}^{T}\big\langle \ma p^{\frac{1}{2}}\big(-(\tilde{\Delta}-\tilde{\nabla}\tilde{\textrm{div}})\left(G({\de})(\nu\tilde{\Delta} V+(\nu+\nu')\tilde{\nabla} D)\right)^2\\
		&+\nu\partial_X\tilde{\textrm{div}} \left(F({\de})\tilde{\nabla} {\de}+G({\de})(\nu\tilde{\Delta} V+(\nu+\nu')\tilde{\nabla} D)\right)\big),\ma p^{\frac{1}{2}}W^2_{\neq}\big\rangle_{H^2} dt.
	\end{align*}
	For $LS^{\nwh}$ and $LE^{\nwh}$, the bootstrap argument implies that
	\begin{align*}
		LS^{\nwh}\lesssim& \|\ma W^2_{\neq}\|_{L^2H^3}\|\ma p^{\frac{1}{2}} W^2_{\neq}\|_{L^2H^2}
		\lesssim \nu^{-1/6}\varepsilon\nu^{-1/2}B_1\varepsilon\lesssim\frac{1}{C_1}(B_1\nu^{-1/3}\varepsilon)^2.
	\end{align*}
	and 
	\begin{align*}
		LE^{\nwh}\lesssim& \nu\|\partial_Xp^{-\frac{1}{2}}m^{\frac{1}{4}}p^{\frac{1}{2}}(W^2_{\neq},\partial_X{\de}_{\neq})\|^2_{L^2H^2}+\nu\|\partial_Xp^{-\frac{1}{2}}m^{\frac{1}{4}}p^{\frac{1}{2}}W^2_{\neq}\|_{L^2H^2}\|m^{\frac{1}{4}}\partial_Xp^{\frac{1}{2}}D_{\neq}\|_{L^2H^2}\\
		&+\nu(\nu+\nu')\|pm^{\frac{1}{4}}W^2_{\neq}\|_{L^2H^2}\|p\partial_XD_{\neq}\|_{L^2H^2}\\
		\lesssim&\nu(B_1\nu^{-1/3}\varepsilon)^2+\nu B_1\nu^{-1/3}\varepsilon B_1\nu^{-5/6}\varepsilon+\nu^2B_1\nu^{-5/6}\varepsilon C_0 B_1\nu^{-4/3}\varepsilon\lesssim\nu^{1/2}C_0(B_1\nu^{-1/3}\varepsilon)^2.
	\end{align*}
	The estimates of $I^{\nwh}_1-I^{\nwh}_6$ are exact the same to that in Section \ref{sec:511}. Therefore, we omit them and arrive at
	\begin{align*}
		\sum_{k=1}^{6}I^{\nwh}_k\lesssim\nu^{-3/2}\varepsilon(\nu^{-1/3}\varepsilon)^2.
	\end{align*} 
	This completes the improvement of \eqref{bs-w1}-\eqref{bs-o3}.

	\medskip
	\noindent {\bf Acknowledgment:}\, FW gratefully  acknowledges support from the National Natural Science Foundation of China (No. 12471223, 12101396, and 12331008).
	
	\medskip
	\noindent{\bf Data availability:} The manuscript contains no associated data.

	\medskip
	\noindent{\bf Conflict of Interest:} The authors declare that they have no conflict of interest.

\end{document}